\documentclass[12pt,french,reqno,openright]{theseamsbook}
\usepackage[T1]{fontenc}
\usepackage{babel,indentfirst}
\usepackage{a4}
\usepackage{graphicx}
\usepackage{amsmath,amsfonts,amssymb,amsthm,array}

\makeindex

\newcommand{\gespace}{\vskip 5mm}
\newcommand{\mespace}{\vskip 3mm}
\newcommand{\pespace}{\vskip 2mm}
\newcommand{\leg}{\em}

\newcommand{\iit}{1\hspace{-.36em}1}
\newcommand{\petitiit}{1 \hskip-.27em 1}
\newcommand{\emphase}{\sl}
\newcommand{\gras}{\emph}

\newcommand{\isomto}{\stackrel{\sim}{\longrightarrow}}
\newcommand{\implique}{\Rightarrow}
\newcommand{\ssi}{\Leftrightarrow}
\newcommand{\IMPLIQUE}{\Longrightarrow}
\newcommand{\SSI}{\Longleftrightarrow}
\newcommand{\gal}{{\displaystyle\nearrow}}
\newcommand{\nongal}{{\displaystyle\nwarrow}}
\newcommand{\galtou}{{\displaystyle\nearrow \hskip -4.05mm \swarrow}}
\newcommand{\nongaltou}{{\displaystyle\searrow \hskip -4.05mm \nwarrow}}
\newcommand{\galsimple}{\text{\rotatebox[origin=c]{315}{\dag}}}
\newcommand{\nongalsimple}{\text{\rotatebox[origin=c]{45}{\dag}}}
\newcommand{\simple}{{\displaystyle/ \hskip -2.1mm \circ}}
\newcommand{\nonsimple}{{\displaystyle\backslash \hskip -2.1mm \circ}}

\renewcommand{\andify}{%
\nxandlist{\unskip, }{\unskip{} et~}{\unskip, et~}}

\newtheorem{Th}{Théorème}[section]
\newtheorem*{Th*}{Théorème}
\newtheorem{prop}[Th]{Proposition}
\newtheorem{cor}[Th]{Corollaire}
\newtheorem{fait}[Th]{Fait}
\newtheorem{lem}[Th]{Lemme}
\newtheorem{propdef}[Th]{Proposition \& Définition}
\newtheorem{Thdef}[Th]{Théorème \& Définition}

\theoremstyle{definition}
\newtheorem{defn}[Th]{Définition}

\newtheorem{defnetconv}[Th]{Définition \& Convention}
\newtheorem{rem}[Th]{Remarque}
\newtheorem{rems}[Th]{Remarques}
\newtheorem{ex}[Th]{Exemple}
\newtheorem{exs}[Th]{Exemples}
\newtheorem{notas}[Th]{Notations}
\newtheorem{qus}[Th]{Questions}

\makeatletter
\renewcommand{\@makechapterhead}[1]{%
   \vspace*{-12pt}%
   \begin{center}%
     \normalfont\sffamily\Large\chaptername\  \thechapter%
   \end{center}%
   \vspace{3pt}
   \begin{center}
    \normalfont\large\itshape\bfseries #1
   \end{center}%
   \nobreak\vspace{9pt}
}
\makeatother

\begin{document}
\thispagestyle{empty}
\begin{center}
UNIVERSITÉ DE VALENCIENNES ET DU HAINAUT - CAMBRÉSIS
\gespace\gespace\gespace\pespace
{\Large \bf THÈSE}\\
\gespace
pour obtenir le grade de
\pespace
{\bf DOCTEUR DE L'UNIVERSITÉ DE VALENCIENNES}\\
\mespace
{\bf{\it Discipline : } Mathématiques Pures}\\
présentée et soutenue publiquement par
\gespace
Emmanuel ANDRÉO\\
le 15 juin 2004
\gespace\gespace
\underline{Titre} :
{\it Dissociation des Extensions Algébriques de Corps par \\ les Extensions Galoisiennes ou
Galsimples non Galoisiennes}
\gespace
\begin{tabular}{ll}
Directeur de thèse :\hskip 5mm    &Richard MASSY \hskip 5mm   Professeur, \\
                        &Université de Valenciennes \\
\end{tabular}
\gespace\gespace
{\bf JURY}\\
\gespace
\begin{tabular}{ll}
Président :     &Christian U. JENSEN   \hskip 5mm  Professor, \\
                &University of Copenhagen \\
\\
Rapporteurs : &Jean-François JAULENT  \hskip 5mm Professeur,\\ 
              &Université de Bordeaux I \\
\\
              &John R. SWALLOW    \hskip 5mm Kimbrough Associate Professor,\\ 
              &Davidson College \\
\\
              &Arne LEDET    \hskip 5mm    Assistant Professor, \\
              &Texas University \\
\\
Examinateurs :   &Pierre DÈBES  \hskip 5mm  Professeur, \\
                 &Université de Lille I \\
\\
                 &Manfred HARTL \hskip 5mm  Professeur,\\ 
                 &Université de Valenciennes \\
\end{tabular}
\end{center}


\chapter*{REMERCIEMENTS}
\thispagestyle{empty}
J'exprime d'abord toute ma gratitude à mon unique directeur de recherche, le Professeur
Richard Massy, pour l'extraordinaire attention qu'il m'a procurée ces cinq
années. C'est grâce à ses conseils que je me suis lancé dans cette formidable
aventure. Qu'il soit également remercié pour la primeur qu'il me laisse de son "Théorème
$M$".

\mespace
Les Professeurs Jean-François Jaulent, Arne Ledet et John
Swallow m'ont fait l'honneur d'accepter d'être les rapporteurs de ce
travail. Merci à eux pour le temps qu'ils m'ont consacré et pour leurs
remarques pertinentes.

\mespace
Je rends hommage au Professeur Christian U. Jensen : il participa au cours de
D.E.A. qui m'introduisit aux études doctorales, et je lui dois plusieurs de mes
exemples. Qu'il figure dans mon Jury est pour moi un accomplissement.

\mespace
Les Professeurs Pierre Dèbes et Manfred Hartl ont bien voulu être les
examinateurs des sept chapitres qui suivent. Que la lecture de ceux-ci leur
laisse le meilleur souvenir...

\mespace
Ma reconnaissance va aussi au Lamath, pour les moyens mis à ma disposition et pour
l'atmosphère chaleureuse que j'y ai trouvée.

\mespace
Enfin merci à toute ma famille et à mes amis pour leurs encouragements constants.

\addtocontents{toc}{\gespace}
\chapter*{INTRODUCTION}
\addtocontents{toc}{\pespace}
\thispagestyle{empty}
Le théorème fondamental de l'Arithmétique factorise tout nombre entier en
produit de nombres premiers. En théorie des groupes, le théorème de
Jordan-Hölder dévisse de nombreux groupes par leurs suites normales qui se
raffinent en suites de composition. Le dernier théorème du chapitre 7 final de
cette thèse dissocie toute extension de corps de degré fini par ses "tours
d'élévation" qui se raffinent en "tours de composition".\index{Dissociation}

Le thème central de ce travail est donc celui de la dissociation des extensions
de corps. Nous prouvons que cette dissociation joue pour les extensions finies
un rôle analogue à celui de la factorisation pour les entiers.\\

Détaillons maintenant ce point de vue. Pour les groupes, on connaît les deux
célébres théorèmes suivants :

\pespace
{\bf Théorème de Schreier} Deux suites normales d'un même groupe admettent des
raffinements équivalents.

\pespace
{\bf Théorème de Jordan-Hölder} Soit $G$ un groupe admettant une suite de
composition.\\
\noindent
(1) Toute suite normale stricte de $G$ admet un raffinement qui est une suite de
composition de $G$.\\
\noindent
(2) Deux suites de composition de $G$ sont équivalentes.\\

La problématique, le questionnement de cette thèse est de se demander s'il
existe pour les extensions de corps des analogues à ces deux profonds théorèmes
qui révèlent la structure des groupes.

Et dans l'affirmative, peut-on obtenir ces analogues à l'instar de la théorie
des groupes, c'est à dire de manière intrinsèque, en restant à l'intérieur des
extensions considérées, par opposition à une approche extrinsèque faisant
intervenir leurs clôtures galoisiennes ? Notre démarche se veut en effet
effective, calculatoire, ce qui exclut de procéder via les clôtures
galoisiennes, beaucoup trop grandes voire inconnues en général.

Les extensions non galoisiennes peuvent être considérées comme chaotiques. Est-il
possible "d'approximer" les extensions algébriques par exemple, ou tout au moins
certaines d'entre elles, par les extensions galoisiennes ? Peut-on dissocier ces
extensions par leurs corps intermédiaires de façon à constituer 
une tour qui comporte le plus grand nombre possible de "marches galoisiennes" ?

\mespace
Dans \cite{M-MD}, Massy a introduit la notion de parallélogramme 
galoisien qui généralise celle d'extension galoisienne. Cependant, le
théorème final ne s'énonce qu'en degrés finis. Pour ne pas limiter notre
analogue du théorème de Schreier aux extensions finies, il a fallu, de manière
déterminante, utiliser certaines propriétés des parallélogrammes galoisiens
infinis.
Leur étude est justement l'objet du chapitre 1, où nous présentons une
théorie de Galois infinie
en dimension 2 généralisant aux parallélogrammes de degré quelconque le
théorème de Krull pour les extensions galoisiennes infinies.

\mespace
Dès le chapitre 2, nous sommes amenés à définir précisément ce que sont les
tours de corps, les marches d'une tour de corps, les tours galoisiennes :\\
Soit $L/K$ une extension de corps. Une "tour $(F)$ de $L/K$" est une suite finie
croissante $\{F_i\}_{0 \leq i \leq m}$ de corps intermédiaires entre $K$ et $L$,
telle que $F_0 = K$ et $F_m = L$ :
$$(F) \qquad K = F_0 \leq F_1 \leq \dots \leq F_i \leq F_{i +1} \leq \dots \leq F_m
= L.$$
Nous disons que les extensions $F_{i +1}/F_i \; (i = 0, \dots,
m-1)$ sont les "marches" de la tour $(F)$. Nous appelons "tour galoisienne de
$L/K$"
une tour de $L/K$ dont toutes les marches sont galoisiennes. En reprenant le
symbole de théorie des groupes exprimant le fait d'être "normal dans", nous
écrirons les tours galoisiennes
$$(F) \qquad K = F_0 \unlhd F_1 \unlhd \dots \unlhd F_i \unlhd F_{i +1} \unlhd \dots
\unlhd F_m = L.$$

Le but du chapitre 2 est d'introduire une généralisation de la notion
d'extension galoisienne : celle d'extension galtourable. Il existe des extensions qui 
ne peuvent se dissocier en une tour galoisienne : mis à part les extensions non
galoisiennes de degré premier, c'est le cas par exemple pour ${\mathbb
Q}(\sqrt[6]{2}) / {\mathbb Q}$. Nous appelons "extension galtourable" une
extension qui admet une tour galoisienne. Toute extension galoisienne est
évidemment galtourable, la réciproque étant fausse. La classe des extensions
galtourables contient donc strictement celle des extensions galoisiennes.
Cependant nous parvenons à étendre aux extensions galtourables les
propriétés essentielles de la théorie de Galois générale classique. Par exemple :\\ 

{\it Soient $K/J$ et $L/J$ deux extensions algébriques. Sous la seule condition que
$K$ et $L$ soient contenus dans un même corps, on a l'implication
$$(L/J \;\text{galtourable }) \qquad \IMPLIQUE \qquad (KL/K \;\text{galtourable
})\:.$$}

De même, tout compositum d'extensions galtourables est galtourable. Précisément
:

{\it Quelles que soient les extensions galtourables $K/J$ et $L/J$ dont les sommets
sont contenus dans un même corps, l'extension compositum $KL/J$ est
galtourable.}\\

Quand on empile deux extensions galoisiennes, on obtient une extension
galtourable non nécessairement galoisienne en général. La situation est
différente avec les extensions galtourables. Précisément :\\

{\it Pour toute tour $K \leq L \leq M$, avoir $L/K$ galtourable et $M/L$ galtourable
implique que $M/K$ est galtourable.}\\

\pespace
Mentionnons que les extensions galtourables constituent une "généralisation
maximale" des extensions galoisiennes : une extension admettant une tour
galtourable est encore une extension galtourable.

\mespace
Le chapitre 3, très technique, est nécessaire pour parvenir à des démonstrations rigoureuses dans la
suite. Nous y introduisons, par analogie avec la théorie des groupes, la notion 
de raffinement de tour de corps :\\
Soient $L/K$ une extension algébrique, et
$$(F) \qquad K=F_0 \leq F_1 \leq \dots \leq F_i \leq F_{i+1} \leq \dots \leq F_m=L$$
une tour de $L/K$.\\
\noindent (1) Nous appelons "raffinement de (F)" toute tour
$$(E) \qquad  K=E_0 \leq E_1 \leq \dots \leq E_j \leq E_{j+1} \leq \dots \leq E_n=L$$
de $L/K$ vérifiant les deux conditions suivantes :
$$\begin{array}{lll}
\text{(RAF1)} 	& m\leq n \;.\hfill & \qquad\qquad\qquad\qquad\qquad\qquad\\
\text{(RAF2)}	& \text{Il existe une suite finie d'indices}	&\\
		& 0 \leq j_0 < j_1 < \dots < j_m \leq n	&
\end{array}$$
telle que
$$ \forall i \in \{ 0, \dots, m \} \quad F_i=E_{j_i} \:.$$

\noindent (2) Nous appelons "raffinement propre de $(F)$" tout
raffinement $(E)$ de $(F)$ qui vérifie la condition supplémentaire
$$\text{(RAF3)} \qquad \exists j \in \{1, \dots, n-1 \} \quad \forall i \in \{0, \dots, m\}
\quad E_j \neq F_i \;.\qquad\qquad$$

\noindent (3) Nous disons que $(E)$ est un "raffinement strict" de $(F)$ si et seulement
si c'est une tour stricte.\\
\noindent (4) Nous disons que $(E)$ est un "raffinement trivial" de $(F)$ si et seulement
si c'est un raffinement de $(F)$ non propre, autrement dit qui vérifie comme
condition supplémentaire la négation de (RAF3) précédente, i.e.
$$\text{(RAFT)} \qquad \forall j \in \{1, \dots, n-1\} \quad \exists i \in \{0, \dots,
m\} \quad E_j=F_i \:. \qquad\qquad\qquad\qquad$$
\noindent (5) Nous disons que $(E)$ est un "raffinement galoisien" de $(F)$ si et seulement
si c'est un raffinement de $(F)$ qui vérifie la condition
supplémentaire
$$\text{(RAFG)} \quad \forall j \in \{1, \dots, n-1\} \quad \text{\rm \Large(} \forall i \in \{0, \dots,
m\} \quad E_j\neq F_i \text{\rm \Large)} \implique E_{j-1} \unlhd E_j\;.$$

Nous démontrons en particulier qu'un raffinement galoisien d'une tour
galoisienne est encore une tour galoisienne (d'où la terminologie).

\mespace
Le chapitre 4 énonce les premiers théorèmes de dissociation. Nous définissons
tout d'abord les tours de composition galoisiennes et
l'équivalence de deux tours galoisiennes :\\
Soient $L/K$ une extension galtourable et
$$(F) \qquad K=F_0 \unlhd \dots \unlhd F_i \unlhd \dots \unlhd F_m=L$$
une tour galoisienne de $L/K$.\\
(1) Nous disons que $(F)$ est "une tour de composition galoisienne de $L/K$" si
et seulement si elle est stricte et n'admet aucun raffinement galoisien propre.
\pespace
\noindent
(2) Soit
$$\begin{array}{c}
(E) \qquad K=E_0 \unlhd \dots \unlhd E_j \unlhd \dots
\unlhd E_n=L\\
\end{array}\\$$
une autre tour galoisienne de $L/K$. Nous disons que $(E)$ et $(F)$ sont
"équivalentes", et nous notons $(E) \sim (F)$, si et seulement si elles ont même
nombre de marches : $m=n$, et si, à permutation près, les groupes de Galois de
ces marches sont isomorphes (topologiquement en degrés infinis) :
$$\exists \sigma \in S_m \quad \forall i \in \{1, \dots, m=n \} \quad
Gal(F_i/F_{i-1}) \isomto Gal(E_{\sigma(i)}/E_{\sigma(i)-1}) \;.$$ 

Avec ces définitions, nous énonçons et prouvons les analogues suivants, pour les extensions 
galtourables, des théorèmes de Schreier et de Jordan-Hölder :

\pespace
\begin{Th*}[]  {($1^{\text{er}}$ théorème de dissociation)} \\ 
Si $L/K$ est une extension galtourable, deux tours galoisiennes de $L/K$
admettent des raffinements équivalents.
\end{Th*}

{\noindent \it Scholie.} L'extension $L/K$ peut être ici de degré infini.

\pespace
\begin{Th*}[] {($3^{\text{ème}}$ théorème de dissociation)}\\
Soit $L/K$ une extension galtourable de degré fini.\\
(1) Toute tour stricte de $L/K$ admet un raffinement galoisien qui est une tour de
composition galoisienne de $L/K$.\\
(2) Deux tours de composition galoisiennes de $L/K$ sont équivalentes.
\end{Th*}

Le deuxième théorème de dissociation fournit une caractérisation des
extensions admettant une tour de composition galoisienne. On sait qu'un groupe
admet une suite de composition si et seulement s'il satisfait la condition de
chaîne normale ; c'est en particulier le cas des groupes finis. Voici la version
galoisienne de ce résultat :

\pespace
\begin{Th*}[] {($2^{\text{ème}}$ théorème de dissociation)} \\
Une extension de corps admet une tour de composition galoisienne si et seulement
si elle est galtourable de degré fini.
\end{Th*}

\mespace
Un autre parallèle avec les groupes est fourni par la notion de "galsimplicité".
En terme de dissociation, les extensions galsimples jouent le rôle des groupes
simples en théorie des groupes. Nous appelons "extension galsimple" une
extension $L/K$ non triviale n'admettant aucune extension quotient galoisienne
propre :
$$(L/K \;\text{galsimple}) \qquad \stackrel{\text{Déf.}}{\SSI} \qquad
\left(\begin{array}{c}
	L \neq K ,\quad \forall F \quad K \leq F \leq L\\
	(F/K \;\text{galoisienne}) \implique (F=K \;\text{ou} \; F=L)
\end{array}\right)\:.$$

On sait que pour qu'une suite normale de groupes soit de composition, il faut et
il suffit que chacun de ses facteurs soit simple. Voici la version galoisienne
de ce résultat :\\

{\it Soit $L/K$ une extension galtourable quelconque. Pour qu'une tour
galoisienne de $L/K$ soit de composition, il faut et il suffit que chacune de
ses marches soit galsimple.}\\

Le chapitre 5 est essentiellement constitué d'exemples des notions précédentes. Il comporte également quelques propriétés des extensions 
galsimples qui nous sont utiles dans les deux derniers chapitres.

\mespace
Le coeur du chapitre 6 est le "théorème $M$", dû à Richard Massy. Celui-ci montre 
qu'à toute extension finie est attachée un invariant, son "corps d'intourabilité", 
au-delà duquel l'extension n'est plus galtourable :
\pespace
\begin{Th*}[] {($4^{\text{ème}}$ théorème de dissociation)} \\
Pour toute extension finie $L/K$ , il existe un corps intermédiaire $M$ et un
seul entre $K$ et $L$, vérifiant à la fois les deux propriétés suivantes :\\
\noindent (1) L'extension $M/K$ est galtourable ;\\
\noindent (2) La sous-extension $L/M$ est soit triviale, soit galsimple non galoisienne.
\end{Th*}

Nous mettons en évidence le rôle central du corps d'intourabilité, noté
$M(L/K)$, en prouvant deux maximalités : pour une relation d'ordre canonique,
l'extension $M(L/K)/K$ (resp. $L/M(L/K)$) se réalise comme l'extension quotient
galtourable (resp. la sous-extension galsimple non galoisienne) maximale de
$L/K$. Nous achevons le chapitre 6 en exhibant une large classe d'exemples de ce
corps d'intourabilité en termes de corps cyclotomiques.

\mespace
Le chapitre 7 final est l'aboutissement des précédents.
Nous y introduisons, grâce au corps d'intourabilité, la notion de "tour
d'élévation" associée à  une tour de corps : elle fait correspondre canoniquement
à toute tour de $L/K$ une tour galtourable de l'extension quotient galtourable
maximale $M(L/K)/K$ de $L/K$. Nous avons noté horizontalement $F \unlhd E$ une
extension galoisienne $E/F$ ; notons $F \lessgtr E$ lorsque $E/F$ n'est que
galtourable. Ceci permet d'écrire les tours d'élévation de $M(L/K)/K$ :\\

{\it Soit $L/K$ une extension finie quelconque. Toute tour
$$(F) \qquad K=F_0 \leq F_1 \leq \dots \leq F_i \leq \dots \leq F_m=L$$
de $L/K$ induit une tour galtourable constituée des corps d'intourabilité 
sur $K$ de chacun des corps de $(F)$ :
$$\begin{array}{r}
K=M_0 := M(F_0/K) \lessgtr M_1:= M(F_1/K) \lessgtr \dots \lessgtr
M_i:=M(F_i/K) \lessgtr \dots \\
\dots \lessgtr M_m:=M(F_m/K)=M(L/K) \;.
\end{array}$$}
Nous l'appelons "tour d'élévation de $M(L/K) \galtou K$ associée à $(F)$"\\

Les tours d'élévation de l'extension $L/K$ elle-même s'obtiennent comme "tours
induites" des tours d'élévation de $M(L/K) \galtou K$. Précisément :\\
Soient $M$ un corps d'intermédiaire entre $K$ et $L$ : $K
\leq M \leq L$, et 
$$(E) \qquad K=E_0 \leq E_1 \leq \dots \leq E_m=M$$
une tour de $M/K$. Nous appelons "tour de $L/K$ induite par $(E)$", et nous
notons
$$((E) \dasharrow L) \;,$$
la tour de $L/K$ définie de la façon suivante
$$((E) \dasharrow L) := \left\{\begin{array}{ll}
(E)						&\text{si } M=L\\
K=E_0 \leq E_1 \leq \dots \leq E_m=M < L	&\text{si } M\neq L
\end{array}\right. \;.$$

\pespace
Soient $L/K$ une extension finie et $(F)$ une tour quelconque de $L/K$.
Nous appelons "tour d'élévation de $L/K$ associée à $(F)$" la tour de $L/K$ induite par la tour d'élévation
associée à $(F)$ de l'extension quotient galtourable maximale de $M(L/K) \galtou
K$ de $L/K$.
 
Ceci nous permet de définir des tours de compositions non nécessairement 
galoisiennes :\\
Soit $L/K$ une extension finie quelconque. Nous appelons "tour de composition de
$L/K$" toute tour d'élévation de $L/K$ stricte qui n'admet aucun raffinement
galoisien propre.\\
Une caractérisation de ces tours de composition est la suivante :\\

{\it Soient $L/K$ une extension finie et
$$(C) \qquad K=C_0 \leq \dots \leq C_i \leq \dots \leq C_m=L$$
une tour de $L/K$. On a l'équivalence :\\
$(C)$ est une tour de composition si et seulement si elle est induite par une
tour de composition galoisienne de l'extension quotient galtourable maximale de
$L/K$.}

\pespace
Nous n'avons jusqu'ici défini l'équivalence de deux tours
d'une même extension que lorsque ces tours sont galoisiennes. La définition
ci-après de l'équivalence de deux tours induites implique en particulier celle
de l'équivalence de deux tours de composition non galoisiennes :\\
Soient $L/K$ une extension finie quelconque, $(T)$ et $(T')$ deux tours
galoisiennes de l'extension quotient galtourable maximale $M(L/K) \galtou K$ de
$L/K$. Nous disons que les tours induites de $L/K$ par $(T)$ et $(T')$ sont
équivalentes si et seulement si les tours galoisiennes $(T)$ et $(T')$ le sont
au sens du chapitre 4 :
$$((T) \dasharrow L) \, \sim \, ((T') \dasharrow L) \quad
\stackrel{\text{Déf.}}{\SSI} \quad (T) \sim (T') \;.$$

Ceci pour aboutir enfin à la généralisation aux extensions finies quelconques
des théorèmes de dissociation obtenus au chapitre 4 pour des extensions
galtourables :

\pespace
\begin{Th*}[] ($5^{\text{ème}}$ théorème de dissociation) \\
Deux tours d'élévation d'une même extension finie quelconque admettent des
raffinements galoisiens qui sont des tours d'élévation équivalentes de cette
extension.
\end{Th*}

\pespace
\begin{Th*}[] ($6^{\text{ème}}$ théorème de dissociation) \\
Soit $L/K$ une extension finie quelconque.\\
\noindent (1) Toute tour d'élévation stricte de $L/K$ admet un raffinement galoisien qui
est une tour de composition de $L/K$.\\
\noindent (2) Deux tours de composition de $L/K$ sont équivalentes.
\end{Th*}

\addtocontents{toc}{\gespace\gespace}
\chapter{PARALLÉLOGRAMMES GALOISIENS INFINIS}
\addtocontents{lof}{\gespace\pespace}
\addtocontents{lof}{\noindent Chapitre \thechapter}
\addtocontents{lof}{\pespace}
\addtocontents{toc}{\pespace}

\section{Introduction}
Dans \cite{M-MD}, Massy a introduit la notion de parallélogramme 
galoisien qui généralise celle d'extension galoisienne. Cependant, le
théorème final se limite au degré fini. Une application est fournie dans
\cite{M-MD2}. Le but de ce premier chapitre est d'étendre les
résultats de \cite{M-MD} aux parallélogrammes de degré infini. Nous
mettons en évidence section 5 une théorie générale des parallélogrammes
galoisiens de nature essentiellement algébrique. Les topologies de Krull sur les
groupes de Galois des extensions constituant ces parallélogrammes
n'interviennent que dans la section 6. Nous y présentons une théorie de Galois infinie
en dimension 2 généralisant aux parallélogrammes de degré quelconque le
théorème de Krull pour les extensions galoisiennes infinies. Cette théorie sera
appliquée aux chapitres 3 et 4 pour développer en toute généralité une notion de
"raffinement de tours galoisiennes" jouant pour les extensions galoisiennes un rôle
analogue à celui du raffinement des suites normales de groupes.

\vskip 10mm

\section{Définitions}
Plusieurs des démonstrations de \cite{M-MD} ne nécessitent pas  que les
sous-groupes considérés soient normaux, autrement dit que les extensions
quotients (cf. infra) soient galoisiennes : voir en particulier la section 5
pour de nouveaux résultats n'utilisant pas la normalité. Ceci justifie que l'on
introduise la notion de quadrilatère corporel suivante qui généralise celle de
parallélogramme galoisien.

\vskip 2mm

\begin{defn}
Nous appelons "quadrilatère corporel" \index{Quadrilatère}(ou "quadrilatère" en abrégé) tout
quadruplet de corps $(J,K,N,L)$ dans lequel :
\begin{align*}
& (Q_{0}) \quad  K \text{ et } L \text{ sont contenus dans un même corps} \:; \\
& (Q_{1}) \quad K \cap L = J \:; \\
& (Q_{2}) \quad KL = N \:.
\end{align*}

\noindent Le quadrilatère  ${}^t (J,K,N,L)=(J,L,N,K)$ sera dit "transposé"
\index{Quadrilatère!transposé}de $(J,K,N,L)$.

\newpage

\begin{figure}[!h]
\begin{center}
\vskip -3mm
\includegraphics[width=6cm]{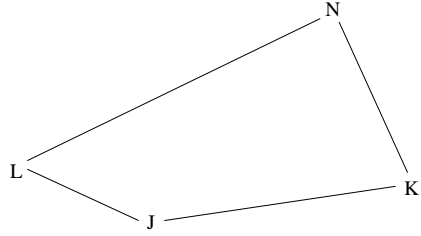}
\end{center}
\vskip -5mm
\rm{\caption{\label{fig:1}\leg Quadrilatère corporel}}
\vskip -4mm
\end{figure}
\end{defn}

\vskip 3mm
La définition suivante figure dans la section 3 de \cite{M-MD} pour des
extensions galoisiennes qui justifient la terminologie.

\begin{defn} 
(1) Soit $E/F$ une extension algébrique. Nous appelons "sous-extension" (resp.
"extension quotient") \index{Sous-extension}\index{Extension!quotient}de $E/F$ toute extension $E/M$ (resp. $M/F$) où $M$ est un corps
intermédiaire : $F \subseteq M \subseteq E$.

(2) Soit $(J,K,N,L)$ un quadrilatère corporel. Nous appelons "sous-quadrilatère"
(resp. "quadrilatère quotient")
\index{Sous-quadrilatère}\index{Quadrilatère!quotient}de $(J,K,N,L)$ tout quadrilatère $(M,E,N,F)$ (resp.
$(J,E,C,F)$ ) où $E$ et $F$ sont deux corps intermédiaires :

$K \subseteq E \subseteq N \; , \; L \subseteq F \subseteq N \qquad ${\Large
(}resp. $J \subseteq E \subseteq K \; , \; J \subseteq F \subseteq L${\Large )}.
\end{defn}

\begin{figure}[!h]
\begin{center}
\vskip -4mm
\includegraphics[width=8cm]{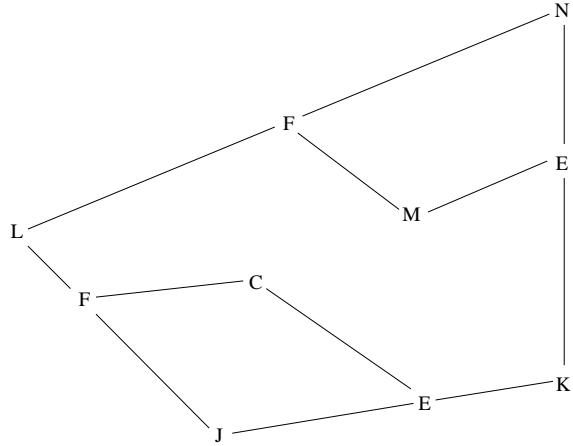}
\end{center}
\vskip -6mm
\rm{\caption{\label{fig:2}\leg Sous-quadrilatère \& quadrilatère quotient}}
\vskip -1mm
\end{figure}
Nous aurons besoin du lemme immédiat suivant pour les monotonies des bijections
de la section 5.

\begin{lem} \label{lem:relordre}
(1) Soit $E/F$ une extension algébrique. Dans l'ensemble des sous-extensions (resp.
des extensions quotients) de $E/F$, la relation définie par

$ (E/M') \leq (E/M) \; \ssi \; M \subseteq M'\quad$ {\Large (}resp. $ (M/F)
\leq  (M'/F) \; \ssi \; M\subseteq M'$ {\Large )} est une relation d'ordre.

(2) Soit $(J,K,N,L)$ un quadrilatère corporel. Dans l'ensemble des
sous-quadrila\-tères
(resp. des quadrilatères quotients) de $(J,K,N,L)$, la relation
définie par
\begin{align*} 
(M',E',N,F') \: \leq \: (M,E,N,F) \quad &\ssi \quad (E \subseteq E'\; , \;F \subseteq F') \\
\text{{\Large (}resp. } (J,E,C,F) \: \leq \: (J,E',C',F') \quad &\ssi \quad (E \subseteq E'\; ,
\; F \subseteq F') \text{{\Large )}}
\end{align*}
est une relation d'ordre.
\end{lem}

\vskip 2mm
\begin{defn}
Nous appelons "parallélogramme galoisien" \index{Parallélogramme galoisien}(ou "parallélogram\-me"  en abrégé) un
quadrilatère corporel $(J,K,N,L)$ dans lequel toutes les arêtes $K/J$, $N/K$,
$N/L$, $L/J$
sont des extensions galoisiennes. Nous le notons alors $[J,K,N,L]$.
\end{defn}

Clairement, pour qu'un quadrilatère $(J,K,N,L)$ soit un parallélogramme, il faut
et il suffit que les extensions $K/J$ et $L/J$ soient galoisiennes.
En convenant de figurer par des segments parallèles de longueurs égales les
extensions dont les groupes de Galois sont isomorphes, on obtient une figure du
type (où les flêches symbolisent les extensions galoisiennes)

\begin{figure}[!h]
\begin{center}
\vskip -4mm
\includegraphics[width=6cm]{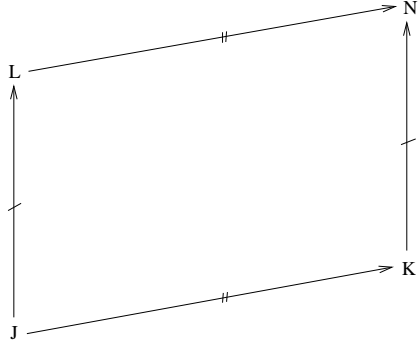}
\end{center}
\vskip -7mm
\rm{\caption{\label{fig:3}\leg Parallélogramme galoisien}}
\vskip -2mm
\end{figure}

\noindent Par composition, la "diagonale $N/J$" \index{Diagonale d'un parallélogramme
galoisien}d'un parallélogramme $[J,K,N,L]$
est nécessairement galoisienne. Dans \cite{M-MD} , "le degré" \index{Degré!d'un parallélogramme
galoisien}d'un
parallélogramme galoisien $[J,K,N,L]$ est défini comme étant le couple des
degrés :
$$ deg[J,K,N,L] := \text{{\Large (}} \,[N:K] \:,\: [K:J] \, \text{{\Large )}}.$$

\noindent Ce degré sera dit "infini" quand l'un des degrés $[N:K]$ ou $[K:J]$ est infini.
Lorsque $K=J$ ou $L=J$, nous disons que le quadrilatère $(J,K,N,L)$ est "plat". Tout
corps s'identifie au quadrilatère plat $(F,F,F,F)$. Toute extension (resp. toute
extension galoisienne) $E/F$ s'identifie au quadrilatère plat (resp. au
parallélogramme plat) $(F,E,E,F)$ (resp. $[F,E,E,F]$) ou à son transposé.
\vskip 5mm

Pour définir enfin le groupe de Galois d'un parallélogramme\index{Groupe de
Galois d'un parallélogramme \\ galoisien}, plaçons-nous dans
la catégorie produit ${\mathbf{Gr^2}} = {\mathbf{Gr}} \times {\mathbf{Gr}}$ de la
catégorie des groupes par elle-même. Nous appelons "bigroupe" un objet de
$\mathbf{Gr^2} $, c'est à dire un couple de groupes. Autrement dit, un
bigroupe est un (objet en) groupe(s) dans la catégorie produit
$\mathbf{Ens^2}$ de la catégorie des ensembles par elle-même.

\vskip 4.5mm
\begin{defn} \cite[Déf.3.6]{M-MD} \label{def:grgal}
Soit $[J,K,N,L]$ un parallélogramme galoisien. Nous appelons "groupe de Galois de
$[J,K,N,L]$ ", et nous notons $Gal[J,K,N,L]$ le bigroupe :
$$ Gal[J,K,N,L] := (Gal(N/K),Gal(N/L)). $$
\end{defn}

Les notions de sous-groupe, de sous-groupe normal et de groupe quotient de
$Gal[J,K,N,L]$ se définissent de manière évidente à partir des objets
correspondants de $\mathbf{Gr^2} $ (cf. \cite[Sect.3]{M-MD} ). Munis des
topologies convenables, nous y reviendrons section 6.


\vskip 8mm
\section{Propriétés topologiques}
Cette section 3 est préparatoire ; elle regroupe des résultats indépendants
nécessaires aux raisonnements des sections 4 à 6.

Soit $E/F$ une extension galoisienne de degré infini. Munissons le groupe de
Galois $G:=Gal(E/F)$ de sa topologie de Krull (cf. \cite{Kr} ou
\cite[p.340]{Ka}). Par les propriétés générales des
groupes topologiques, on sait que pour tout sous-groupe $H$ de $G$, l'adhérence
$\overline{H}$ de $H$ est un sous-groupe de $G$ \cite[TG III.7]{Bo}. C'est le
groupe de Galois de $E$ sur le corps des invariants de $H$ dans $E$ 
\cite[p.344]{Ka}  i.e.
$$\begin{array}{cr}
\qquad \qquad \qquad \qquad \qquad \qquad\overline{H}=Gal(E/E^{H}).\qquad \qquad
\qquad \qquad \qquad \qquad 	&(0)
\end{array}$$

\vskip 3mm
\begin{prop} \label{prop:adh} 
(1) La normalité d'un sous-groupe de $G$ dans un autre implique la normalité de leurs
adhérences :
$$ \forall A \leq G \quad \forall B \leq G \quad \qquad A \unlhd B \quad \implique \quad \overline{A} \unlhd
\overline{B}.$$

(2) Pour tout sous-groupe $H$ de $G$, le corps des invariants dans E du sous-groupe H et de son
adhérence $\overline{H}$ sont égaux :
$$E^H=E^{\overline{H}} \:.$$
\end{prop}

\begin{proof} 
(1) Fixons-nous $\alpha \in A$, et considérons l'application
\begin{align*}
f_{\alpha}\: : \: &G \longrightarrow G \\
&\gamma \longmapsto \alpha ^{\gamma}=\gamma^{-1}\alpha\,\gamma
\end{align*}
Par la normalité de $A$ dans $B$, on a clairement $f_{\alpha}(B) \subseteq A$, et comme
$f_{\alpha}$ est continue \cite[TGI.9]{Bo}
$$f_{\alpha}(\overline{B}) \subseteq \overline{f_{\alpha}(B)} \subseteq \overline{A}.$$
Fixons-nous ensuite un $\beta \in \overline{B}$. Pour l'automorphisme intérieur
\begin{align*}
g_{\beta}\: : \: &G \longrightarrow G \\
&\gamma \longmapsto \gamma^{\beta},
\end{align*}
on a
$$\forall \alpha \in A \quad g_{\beta}(\alpha)=f_{\alpha}(\beta) \in \overline{A},$$
de sorte que
$$\forall \beta \in \overline{B} \quad g_{\beta}(A) \subseteq \overline{A}.$$
Par la continuité de $g_{\beta}$, on en déduit que
$$g_{\beta}(\overline{A}) \subseteq \overline{g_{\beta}(A)} \subseteq \overline{A}.$$
Donc
$$\forall \beta \in \overline{B} \quad \forall \alpha \in \overline{A} \quad
g_{\beta}(\alpha)=\alpha^{\beta} \in \overline{A},$$
ce qui exprime précisément que $\overline{A} \unlhd \overline{B}$.

\noindent (2) D'après le (0) ci-dessus et le fait que la sous-extension $E/E^H$ soit
galoisienne, on a directement \vskip -3mm
$$E^{\overline{H}}=E^{Gal(E/E^H)}=E^H \:.$$ \vskip -5mm
\end{proof}

\vskip 3mm
\begin{prop} \label{prop:topind}
Soit $N/K$ une extension galoisienne. Pour tout corps intermédiaire $E$, $K
\subseteq E \subseteq N$, la topologie de Krull de $Gal(N/E)$ est égale à la topologie induite sur
$Gal(N/E)$ par la topologie de Krull de $Gal(N/K)$.
\end{prop}

\begin{proof}
Posons $\Gamma := Gal(N/K)$, $A :=Gal(N/E)$ et soit $\alpha$ un élément quelconque
de $A$.
Prouvons d'abord que tout voisinage $V$ de $\alpha$ pour la topologie de Krull
de $A$ est aussi un voisinage de $\alpha$ pour la topologie induite sur $A$ par
celle\- de $\Gamma$. Par définition de la topologie de Krull de $A$, il existe une
extension galoisienne finie $M/E$, avec $M \subseteq N$, telle que $V \supseteq
\alpha \,Gal(N/M)$.

\begin{figure}[!h]
\begin{center}
\vskip -5mm
\includegraphics[width=5.5cm]{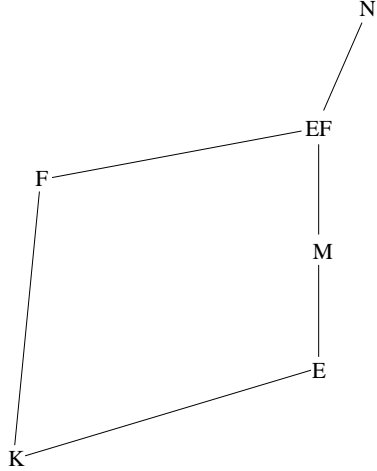}
\end{center}
\vskip -8mm
\rm{\caption{\label{fig:4} \leg Topologie de Krull induite}}
\vskip -2mm
\end{figure}

\noindent En vertu du théorème de l'élément primitif, il existe d'autre part
un élément $x \in M$ tel que $M=E(x)$. Soient $P(X)=Irr(x,K,X)$ le polynôme
minimal de $x$ sur $K$, et $R:=\{x{\scriptstyle =}x_1\,,\,\dots\,,\,x_n\}$ l'ensemble des racines de $P(X)$
dans une clôture algébrique fixée de $K$. Puisque l'extension $N/K$ est normale, on 
a $R \subseteq N$. Considérons alors le corps intermédiaire $F := K(R)$. C'est un
corps de décomposition, donc l'extension $F/K$ est galoisienne finie, et comme $x=x_1
\in R $, on a clairement $M=E(x) \subseteq EF$. On en déduit que
$$Gal(N/M) \geq Gal(N/EF)=Gal(N/E) \cap Gal(N/F)=A \cap Gal(N/F) \:,$$
d'où
$$V \supseteq \alpha \,Gal(N/M) \supseteq \alpha \, (A \cap Gal(N/F)) = \alpha A \cap
\alpha \, Gal(N/F) \:.$$
Ainsi $V \supseteq A \cap \alpha \, Gal(N/F)$ car $\alpha \in A$. Or $\alpha \, Gal(N/F)$
est un voisinage de $\alpha$ pour la topologie de Krull de $\Gamma$ puisque $F/K$
est galoisienne finie. Ceci prouve que $V$ est bien un voisinage de $\alpha$ pour la
topologie induite sur A par celle de $\Gamma$.

La réciproque reprend certains des arguments précédents dans l'ordre inverse.
Soit $V=U \cap A$ un voisinage de $\alpha$ dans lequel $U$ est un voisinage de
$\alpha$ pour la topologie de Krull de $\Gamma$. Il existe une extension galoisienne
finie $F/K \,,\:F \subseteq N$, telle que $U \supseteq \alpha \, Gal(N/F)$. Ainsi
$$V \supseteq \alpha \, Gal(N/F) \cap A = \alpha \, Gal(N/F) \cap \alpha 
A=\alpha \, (Gal(N/F) \cap A).$$
Or
$$Gal(N/F) \cap A=Gal(N/F) \cap Gal(N/E)=Gal(N/EF),$$
d'où $V \supseteq \alpha \, Gal(N/EF)$. Comme l'extension $EF/E$ est galoisienne
finie par translation de $F/K$ par $E/K$, on a bien prouvé que $V$ est un voisinage
de $\alpha$ pour la topologie de Krull de $A$.
\end{proof}

\begin{lem} \label{lem:topquo}
Soient $N/K$ une extension galoisienne et $E/K,\; E \subseteq N$, une extension
galoisienne quotient de $N/K$. Soit $\rho_E$ l'homomorphisme de restriction de
$Gal(N/K)$ sur $Gal(E/K)$. Alors, pour toute extension galoisienne quotient $F/K$ de
$N/K$,

\begin{figure}[!h]
\begin{center}
\vskip -5mm
\includegraphics[width=7cm]{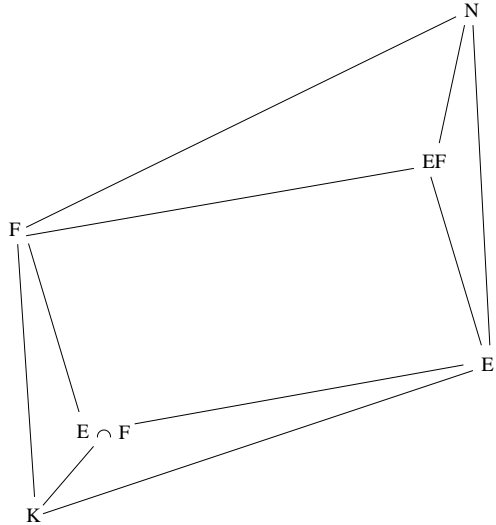}
\end{center}
\vskip -7mm
\rm{\caption{\label{fig:5}\leg Parallélogramme galoisien inscrit}}
\vskip -5mm
\end{figure}

\newpage
\noindent on a
$$ \rho_E(Gal(N/F))=Gal(E/E \cap F).$$
\end{lem}

\begin{proof}
En translatant l'extension galoisienne $E/K$ par $F/K$, on obtient l'extension
galoisienne $EF/F$ et un isomorphisme de restriction $\mid_E$ du groupe
$Gal(EF/F)$ sur $Gal(E/E\cap F)$. Soit donc $r \in Gal(E/E \cap F)$ ; il existe $s \in Gal(EF/F)$
tel que $s_{\mid_E}=r$. Et en considérant $EF/F$ comme une extension galoisienne
quotient de $N/F$, il existe $t\in Gal(N/F)$ tel que $t_{\mid_{EF}}=s$. Alors, pour tout $x
\in E$,
$$(\rho_E(t))(x)=t_{\mid_{EF}} (x)=s_{\mid_E} (x)=r(x)\;,$$
de sorte que $\rho_E(t)=r$, et $Gal(E/E \cap F) \subseteq \rho_E(Gal(N/F))$. L'autre
inclusion est claire.
\end{proof}

\vskip 5mm
Le lemme précédent nous permet d'énoncer un analogue de la proposition
\ref{prop:topind} en remplaçant topologie induite par topologie quotient.

\begin{prop}
Soit $N/K$ une extension galoisienne. Pour tout corps intermédiaire $E,\,
K\subseteq E\subseteq N$, tel que l'extension $E/K$ soit galoisienne, la topologie
de Krull de $Gal(E/K)$ est égale à la topologie quotient sur $Gal(E/K)$ de la
topologie de Krull de $Gal(N/K)$.
\end{prop}

\begin{proof}
Soit 
\begin{align*}
\rho_E\: :\quad \Gamma :=Gal(N/K) & \longrightarrow G:=Gal(E/K) \\
\gamma &\longmapsto \gamma_{\mid_E}
\end{align*}
l'homomorphisme de restriction à $E$. Considérons d'abord un voisinage $V$ d'un
élément $\gamma_{\mid_E}=\rho_{E}(\gamma)\,,\,\gamma\in\Gamma$, de $G$ muni de sa
topologie de Krull. Par définition de celle-ci, il existe une extension galoisienne
finie $E_{\gamma}/K\,,\:E_{\gamma}\subseteq E$, telle que $V\supseteq
\gamma_{\mid_E}\,Gal(E/E_{\gamma})$. D'après le lemme \ref{lem:topquo},
$$\rho_E(\gamma\,Gal(N/E_{\gamma}))=\gamma_{\mid_E}\,Gal(E/E \cap
E_{\gamma})=\gamma_{\mid_E}\,Gal(E/E_{\gamma}).$$  
Il s'ensuit que
$$\rho_E^{-1} (V) \supseteq 
\rho_E^{-1} (\gamma_{\mid_E} \, Gal(E/E_{\gamma}))=\rho_E^{-1} 
(\rho_E(\gamma\,Gal(N/E_{\gamma})))\supseteq \gamma\,Gal(N/E_{\gamma}).$$
Comme $\gamma\,Gal(N/E_{\gamma})$ est un ouvert pour la topologie de Krull
de $\Gamma$, on obtient que $\rho_E^{-1}(V)$ est un voisinage de $\gamma$ pour cette
topologie. Autrement dit, la topologie de Krull sur $G$ rend $\rho_E$ continue. Or
la topologie quotient sur $G$ est par définition la plus fine de celles rendant la
surjection $\rho_E$ continue. Ceci signifie que toute partie ouverte pour la
topologie de Krull de $G$ est ouverte pour la topologie quotient sur $G$
\cite[TGI.11]{Bo}.

Considérons maintenant un ouvert $\Theta$ pour la topologie quotient sur $G$ de la
topologie de Krull de $\Gamma$. Par la continuité de $\rho_E$ pour cette topologie,
$\Omega:=\rho_E^{-1}(\Theta)$ est un ouvert de $\Gamma$, donc voisinage de chacun
de ses points. Par définition de la topologie de Krull sur $\Gamma$, il existe, pour
tout $\omega \in \Omega$, une extension galoisienne finie
$E_{\omega}/K\,,\: E_{\omega} \subseteq N$, telle que
$$\omega \, Gal(N/E_{\omega}) \subseteq \Omega \:.$$
Donc clairement
$$ \Omega = \bigcup_{\omega \in \Omega} \omega \,Gal(N/E_{\omega}) \: .$$
Et comme $\rho_E$ est surjective
$$\Theta=\rho_E(\rho_E^{-1} (\Theta))=\rho_E(\Omega)=\bigcup_{\omega \in \Omega}
\omega_{\mid_E} \, \rho_E( \, Gal(N/E_{\omega}) \, ) \: .$$
En appliquant le lemme \ref{lem:topquo} à l'extension galoisienne quotient
$E_{\omega}/K$ de $N/K$, on obtient donc que
$$\Theta=\bigcup_{\omega \in \Omega} \omega_{\mid_E} \,Gal(E/E
\cap E_{\omega}) \: .$$
Or les extensions $E\cap E_{\omega}/K$ sont galoisiennes comme intersections
d'extensions galoisiennes de $K$, et elles sont finies comme chaque
$E_{\omega}/K$.
Par conséquent les $\omega_{\mid_E} \,Gal(E/E \cap E_{\omega})\quad (\omega \in
\Omega)$ sont des ouverts pour la topologie de Krull de $G$, et il en est de même
de $\Theta$ par union. On a donc montré que toute partie ouverte pour la
topologie quotient sur $G$ est ouverte pour la topologie de Krull de $G$. D'où
la conclusion.
\end{proof}


\vskip 8mm
\section{Propriétés générales}

La propriété suivante de décomposition en produit direct du groupe de Galois de
la diagonale d'un parallélogramme intervient fréquemment dans les démonstrations
des sections 4 à 6.

\begin{prop}{\rm (dite de "scindement de la diagonale")} \label{prop:scind} \\
Pour tout parallélogram\-me galoisien $[J,K,N,L]$, le groupe de Galois de la
diagonale $N/J$ se décompose en produit direct sous la forme
$$Gal(N/J)=Gal(N/K)\, \times \, Gal(N/L).$$
\end{prop}

\begin{proof}
Posons pour abréger
$$\Delta:=Gal(N/J),\quad \Gamma:=Gal(N/K), \quad \Lambda:=Gal(N/L).$$
On a $\Gamma \cap \Lambda= 1$ car tout élément de l'intersection doit laisser fixe le
compositum $KL=N$. De plus, les extensions $K/J$ et $L/J$ étant galoisiennes,
les sous-groupes $\Gamma$ et $\Lambda$ sont normaux dans $\Delta$. On en déduit
que leurs éléments commutent ; d'où l'existence du produit direct $\Gamma \times
\Lambda$. Reste à prouver que $\Delta=\Gamma\,\Lambda$. Soit $\delta \in\Delta$.
Dans le parallélogramme $[J,K,N,L]$, la restriction à $K$ 
\begin{align*}
\rho_K\: : \: \Lambda &\isomto Gal(K/J) \\
\lambda &\longmapsto \lambda_{\mid_K}
\end{align*}
est un isomorphisme de groupes. Considérons l'antécédent
$\lambda:=\rho_K^{-1}(\delta_{\mid_K})$, et posons $\gamma:=\delta\lambda^{-1}$.
Comme $\lambda_{\mid_K} = \delta_{\mid_K}$, on a pour tout élément $k\in K$
$$\gamma (k)=\delta ((\lambda^{-1})_{\mid_K}(k))=\delta_{\mid_K} 
((\lambda_{\mid_K})^{-1} (k))=\lambda_{\mid_K} \circ (\lambda_{\mid_K})^{-1}
 (k)=k,$$
ce qui prouve que $\gamma\in\Gamma$. Donc
$\delta=\gamma\lambda\in\Gamma\Lambda$, ce que l'on voulait.
\end{proof}

\vskip 5mm
Nous généralisons maintenant la partie I du théorème 1.3 de \cite{M-MD}.
Il est surprenant de constater qu'aucune propriété topologique (de fermeture)
n'est exigée sur les groupes considérés.

\begin{Th} \label{th:propfonct}
Soit $[J,K,N,L]$ un parallélogramme galoisien de degré quelconque.\\ 
(1) Pour tout sous-groupe $A$ de $Gal(L/J)$ :

(1-1) On a le sous-parallélogramme $[L^A,KL^A,N,L]$ où $L^A$ désigne le corps des
invariants dans $L$ de $A$.

(1-2) Si de plus $A$ est normal dans $Gal(L/J)$, on a le parallélogramme quotient
$[J,K,KL^A,L^A]$.  

\noindent (2) Pour tout sous-groupe $A$ de $Gal(N/K)$ :

(2-1) On a le sous-parallélogramme $[L^{(A_{\mid_L})},N^A,N,L]$ où $L^{(A_{\mid_L})}$ désigne
le corps des invariants dans $L$ de l'image de $A$ par la restriction à $L$.

(2-2) Si de plus $A$ est normal dans $Gal(N/K)$, on a le parallélogramme quotient
$[J,K,N^A,L^{(A_{\mid_L})}]$. 

\noindent (3) Pour tous sous-groupes $A_0$ et $A_1$ de $Gal(L/J)$ (resp. $Gal(N/K)$),
avoir $A_1$ normal dans $A_0$: $A_1 \unlhd A_0$, implique que l'on ait le
parallélogramme
$$[L^{A_0},KL^{A_0},KL^{A_1},L^{A_1}] \qquad \text{{\Large (}resp. } [L^{(A_{0\mid L})},N^{A_0},N^{A_1},
L^{(A_{1\mid L})}] \text{ {\Large )}}.$$
\end{Th}

\begin{proof} 
(1) Posons pour simplifier $F:=L^A$.
 
(1-1) La donnée du parallélogramme
$[J,K,N,L]$ induit l'isomorphisme de restriction à $K$
\begin{align*}
\rho_K\: : \: Gal(N/L) &\isomto Gal(K/J) \: .\\
\lambda &\longmapsto \lambda_{\mid_K}
\end{align*}
Clairement, $K \cap L =J$ implique $K \cap F=J$. En translatant l'extension
galoisienne $K/J$ par $F/J$, on obtient l'extension galoisienne $KF/F$ et
l'isomorphisme de restriction à $K$
$$ r_K\: : \: Gal(KF/F) \isomto Gal(K/J). $$
\cite[p.266,Th.1.12]{La}. Translatons ensuite l'extension galoisienne $KF/F$
par $L/F$ : comme $KFL=N$, on obtient l'isomorphisme de restriction à $KF$
$$\rho_{KF}\: : \: Gal(N/L) \isomto Gal(KF/KF \cap L) \leq Gal(KF/F).$$
Il est clair que $\rho_K=r_K \circ \, \rho_{KF}$. Soit alors $\gamma\in Gal(KF/F)$.
Il existe $\lambda \in Gal(N/L)$ tel que $r_K(\gamma)=\rho_K(\lambda)$, d'où
$$r_K(\gamma)=r_K(\rho_{KF} (\lambda)) \quad \ssi \quad
\gamma=\rho_{KF} (\lambda)$$
par injectivité de $r_K$. On en déduit que $\gamma$ appartient à $Gal(KF/KF \cap L)$ et
l'égalité $Gal(KF/F)=Gal(KF/KF \cap L)$. Ainsi :
$$F=(KF)^{Gal(KF/F)} = (KF)^{Gal(KF/KF \cap L)} =KF \cap L \:; $$
d'où le parallélogramme $[F,KF,N,L]=[L^A,KL^A,N,L]$.

\vskip 5mm
(1-2) D'après la proposition \ref{prop:adh}(1), la normalité de $A$ dans
$Gal(L/J)$ implique celle de son adhérence pour la topologie de Krull de
$Gal(L/J)$ :
$$\overline{A}=Gal(L/L^A) \unlhd Gal(L/J).$$
On en déduit que l'extension $F=L^A/J$ est galoisienne. En la translatant par
$K/J$, on obtient l'extension galoisienne $KF/K$. Comme $KF/F$ est galoisienne
par le (1-1), ceci suffit à prouver l'existence du parallélogramme $[J,K,KF,F]$.

\vskip 5mm
\noindent (2) (2-1) Par la proposition \ref{prop:scind}, $Gal(N/J)=Gal(N/K) \times
Gal(N/L)$. En particulier :
$$\forall \delta \in Gal(N/L^{(A_{\mid_L})}) \quad \exists! \kappa \in Gal(N/K) \quad
\exists! \lambda \in Gal(N/L) \qquad \delta=\kappa \lambda,$$
et par restriction à $L$ dans $Gal(N/J) \, : \,
\delta_{\mid_L}=\kappa_{\mid_L}$. Restreinte au sous-groupe $Gal(N/L^{(A_{\mid_L})})$,
cette restriction à $L$ est à valeurs dans $Gal(L/L^{(A_{\mid_L})})$, de sorte que
$\kappa_{\mid_L} \in Gal(L/L^{(A_{\mid_L})})$. De plus, en vertu du parallélogramme
$[J,K,N,L]$, on a l'homéomorphisme de groupes profinis munis de leurs topologies
de Krull :
\begin{align*}
\rho_L\: : \: Gal(N/K) &\isomto Gal(L/J) \;. \\
\gamma &\longmapsto \gamma_{\mid_L}
\end{align*}
Or, par un homéomorphisme, l'adhérence de l'image d'une partie est égale à
l'image de l'adhérence de cette partie. D'où
$$\overline{\rho_L(A)}=\rho_L(\overline{A})\:.$$
Ainsi, par le (0) de la section 3,
$$Gal(L/L^{(A_{\mid_L})})=Gal(L/L^{\rho_L(A)})=\rho_L(\overline{A})=\rho_L(Gal(N/N^A)).$$
Donc $\kappa_{\mid_L} =\rho_L(\kappa) \in \rho_L(Gal(N/N^A))$, et par
injectivité, $\kappa$ appartient nécessairement à $Gal(N/N^A)$. Ceci prouve que
$Gal(N/L^{(A_{\mid_L})})$ est inclus dans le produit direct $Gal(N/N^A) \times
Gal(N/L)$. Comme par ailleurs $Gal(N/N^A)$ et $Gal(N/L)$ sont inclus dans
$Gal(N/L^{(A_{\mid_L})})$, on a l'égalité
$$Gal(N/L^{(A_{\mid_L})})=Gal(N/N^A) \times Gal(N/L).$$
On en déduit que
\begin{align*} 
N^A \cap L &= N^{Gal(N/N^A)} \cap N^{Gal(N/L)} \\ 
&=N^{< Gal(N/N^A),Gal(N/L) >} \\ &=N^{Gal(N/L^{(A_{\mid_L})})}=L^{(A_{\mid_L})}.
\end{align*}
De plus, $N=KL\subseteq N^A L \subseteq N$ d'où $N=N^A L$, et l'on a bien le
parallélogramme $[L^{(A_{\mid_L})},N^A,N,L]$.

(2-2) Posons pour simplifier $F:=L^{(A_{\mid_L})}$ et $E:=N^A$. En vertu du (2-1)
précèdent, on a le parallélogramme $[F,E,N,L]$ et l'extension $E/F$ est
galoisienne. Dans les notations de la démonstration de ce même (2-1), on a
$$Gal(L/F)=\rho_L(Gal(N/E)) \:.$$
Or par la proposition \ref{prop:adh}(1), la
normalité de $A$ dans $Gal(N/K)$ implique celle de $\overline{A}=Gal(N/E)$ qui
se transmet par l'isomorphisme $\rho_L$ à $Gal(L/F)$, de sorte que $F/J$ est une
extension galoisienne. Clairement d'autre part, $K \cap F=J$ (car $K \cap L=J$)
et $KF \subseteq E$. Il reste à prouver que cette dernière inclusion est une
égalité. Appliquons pour cela la proposition \ref{prop:scind} dans les
parallélogrammes $[F,E,N,L]$ et $[F,KF,N,L]$ (cf. (1-1)). On a :
$$Gal(N/F)=Gal(N/E) \times Gal(N/L) = Gal(N/KF) \times Gal(N/L).$$
Pour tout $\kappa \in Gal(N/KF)$, il existe donc $\alpha \in Gal(N/E)$ et
$\lambda \in Gal(N/L)$ tels que $\kappa \, id_N=\alpha \lambda$. Or, de $KF
\subseteq E$ suit $Gal(N/E) \subseteq Gal(N/KF)$ d'où $\alpha \in Gal(N/KF)$.
Par unicité des décompositions dans un produit direct, on en déduit que $\kappa
= \alpha \in Gal(N/E)$, ce qui prouve que $Gal(N/KF)=Gal(N/E)$. Mais alors :
$$E=N^{Gal(N/E)}=N^{Gal(N/KF)}=KF$$
ce que l'on voulait.

\noindent (3) Par le (1-1) (resp. le (2-1)), la donnée d'un sous-groupe $A_0$ de
$Gal(L/J)$ \- (resp. $Gal(N/K)$ ) induit le parallélogramme
$$[L^{A_0},KL^{A_0},N,L] \quad \text{(resp.} \; [L^{A_{0\mid_L}},N^{A_0},N,L] \;
\text{).} $$
De plus, d'après la proposition \ref{prop:adh}(1), avoir $A_1 \unlhd A_0$
implique que $\overline{A_1} \unlhd \overline{A_0}$. Si
$\overline{A_0}=Gal(L/L^{A_0})$, le (1-2) fournit donc le parallélogramme
$$[L^{A_0},KL^{A_0},KL^{A_0}L^{\overline{A_1}},L^{\overline{A_1}}]=[L^{A_0},KL^{A_0}
,KL^{A_1},L^{A_1}] $$
en vertu du (2) de la proposition \ref{prop:adh}. Enfin, si
$\overline{A_0}=Gal(N/N^{A_0})$, en utilisant que par l'homéomorphisme
de restriction à L
$$(\overline{A_1})_{\mid_L}=\overline{({A_1}_{\mid_L})}\;,$$
on obtient par le (2-2) le parallélogramme
$$[L^{({A_0}_{\mid_L})},N^{A_0},N^{\overline{A_1}},L^{\overline{({A_1}_{\mid_L})}}]=
[L^{({A_0}_{\mid_L})},N^{A_0},N^{A_1},L^{({A_1}_{\mid_L})}] \:.$$
\vskip -5mm
\end{proof}

\begin{cor} \label{cor:propfonct} 
Soit $[J,K,N,L]$ un parallélogramme de degré quelconque. \\ 
(1) Pour tout corps intermédiaire $J \subseteq F \subseteq L$ :\\
\indent (1-1) On a le sous-parallélogramme $[F,KF,N,L]$. \\ 
\indent (1-2) Le fait d'avoir l'extension quotient $F/J$ galoisienne implique
l'existence du parallélogramme quotient $[J,K,KF,F]$.\\ 
(2) Pour tout corps intermédiaire $K \subseteq E \subseteq N$ :\\
\indent (2-1) On a le sous-parallélogramme $[E \cap L,E,N,L]$.\\
\indent (2-2) Le fait d'avoir l'extension quotient $E/K$ galoisienne implique
l'existence du parallélogramme quotient $[J,K,E,E \cap L]$.
\end{cor}


\vskip 12mm
\section{Théorie de Galois générale algébrique en dimension 2}

On développe dans cette section une théorie de Galois générale en dimension 2
dont les résultats sont indépendants de toute topologie. Elle contient la théorie
de Galois générale des extensions de corps, celles-ci n'étant que des parallélogrammes
plats (Sect.2).

Le théorème suivant associe à tout sous-groupe du groupe de Galois d'un \-
parallélogramme (Déf. \ref{def:grgal}) un sous-parallélogramme et un
quadrilatère quotient.

\begin{Th} \label{th:exalg} 
Soit $[J,K,N,L]$ un parallélogramme galoisien de degré quelconque.

(1) Pour tout sous-groupe $A$ (resp. $B$) de $Gal(N/K)$ (resp. $Gal(N/L)$), on a
le sous-parallélogramme galoisien
$$[N^{A \times B},N^A,N,N^B]\:,$$
et le quadrilatère corporel $$(J,K^{(B_{\mid_K})},N^{A \times
B},L^{(A_{\mid_L})})\:,$$
où $A_{\mid_L}$ (resp. $B_{\mid_K}$) est l'image de $A$ (resp. $B$) par la
restriction à $L$ (resp. $K$).

(2) Si de plus $A$ (resp. $B$) est normal dans $Gal(N/K)$ (resp. $Gal(N/L)$), on
a le parallélogramme galoisien quotient $[J,K^{(B_{\mid_K})},N^{A \times
B},L^{(A_{\mid_L})}]$. 
\end{Th}

\begin{figure}[!h]
\begin{center}
\vskip -3mm
\includegraphics[width=10cm]{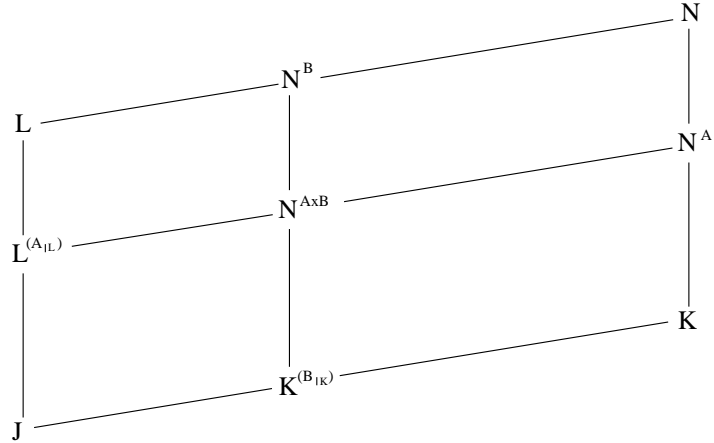}
\end{center}
\vskip -6mm
\rm{\caption{\label{fig:6}\leg Sous-parallélogramme \& parallélogramme quotient}}
\vskip -1mm
\end{figure}

\newpage
\begin{proof}
 (1) \emph {Existence de} $[N^{A \times B},N^A,N,N^B]$. Clairement, le sous-groupe de
 $Gal(N/J)$ engendré par $A$ et $B$ est égal au produit direct de $A$ par $B$,
 d'où
 $$N^A \cap N^B = N^{<A,B>} = N^{A \times B} \: .$$
 Ensuite, on déduit de $K \subseteq N^A \subseteq N$ et $L \subseteq N^B
 \subseteq N$ que
 $$KL=N \subseteq N^A N^B \subseteq N \: ,$$
 d'où $N^A N^B =N$. Ceci prouve l'existence du sous-quadrilatère $(N^{A \times
 B},N^A,N,N^B)$. De plus, on a le sous-parallélogramme
 $[L^{(A_{\mid_L})},N^A,N,L]$ d'après le (2-1) du théorème \ref{th:propfonct}. 
 En particulier, l'extension $N^A/L^{(A_{\mid_L})}$ est galoisienne. Comme
 $$L^{(A_{\mid_L})}=N^A \cap L \subseteq N^A \cap N^B = N^{A \times B} \: ,$$
 la sous-extension $N^A/N^{A \times B}$ est galoisienne. Par le même raisonnement
 dans le parallélogramme transposé $[J,L,N,K]$, on déduit cette fois du sous-parallélogramme\\
 $[K^{(B_{\mid_K})},N^B,N,K]$ que l'extension $N^B/N^{A \times B}$ est
 galoisienne.

  \emph {Existence de} $(J,K^{(B_{\mid_K})},N^{A \times
 B},L^{(A_{\mid_L})})$. A l'évidence, $K \cap L=J$ implique $K^{(B_{\mid_K})}
 \cap L^{(A_{\mid_L})}=J$. Reste à prouver que $K^{(B_{\mid_K})} \, L^{(A_{\mid_L})}= N^{A \times B}$. En appliquant le scindement de la diagonale
 (Prop.\ref{prop:scind}) dans les parallélogrammes galoisiens $[K^{(B_{\mid_K})},N^B,N,K]$
 et $[L^{(A_{\mid_L})},N^A,N,L]$, on a respectivement
 $$Gal(N/K^{(B_{\mid_K})})=Gal(N/K) \times Gal(N/N^B) \;, $$
 $$Gal(N/L^{(A_{\mid_L})})=Gal(N/N^A) \times Gal(N/L) \;. $$
 Ainsi, l'intersection
 $$Gal(N/K^{(B_{\mid_K})}) \cap Gal(N/L^{(A_{\mid_L})})=Gal(N/K^{(B_{\mid_K})} \,
 L^{(A_{\mid_L})})$$
 est égale à
 $$Gal(N/N^A) \times Gal(N/N^B) = Gal(N/N^{A \times B})$$
 en vertu du sous-parallélogramme $[N^{A \times B},N^A,N,N^B]$.
 On en déduit
 $$K^{(B_{\mid_K})} \, L^{(A_{\mid_L})} =N^{Gal(N/K^{(B_{\mid_K})}L^{(A_{\mid_L})})}=N^{Gal(N/N^{A \times B})}=N^{A \times B}$$
 ce que l'on voulait.
 
 (2) D'après le (2-2) du théorème \ref{th:propfonct}, le fait que $A$ (resp.
 $B$) soit normal dans $Gal(N/K)$ (resp. $Gal(N/L)$) induit le parallélogramme
 $$[J,K,N^A,L^{(A_{\mid_L})}] \qquad \text{{\Large (} resp. }
 [J,L,N^B,K^{(B_{\mid_K})}]  \text{ {\Large )}.}$$
  En particulier, les extensions $L^{(A_{\mid_L})}/J$ et $K^{(B_{\mid_K})}/K$  sont
 galoisiennes.
 Ceci suffit à prouver que le quadrilatère $(J,K^{(B_{\mid_K})},N^{A \times
 B},L^{(A_{\mid_L})})$ du (1) est bien un parallélogramme.
\end{proof}

\vskip 2mm
La proposition suivante est générale et algébrique dans la mesure où elle
s'énonce sans argument topologique. Elle conduira, par restriction aux
sous-groupes fermés pour la topologie de Krull, au théorème \ref{th:final}.

\begin{prop} \label{prop:phialg}
Soit $[J,K,N,L]$ un parallélogramme galoisien de degré quelconque.\\ 
(1) Sous-parallélogrammes galoisiens

(1-1) L'application
$$\Phi_s \; : \; [M,E,N,F] \longmapsto Gal[M,E,N,F]$$
est une injection de l'ensemble des sous-parallélogrammes de $[J,K,N,L]$ dans
l'ensemble des sous-bigroupes du groupe de Galois de $[J,K,N,L]$ (cf. Déf.
\ref{def:grgal}).

(1-2) L'application
$$\Psi_s \; : \; (A,B) \longmapsto [N^{A \times B},N^A,N,N^B]$$
est une surjection de l'ensemble des sous-bigroupes de $Gal[J,K,N,L]$
sur l'ensemble des sous-parallélogrammes de $[J,K,N,L]$.

(1-3) Le composé $\Psi_s \circ \Phi_s$ est l'identité. 
\vskip 3mm
\noindent (2)  Parallélogrammes galoisiens quotients

(2-0) Pour tout parallélogramme quotient $[J,E,C,F]$ de $[J,K,N,L]$,
il existe un unique sous-bigroupe normal $(A,B)$ de $Gal[J,K,N,L]$ tel que l'on
ait par restriction $A_{\mid_L}=Gal(L/F)$ et $B_{\mid_K}=Gal(K/E)$. Précisément :
$$A=Gal(N/KF) \quad , \qquad B=Gal(N/EL) \:.$$

(2-1) Dans les notations du (2-0), l'application
$$\Phi_q \; : \; [J,E,C,F] \longmapsto (A,B)$$
est une injection de l'ensemble des parallélogrammes quotients de
$[J,K,N,L]$ dans l'ensemble des sous-bigroupes normaux de $Gal[J,K,N,L]$.

(2-2) L'application
$$\Psi_q \; : \; (A,B) \longmapsto [J,K^{(B_{\mid_K})},N^{A \times
 B},L^{(A_{\mid_L})}]$$
 est une surjection de l'ensemble des sous-bigroupes normaux de $Gal[J,K,N,L]$ sur
 l'ensemble des parallélogrammes quotients de $[J,K,N,L]$.
 
 (2-3) Le composé $\Psi_q \circ \Phi_q$ est l'identité.
 
 (2-4) Dans les notations du (2-0), on a l'isomorphisme
 $$Gal[J,E,C,F] \isomto  Gal[J,K,N,L] \,/\, (A,B)\:.$$
\end{prop}

\begin{proof} 
(1) (1-1) Par la définition \ref{def:grgal}, 
$$Gal[M,E,N,F] = Gal[M',E',N,F'] \quad \ssi \quad 
\left\{ \begin{array}{rl}
Gal(N/E) &= Gal(N/E') \\
Gal(N/F) &= Gal(N/F')
\end{array} \right. \: .$$
En prenant les invariants dans $N$ de ces groupes, on en déduit que $E=E'$ et
$F=F'$ ; d'où $M=E \cap F=E' \cap F'=M'$, et l'injectivité de $\Phi_s$ est
prouvée.

(1-2) L'application $\Psi_s$ existe en vertu du (1) du théorème \ref{th:exalg}.
Sa surjectivité résulte immédiatement du (1-3).

(1-3) Soit $[M,E,N,F]$ un sous-parallélogramme de $[J,K,N,L]$. D'après la
proposition \ref{prop:scind}, $Gal(N/E) \times Gal(N/F)=Gal(N/M)$, de sorte que
\begin{align*}
\Psi_s \circ \Phi_s ([M,E,N,F]) &= \Psi_s(Gal(N/E)\, , \, Gal(N/F)) \\
&= [N^{Gal(N/M)},N^{Gal(N/E)},N,N^{Gal(N/F)}] \\
&= [M,E,N,F] \: .
\end{align*}

\noindent (2) (2-0) D'après le (1-1) du corollaire \ref{cor:propfonct}, on a le
parallélogramme $[F,KF,N,L]$ dans lequel 
$Gal(N/KF)_{\mid_L} = Gal(L/F) \: .$
De plus, dans $[J,E,C,F]$, l'extension $F/J$ est galoisienne
et, d'après le (1-2) de ce même corollaire, on a le parallélogramme quotient
$[J,K,KF,F]$. En particulier l'extension $KF/K$ est galoisienne et
$Gal(N/KF)$ est un sous-groupe normal de $Gal(N/K)$, ce qui permet de poser
$A:=Gal(N/KF)$. Même raisonnement dans le parallélogramme transposé $[J,L,N,K]$
en prenant $B:=Gal(N/EL)$. De plus, dans $[J,K,N,L]$, les restrictions à $L$ et
$K$ sont des isomorphismes ; donc les sous-groupes $A$ et $B$ sont nécessairement uniques.

(2-1),(2-2),(2-3) L'existence de l'application $\Phi_q$ résulte de l'unicité du
sous-bigroupe normal $(A,B)$ du (2-0). L'application $\Psi_q$ résulte quant à
elle du (2) du théorème \ref{th:exalg}. Pour $A=Gal(N/KF)$ et $B=Gal(N/EL)$, on
a par (2-0) :
\begin{align*}
\Psi_q \circ \Phi_q ([J,E,C,F]) &= \Psi_q(A,B) \\
&=[J,K^{Gal(K/E)},N^{A \times B},L^{Gal(L/F)}] \\
&=[J,E,N^{A \times B},F] \: .
\end{align*}

\noindent De ce dernier parallélogramme, on déduit en particulier que $EF=N^{A \times B}$, et
dans $[J,E,C,F]$, on a $EF=C$. Finalement le composé $\Psi_q \circ \Phi_q$ est
l'identité, ce qui implique que $\Phi_q$ est injective et $\Psi_q$ surjective.

(2-4) Par définition
$$Gal[J,E,C,F]=(Gal(C/E),Gal(C/F))\isomto (Gal(F/J),Gal(E/J)).$$
De l'existence des parallélogrammes $[J,K,KF,F]$ et $[J,L,EL,E]$ (cf. Th
\ref{th:propfonct} \\(1-2)), on déduit alors que
\begin{align*}
Gal[J,E,C,F]&\isomto (Gal(KF/K),Gal(EL/L)) \\
&\isomto (Gal(N/K)/Gal(N/KF)\, ,\, Gal(N/L)/Gal(N/EL)) \\
&\isomto Gal[J,K,N,L] /(A,B)
\end{align*}
par (2-0).
\end{proof}

\vskip 5mm
Le théorème suivant fournit des égalités générales entre corps assez
surprenantes au vu des hypothèses minimalistes. Elles laissent entrevoir des
propriétés inattendues des extensions galoisiennes qui seront étudiées
au chapitre 4 avec la notion de tour galoisienne de composition.
De plus, ces égalités induisent comme une dualité entre le compositum et l'intersection 
de deux corps $K$ et $L$. Nous la traduisons pour ce qui nous concerne en termes de quadrilatères 
dans la proposition \ref{prop:rets} à suivre : il y a correspondance biunivoque entre les sous-quadrilatères 
et les quadrilatères quotients de tout parallélogramme $[K\cap L,K,KL,L]$.

\vskip 2mm
\begin{Th}{\rm( dit "de l'écartelé"\,\footnote[2]{\,du nom du blason qu'évoquent, en héraldique, les
figures correspondant au théorème.})} \label{th:ecart}\index{Théorème!de
l'écartelé}\index{Ecartelé@Écartelé} \\
Soient $K$ et $L$ deux corps contenus dans un même corps et $J:=K \cap L$. On
suppose seulement les extensions $K/J$ et $L/J$ galoisiennes, leurs degrés étant
quelconques. Alors :\\ 
(1) Pour tous corps intermédiaires $E$ et $F$ : $\; J \subseteq E \subseteq K
\;, \; J \subseteq F \subseteq L \: ,$
on a l'égalité
$$KF \cap EL =EF \:.$$

\begin{figure}[!h]
\begin{center}
\vskip -4mm
\includegraphics[width=8.2cm]{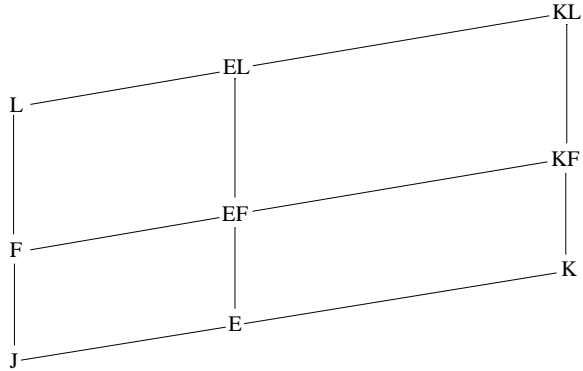}
\end{center}
\vskip -7mm
\rm{\caption{\label{fig:7}\leg Ecartelé compositum}}
\vskip -2mm
\end{figure}

\vskip 2mm
\noindent (2) Pour tous corps intermédiaires $E$ et $F$ : $\; K \subseteq E
\subseteq KL \;, \; L \subseteq F \subseteq KL \: ,$
on a l'égalité
$$(K \cap F) (E \cap L) = E \cap F \:.$$

\begin{figure}[!h]
\begin{center}
\vskip -4mm
\includegraphics[width=8.2cm]{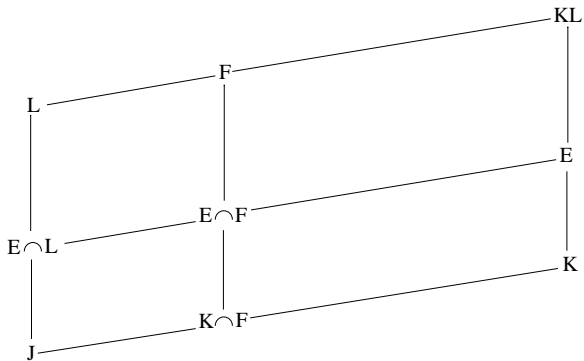}
\end{center}
\vskip -7mm
\rm{\caption{\label{fig:8} \leg Ecartelé intersection}}
\vskip -3mm
\end{figure}

\end{Th}

\newpage
\begin{proof}
Posons $N:=KL$. Comme les extensions $K/J$ et $L/J$ sont galoisiennes, on dispose
du parallélogramme galoisien $[J,K,N,L]$.\\ 
(1) D'après le (1) du corollaire \ref{cor:propfonct},
on a le sous-parallélogramme $[F,KF,N,L]$. Dans celui-ci, le scindement de la
diagonale (Prop. \ref{prop:scind}) fournit l'égalité
$$Gal(N/F) = Gal(N/KF) \times Gal(N/L) \:.$$
De même, dans le parallélogramme transposé $[J,L,N,K]$, on a le
sous-parallélo\-gramme $[E,EL,N,K]$ et l'égalité
$$Gal(N/E)=Gal(N/K) \times Gal(N/EL) \:.$$
On en déduit que
$$Gal(N/EF)=Gal(N/F) \cap Gal(N/E) = Gal(N/KF) \times Gal(N/EL)$$
et ainsi
$$EF=N^{Gal(N/EF)}=N^{Gal(N/KF) \times Gal(N/EL)} \:.$$
Il en résulte, par le (1-1) de la proposition \ref{prop:phialg}, que
$$\Psi_s(Gal(N/KF),Gal(N/EL))=[EF,KF,N,EL]\:,$$
et en particulier $KF \cap EL =EF$.

\noindent (2) Posons $A:=Gal(N/E),\: B:=Gal(N/F)$. D'après le (2-1) du théorème
\ref{th:propfonct}, on a les parallélogrammes $[L^{(A_{\mid_L})},E,N,L] \; , \;
[K^{(B_{\mid_K})},F,N,K]$, et donc 
$$E \cap L= L^{(A_{\mid_L})} \; , \; K \cap F= K^{(B_{\mid_K})} \:.$$
D'autre part, d'après le (1) du théorème \ref{th:exalg}, on a le quadrilatère
$$(J,K^{(B_{\mid_K})},N^{A \times B},L^{(A_{\mid_L})}) $$
où $N^{A \times B}=N^A \cap N^B$. Ainsi :
$$E \cap F =N^A \cap N^B=N^{A \times B}=K^{(B_{\mid_K})} L^{(A_{\mid_L})}=(K
\cap F)(E \cap L) \: .$$ \vskip -5mm
\end{proof}

\vskip 5mm
\begin{cor}
Soient $K$ et $L$ deux corps contenus dans un même corps. On suppose seulement
que les extensions $K/(K \cap L)$ et $L/(K \cap L)$ sont galoisiennes.

Alors :

\noindent (1) Pour tous corps intermédiaires $E$ et $F$ :$\; K \cap L \subseteq
E \subseteq K \; , \; K \cap L \subseteq F \subseteq L \:,$
on a le parallélogramme galoisien
$$[EF,KF,KL,EL]$$
et le quadrilatère corporel
$$(K \cap L,E,EF,F) \:.$$

\noindent (2) Pour tous corps intermédiaires $E$ et $F$ :$\; K \subseteq E
\subseteq KL \; , \; L \subseteq F \subseteq KL \:,$
on a le parallélogramme galoisien
$$[E \cap F,E,KL,F]$$
et le quadrilatère corporel
$$(K \cap L,K \cap F,E \cap F,E \cap L) \:.$$
\end{cor}

\begin{proof} 
(1) Le parallélogramme $[EF,KF,KL,EL]$ apparaît déjà dans la démonstration du (1)
du théorème \ref{th:ecart}. Et il est clair que $E \cap F = K \cap L$.

\noindent (2) Soit $N:=KL$. D'après le (1) du théorème \ref{th:exalg} appliqué avec
$A=Gal(N/E)$ et $B=Gal(N/F)$, on a le sous-parallélogramme
$$[N^{A \times B},E,N,F]=[E \cap F,E,N,F]$$
en vertu du (2) de la démonstration du théorème \ref{th:ecart} qui donne aussi le
quadrilatère $$(K \cap L,K \cap F,E \cap F,E \cap L) \:.$$
\vskip -7mm
\end{proof}

\vskip 5mm
\begin{prop} \label{prop:rets} 
Soit $[J,K,N,L]$ un parallélogramme galoisien de degré quelconque. Notons 
$${\mathcal S}quad \, [J,K,N,L] \;\text{ou}\;\: {\mathcal S}quad \; \text{{\Large (}
resp. }{\mathcal R}quad \, [J,K,N,L] \;\text{ou}\:\;{\mathcal R}quad \text{{\Large
)}}$$
l'ensemble des sous-quadrilatères (resp. des quadrilatères quotients) de
$[J,K,N,L]$. Pour les relations d'ordre du (2) du lemme \ref{lem:relordre} :

\noindent (1) L'application
\vskip -7mm
\begin{align*}
{\mathcal R} \: : \: {\mathcal S}quad &\longrightarrow {\mathcal R}quad \\ 
(M,E,N,F) &\longmapsto (J,K \cap F,M,E \cap L) 
\end{align*}
est une bijection décroissante.

\begin{figure}[!h]
\begin{center}
\vskip -5mm
\includegraphics[width=8cm]{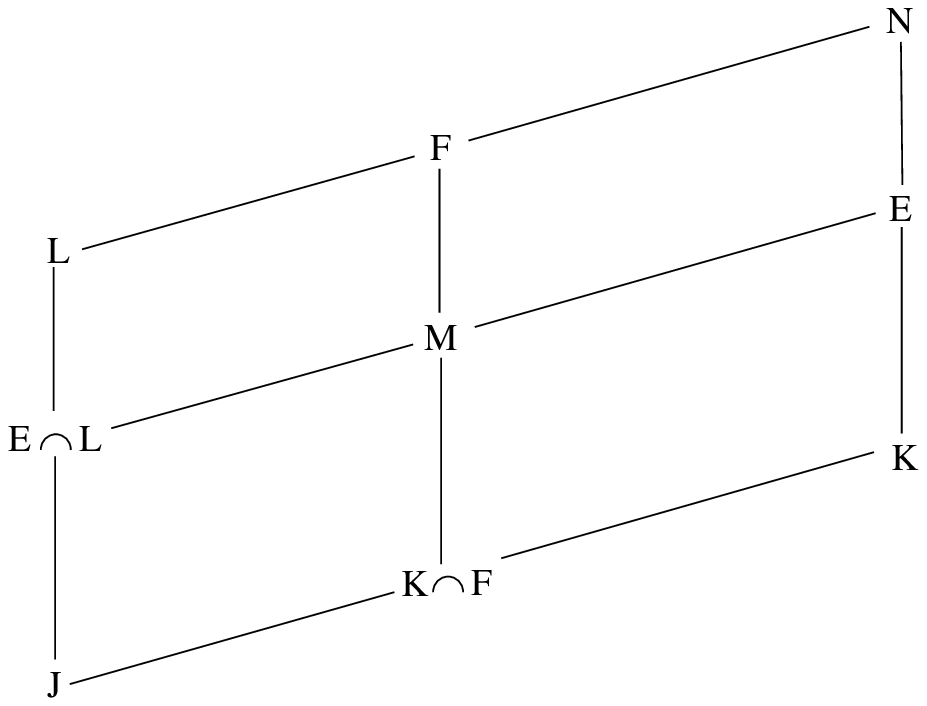}
\end{center}
\vskip -9mm
\rm{\caption{\label{fig:9}\leg Bijection $\mathcal R$}}
\vskip -1mm
\end{figure}

\noindent (2) L'application
\vskip -7mm
\begin{align*}
{\mathcal S} \: : \: {\mathcal R}quad &\longrightarrow {\mathcal S}quad \\ 
(J,E,C,F) &\longmapsto (C,KF,N,EL) 
\end{align*}
est une bijection décroissante.

\begin{figure}[!h]
\begin{center}
\vskip -5mm
\includegraphics[width=8cm]{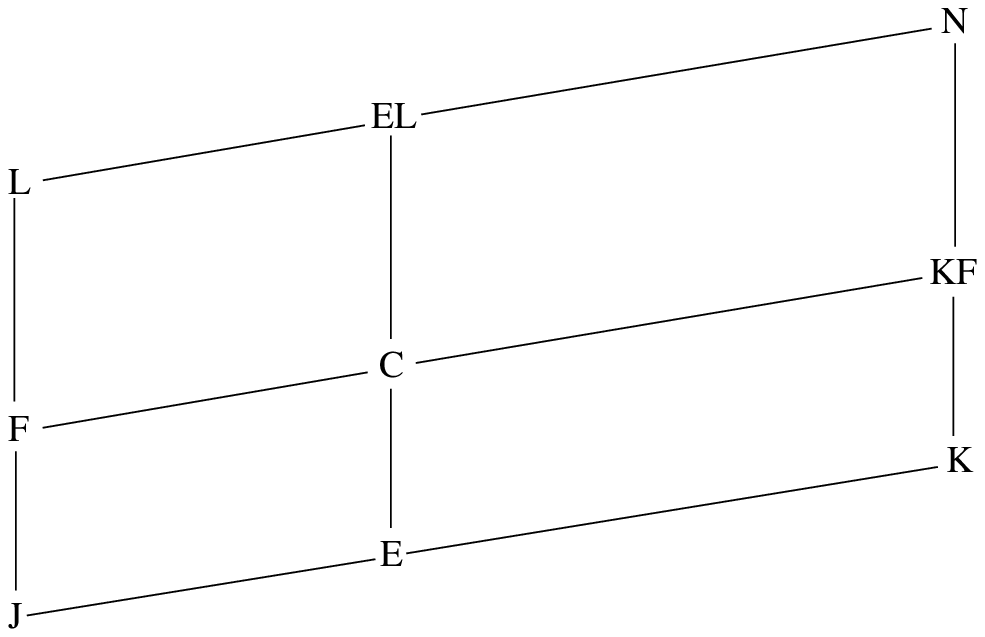}
\end{center}
\vskip -9mm
\rm{\caption{\label{fig:10}\leg Bijection $\mathcal S$}}
\vskip -1mm
\end{figure}

\newpage
\noindent (3) Les applications ${\mathcal R}$ et ${\mathcal S}$ sont réciproques l'une
de l'autre.

\noindent (4) On a l'égalité des cardinaux (éventuellement infinis)
$$\mid {\mathcal R}quad \mid \:=\: \mid {\mathcal S}quad \mid \:.$$
\end{prop}

\gespace
\begin{proof}
Dans le parallélogramme $[J,K,N,L]$, les extensions $K/J$ et $L/J$ sont
galoisiennes.\\ 
(1) L'application ${\mathcal R}$ existe bien car d'après le (2) du théorème de
l'écartelé
$$(K \cap F) (E \cap L) = E \cap F =M \:,$$
d'où le quadrilatère quotient $(J,K\cap F, M,E \cap L)$. La décroissance de
${\mathcal R}$ résulte directement de la définition des relations d'ordre du
lemme \ref{lem:relordre}(2) car
$$(M',E',N,F') \leq (M,E,N,F) \: \ssi \: (E \subseteq E' \, ,\: F \subseteq F')$$
implique que
\begin{align*} 
(E \cap L \subseteq E' \cap L \, &, \: K \cap F \subseteq K \cap F') \\
&\Updownarrow \\ 
(J,K \cap F,M,E \cap L) &\leq (J,K \cap F',M',E' \cap L) \:.
\end{align*}
La bijectivité de ${\mathcal R}$ est une conséquence du (3) ci-dessous.

\pespace
\noindent (2) L'application  ${\mathcal S}$ existe bien car d'après le (1) du théorème de
l'écartelé
$$KF \cap EL=EF \:,$$
d'où le sous-quadrilatère $(C,KF,N,EL)$. La décroissance de ${\mathcal S}$
résulte directement du lemme \ref{lem:relordre}(2) car
$$(J,E,C,F) \leq (J,E',C',F') \: \ssi \: (E \subseteq E' \, ,\: F \subseteq F')$$
implique que
\begin{align*} 
(KF \subseteq KF' \, &, \: EL \subseteq E'L) \\ 
&\Updownarrow \\ 
(C',KF',N,E'L) &\leq (C,KF,N,EL) \:.
\end{align*}
La bijectivité de ${\mathcal S}$ est une conséquence du (3) suivant.

\noindent (3) Pour tout sous-quadrilatère $(M,E,N,F) \in {\mathcal S}quad$, on a 
$${\mathcal S} \circ {\mathcal R} (M,E,N,F) \:=\: {\mathcal S}(J,K \cap F, M, E \cap
L) \:=\: (M,K(E \cap L),N,(K \cap F) L) \:.$$
Or d'après le (2) du théorème de l'écartelé appliqué à $E$ et $N$ (resp. $N$ et
$F$) on obtient
$$K(E \cap L)=E \qquad \text{{\Large (} resp. } (K \cap F) L=F \text{\Large )}
\: ,$$
ce qui prouve que ${\mathcal S} \circ {\mathcal R}=id_{{\mathcal S}quad}$.\\ 
Pour tout quadrilatère quotient $(J,E,C,F) \in {\mathcal R}quad$, on a
$${\mathcal R} \circ {\mathcal S} (J,E,C,F) \:=\: {\mathcal R}(C,KF,N,EL)\:=\:
(J,K \cap EL,C,KF \cap L) \:.$$
Or d'après le (1) du théorème de l'écartelé appliqué à $J$ et $E$ (resp. $F$ et
$J$ ), on obtient
$$ K \cap EL=E\qquad \text{{\Large (} resp. } KF \cap L =F \text{\Large )} \:
,$$
ce qui prouve que ${\mathcal R} \circ {\mathcal S}=id_{{\mathcal R}quad}$.
\end{proof}


\vskip 10mm
\section[Généralisation de la théorie de Galois finie en dimension 2]
{Généralisation topologique de la théorie de Galois finie \newline en dimension
2 \hfill {} } 

Dans cette section, nous enrichissons la proposition algébrique
\ref{prop:phialg} en munissant les groupes de Galois de leur topologie de Krull.
Considérées dans la catégorie produit $\mathbf {ProGr^2}$ de la catégorie
 des groupes profinis $\mathbf {ProGr}$ par elle-même, les applications $\Phi$ et
$\Psi$ de la proposition \ref{prop:phialg} deviennent des bijections. Le
théorème \ref{th:final} généralise ainsi, d'une part le théorème principal de
la théorie de Galois finie en dimension 2 \cite[Th.4-2]{M-MD} généralisant
lui-même la bijection de Galois classique, d'autre part le théorème de Krull qui
se retrouve en particularisant à des parallélogrammes galoisiens infinis plats.

Par définition dans \cite[p.101]{P}, tout sous-groupe d'un sous-groupe
topologique est fermé. On sait que tout sous-groupe fermé $H$ d'un groupe
profini $G$ est profini, et si $H$ est normal, le quotient $G/H$ est aussi
profini. Ceci nous conduit à poser la définition suivante.

\begin{defn} \label{def:prof} 
(1) Nous appelons "sous-groupe profini" (resp. "sous-groupe profini normal")
d'un groupe profini $G$ tout sous-groupe (resp. sous-groupe normal) $H$ de $G$
fermé pour la topologie de $G$. Nous écrirons
$$H \leq_c G \qquad \text{{\Large (} resp. } H \unlhd_c G \text{\Large )} \:
.$$
(2) Nous appelons "sous-bigroupe profini" (resp. "sous-bigroupe profini normal")
d'un bigroupe profini $(G_1,G_2)$ tout bigroupe $(H_1,H_2)$ tel que l'on ait
$$H_i \leq_c G_i \qquad \text{{\Large (} resp. } H_i \unlhd_c G_i \text{\Large
)} \quad \text{(i=1,2)}  \:. $$
Nous écrirons
$$(H_1,H_2) \leq_c (G_1,G_2) \qquad \text{{\Large (} resp. } (H_1,H_2) \unlhd_c (G_1,G_2) \text{\Large
)}   \:.$$
(3) Nous appelons "bigroupe profini quotient" d'un bigroupe profini $(G_1,G_2)$
par un sous-bigroupe profini normal $(H_1,H_2)$ le bigroupe $(G_1/H_1,G_2/H_2)$.
Nous écrirons
$$(G_1,G_2)/(H_1,H_2):=(G_1/H_1,G_2/H_2) \:.$$
(4) Nous appelons "isomorphisme de bigroupes profinis"
$$(f_1,f_2) \: : \: (G_1,G_2) \isomto (G'_1,G'_2)$$
un morphisme de $\mathbf {ProGr^2}$ tel que chacun des
$$ f_i \: : \: G_i \isomto G'_i \quad (i=1,2)$$
soit un isomorphisme de groupes profinis.
\end{defn}

Pour éviter toute ambiguïté dans la démonstration de la proposition
\ref{prop:relordre} ci-dessous, sortons du contexte le fait général suivant.

\begin{lem} \label{lem:transferm}
Soient $X$ un espace topologique et $A$ une partie fermée de $X$. Pour toute
partie $B$ de $A$, avoir $B$ fermée dans $A$ muni de la topologie induite par
celle de $X$ équivaut à avoir $B$ fermée dans $X$.
\end{lem}

\begin{proof}
Si $B$ est fermée dans $A$, il existe un fermé $F$ de $X$ tel que $B=F \cap A$,
et inversement, il suffit d'écrire $B=B \cap A$.
\end{proof}

\begin{lem} \label{lem:relordrebis}
Soit $E/F$ une extension galoisienne. Pour tous corps intermédiaires $M$ et $M'$
entre $F$ et $E$, on a les équivalences
$$M \subseteq M' \; \ssi \; Gal(E/M') \leq Gal(E/M) \; \ssi \; Gal(E/M') \leq_c
Gal(E/M)$$
où les groupes $Gal(E/M)$ et $Gal(E/M')$  sont munis de leurs topologies de
Krull.
\end{lem}

\begin{proof}
La première équivalence étant claire, il suffit de prouver la seconde et plus
précisément le sens direct de celle-ci. Appliquons le lemme \ref{lem:transferm}
précédent avec
$$X:=Gal(E/F) \: , \; A:=Gal(E/M) \: , \; B:=Gal(E/M') \:.$$
Par hypothèse $B \subseteq A$, et d'après le théorème de Krull classique, on a
$$A \leq_c X \: , \; B \leq_c X \:.$$
Donc $B$ est fermé dans $A$ muni de la topologie induite par celle de $X$. Or, en vertu de la
proposition \ref{prop:topind}, cette topologie induite sur $A$ coïncide avec la
topologie de Krull de $A$. On a donc bien montré que $Gal(E/M')$ est fermé dans
$Gal(E/M)$ muni de sa topologie de Krull.
\end{proof}
\vskip 2mm
Les groupes profinis qui interviennent dans la suite sont des groupes de Galois.
Nous convenons une fois pour toutes qu'étant donnée une extension galoisienne,
son groupe de Galois est muni de sa topologie de Krull.

\begin{prop} \label{prop:relordre}
Lorsque $E/F$ est une extension galoisienne, la relation d'ordre du lemme
\ref{lem:relordre}.(1) dans l'ensemble des sous-extensions de $E/F$ (resp. des
extensions galoisiennes quotients de $E/F$) s'écrit
\begin{align*} 
(E/M') \leq (E/M) \quad &\ssi \quad Gal(E/M') \leq_c Gal(E/M) \\ 
\text{{\Large (} resp. } (M/F) \leq (M'/F) \quad &\ssi \quad Gal(E/M') \leq_c
Gal(E/M) \: \text{\Large )} \:. 
\end{align*}
\end{prop}

\begin{proof}
Immédiate par le lemme \ref{lem:relordrebis}.
\end{proof}
\vskip 2mm
Pour le généraliser en dimension 2, reformulons maintenant le théorème de Krull
classique \cite[AV.64,Th.4]{Bo2}.

\begin{Th}{\rm(Théorème de Krull revisité)}\index{Théorème!de Krull}

Soit $N/K$ une extension galoisienne de degré quelconque. On munit l'ensemble
des sous-extensions (resp. des extensions galoisiennes quotients) de $N/K$ de la
relation d'ordre de la proposition \ref{prop:relordre} ci-dessus. Alors : \\ 
(1) L'application
$$(N/E) \longmapsto Gal(N/E)$$
est une bijection croissante de l'ensemble des sous-extensions de $N/K$ sur
l'ensemble des sous-groupes profinis de $Gal(N/K)$. Sa réciproque est
l'application, elle-même croissante, 
$$H \longmapsto (N/N^H) \:.$$

\noindent (2) L'application
$$(E/K) \longmapsto Gal(N/E)$$
est une bijection décroissante de l'ensemble des extensions galoisiennes
quotients de $N/K$ sur l'ensemble des sous-groupes profinis normaux de
$Gal(N/K)$. Sa réciproque est l'application, elle-même décroissante, 
$$H \longmapsto (N^H/K) \:.$$
De plus, la restriction à $E$ induit un isomorphisme de groupes profinis
$$Gal(E/K) \isomto Gal(N/K) / Gal(N/E) \:.$$
\end{Th}

\begin{proof}
Tout résulte directement du théorème de Krull classique, à l'exception de la
monotonie des réciproques (lorsque les ordres sont partiels, la réciproque d'une
bijection monotone n'est pas nécessairement monotone). \\ 
(1) Soient $H_1$ et $H_2$ deux sous-groupes profinis de $Gal(N/K)$ tels que $H_1
\leq_c H_2$. Par définition, on a
$$H_1=\overline{H_1}=Gal(N/N^{H_1}) \leq_c H_2=\overline{H_2}=Gal(N/N^{H_2}) \:,$$
ce qui équivaut à $(N/N^{H_1}) \leq (N/N^{H_2})$ par la proposition
\ref{prop:relordre}. \\ 
(2) Dans les notations du (1), supposons de plus $H_1$ et $H_2$ normaux dans
$Gal(N/K)$. On a toujours $Gal(N/N^{H_1}) \leq_c Gal(N/N^{H_2})$, ce qui
équivaut par la proposition \ref{prop:relordre} à avoir $(N^{H_2}/K) \leq
(N^{H_1}/K)$.
\end{proof}

\begin{prop}  \label{prop:relordre2}
Soit $[J,K,N,L]$ un parallélogramme galoisien. La relation d'\\ordre du lemme
\ref{lem:relordre} dans l'ensemble des sous-parallélogrammes galoisiens
(resp. des parallélogrammes galoisiens quotients) de $[J,K,N,L]$
s'écrit 
$$[M',E',N,F'] \leq [M,E,N,F] \quad \ssi \quad Gal[M',E',N,F'] \leq_c
Gal[M,E,N,F] $${\Large (} resp. 
$$[J,E,C,F] \leq [J,E',C',F'] \; \ssi \;
Gal[C',KF',N,E'L] \leq_c Gal[C,KF,N,EL] \; \text{\Large )} \:.$$
\end{prop}

\begin{proof} 
(1) \gras {Sous-parallélogrammes}. Par définition 
\begin{align*} 
[M',E',N,F'] \leq [M,E,N,F] \quad &\ssi \quad (E \subseteq E' \; ,
\; F \subseteq F') \\ 
\text{et en vertu du lemme \ref{lem:relordre}} \qquad \quad \qquad& \\
\ssi \left\{ 	\begin{array}{rl} Gal(N/E') &\leq_c Gal(N/E) \\
		Gal(N/F') &\leq_c Gal(N/F) 
		\end{array} \right.  
&\ssi Gal[M',E',N,F'] \leq_c Gal[M,E,N,F] \:.
\end{align*}

\noindent (2) \gras {Parallélogrammes quotients.} D'après la proposition
\ref{prop:rets}, on a l'équivalence
$$[J,E,C,F] \leq [J,E',C',F'] \; \ssi \; [C',KF',N,E'L] \leq [C,KF,N,EL] \:;$$
d'où la conclusion par le (1) précédent.
\end{proof}

\vskip 2mm
Le théorème \ref{th:final} qui suit généralise en degré quelconque le théorème
principal de la théorie de Galois finie en dimension 2 (Th.4.2. de \cite{M-MD}).
En se limitant à des sous-groupes fermés, il rend bijectives les injections
$\Phi$ et les surjections $\Psi$ de la proposition algébrique \ref{prop:phialg}.
De plus, en se limitant à des parallélogrammes plats, il redonne exactement la
double bijection du théorème de Krull revisité. Ainsi, le théorème
\ref{th:final} suivant généralise en dimension 2 le théorème de Krull, tout
comme le théorème 4.2. de \cite{M-MD} généralisait le théorème de Galois
classique pour des extensions finies.
 
\begin{Th} \label{th:final}\index{Théorème!de Krull}
Soit $[J,K,N,L]$ un parallélogramme galoisien de degré quelconque, de groupe de Galois
$Gal[J,K,N,L]$ (Déf. \ref{def:grgal}). On munit l'ensemble des sous-parallélogrammes galoisiens (resp.
des parallélogrammes galoisiens quotients) de $[J,K,N,L]$ de la relation d'ordre
de la proposition \ref{prop:relordre2}. Alors : \\ 
(1) Sous-parallélogrammes galoisiens

 L'application 
$$[M,E,N,F] \longmapsto Gal[M,E,N,F] $$
est une bijection croissante de l'ensemble des sous-parallélogrammes de $[J,K,N,L]$ sur l'ensemble des sous-bigroupes profinis de $Gal[J,K,N,L]$,
dont la réciproque est l'application, elle-même croissante,
$$(A,B) \longmapsto [N^{A \times B},N^A,N,N^B] \:.$$

\noindent (2) Parallélogrammes galoisiens quotients

(2-0) Pour tout parallélogramme quotient $[J,E,C,F]$ de $[J,K,N,L]$,
il existe un unique sous-bigroupe profini normal $(A,B)$ de $Gal[J,K,N,L]$ tel
que l'on ait $A_{\mid_L}=Gal(L/F)$ et $B_{\mid_K}=Gal(K/E)$. Précisément :
$$A=Gal(N/KF) \quad , \qquad B=Gal(N/EL) \:.$$

\begin{figure}[!h]
\begin{center}
\vskip -5mm
\includegraphics[width=8cm]{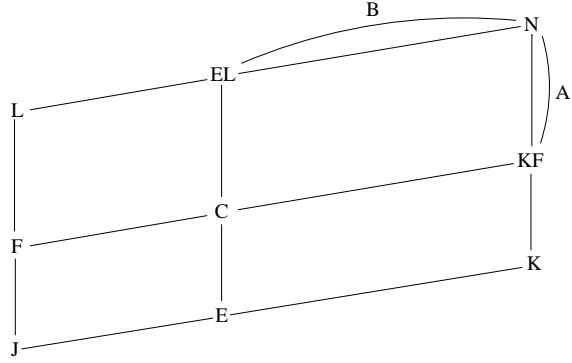}
\end{center}
\vskip -7mm
\rm{\caption{\label{fig:11}\leg Sous-bigroupe associé à un parallélogramme
quotient}}
\vskip 1mm
\end{figure}

(2-1) Dans les notations du (2-0), l'application
$$[J,E,C,F] \longmapsto (A,B)$$
est une bijection décroissante de l'ensemble des parallélogrammes quotients de
$[J,K,N,L]$ dans l'ensemble des sous-bigroupes profinis normaux de
$Gal[J,K,N,L]$ dont la réciproque est l'application, elle-même décroissante, 
$$(A,B) \longmapsto [J,K^{(B_{\mid_K})},N^{A \times B},L^{(A_{\mid_L})}]\:.$$

(2-2) Dans les notations du (2-0), on a l'isomorphisme de bigroupes profinis
 $$Gal[J,E,C,F] \isomto   Gal[J,K,N,L] \,/\, (A,B)\:.$$
\end{Th}

\begin{proof}
On reprend les notations $\Phi$ et $\Psi$ des applications de la proposition
\ref{prop:phialg} en ajoutant des tildes pour marquer que l'on se limite aux
sous-bigroupes profinis de $Gal[J,K,N,L]$. \\ 
(1) L'application $\widetilde{\Phi}_s \; : \; [M,E,N,F] \longmapsto
Gal[M,E,N,F]$ du (1) existe bien car par définition
$Gal[M,E,N,F]=(Gal(N/E),Gal(N/F))$
où $Gal(N/E)$ (resp. $Gal(N/F)$) est fermé dans $Gal(N/K)$ (resp. $Gal(N/L)$) en
vertu du théorème de Krull, de sorte que
$$(Gal(N/E),Gal(N/F)) \leq_c Gal[J,K,N,L]\;.$$
Soit $\widetilde{\Psi}_s \; :
\;(A,B) \longmapsto [N^{A \times B},N^A,N,N^B] \:$ la restriction à l'ensemble
des sous-bigroupes profinis de $Gal[J,K,N,L]$ de
l'application $\Psi_s$ du (1-2) de la proposition \ref{prop:phialg}. D'après le
(1-3) de cette même proposition, le composé $\widetilde{\Psi}_s \circ
\widetilde{\Phi}_s$ est l'identité. De plus, pour tout sous-bigroupe profini
$(A,B)$ de $Gal[J,K,N,L]$, on a 
$$\widetilde{\Phi}_s \circ \widetilde{\Psi}_s (A,B) \:=\: \widetilde{\Phi}_s ([N^{A
\times B},N^A,N,N^B])\:=\:(Gal(N/N^A),Gal(N/N^B)) \:=\:(\overline{A},\overline{B}) \:.$$
Mais par définition 
$$(A,B) \leq_c Gal[J,K,N,L] \quad \implique \quad \left\{ 
	\begin{array}{l} 
	A \leq_c Gal(N/K) \\
	B \leq_c Gal(N/L) \\
	\end{array} \right.  $$
de sorte que $\widetilde{\Phi}_s \circ \widetilde{\Psi}_s (A,B)=(A,B) \:.$
Les applications $\widetilde{\Phi}_s$, $\widetilde{\Psi}_s$ sont donc
réciproques l'une de l'autre, et en particulier bijectives. Montrons qu'elles
sont croissantes. C'est clair directement pour $\widetilde{\Phi}_s$ d'après la
proposition \ref{prop:relordre2}. Pour $\widetilde{\Psi}_s$, considérons deux
sous-bigroupes profinis de $Gal[J,K,N,L]$
$$(A',B') \leq_c (A,B) \:;$$
et posons
$$[M,E,N,F] := \widetilde{\Psi}_s(A,B) \: , \quad [M',E',N,F'] :=
\widetilde{\Psi}_s(A',B') \:.$$
Par ce qui précède
$$Gal[M,E,N,F]=\widetilde{\Phi}_s ([M,E,N,F])=\widetilde{\Phi}_s \circ
\widetilde{\Psi}_s (A,B)=(A,B) \:.$$
Donc
$$Gal[M',E',N,F']=(A',B') \leq_c (A,B)=Gal[M,E,N,F] \:,$$
ce qui équivaut, par la proposition \ref{prop:relordre2}, à $[M',E',N,F']\leq
[M,E,N,F] \:.$

\noindent (2) (2-0) Ce n'est que le (2-0) de la proposition \ref{prop:phialg} en
rajoutant le fait que $A=Gal(N/KF)$ et $B=Gal(N/EL)$ sont des sous-groupes
fermés en vertu du théorème de Krull.

(2-1) L'application $\widetilde{\Phi}_q \; : \; [J,E,C,F] \longmapsto (A,B)$
du (2-1) existe bien car d'après le (2-0) précèdent, on a
$$(A,B) \unlhd_c Gal[J,K,N,L]$$
Soit $\widetilde{\Psi}_q\; : \;
(A,B) \longmapsto [J,K^{(B_{\mid_K})},N^{A \times B},L^{(A_{\mid_L})}]$
la restriction à l'ensemble des sous-bigroupes profinis normaux de $Gal[J,K,N,L]$
de l'application $\Psi_q$ du (2-2) de la proposition \ref{prop:phialg}.
D'après le (2-3) de cette même proposition, le composé $\widetilde{\Psi}_q \circ
\widetilde{\Phi}_q$ est l'identité. De plus, pour tout sous-groupe profini
normal $(A,B)$ de $Gal[J,K,N,L]$, on a
$$\widetilde{\Phi}_q \circ \widetilde{\Psi}_q (A,B) \:=\:
\widetilde{\Phi}_q([J,K^{(B_{\mid_K})},N^{A \times
B},L^{(A_{\mid_L})}])\:=\:(A',B')$$
où par définition $A'$ (resp. $B'$) est l'unique sous-groupe de $Gal(N/K)$
(resp. $Gal(N/L)$) tel que $A'_{\mid_L}=Gal(L/L^{(A_{\mid_L})})\quad
\text{{\Large (} resp. } B'_{\mid_K}=Gal(K/K^{(B_{\mid_K})}  \text{\Large )}\:.$
Or dans le parallélogramme galoisien $[J,K,N,L]$, la restriction à $L$ est un
isomorphisme de groupes profinis de $Gal(N/K)$ sur $Gal(L/J)$. Ainsi :
$\overline{(A_{\mid_L})}=(\overline{A})_{\mid_L}$. Comme $A$ est fermé, on en
déduit que $Gal(L/L^{(A_{\mid_L})})=A_{\mid_L} \:.$ De la même façon
$Gal(K/K^{(B_{\mid_K})})=B_{\mid_K} \:.$ Il en résulte par unicité que
$\widetilde{\Phi}_q \circ \widetilde{\Psi}_q (A,B)=(A,B) \:.$
Les applications $\widetilde{\Phi}_q$, $\widetilde{\Psi}_q$ sont donc
réciproques l'une de l'autre, et en particulier bijectives. Montrons qu'elles
sont décroissantes. Pour $\widetilde{\Phi}_q$, on a d'après la proposition
\ref{prop:relordre2}
\begin{align*} 
[J,E,C,F] \leq [J,E',C',F'] \quad \ssi \quad &Gal[C',KF',N,E'L] \leq_c
Gal[C,KF,N,EL] \:,\\ 
\text{et par définition} \qquad \qquad\qquad\qquad & \\
\ssi \left\{ \begin{array}{l}
	Gal(N/KF') \leq_c Gal(N/KF) \\
	Gal(N/E'L) \leq_c Gal(N/EL)
	\end{array} \right.       &\ssi \: \widetilde{\Phi}_q([J,E',C',F']) \: \leq_c \:
\widetilde{\Phi}_q([J,E,C,F]) \:.
\end{align*}
Pour $\widetilde{\Psi}_q$, considérons deux sous-bigroupes profinis normaux de
$Gal[J,K,N,L]$
$$(A',B') \leq_c (A,B) \:.$$
Comme le composé $\widetilde{\Phi}_q \circ \widetilde{\Psi}_q$ est l'identité,
on a d'après l'équivalence obtenue précedemment pour la décroissance de
$\widetilde{\Phi}_q$,
\begin{align*} 
(A',B') \leq_c (A,B) \quad &\ssi \quad \widetilde{\Phi}_q ( \widetilde{\Psi}_q (A',B'))
\: \leq_c \: \widetilde{\Phi}_q ( \widetilde{\Psi}_q (A,B)) \\ 
&\ssi \quad \widetilde{\Psi}_q (A,B) \: \leq_c \: \widetilde{\Psi}_q (A',B')
\end{align*}
ce qui exprime la décroissance de $\widetilde{\Psi}_q$.

(2-2) Tous les isomorphismes du (2-4) de la proposition \ref{prop:phialg} sont
topologiques.
\end{proof}

\addtocontents{toc}{\mespace\pespace}
\chapter{EXTENSIONS GALTOURABLES}
\addtocontents{lof}{\gespace}
\addtocontents{lof}{\noindent Chapitre \thechapter}
\addtocontents{lof}{\pespace}
\addtocontents{toc}{\pespace}

\gespace

Il s'agit ici d'introduire une notion nouvelle, celle d'extension galtourable,
qui généralise celle d'extension galoisienne. Cette notion est la clef de notre
démarche pour étudier les tours de corps (chapitre 3 et suivants).

On expose dans ce chapitre les premières propriétés des extensions galtourables.
Nous examinons plus particulièrement les propriétés des extensions galoisiennes
qui se généralisent aux extensions galtourables.

\gespace
\gespace
\section{Définitions, notations, exemples}
Formulons tout d'abord un certain nombre de définitions.

A l'instar de la notation de la théorie des groupes désignant les sous-groupes,
nous remplaçons l'inclusion $\subseteq$ par un $\leq$ pour exprimer qu'il y a
conservation de la structure de corps :
\vskip -3mm \begin{equation}
\forall E \text{ corps,} \quad F \leq E \quad \stackrel{\text{déf.}}{\SSI} \quad (F
\text{ sous-corps de } E \text{).} \tag{1-0}
\end{equation}

\gespace
\begin{defnetconv}[] \label{def:tour} Soit $L/K$ une extension quelconque.

\pespace
\noindent (1) On appelle "tour (de corps)\index{Tour} de $L/K$" une suite
finie\footnote{\,C'est le point de vue de cette thèse (à l'instar de la théorie
des groupes), bien que des auteurs récents appellent "tour de corps" une suite
infinie \cite{Y}.} croissante
$\{F_i\}_{0 \leq i \leq m}$ de corps intermédiaires de $L/K$ dans laquelle
$F_0=K$ et $F_m=L$.\\
Nous notons une telle tour
$$(F) \qquad K=F_0 \leq F_1 \leq \,\dots \, \leq F_i \leq F_{i+1} \leq \,\dots \,
\leq F_m=L \:.$$
Nous appelons	\begin{itemize}
		\item "marche de la tour $(F)$"\index{Marche d'une tour de
		corps} : toute extension $F_{i+1}/F_i
		\quad (i=0,\dots ,m-1)$ ;
		\item "hauteur de la tour $(F)$"\index{Hauteur d'une tour de
		corps} : l'entier $m$.
		\end{itemize}

\pespace		
\noindent Nous disons que la tour $(F)$ est "triviale"\index{Tour!triviale} si et seulement si elle est de hauteur
nulle : $m=0$. Lorsque
$$\forall i \in \{0,\dots,m-1\} \quad F_{i+1}\neq F_i$$
 \vskip -3mm \noindent la tour 
$$\begin{array}{c}
(F_<):=(F) \qquad K=F_0 <F_1 < \,\dots \, < F_i < F_{i+1} < \,\dots \,
< F_m=L \:. \\
\end{array}$$
est dite "stricte"\index{Tour!stricte}. En particulier, la tour
triviale est stricte.\\
Nous convenons que pour $m=0$ la tour $(F_<)$ ci-dessus se réduit à
$$(F_<) \qquad K=F_0=L \;.$$

\pespace
\noindent (2) Lorsque $L/K$ est algébrique, nous appelons "tour
galoisienne"\index{Tour!galoisienne} de
$L/K$ toute tour
$$(T) \qquad K=T_0 \leq T_1 \leq \,\dots \, \leq T_i \leq T_{i+1} \leq \,\dots \,
\leq T_m=L$$
dont toutes les marches $T_{i+1}/T_i \quad (i=0,\dots,m-1)$ sont des extensions
galoisiennes.

\pespace
\noindent (3) Deux tours de $L/K$ 
$$\begin{array}{cc}
(F) \quad	& K=F_0 \leq F_1 \leq \,\dots \, \leq F_i \leq \,\dots \,
\leq F_m=L  \\
&\\
(E) \quad	& K=E_0 \leq E_1 \leq \,\dots \, \leq E_j \leq \,\dots \,
\leq E_n=L \\
\end{array}$$
sont égales si et seulement si les deux suites $\{F_i\}_{0 \leq i \leq m}$ et
$\{E_j\}_{0 \leq j \leq n}$ le sont ; autrement dit lorsque $m=n$ et
$$\forall i \in \{0,\dots,m\} \qquad F_i=E_i \:.$$
\end{defnetconv}

\gespace
\begin{rems} 
 Les notations sont celles de la définition \& convention \ref{def:tour}.\\ 
(1) Pour une tour $(F)$ triviale, on a nécessairement $L=K$. Mais on peut avoir
$L=K$ et $(F)$ non triviale, i.e. de hauteur non nulle : $m\neq 0$ (répétition
de corps).

\pespace
\noindent (2) Notons que la tour triviale est la seule tour stricte de $L/K$
lorsque $L=K$, et qu'elle est toujours galoisienne.

\pespace
\noindent (3) Il ne suffit pas d'avoir 
$$\{E_0,\dots,E_n\}=\{F_0,\dots,F_m\}$$
pour que les tours $(F)$ et $(E)$ soient égales.
\end{rems}

\gespace
\begin{fait} \label{fait:hauteur}
La hauteur d'une tour stricte d'une extension finie $L/K$ est au plus égale au
nombre de diviseurs premiers, comptés avec leur multiplicité, du degré de $L/K$.
\end{fait}

\mespace
\begin{proof}
Cela découle directement de la transitivité des degrés.
\end{proof}

\mespace
De même que les extensions algébriques (même séparables) ne sont pas
nécessairement galoisiennes, elles n'admettent pas nécessairement de tour
galoisienne (cf. exemples \ref{ex:existence} ), ce qui justifie la définition suivante.

\mespace
\begin{defn} \label{def:galtourable}
Nous disons qu'une extension $L/K$ est
"galtourable"\index{Extension!galtourable} si et seulement si elle
admet une tour galoisienne au sens de la définition \& convention \ref{def:tour}.(2).
\end{defn}

\mespace
{\it \noindent Scholies. } 
\noindent (1) Par la transitivité des notions d'algébricité et de séparabilité,
une extension galtourable est toujours algébrique et séparable.

\pespace
\noindent (2) Une extension galoisienne est évidemment galtourable. Mais une
extension galtourable n'est pas nécessairement galoisienne vu la non
transitivité de la normalité. La notion d'extension galtourable
généralise donc celle d'extension galoisienne.

\pespace
\noindent (3) Dans \cite{BM-LE} est introduite la notion d' "extension radicale
répétée". Il s'agit des extensions admettant une tour de corps dont chaque
marche est une extension radicale. Nous dirions quant à nous que ce sont des
extensions "radtourables".

\gespace
Les tours galoisiennes introduites dans la définition \& convention \ref{def:tour}.(2) se
généralisent en des tours dont les marches ne sont que galtourables.

\mespace
\begin{defn} \label{def:tourgaltourable} 
Soit $L/K$ une extension algébrique. Nous appelons "tour
galtourable\index{Tour!galtourable} de $L/K$"
toute tour
$$(F) \qquad K=F_0 \leq F_1 \leq \,\dots \, \leq F_i \leq F_{i+1} \leq \,\dots \,
\leq F_m=L$$
dont toutes les marches $F_{i+1}/F_i \quad (i=0,\dots,m-1)$ sont des extensions
galtourables.
\end{defn}

\gespace
Pour une extension donnée, l'existence d'une telle tour galtourable n'est pas
automatique.
Elle est discutée dans le corollaire 1.11 du chapitre 3.  

\gespace
\begin{defn} \label{def:galsimple}
(1) Nous appelons "extension simple"\index{Extension!simple} toute extension $L/K$ non triviale
n'admettant aucun corps intermédiaire propre :
$$\begin{array}{ccl}
(L/K \text{ simple}) &\stackrel{\text{déf.}}{\SSI}&  (L \neq K, \quad \forall F \quad K\leq F \leq L
\:\implique\: (F=K\: \text{ou} \: F=L)\;) \\
&\SSI &(L \neq K, \quad \forall F \quad K\leq F < L
\:\implique\: F=K)\:.\\
\end{array}$$

\pespace
\noindent (2) Nous appelons "extension galsimple"\index{Extension!galsimple} toute extension $L/K$ non triviale
n'admettant aucune extension quotient galoisienne propre :
$$\begin{array}{ccl}
(L/K \text{ galsimple}) &\stackrel{\text{déf.}}{\SSI}&  \left(\begin{array}{c}
					L \neq K, \quad \forall F \quad K\leq F
					\leq L\\
					(F/K \text{ galoisienne})\:\implique\: (F=K\: \text{ou} \: F=L)\;
					\end{array}\right) \\
		&&\\			
		&\SSI 			&\left(\begin{array}{c}
					L \neq K, \quad \forall F \quad K\leq F
					< L\\
					(F/K \text{ galoisienne})\:\implique\: F=K\\
					\end{array}\right) \:.\\
\end{array}$$
\end{defn}

\gespace
\begin{rem}
Dans la définition ci-dessus de la galsimplicité, aucune hypothèse n'est faite
sur la sous-extension $L/F$ : elle peut être galoisienne ou ne pas l'être.
\end{rem}

\gespace
\begin{prop}
(1) Toute extension simple est galsimple.

\pespace
\noindent (2) Une extension galsimple est galtourable si et seulement si elle est galoisienne.
\end{prop}

\mespace
\begin{proof}
Le (1) est clair. Prouvons le sens direct du (2). Soit $L/K$ une extension
galsimple galtourable. Par la définition \ref{def:galtourable}, elle admet une
tour galoisienne (définition \& convention  \ref{def:tour}.(2) )
$$(T) \qquad K=T_0 \leq T_1 \leq \,\dots \, \leq T_i \leq T_{i+1} \leq \,\dots \,
\leq T_m=L \:.$$
Comme $L/K$ est galsimple (Déf. \ref{def:galsimple}.(2) ), on tire de $K \leq T_1
\leq L$ l'implication
$$(T_1/K \:\text{galoisienne}) \quad \implique \quad (T_1=K \: \text{ou} \:
T_1=L)\:.$$
Si $T_1=L$, $(T_1/K)=(L/K)$ est bien galoisienne. Si $T_1 \neq L$, on a la tour 
$$(T) \qquad K=T_0 = T_1 \leq T_2 \leq \,\dots \, \leq T_i \leq T_{i+1} \leq \,\dots \,
\leq T_m=L \:.$$
En itérant le procédé sans avoir déjà conclu, on obtient
$$(T) \qquad K=T_0 = T_1 =\,\dots \, = T_{m-2} \leq T_{m-1} \leq T_m=L \:.$$
Donc, ou bien $T_{m-1}=K$, ou bien $T_{m-1}=L$, et dans les deux cas $L/K$ est
galoisienne.
\end{proof}

\mespace
Nous étudierons les extensions simples ou galsimples au chapitre 5. La
galsimplicité nous permettra d'énoncer le "Théorème $M$" (Chap. 6) qui est le
point crucial de la dissociation de toutes les extensions finies.

\mespace
Afin d'alléger le texte qui suit, convenons de quelques notations. Elles
concernent les définitions \ref{def:tour}, \ref{def:galtourable},
\ref{def:tourgaltourable} et \ref{def:galsimple} précédentes, et ont été
rassemblées ici pour y référer facilement.

\gespace
\begin{notas}
(1) (Extension stricte $L/K$) $\quad \ssi \quad K< L \quad \ssi \quad L
> K \:.$\\
\pespace \noindent
(2) (Extension galoisienne $L/K$) $\quad \ssi \quad L \gal K \quad
\ssi \quad K\unlhd L \quad \ssi \quad L
\unrhd K \:.$\\
\pespace \noindent
(3) (Extension stricte galoisienne $L/K$) $\quad \ssi \quad K\lhd L \quad
\ssi \quad L \rhd K \:.$\\
\pespace \noindent
(4) (Extension non galoisienne $L/K$) $\quad \ssi \quad L \nongal K \:.$\\
\pespace \noindent
(5) (Extension galtourable $L/K$) $\quad \ssi \quad L \galtou K \quad
\ssi \quad K      \lessgtr L \quad \ssi \quad L \gtrless K 
\:.$\\
\pespace \noindent
(6) (Extension non galtourable $L/K$) $\quad \ssi \quad L \nongaltou K \:.$\\
\pespace \noindent
(7) (Extension galtourable non galoisienne $L/K$) $\quad \ssi \quad L
\galtou \nongal K \:.$\\
\pespace \noindent
(8) (Extension simple $L/K$) $\quad \ssi \quad L \,\simple K \:.$\\
\pespace \noindent
(9) (Extension non simple $L/K$) $\quad \ssi \quad L \,\nonsimple K \:.$\\
\pespace \noindent
(10) (Extension galsimple $L/K$) $\quad \ssi \quad L \galsimple K  \:.$\\
\pespace \noindent
(11) (Extension non galsimple $L/K$) $\quad \ssi \quad L \nongalsimple K$\\
\pespace \noindent
(12) (Extension galsimple galoisienne $L/K$) $\quad \ssi \quad L \galsimple \gal
K \:.$\\
\pespace \noindent
(13) (Extension galsimple non galoisienne $L/K$)
$$ \quad \ssi \quad L \galsimple \nongal K  \quad \ssi \quad K \blacktriangleleft L 
\quad \ssi \quad L \blacktriangleright K \:.$$
\end{notas}

\mespace
Les notations précédentes ne sont évidemment pas redondantes ; en voici quelques
exemples.

\mespace
\begin{exs} \label{ex:existence}
Pour un réel $r \in {\mathbb R}_+$ et un entier $n \in {\mathbb N}$, on note,
 $\sqrt[n]{r}$ la racine $n$-ième réelle de $r$ : $ \sqrt[n]{r}
\in {\mathbb R} \:.$ 
\pespace
\noindent (i) Il existe des extensions séparables non galtourables :
$${\mathbb Q}(\sqrt[6]{2}) \nongaltou {\mathbb Q} \:.$$
(ii) Il existe des extensions galtourables non galoisiennes :
$${\mathbb Q}(\sqrt[4]{2}) \galtou \nongal {\mathbb Q} \:.$$
(iii) \begin{itemize}
\item Toute extension cyclique de degré premier est simple.
\item Quel que soit l'entier $n \geq 3$, l'extension ${\mathbb Q}(\theta) / {\mathbb
Q}$ où 
\end{itemize}
$$\theta^n - \theta - 1 =0$$
est simple, mais non galoisienne :
$${\mathbb Q}(\theta)\, \simple \nongal {\mathbb Q} \:.$$
(iv) Il existe des extensions galsimples non simples non galoisiennes :
$${\mathbb Q}(\sqrt[9]{2}) \galsimple \nonsimple \nongal {\mathbb Q} \:.$$
(v) Il existe des extensions galsimples non simples galoisiennes : soit $L$ le
corps de décomposition dans $\mathbb C$ du polynôme 
$$X^5 + 20 X +16 \qquad \text{(resp. } \sum\limits_{n=0}^8 \frac{X^n}{n!} \text{ ).}$$
On a :
$$Gal(L/{\mathbb Q}) \isomto A_5 \qquad \text{(resp. }Gal(L/{\mathbb Q}) \isomto A_8 \text{ )} $$
$$[L:{\mathbb Q}]=60  \qquad \quad\text{(resp. } [L:{\mathbb Q}]=20160 \text{
)}\;,$$
et $L/{\mathbb Q}$ est une extension galsimple non simple
galoisienne :
$$ L \galsimple \nonsimple \gal {\mathbb Q}\:.$$
\end{exs}

\mespace
\begin{proof}
(i) On a $[{\mathbb Q}(\sqrt[6]{2}):{\mathbb Q}]=6$ par le critère d'Eisenstein.
Supposons que l'extension ${\mathbb Q}(\sqrt[6]{2})/{\mathbb Q}$ soit
galtourable, i.e. par définition qu'il existe une tour galoisienne
$$(T) \qquad {\mathbb Q}=T_0 \unlhd \,\dots \, \unlhd T_i \unlhd \,\dots \,
\unlhd T_m={\mathbb Q}(\sqrt[6]{2}) \:.$$
Prouvons qu'alors ${\mathbb Q}(\sqrt[6]{2})/{\mathbb Q}$ admet nécessairement
une tour galoisienne stricte à deux marches (ceci est un cas particulier d'un
résultat général démontré au corollaire 2.6 du chapitre 3). 
Posons :
$$j:=min\{i\in\{0,\dots,m\} \; | \; T_i \neq {\mathbb Q}\}\:, \quad
k:=max\{i\in\{0,\dots,m\} \; | \; T_i \neq T_m\}\:.$$
De $j>k$ suivrait
$${\mathbb Q}(\sqrt[6]{2})=T_{k+1} \gal T_k={\mathbb Q} \quad : \;
\text{absurde.}$$
Donc nécessairement $j\leq k$ et
$$(T) \quad {\mathbb Q}=T_0 = \,\dots \, = T_{j-1} \lhd T_{j} \unlhd \,\dots \,
\unlhd T_{k} \lhd T_{k+1}=\,\dots\,=T_m={\mathbb Q}(\sqrt[6]{2}) \:.$$
On en déduit la tour à trois marches
$${\mathbb Q} \lhd T_j \leq T_k \lhd {\mathbb Q}(\sqrt[6]{2})$$
qui ne peut pas être stricte, en vertu du Fait \ref{fait:hauteur}. C'est donc que $T_j=T_k$ et, dans notre hypothèse, on a ainsi prouvé l'existence d'une tour galoisienne stricte à deux marches
$${\mathbb Q}=F_0 \lhd F_1 \lhd F_2={\mathbb Q}(\sqrt[6]{2}) \;.$$
Par ailleurs, on a la \\ 

{\bf Proposition. }{\it Soient $n \in {\mathbb N}$ et $a \in {\mathbb Q}_+$ un
rationnel positif. On suppose que pour tout diviseur $d$ de $n$, $a \notin
{\mathbb Q}^d$. Alors, pour tout corps
intermédiaire $F$ entre ${\mathbb Q}$ et $L:={\mathbb Q}(\sqrt[n]{a})$ :
${\mathbb Q} \leq F \leq L$, il existe un entier $d$ divisant $n$ : $d \mid n$,
tel que $F={\mathbb Q}(\sqrt[d]{a})$.}

\begin{proof}
Conséquence directe du Théorème 2.1 de \cite{A-V} ou du Théorème 2.2 de
\cite{Vi}.
\end{proof}

Avec ce qui précède, il résulte de la proposition ci-dessus que les seules tours
strictes à deux marches de ${\mathbb Q}(\sqrt[6]{2}) / {\mathbb Q}$ sont
$${\mathbb Q}=F_0 \lhd F_1 = {\mathbb Q}(\sqrt{2}) < F_2={\mathbb
Q}(\sqrt[6]{2})$$
et
$${\mathbb Q}=F_0 < F_1 = {\mathbb Q}(\sqrt[3]{2}) \lhd F_2={\mathbb
Q}(\sqrt[6]{2}) \;.$$
Or ni l'une ni l'autre n'est galoisienne, d'où la contradiction.

\pespace
\noindent (ii) Il suffit de considérer la tour galoisienne
$${\mathbb Q} \lhd {\mathbb Q}(\sqrt{2}) \lhd {\mathbb Q}(\sqrt[4]{2})\:.$$

\pespace
\noindent (iii) La première assertion est claire ; quant à la seconde, elle sera
démontrée au chapitre 5. 

\pespace
\noindent (iv) Posons $L:={\mathbb Q}(\sqrt[9]{2})$. Clairement, l'extension $L/{\mathbb Q}$ est non simple et non galoisienne. Si elle n'était pas galsimple, il existerait un corps intermédiaire $N$ tel que ${\mathbb Q} \lhd N < L$. Or, en vertu de la
proposition de la démonstration du (i) ci-dessus, on a nécessairement
$N={\mathbb Q}(\sqrt[3]{2})$ qui n'est pas galoisien sur ${\mathbb Q}$.

\pespace
\noindent (v) Pour $X^5 + 20 X + 16$, on vérifie que $Gal(L/{\mathbb Q}) \isomto A_5$
par PARI \cite{Co2}. Pour $\sum\limits_{n=0}^8 \frac{X^n}{n!}$, on a $Gal(L/{\mathbb Q})
\isomto A_8$ d'après un résultat de Schur (\cite{Sc} ou \cite{Col}). Les
sous-groupes de Sylow de $A_5$ (resp. $A_8$) assurent que $L/{\mathbb Q}$ n'est
pas simple. Mais $L \gal {\mathbb Q}$ est galsimple puisque $A_5$ (resp. $A_8$) est
un groupe simple (c'est un phénomène général : cf. Chap. 4, Prop. 2.7).
\end{proof} 


\gespace
\gespace
\section{Théorie générale des extensions galtourables}
On dispose d'une théorie générale des extensions galoisiennes finies ou
infinies. Nous allons montrer qu'il existe aussi une théorie générale des
extensions galtourables. Rappelons d'abord le bien connu "théorème de
l'extension galoisienne translatée"  \cite[p.266,Th.1.12]{La} (dit "natural
irrationalities" dans \cite{Mor}).

\gespace
\begin{Th} \label{th:galtranslatée} \index{Théorème!de l'extension galoisienne
translatée} 
Soient $K/J$ et $L/J$ deux extensions algébriques. Sous la seule hypothèse que
$K$ et $L$ soient contenus dans un même corps, avoir $L/J$ galoisienne implique
que $KL/K$ est galoisienne :
$$(L \gal J) \qquad \IMPLIQUE \qquad (KL \gal K)\:.$$
\end{Th}

\gespace
\begin{cor}
La translatée d'une extension galoisienne par un corps donné est toujours une
extension galoisienne.\\
Précisément, soit $L/J$ une extension galoisienne. Pour tout corps $C$ contenu
dans une clôture algébrique de $J$ contenant $L$, l'extension $CL/CJ$ est encore
une extension galoisienne. Autrement dit, on a l'implication
$$(L \gal J) \qquad \IMPLIQUE \qquad (CL \gal CJ)\:.$$
\end{cor}

\mespace
\begin{proof}
Il suffit d'appliquer le théorème \ref{th:galtranslatée} en translatant
l'extension galoisienne $L/J$ par $CJ/J$.
\end{proof}

\gespace
Ces énoncés se généralisent aux extensions galtourables.

\gespace
\begin{Th} \label{th:translatée}\index{Théorème!de l'extension galtourable
translatée} 
Soient $K/J$ et $L/J$ deux extensions algébriques. Sous la seule condition que
$K$ et $L$ soient contenus dans un même corps, on a l'implication
$$(L/J \;\text{galtourable }) \qquad \IMPLIQUE \qquad (KL/K \;\text{galtourable
})\:.$$
\end{Th}

\gespace
\begin{proof}
Soit $(F)$ une tour galoisienne de l'extension galtourable $L \galtou J$ :
$$(F) \qquad J=F_0 \unlhd F_1 \unlhd \,\dots \, \unlhd F_i \unlhd F_{i+1} \unlhd \,\dots \,
\unlhd F_m=L \:.$$
Posons $E_i:=KF_i \quad(i=0,\dots,m)$, et appliquons le théorème
\ref{th:galtranslatée} en translatant chacune des marches galoisienne $F_{i+1} \gal F_i
\quad (i=0,\dots,m-1)$ par l'extension algébrique $E_i/F_i$.

\begin{figure}[!h]
\begin{center}
\vskip -7mm
\includegraphics[width=9cm]{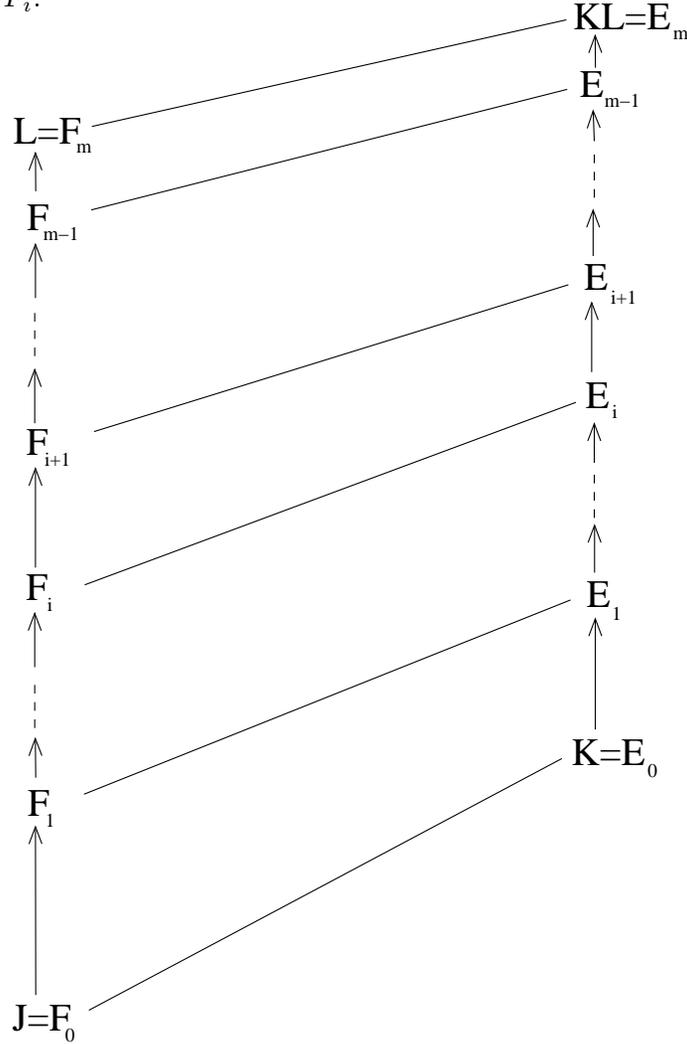}
\end{center}
\vskip -4mm
\rm{\caption{\label{fig:12}\leg Translatée d'une tour galoisienne}}
\vskip 1mm
\end{figure}

\noindent On en déduit que les extensions
$$ E_i\,F_{i+1}=K\,F_{i+1}=E_{i+1} \:/\: E_i \qquad (i=0,\dots,m-1)$$
sont galoisiennes. On obtient donc la tour galoisienne
$$(E) \qquad K=KF_0=E_0 \unlhd \,\dots \, \unlhd E_i \unlhd E_{i+1}
\unlhd \,\dots \, \unlhd E_m=KF_m=KL \:.$$
Or ceci n'exprime rien d'autre que la galtourabilité de $KL/K$.
\end{proof}

\gespace
Enonçons maintenant trois corollaires au théorème \ref{th:translatée}.
\mespace

\gespace
\begin{cor}
La translatée d'une extension galtourable par un corps donné est toujours une
extension galtourable.\\
Précisément, soit $L/J$ une extension galtourable. Pour tout corps $C$ contenu dans
une clôture algébrique de $J$ contenant $L$, l'extension $CL/CJ$ est
encore galtourable. Autrement dit on a l'implication
$$(L \galtou J) \quad \IMPLIQUE \quad (CL \galtou CJ) \:.$$
\end{cor}

\pespace
\begin{proof}
On applique le théorème \ref{th:translatée} avec $K:=CJ$. 
\end{proof}

\gespace
\begin{fait} 
La réciproque du théorème \ref{th:translatée} est fausse.
\end{fait}

\mespace
{\noindent \it 
Scholie.} Il en est de même de la réciproque du théorème de l'extension
galoisienne translatée \ref{th:galtranslatée}.

\mespace
\begin{proof}
Reprenons l'extension $L:={\mathbb Q}(\sqrt[6]{2})/{\mathbb Q}$ de l'exemple
\ref{ex:existence}.(i) et translatons la par l'extension $K:={\mathbb Q}(j) /
{\mathbb Q}$ où $j:=e^{2i\pi/3}$.
Comme $K$ contient les racines $6^{\text{èmes}}$ de l'unité, l'extension $KL=
K(\sqrt[6]{2}) / K$ est kummérienne, donc galoisienne et a fortiori galtourable.
Tandis que l'on a prouvé que l'extension $L:={\mathbb Q}(\sqrt[6]{2})/{\mathbb
Q}$ n'est même pas galtourable.

\begin{figure}[!h]
\begin{center}
\vskip -5mm
\includegraphics[width=6.5cm]{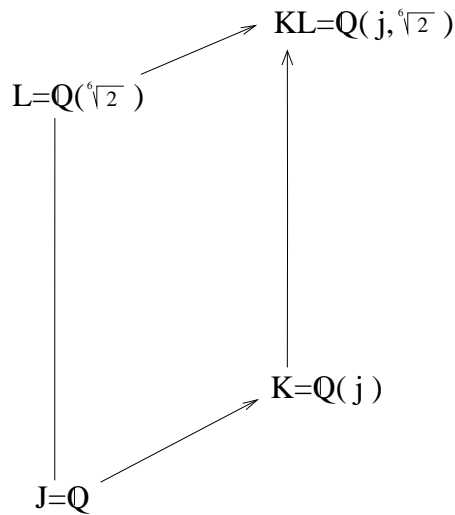}
\end{center}
\vskip -7mm
\rm{\caption{\label{fig:13}\leg Descente non galtourable}}
\vskip -10mm
\end{figure}

\end{proof}

\begin{cor}
La translatée d'une tour galtourable par une extension algébrique est une tour
galtourable.

\noindent Précisément, soit $L/J$ une extension admettant une tour galtourable (cf. Déf.
\ref{def:tourgaltourable})
$$(F) \qquad J=F_0 \lessgtr \,\dots \, \lessgtr F_i \lessgtr F_{i+1}
\lessgtr \,\dots \, \lessgtr F_m=L \:.$$
Alors, pour toute extension algébrique $K/J$ telle que $K$ et $L$ soient
contenus dans un même corps, on a la tour galtourable
$$(E) \qquad K=E_0 \lessgtr \,\dots \, \lessgtr E_i \lessgtr E_{i+1}
\lessgtr \,\dots \, \lessgtr E_m=KL$$
où l'on a posé $\: E_i:=KF_i \quad (i=0,\dots,$ $m)$.
\end{cor}

\mespace
\begin{proof}
Il suffit d'appliquer le théorème \ref{th:translatée} aux extensions
galtourables $F_{i+1} \galtou F_i$  en les translatant par les extensions
algébriques $E_i /F_i$, pour $i \in \{0, \dots,\\ m-1\}$.
\end{proof}

\mespace
Quand on empile deux extensions galoisiennes, on obtient une extension
galtourable non nécessairement galoisienne en général (cf. exemple
\ref{ex:existence}.(ii)). Les extensions galtourables n'ont pas ce défaut : quand on
empile deux extensions galtourables, on obtient toujours une extension
galtourable. Précisément :

\gespace
\begin{fait} \label{fait:concaténation}
Pour toute tour $K \leq L \leq M$, avoir $L/K$ galtourable et $M/L$ galtourable
implique que $M/K$ est galtourable. Autrement dit on a l'implication
$$(\, L \galtou K\:, \; M \galtou L \,) \quad \IMPLIQUE \quad (M \galtou K) \:.$$
\end{fait}

\mespace
{\it \noindent Scholie.} Nous avons déjà prouvé que l'implication inverse est
fausse en général : cf. \ref{fait:quotient}.

\mespace
\begin{proof}
Il suffit d'utiliser que la juxtaposition de tours galoisiennes est encore une
tour galoisienne. Précisément, si
$$(F) \qquad F_0 \unlhd F_1 \unlhd \,\dots \, \unlhd F_i \unlhd \,\dots \,
\unlhd F_m $$
et
$$(E) \qquad F_m=E_0 \unlhd E_1 \unlhd \,\dots \, \unlhd E_j \unlhd \,\dots \,
\unlhd E_n $$
sont deux tours galoisiennes, la tour
$$\begin{array}{lr}
(T)\quad	& F_0=:T_0  \unlhd \,\dots \, \unlhd T_i:=F_i  \unlhd \,\dots \,
\unlhd T_m:=F_m=E_0  \unlhd T_{m+1}:=E_1 \unlhd \,\dots \\
	& \dots \unlhd \, T_{m+j}:=E_j \,\unlhd \dots \,\unlhd T_{m+n}=E_n
\end{array}$$
est évidemment galoisienne.
\end{proof}

\pespace
\begin{cor} \label{cor:compositum}
Tout compositum d'extensions galtourables est galtourable. Précisément :
quelles que soient les extensions galtourables $K/J$ et $L/J$ dont les sommets
sont contenus dans un même corps, l'extension compositum $KL/J$ est galtourable.
Autrement dit, on a l'implication
$$(\, K \galtou J\:, \; L \galtou J \,) \quad \IMPLIQUE \quad (\,KL \galtou J
\,) \:.$$
\end{cor}

\pespace
\begin{proof}
Le théorème \ref{th:translatée} assure que $KL \galtou K$. Comme de plus $K
\galtou J$, le Fait \ref{fait:concaténation} précédent entraine que $KL \galtou
J$.
\end{proof}

\gespace
\begin{prop} \label{prop:sousextension}
Toute sous-extension d'une extension galtourable est galtourable. Précisément, 
soit $L \galtou K$ une extension galtourable. Pour tout corps intermédiaire $M$,
$K \leq M \leq L$, la sous extension $L/M$ est galtourable. 
\end{prop}

\mespace
\begin{proof}
Il suffit d'appliquer le théorème \ref{th:translatée} en translatant
l'extension galtourable $L \galtou K$ par l'extension algébrique $M/K$.
\end{proof}

\gespace
La proposition précédente généralise une propriété analogue des extensions
galoisiennes. Le fait suivant établit que la question "duale", c'est à dire pour
les extensions quotients, trouve la même réponse dans le cas galtourable que
dans le cas galoisien.

\mespace
\begin{fait} \label{fait:quotient}
En général, une extension quotient d'une extension galtourable n'est pas
galtourable.
\end{fait}

\begin{proof}
Considérons la situation

\begin{figure}[!h]
\begin{center}
\vskip -5mm
\includegraphics[width=7cm]{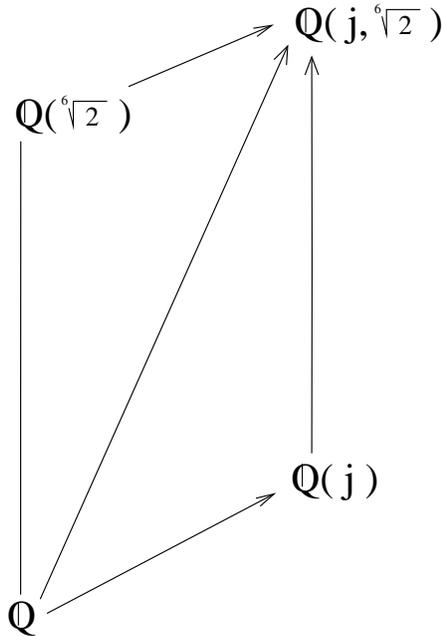}
\end{center}
\vskip -8mm
\rm{\caption{\label{fig:14}\leg Extension quotient non galtourable}}
\vskip -1mm
\end{figure}

\noindent L'extension ${\mathbb Q}(j,\sqrt[6]{2}) / {\mathbb Q}$ est galoisienne
(puisque son corps sommet est le corps de décomposition de $X^6 -2$ sur $\mathbb
Q$). Mais on a déjà vu que l'extension quotient
${\mathbb Q}(\sqrt[6]{2}) / {\mathbb Q}$ n'est pas galtourable (cf. exemple
\ref{ex:existence}.(i) ).
\end{proof}

\mespace
Les différentes propriétés précédentes constituent le début d'une théorie de la
galtourabilité, similaire à la théorie de Galois générale. On peut numériquement
décider si une extension donnée est galtourable ou non, de même que l'on peut
décider si elle est galoisienne ou pas. Cependant il paraît illusoire de
chercher une caractérisation "universelle" des extensions galtourables, ou même
seulement des extensions galtourables finies. Rappelons qu'une telle
caractérisation n'existe déjà pas pour la classe, plus restreinte, des
extensions galoisiennes.

\gespace
Un important problème de théorie de Galois est de conserver le caractère
galoisien d'une extension en la faisant glisser sur un sous-corps de son corps
de base. Ceci constitue le délicat

\mespace
{\bf "Problème de la descente galoisienne".}\\
Soient $N \gal K$ une extension galoisienne et $K/J$ une extension algébrique.
Existe-t-il un sous corps $L$ de $N$ : $L \leq N$, tel que les trois propriétés
suivantes soient vérifiées :
$$\begin{array}{ll}
(D_0) &L/J \text{ est galoisienne} \:: L \gal J\:;\\
(D_1) &\text{L'intersection de }K  \text{ et de } L \text{ est égale à }J \:: K
\cap L =J \:;\\
(D_2) &\text{Le compositum de }K  \text{ et de } L \text{ est égal à }N \:: KL=N \:.\\
\end{array} $$

\gespace
\noindent Ce problème, maintes fois abordé dans la littérature
(\cite{Mo}, \cite{MD}, \cite{M-MD2}, \cite{Ka}, \cite{Co}, ...) se
généralise en le problème suivant sur lequel tout reste à découvrir.

\mespace
{\bf "Problème de la descente galtourable".}\\
Soient $N \galtou K$ une extension galtourable et $K/J$ une extension
algébrique. Existe-t-il un sous corps $L$ de $N$ tel que
$$\begin{array}{ll}
(D_0) & L/J \text{ est galtourable} \:: L \galtou J \:; \;\quad\qquad\qquad\qquad \qquad \qquad\\
(D_1) & K \cap L =J \:;\\
(D_2) & KL=N \:.\\
\end{array} $$

\gespace
Pour le problème de la descente galoisienne, Massy a introduit la notion de
parallélogramme galoisien \cite{M-MD}, essentiellement en degrés finis. Une
généralisation en degrés infinis figure dans \cite{A-M2} (confer Chap. 1). Dans
le cas galtourable introduisons la

\mespace
\begin{defn} 
Nous appellons "quadrilatère galtourable"\index{Quadrilatère!galtourable} tout quadrilatère corporel $(J,K,N,L)$
(cf. Chap. 1, Déf. 1.1) dans lequel les quatre extensions sont galtourables :
$$K \galtou J\,, \: N \galtou K \,,\: N \galtou L \,,\: L \galtou J \;.$$

\begin{figure}[!h]
\begin{center}
\vskip -5mm
\includegraphics[width=5.5cm]{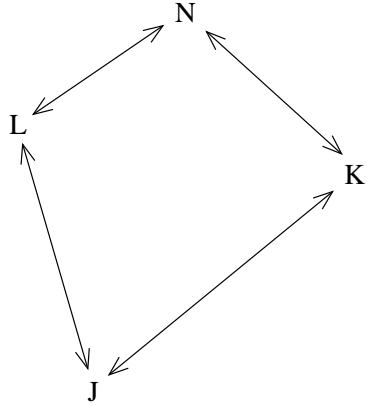}
\end{center}
\vskip -8mm
\rm{\caption{\label{fig:15}\leg Quadrilatère galtourable}}
\vskip 1mm
\end{figure}

\end{defn}

\newpage
\begin{rem}
D'après le théorème \ref{th:translatée}, la galtourabilité de $L/J$ (resp.
$K/J$) implique celle de $N/K$ (resp. $N/L$) :
$$(L \galtou J)\: \implique \: (N \galtou K) \quad \text{{\Large (} resp. }\: (K \galtou
J) \: \implique \: (N \galtou L) \text{ \Large )}\:.$$ 
\end{rem}

\mespace
\begin{prop}
Soient $(J,K,N,L)$ un quadrilatère galtourable, et $F$ un corps
intermédiaire entre $J$ et $L$, $J \leq F \leq L$.\\
(1) Si $KF \cap L =F$, on a le sous-quadrilatère galtourable $(F,KF,N,L)$.\\
(2) Si $F/J$ est galtourable, on a le quadrilatère quotient galtourable
$(J,K,KF,F)$.

\begin{figure}[!h]
\begin{center}
\vskip -5mm
\includegraphics[width=5.5cm]{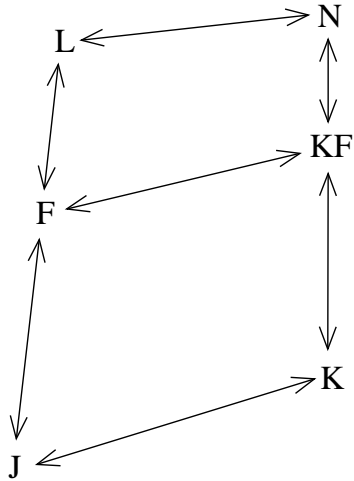}
\end{center}
\vskip -8mm
\rm{\caption{\label{fig:16}\leg Sous galtourabilité \& galtourabilité quotient}}
\vskip 1mm
\end{figure}

\end{prop}

\mespace
\begin{proof} 
(1) D'après le théorème \ref{th:translatée},
$$(K \galtou J) \quad \implique \quad (KF \galtou F) \quad \implique \quad (N \galtou L)\:.$$
D'après la proposition \ref{prop:sousextension} et le théorème \ref{th:translatée},
$$(L \galtou J) \quad \implique \quad (L \galtou F) \quad \implique \quad (N \galtou KF)\:.$$
D'où le quadrilatère galtourable annoncé puisque $F=KF \cap L$ par hypothèse.
\pespace
\noindent (2) Un nouvelle fois par le théorème \ref{th:translatée}
$$(F \galtou J) \quad \implique \quad (KF \galtou K) \:.$$
D'où la conclusion par la démonstration précédente du (1).
\end{proof}

\pespace
Le lecteur qui se sera référé au théorème 1.3 de \cite{M-MD} et au théorème
4.2 
du chapitre 1 aura noté que la proposition précédente est bien moins
consistante que ceux-ci. En effet, plusieurs questions demeurent irrésolues. En
voici un résumé.

\mespace
\begin{qus}
On se donne un quadrilatère galtourable $(J,K,N,L)$. Dans quels cas peut-on
répondre par l'affirmative aux questions suivantes :\\
(1) Si $F$ est un corps intermédiaire entre $J$ et $L$, $J \leq F \leq L$,
$$KF \cap L=F \quad ?$$

\noindent (2) Si $E$ est un corps intermédiaire entre $K$ et $N$, $K \leq E \leq N$,
$$\begin{array}{ll}
\text{(2-1)} &E /(E \cap L) \text{ est-elle galtourable :} \: E \galtou (E\cap L) \quad
? \: \\
\text{(2-2)} &\text{ Pour }  E / K  \text{ galtourable } \\
                    &\text{(2-2-1)} \quad  K(E \cap L)=E	\quad ?	  \\
                    &\text{(2-2-2)} \quad  (E \cap L) \galtou J  \quad ?  \\

\end{array}$$
\pespace
Lorsque $(J,K,N,L)=[J,K,N,L]$ est un parallélogramme galoisien, toutes les
réponses sont affirmatives en vertu du corollaire 4.3 du chapitre 1. 
\end{qus}


\gespace
\gespace
\section{Extensions galtourables définies par un polynôme}
La galtourabilité d'une extension définie par une racine d'un polynôme
irréductible dépend-elle de cette racine ? Nous allons prouver qu'il n'en est
rien, bien qu'à deux racines différentes correspondent en général deux
extensions différentes (contrairement au cas galoisien). La proposition suivante
permet en particulier d'utiliser le logiciel PARI sur les extensions de corps
définies par un polynôme pour prouver la galtourabilité d'une extension sans
avoir d'état d'âme sur la racine considérée.

\mespace
\begin{prop}	\label{prop:conjuguée}
Soient $J$ un corps quelconque et $\tilde J$ une clôture algébrique fixée de
$J$.
Soit $P(X)$ un polynôme séparable irréductible de $J[X]$. Quelles que soient les
racines $t_1$ et $t_2$ de $P(X)$ dans $\tilde J$ :
$$P(t_1)=P(t_2)=0 \:,$$
l'extension $J(t_1)/J$ est galtourable si et seulement si l'extension $J(t_2)/J$
est galtourable :
$$(J(t_1) \galtou J) \quad \SSI \quad (J(t_2) \galtou J) \:.$$
\end{prop}

\mespace
{\it \noindent Scholie.} La séparabilité de $P(X)$ est exigée par la
galtourabilité des extensions.

\mespace
\begin{proof}
Supposons l'extension $L^1:=J(t_1) / J$ galtourable. Soit donc une
tour galoisienne $(T^1)$ de $L^1 \galtou J$
$$(T^1) \qquad J=T^1_0 \unlhd \,\dots \, \unlhd T^1_{i-1} \unlhd T^1_i \unlhd \,\dots \,
\unlhd T^1_m=L^1 \:.$$
Montrons que $(T^1)$ détermine canoniquement une tour galoisienne $(T^2)$ de
$L^2:=J(t_2) / J$.
L'extension $L^1/J$ est finie (de degré celui de $P(X)$) ; chacune des marches
galoisiennes $T^1_i/T^1_{i-1} \quad (i=1,\dots,m)$ est ainsi séparable de
degré fini. Donc par le théorème de l'élément primitif
$$\forall i \in \{1,\dots,m\} \quad \exists \theta_{i} \in T^1_{i}
\qquad T^1_i=T^1_{i-1}(\theta_{i}) \:.$$
On sait que les corps $J(t_1)$ et $J(t_2)$  sont conjugués, i.e. il
existe un $J$-isomorphisme $\Phi$ induit par $\Phi(t_1)=t_2$ :
$$\begin{array}{cccc}
\Phi \,: & L^1 & \isomto & L^2 \\
         & t_1 & \longmapsto          & t_2 \\
\end{array} \quad.$$
Posons alors
$$ v_i:=\Phi(\theta_i) \qquad (i=1,\dots,m)$$
et
$$T^2_0:=J \:, \quad T^2_i:=T^2_{i-1}(v_i)\:.$$
Il est clair par récurrence que
$$T^2_i=\Phi(T^1_i) \qquad (i=0,\dots,m)$$
car
$$\begin{array}{cccccl}
\Phi(T^1_0)	&=& \Phi (J)			&=& J		& \\
\vdots		& &\vdots			& &\vdots 	& \\
\Phi(T^1_i)	&=& \Phi (T^1_{i-1}(\theta_i))	&=& \Phi
(T^1_{i-1})(\Phi(\theta_i))			&\\
		& &				&=&T^2_{i-1}(v_i)&= T^2_i \:.\\
\end{array} $$
\pespace
\noindent En prenant l'image de $(T^1)$ par $\Phi$, on obtient donc canoniquement la tour
\pespace
$$\begin{array}{*{12}{c}}
(T^2)	& \Phi(J)	&=	&\Phi(T^1_0)	&\leq \dots
\leq 	&\Phi(T^1_{i-1})	&\leq 	&\Phi(T^1_i)
&\leq  \dots \leq 	&\Phi(T^1_m)	&=	&\Phi(L^1) \\
			&||		&	&||		&	
	&||		&	&||
&			&||		&	&||	\\
			& J		&=	&T^2_0		&\leq  \dots
			\leq 	&T^2_{i-1}		&\leq &T^2_i	
&\leq  \dots \leq 	&T^2_m		&=	&L^2	\\

\end{array}$$
\pespace
\noindent Il reste à prouver que la tour $(T^2)$ est bien galoisienne.
Posons :
$$P_i(X):=Irr(\theta_i,T^1_{i-1},X) \quad (i=1,\dots,m)$$
et ${\mathcal R}_i:=\{\text{ racines de } P_i(X) \text{ dans } \tilde J\:\}$.
Puisque, par hypothèse, l'extension $T^1_i / T^1_{i-1}$ est galoisienne, on a
$${\mathcal R}_i \subseteq T^1_i  \quad (i=1,\dots,m) \:.$$
Donc directement
$$\Phi({\mathcal R}_i) \subseteq \Phi(T^1_i)= T^2_i \:.$$
Or $\Phi({\mathcal R}_i) =\{\text{ racines de } P^{\Phi}_i(X) \text{ dans }\tilde
J\:\} $
car $P_i(X)$ se décompose en facteurs du premier degré dans $T^1_i$ :
$$P_i(X)=(X-z_{i\,1}{\scriptstyle \,=\,}\theta_i) \dots (X-z_{i\,d_i}) \:,$$
et donc
$$P^{\Phi}_i(X)=(X-\Phi(\theta_i){\scriptstyle =}v_i) \dots (X-\Phi(z_{i\,d_i})) \:.$$
On en déduit que $T^2_i$ est le corps de décomposition sur $T^2_{i-1}$ du
polynôme $P^{\Phi}_i(X)$. En effet comme $\Phi({\mathcal R}_i) \subseteq
T^2_i$, on a trivialement
$$T^2_i=T^2_i(\Phi(z_{i\,1}){\scriptstyle \,=\,}v_i,\dots,\Phi(z_{i\,d_i}))\:;$$
et puisque $T^2_i=T^2_{i-1}(v_i)$,
$$T^2_i=T^2_{i-1}(\Phi(z_{i\,1}){\scriptstyle \,=\,}v_i,\dots,\Phi(z_{i\,d_i}))
 \:.$$
L'extension $T^2_i / T^2_{i-1}$ est donc normale, et par suite galoisienne.
\end{proof}

\addtocontents{toc}{\mespace\pespace}
\chapter{RAFFINEMENTS DE TOURS DE CORPS}
\addtocontents{toc}{\pespace}

\gespace
Le chapitre précédent introduisait la notion de tour galoisienne jouant pour les
corps un rôle analogue à celui des suites normales pour les groupes. Ce
parallèle avec les groupes se poursuivra au chapitre 4 en un quasi-dictionnaire
par des analogues, au départ inespérés, des  célèbres théorèmes de Schreier et
de Jordan-Hölder. Nous nous intéressons ici à la notion clef de ces futurs
énoncés, celle de raffinement d'une tour (galoisienne) de corps qui nécessite,
hélas, une mise au point technique assez lourde.

\gespace
\gespace
\section{Définition d'un raffinement et d'un raffinement galoisien}

\begin{defnetconv} \label{def:raffinement}
Soient $L/K$ une extension algébrique, et
$$(F) \qquad K=F_0 \leq F_1 \leq \dots \leq F_i \leq F_{i+1} \leq \dots \leq F_m=L$$
une tour de $L/K$ (cf. définition \& convention 1.1 du chapitre 2).

\noindent (1) Nous appelons "raffinement de (F)"\index{Raffinement} toute tour
$$(E) \qquad  K=E_0 \leq E_1 \leq \dots \leq E_j \leq E_{j+1} \leq \dots \leq E_n=L$$
de $L/K$ vérifiant les deux conditions suivantes :
$$\begin{array}{lll}
\text{(RAF1)} 	& m\leq n \hfill & \qquad\qquad\qquad\qquad\qquad\qquad\\
\text{(RAF2)}	& \text{il existe une suite finie d'indices}	&\\
		& 0 \leq j_0 < j_1 < \dots < j_m \leq n	&
\end{array}$$
telle que
$$ \forall i \in \{ 0, \dots, m \} \quad F_i=E_{j_i} \:.$$
En particulier pour la tour triviale $(F)$ (cf. Chap. 2, Déf. \& Conv. 1.1), 
nous convenons que pour $m=0$, la suite finie d'indices ci-dessus se réduit à
$$0 \leq j_0 \leq n \;.$$

\noindent (2) Nous appelons "raffinement propre \index{Raffinement!propre}de $(F)$" tout
raffinement $(E)$ de $(F)$ qui vérifie la condition supplémentaire
$$\text{(RAF3)} \qquad \exists j \in \{1, \dots, n-1 \} \quad \forall i \in \{0, \dots, m\}
\quad E_j \neq F_i \;.\qquad\qquad$$

\end{defnetconv}

\pespace
{\noindent \emphase Scholie.} Cette définition est un analogue exact, mutatis mutandis, de celle
utilisée pour les suites normales de groupes qui figure clairement dans
\cite[p.120]{Ro}.
\\ \phantom{ dd}

\mespace
\begin{rems}[] \label{rem:raffinement}
\noindent (1) La condition (RAF1) est redondante : elle est impliquée par
(RAF2), puisque l'on a clairement
$$\forall i \in \{0, \dots, m\} \quad i \leq j_i \:.$$
Cependant (RAF1) offre un moyen commode pour présélectionner les tours
susceptibles d'être des raffinements.

\noindent (2) Le raffinement $(E)$ de la définition \& convention \ref{def:raffinement}.(1)
peut s'écrire
$$\begin{array}{r}
(E) \qquad K=E_0 \leq \dots \leq E_{j_0}=F_0=K \leq \dots \leq E_{j_1}=F_1 \leq
\dots \leq E_{j_i}=F_i \leq \dots \\ 
\dots \leq E_j \leq E_{j+1} \leq \dots \leq E_{j_m}=F_m=L \leq \dots \leq E_n=L
\:.\\
\end{array}$$

\noindent (3) Toutes les répétitions de la tour $(F)$ sont reproduites dans la
tour $(E)$. Il ne peut exister de raffinement supprimant des corps. Un
raffinement qui ajoute un corps nouveau, distinct de tous ceux de la tour de
départ (condition (RAF3)), est "propre".
\end{rems}

\gespace
La définition suivante précise la définition \& convention générale \ref{def:raffinement}
précédente, à l'instar de ce qui a été fait au chapitre 2 pour les tours
strictes, galoisiennes, etc.

\mespace
\begin{defn}[] \label{def:raffinementsuite}
Soient $L/K$ une extension algébrique, $(F)$ une tour de $L/K$ et $(E)$ un
raffinement de $(F)$ (cf. Déf. \& Conv. \ref{def:raffinement}).\\
(1) Nous disons que $(E)$ est un "raffinement strict" \index{Raffinement!strict}de $(F)$ si et seulement
si c'est une tour stricte (cf. Chap. 2, Déf. \& Conv. 1.1
.(1) ).\\
(2) Nous disons que $(E)$ est un "raffinement trivial" \index{Raffinement!trivial}de $(F)$ si et seulement
si c'est un raffinement de $(F)$ non propre, autrement dit qui vérifie comme
condition supplémentaire la négation de (RAF3) dans la définition \& convention 
\ref{def:raffinement}.(2), i.e.
$$\text{(RAFT)} \qquad \forall j \in \{1, \dots, n-1\} \quad \exists i \in \{0, \dots,
m\} \quad E_j=F_i \:. \qquad\qquad\qquad\qquad$$
(3) Nous appelons "raffinement identité" \index{Raffinement!identité}de $(F)$, la tour $(F)$ elle-même.\\
(4) Nous disons que $(E)$ est un "raffinement galoisien" \index{Raffinement!galoisien}de $(F)$ si et seulement
si c'est un raffinement de $(F)$ qui vérifie la condition
supplémentaire
$$\begin{array}{clcl}
\text{(RAFG)} \quad 	& \forall j \in \{1, \dots, n-1\} \quad \text{\rm \Large(} \forall i \in \{0, \dots,
m\} \quad E_j\neq F_i \text{\rm \Large)} \implique	& E_{j-1} \unlhd E_j & \\
		& 	&\Updownarrow	&\\
		&	&E_j \gal E_{j-1} \;.	&\\
\end{array}$$
\end{defn}

\pespace
{\noindent \emphase Scholie.} Toutes ces notions sont évidemment cumulables : raffinement
strict propre, raffinement galoisien strict, etc.

\begin{rems} \label{rem:raffinementsuite}
(1) Une tour $(F)$ non stricte (cf. Chap. 2, Déf. \& Conv. 1.1
) n'admet pas de raffinement strict.\\
(2) Le raffinement identité est un raffinement galoisien trivial.\\
(3) Un raffinement galoisien $(E)$ d'une tour $(F)$ n'est pas nécessairement une
tour galoisienne : le raffinement identité d'une tour non galoisienne par
exemple.
\end{rems}

\mespace
\begin{fait} \label{fait:raffinement}
(1) Tout raffinement trivial est galoisien.\\
(2) Un raffinement qui est une tour galoisienne est un raffinement galoisien.
\end{fait}

\pespace
\begin{proof}
(1) Soient $L/K$ une extension,
$$(F) \qquad K=F_0 \leq \dots \leq F_i \leq \dots \leq F_m=L$$
une tour de $L/K$, et
$$(E) \qquad K=E_0 \leq \dots \leq E_j \leq \dots \leq E_n=L$$
un raffinement trivial de (F). Dire que $(E)$ est galoisien signifie que que dès
qu'un corps $E_j$ est distinct de tous les $F_i$, il est nécessairement
galoisien sur $E_{j-1}$ (condition (RAFG) de la définition
\ref{def:raffinementsuite}). Mais la trivialité de $(E)$ signifie que tous les
$E_j$ sont des $F_i$ (condition (RAFT)). La condition (RAFG) est donc satisfaite
puisqu'elle n'impose rien.\\
(2) Dire que la tour $(E)$ est galoisienne revient à dire (cf. Chap. 2, Déf. \&
Conv. 1.1.(2)) que 
toutes les marches $E_j/E_{j-1} \;(j=1,\dots,n)$ sont galoisiennes. La condition
(RAFG) est dès lors clairement vérifiée.
\end{proof}

\gespace
\begin{fait} ({\emphase "Transitivité des notions de raffinement et de raffinement
galoisien"})\\
Soient $L/K$ une extension et 
$$(F) \qquad K=F_0 \leq \dots \leq F_i \leq \dots \leq F_m=L$$
une tour de $L/K$. Tout raffinement (resp. raffinement galoisien) $(R)$ d'un
raffinement (resp. d'un raffinement galoisien) $(E)$ de $(F)$ est un raffinement
(resp. un raffinement galoisien) de $(F)$.
\end{fait}

\begin{proof}
En vertu de la remarque \ref{rem:raffinement}.(2)
$$\begin{array}{r}
(E) \qquad K=E_0 \leq \dots \leq E_{j_0}=F_0=K \leq \dots \leq E_{j_i}=F_i \leq
\dots \leq E_j \leq \dots \\
\dots \leq E_{j_m}=F_m \leq \dots \leq E_n=L
\end{array}$$
où $0 \leq j_0 < j_1 < \dots < j_m \leq n$ et $F_i=E_{j_i}$ pour tout $i \in
\{0,\dots,m\}$. De même, comme $(R)$ est un raffinement de $(E)$, on a en particulier
$$\begin{array}{r}
(R) \qquad K=R_0 \leq \dots \leq R_{k_0}=E_0=K \leq \dots \leq
R_{k_{j_i}}=E_{j_i} \leq \dots \leq R_k \leq \dots \\
\dots \leq R_{k_{j_m}}=E_{j_m} \leq \dots \leq R_l=L
\end{array}$$
avec $\,0 \leq k_{j_0} < k_{j_1} < \dots < k_{j_m} \leq l\,$ et
$F_i=E_{j_i}=R_{k_{j_i}}$ pour tout $ i \in \{0,\dots,m\}$.
Ceci suffit à établir que $(R)$ est un raffinement de $(F)$ elle-même.
\pespace 
Montrons maintenant que lorsque $(E)$ (resp. $(R)$) est en fait un raffinement
galoisien de $(F)$ (resp. $(E)$), $(R)$ est nécessairement un raffinement
galoisien de $(F)$, i.e. vérifie la condition (cf. Déf.
\ref{def:raffinementsuite}.(4))
$$\text{(RAFG)} \quad  \forall k \in \{1, \dots, l-1\} \quad \text{\rm \Large(} \forall i \in \{0, \dots,
m\} \quad R_k\neq F_i  \text{\rm \Large)} \; \implique \; R_{k-1} \unlhd R_k \:.$$
Soit $k \in \{1, \dots, l-1\}$ tel que $R_k\neq F_i$ pour tout $i \in \{0, \dots,
m\}$. De deux choses l'une :\\
\indent - Ou bien $R_k$ est distinct de tous les $E_j$, i.e. 
$$\forall j \in \{0, \dots, n\} \quad R_k \neq E_j \;,$$
et alors la condition (RAFG) traduisant que $(R)$ est un
raffinement galoisien de $(E)$ implique immédiatement $R_{k-1} \unlhd R_k$.\\
\indent - Ou bien $R_k$ est déjà un corps de la tour $(E)$, i.e. $$\exists j \in \{0, \dots,
n\} \quad R_k=E_j\:.$$
Alors, par définition de $k$,
$$\forall i \in \{0, \dots, m\} \quad R_k=E_j\neq F_i \:.$$
Si l'on traduit cette fois que $(E)$ est un raffinement galoisien de $(F)$, nous
obtenons 
$$R_{k_{j-1}}=E_{j-1} \unlhd E_j=R_{k_j} \:.$$
Supposons que l'on ait choisi $j$ minimal dans l'ensemble des indices $j$ tels
que $R_k=E_j$. Notons tout d'abord que $j \geq 1$, sans quoi $j=0$ et
$R_k=E_0=K=F_0$ qui contredit notre hypothèse initiale sur $k$. Maintenant si
l'on avait $k-1 <k_{j-1},\text{ i.e. } k \leq k_{j-1}$, on aurait par croissance
de la suite des $R_k$ (Chap. 2, Déf. \& Conv. 1.1
)
$$E_j=R_k \leq R_{k_{j-1}}=E_{j-1} \leq E_j \;;$$
d'où $R_k=E_{j-1}$, qui contredirait la minimalité de $j$. Donc
nécessairement $k_{j-1}\leq k-1$ et 
$$\begin{array}{ccc}
R_{k_{j-1}} 	&\leq R_{k-1} \leq 	&R_k \\
\text{\rotatebox[origin=c]{90}{=}}	&			&\text{\rotatebox[origin=c]{90}{=}}\\
E_{j-1}		&\unlhd			&E_j \,.
\end{array}$$
L'extension $R_k / R_{k-1}$ est ainsi bien galoisienne.
\end{proof}

On a déjà dit dans la remarque \ref{rem:raffinementsuite}.(3) qu'un raffinement galoisien d'une
tour quelconque n'est pas nécessairement une tour galoisienne. Cependant, nous
disposons de la proposition suivante, que l'on peut regarder comme une
justification de notre définition d'un raffinement galoisien.

\pespace
\begin{prop}[] \label{prop:ragtourgal}
Un raffinement galoisien d'une tour galoisienne est encore une tour
galoisienne. Précisément, soit
$$(T) \qquad  K=T_0 \unlhd T_1 \unlhd \dots \unlhd T_i \unlhd \dots \unlhd
T_m=L$$
une tour galoisienne (cf. Chap. 2, Déf. \& Conv. 1.1
.(2)). Tout raffinement galoisien de $(F)$ 
$$\begin{array}{r}
(R) \qquad K=R_0 \leq \dots \leq R_{j_0}=T_0=K \leq \dots \leq R_{j_1}=T_1
\leq \dots \leq R_{j_i}=T_i \leq \dots \\
\dots \leq R_j \leq \dots \leq R_{j_m}=T_m=L \leq \dots \leq R_n=L
\end{array}
$$
est une tour galoisienne :
$$ \forall j \in \{0,\dots,n-1\} \quad R_j \unlhd R_{j+1} \ssi R_{j+1} \gal R_j
\:.$$
\end{prop}

\pespace
\begin{proof}
Raisonnons par l'absurde en supposant que la tour $(R)$ ne soit pas
galoisienne, donc comporte au moins une marche non galoisienne :
$$\exists g \in \{0, \dots, n-1 \} \quad R_{g+1} \nongal R_g \:.$$
Considérons alors l'ensemble
$$J:=\{j \in \{1, \dots, n-1\} \; | \; \forall i \in \{0, \dots, m \}
\quad R_j \neq T_i \} \:.$$
L'indice $g+1$ ne peut être dans $J$, sans quoi l'on aurait 
$$\forall i \in \{0, \dots, m \} \quad R_{g+1} \neq T_i \;,$$
et par la condition (RAFG) de $(R)$ (cf. Déf. \ref{def:raffinementsuite}.(4))
$$R_g \unlhd R_{g+1} \; : \; \text{contradiction.}$$
Si $g+1=n$, $R_{g+1}=L=T_m$. Si $g+1 \in \{1, \dots, n-1\}$,
$$g+1 \notin J \; \SSI \; (\exists i \in \{0, \dots, m \} \quad R_{g+1}=T_i) \:.$$
Donc, dans tous les cas,
$$\exists i \in \{0, \dots, m \} \quad R_{g+1}=T_i \:.$$
Soit $i_p$ le plus petit des indices $i$ tels que l'on ait $R_{g+1}=T_i$ :
$$i_p:=\min \{i \in \{0, \dots, m \} \; | \; R_{g+1}=T_i \} \:.$$
Nécessairement $i_p \geq 1$, car sinon $i_p=0$ et
$$K=T_0=R_{g+1} \geq R_g \geq K$$
qui impliquerait $R_{g+1}=R_g=K$ et l'extension $R_{g+1}/R_g$ serait
galoisienne : contradiction. Donc $i_p-1 \in \{0, \dots, m \}$. La minimalité de $i_p$ interdit que
$R_{g+1}=T_{i_p-1}$. Ainsi $R_{g+1} \neq T_{i_p-1}$ et par la croissance de la
suite $\{T_i\}_{0\leq i \leq m}$,
$$\left.\begin{array}{c}
T_{i_p-1} \leq T_{i_p}=R_{g+1}\\
T_{i_p-1} \neq R_{g+1}
\end{array}\right\} \; \implique \; T_{i_p-1} < R_{g+1} \:.$$
Par ailleurs, comme $(R)$ est un raffinement de $(T)$, $T_{i_p-1}$ est l'un des
corps de la tour $(R)$. Écrivons pour alléger la notation $j_p:=j_{i_p-1}$ ; alors 
$$R_{j_p}=T_{i_p-1} < R_{g+1} \; \implique \; j_p < g+1 \; \ssi \; j_p \leq g$$
(puisque $j_p \geq g+1$ impliquerait $R_{j_p} \geq R_{g+1}$ : contradiction).
Donc
$$T_{i_p-1}=R_{j_p} \leq R_g \;.$$
Finalement
$$T_{i_p-1} \leq R_g \leq  R_{g+1}=T_{i_p} \:.$$
L'hypothèse que $(T)$ est une tour galoisienne nous assure en particulier que
l'extension $T_{i_p} / T_{i_p-1}$ est galoisienne.
Il en est donc de même de sa sous-extension $R_{g+1}/R_g$, ce qui contredit
notre hypothèse initiale $R_{g+1} \nongal R_g$. Cette dernière ne peut être
faite ; c'est donc que la tour $(R)$ est bien galoisienne.
\end{proof}

\gespace
La proposition précédente admet un analogue galtourable.
\mespace

\begin{prop} []
Tout raffinement galoisien d'une tour galtourable est une tour galtourable.
\end{prop}

\pespace
\begin{proof}
Celle de la proposition \ref{prop:ragtourgal}, mutatis mutandis, puisque l'on
sait que toute sous-extension d'une extension galtourable est galtourable (cf.
Prop. 2.9 du chapitre 2 
).
\end{proof}

\gespace
Précisons maintenant le lien entre tours galtourables et tours galoisiennes.
\pespace

\begin{prop}[] \label{prop:ragtourgaltou}
Toute tour galtourable admet un raffinement galoisien qui est une tour
galoisienne.
\end{prop}

\begin{proof}
Soit
$$(F) \qquad K=F_0 \lessgtr F_1 \lessgtr \,\dots \, \lessgtr F_i \lessgtr \,\dots \,
\lessgtr F_m=L $$
une tour galtourable (cf. Déf. 1.5 
et Not. 1.9.(5) du chapitre 2). 
En tant qu'extension galtourable, chacune des marches $F_{i+1}/F_i$ admet une
tour galoisienne
$$(T_i) \qquad F_i=T_{j_i} \unlhd T_{j_i+1} \unlhd \dots \unlhd
T_{j_{i+1}}=F_{i+1}$$
où
$$j_i < j_i+1 \leq j_{i+1} \quad \implique \quad j_i < j_{i+1} \:.$$
En juxtaposant, c'est à dire en mettant bout à bout, les tours galoisiennes
$(T_i)$, on obtient la tour galoisienne
$$(T) \qquad K=F_0=T_{j_0} \unlhd T_{j_0+1} \unlhd \dots \unlhd
T_{j_1}=F_{1} \unlhd \dots \unlhd T_{j_m}=F_m=L$$
avec
$$0 \leq j_0 < j_1 < \dots < j_m \leq j_m$$
et
$$F_i=T_{j_i} \quad (i =0, \dots, m) \:,$$
ce qui montre que $(T)$ est un raffinement de $(F)$.
Comme c'est une tour galoisienne, c'est un raffinement galoisien de $(F)$ (cf.
Fait \ref{fait:raffinement}.(2)). 
\end{proof}

\pespace
{\noindent \emphase Scholie.} La démonstration précédente n'est qu'une généralisation
de celle du Fait 2.7 du chapitre 2
.

\gespace
En contraste avec les extensions galoisiennes la proposition
\ref{prop:ragtourgaltou} admet le

\pespace
\begin{cor} \label{cor:tourgaltouextgaltou}
Soit $L/K$ une extension admettant une tour galtourable. Alors $L/K$ est une
extension galtourable.
\end{cor}

\begin{proof}
D'après la proposition \ref{prop:ragtourgaltou}, cette tour galtourable
admet un raffinement galoisien qui est une tour galoisienne (de $L/K$). Donc
$L/K$ admet une tour galoisienne, i.e. est galtourable (cf. Chap. 2, Déf. 1.4 
).
\end{proof}

\pespace
A partir de la notion d'extension galoisienne, nous avons défini la classe plus
large des extensions admettant une tour galoisienne. Le corollaire
\ref{cor:tourgaltouextgaltou} précédent exprime que l'on ne peut pas aller plus
loin dans cette direction : la classe des extensions admettant une tour
galtourable n'est que celle des extensions galtourables elle-même. Précisément :

\begin{cor}
Soit $\Omega$ un corps fixé. Dans $\Omega$, l'ensemble des extensions
galtourables est égal à celui des extensions admettant une tour
galtourable.
\end{cor}

\begin{proof}
Toute extension galtourable $L \galtou K$ admet la tour galtourable 
$$K=F_0 \lessgtr F_1=L \;.$$ L'autre inclusion est fournie par le corollaire
\ref{cor:tourgaltouextgaltou}.
\end{proof}


\gespace
\gespace
\section{Tour stricte associée}

La définition même d'un raffinement suppose la conservation des répétitions des
corps dans une tour donnée. Le but de ce qui suit est d'introduire un moyen
technique de "supprimer" ces répétitions, ce qui s'avérera nécessaire pour
parler de tour de composition (cf. Chap. 4).

\pespace
\begin{propdef} [] \label{propdef:tourstricte}
Soient $L/K$ une extension quelconque et
$$(F) \qquad K=F_0 \leq \dots \leq F_i \leq \dots \leq F_m=L$$
une tour de $L/K$. Il existe une tour stricte $(S)$ de $L/K$ (cf. Chap. 2, Déf. \&
Conv. 1.1
.(1)) et une seule, telle que $(F)$ soit un raffinement trivial de $(S)$ (cf. Déf.
\ref{def:raffinementsuite}.(2)). Nous appelons $(S)$ "la tour stricte associée
\index{Tour!stricte associée}à
$(F)$", et nous la notons
$$(F_<):=(S) \;.$$
\end{propdef}

\begin{proof}
(1) {\emphase Existence.} Il s'agit en fait d'effacer les répétitions de la tour
$(F)$. Précisément, notons $\mathcal F$ l'ensemble des corps de $(F)$ :
$${\mathcal F} :=\{F_0,\dots, F_i, \dots , F_m \} \:;$$
et soit 
$$n:=|{\mathcal F}|-1 \leq m \:.$$
Renumérotons ${\mathcal F}$ de manière croissante en écrivant
$${\mathcal F}=\{F_{j_0}, \dots, F_{j_i}, \dots ,F_{j_n}\}$$
avec, donc,
$$\forall (k,l) \in \{0, \dots , n \}^2 \quad k<l \implique F_{j_k}<F_{j_l} \:.$$
En posant $ S_i:= F_{j_i} \; (i \in \{0, \dots, n \})$ , on obtient clairement la
tour stricte
$$(S) \qquad S_0=F_{j_0}=F_0=K < \dots < S_i=F_{j_i} < \dots <
S_n=F_{j_n}=F_m=L$$
de $L/K$, avec la suite finie d'indices
$$0 \leq j_0 < j_1 < \dots < j_n \leq m \:.$$
En effet, on aurait dans le cas contraire l'existence d'un couple $(k,l) \in
\{0, \dots, n \}^2$ avec $k <l$ et $j_l \leq j_k$, et par croissance de la suite
$\{F_i\}_{0 \leq i \leq m}$ , $F_{j_l} \leq F_{j_k}$ qui contredit la croissance de
la numérotation de ${\mathcal F}$.\\
Comme on a $S_i=F_{j_i} \; (i \in \{0, \dots, n \})$ par construction, ceci
prouve la condition (RAF2) de la définition \& convention \ref{def:raffinement}.(1), et $(F)$
est un raffinement de $(S)$. La condition (RAFT) de la définition
\ref{def:raffinementsuite}.(2) étant évidemment vérifiée, il est de plus
trivial.
\pespace 
(2) {\emphase Unicité.} Soit maintenant
$$(S') \qquad K=S'_0 < \dots < S'_{i'} < \dots < S'_{n'}=L$$
une autre tour stricte de $L/K$ de raffinement trivial $(F)$. On a aussi
$${\mathcal F}=\{S'_0 , \dots , S'_{n'} \} \:.$$
En effet soit $F_i \in {\mathcal F} =\{F_0,\dots, F_i, \dots , F_m \}$. Si l'on
avait $F_i \notin \{S'_0 , \dots , S'_{n'} \}$, la tour $(F)$ serait un raffinement
propre (cf. Déf. \& Conv.  \ref{def:raffinement}.(2)), i.e. non trivial : contradiction
avec la définition de $(S')$. L'inclusion inverse traduit simplement que $(F)$ est un
raffinement de $(S')$. L'égalité précédente implique que
$$n+1= |{\mathcal F}|=|\{S'_0 , \dots , S'_{n'} \}| \:.$$
Mais $(S')$ étant stricte,
$$|\{S'_0 , \dots , S'_{n'} \}|=n'+1 \:,$$
de sorte que nécessairement $n'=n$. Raisonnons maintenant par l'absurde en supposant que
$(S) \neq (S')$. Par le (3) de la définition \& convention 1.1  du chapitre 2 
cela équivaut à ce que
$$\{i \in \{0, \dots , n \} \; | \; S_i \neq S'_i \,\} \neq \emptyset\:.$$
Soit $l$ le plus petit élément de cet ensemble non vide d'entiers. Comme
d'après ce qui précède
$${\mathcal F}=\{S_0 , \dots , S_n \}=\{S'_0 , \dots , S'_n \} \:,$$
il existe $j$ dans $\{0, \dots, n \}$ tel que
$$j \neq l \:, \quad  S_l=S'_j \:.$$
Si l'on avait $j<l$, on déduirait de la minimalité de $l$ que $S_j=S'_j=S_l$ qui
contredit que $(S)$ est une tour stricte. Donc $j \geq l$, i.e. $j > l$ (puisque
$j \neq l$ par définition).
Comme $(S')$ est stricte, on obtient alors $S'_l < S'_j$. De même, il existe $k$
dans $\{0, \dots, n \}$ tel que $S'_l=S_k$. Mutatis mutandis, on
montre que $k >l$ et $S_l < S_k$. Mais alors
$$\begin{array}{cccc}
S_l					& <	& S_k &\\
\text{\rotatebox[origin=c]{90}{=}}	&
&\text{\rotatebox[origin=c]{90}{=}}	&  \text{ contradiction.}\\
S'_j					&>	& S'_l&
\end{array}$$
\end{proof}

\mespace
\begin{rems} \label{rem:tourstricte}
(1) Lorsque $(F)$ est stricte, on a clairement $(F_<)=(F)$. En particulier, la
tour triviale est stricte (cf. Chap. 2, Déf. \& Conv. 1.1.(1)) et
$$(F)=(F_<) \qquad K=F_0=L \;.$$ 
(2) La démonstration de l'unicité de la tour stricte consiste en fait à montrer
que celle-ci ne peut être obtenue que par renumérotation croissante de
${\mathcal F}$ (numérotation qui est unique).
\end{rems}

\gespace
La notion de tour stricte associée à une tour donnée est compatible avec celle
de raffinement.

\pespace
\begin{prop} [] \label{prop:raffinementstrict}
Soient $L/K$ une extension quelconque et
$$(F) \qquad K=F_0 \leq \dots \leq F_i \leq \dots \leq F_m=L$$
une tour de $L/K$. Pour tout raffinement
$$(E) \qquad K=E_0 \leq \dots \leq E_j \leq \dots \leq E_n=L$$
de $(F)$, $(E_<)$ est un raffinement de $(F_<)$ (Prop.\&Déf.
\ref{propdef:tourstricte}).
\end{prop}

\pespace
{\noindent \emphase Scholie.} En général $(E_<)$ ne raffine pas $(F)$, car les
répétitions éventuelles de $(F)$ ont disparu dans $(F_<)$, et ne peuvent se
retrouver dans la tour stricte $(E_<)$.

\mespace
\begin{proof}
\'Ecrivons 
$$(F_<) \qquad K=F_0=F_{<0} < \dots < F_{<m'}= F_m=L \:;$$ 
$$(E_<) \qquad K=E_0=E_{<0} < \dots < E_{<n'}= E_n=L \:.$$
Comme $(F)$ est un raffinement trivial de $(F_<)$, on a
$${\mathcal F}:=\{F_0,\dots , F_m \}=\{F_{<0}, \dots , F_{<m'} \} \:;$$
de même
$${\mathcal E}:=\{E_0,\dots , E_n \}=\{E_{<0}, \dots , E_{<n'} \} \:.$$
L'hypothèse que $(E)$ est un raffinement de $(F)$ signifie qu'il existe une
suite finie d'entiers
$$0 \leq j_0 < j_1 < \dots < j_m \leq n$$
telle que
$$\forall i \in \{0, \dots, m \} \qquad F_i=E_{j_i} \:.$$
Cela implique en particulier que ${\mathcal F} \subseteq {\mathcal E}$. Ainsi
$$\forall i' \in \{0, \dots, m' \} \quad \exists j'_{i'} \in \{0, \dots, n' \}
\qquad F_{<i'}=E_{<j'_{i'}} \:.$$
Fixons une suite de tels entiers : $\{j'_{i'}\}_{0 \leq i' \leq m'}$. Elle est
strictement croissante. En effet, raisonnons par l'absurde en supposant qu'il
existe $0 \leq i'_1 < i'_2 \leq m'$ avec $j'_{i'_2} \leq j'_{i'_1}$. On a alors
$E_{<j'_{i'_2}} \leq E_{<j'_{i'_1}} $ par croissance de toute tour de corps.
Tandis que, comme la tour $(F_<)$ est stricte, l'inégalité $i'_1 < i'_2$
entraine
$$E_{<j'_{i'_1}}=F_{<i'_1}<F_{<i'_2}=E_{<j'_{i'_2}} \quad : \;
\text{contradiction.}$$
Par conséquent
$$0 \leq j'_0 < j'_1 < \dots < j'_{m'} \leq n'$$
avec
$$\forall i' \in \{0, \dots, m' \} \qquad F_{<i'}=E_{<j'_{i'}} \:.$$
Or ceci exprime précisément que $(E_<)$ est un raffinement de $(F_<)$
(Déf. \& Conv.  \ref{def:raffinement}.(1)).
\end{proof}

\gespace
\begin{cor} []  \label{cor:raffinementstrict}
Pour toute tour stricte $(F)$ de $L/K$ et tout raffinement $(E)$ de $(F)$,
$(E_<)$ est encore un raffinement de $(F)$.
\end{cor}

\pespace
\begin{proof}
Cela découle immédiatement de ce que $(F_<)=(F)$ (cf. Rem.
\ref{rem:tourstricte}.(1)) et de la proposition \ref{prop:raffinementstrict}
précédente.
\end{proof}

\gespace
Nous énonçons maintenant un fait bien intuitif, de démonstration élémentaire
mais subtile. Il sera précisé davantage au chapitre 4. Nous le faisons figurer
ici car il est indispensable pour généraliser la proposition
\ref{prop:raffinementstrict} aux raffinements galoisiens.

\gespace
\begin{fait} [] \label{fait:marchetourstricte}
Soient $L/K$ une extension quelconque, 
$$(F) \qquad K=F_0 \leq \dots \leq F_i \leq \dots \leq F_m=L$$
une tour de $L/K$, et
$$(F_<) \qquad K=F_0=F_{<0} <\dots < F_{<j} < \dots < F_{<m'}=F_m=L$$
la tour stricte associée à $(F)$ (cf. Prop. \& Déf. \ref{propdef:tourstricte}).
Alors toute marche de $(F_<)$ est une marche de $(F)$.
\end{fait}

\pespace
\begin{proof}
Le résultat est trivial si $K=L$. Supposons $L \neq K$ de sorte que l'on puisse
considérer $F_{<j+1}/F_{<j} \; (j \in \{0, \dots, m'-1 \})$ une marche de
$(F_<)$. Comme $(F)$ est un raffinement de $(F_<)$, on sait
$$\exists i_j < i_{j+1} \quad F_{i_j}=F_{<j}<F_{<j+1}=F_{i_{j+1}} \:.$$
Considérons l'ensemble d'entiers
$$I_{j+1}:=\{i \in \{0, \dots, m \} \; | \; F_i=F_{<j+1} \}\:.$$
Comme l'ensemble $I_{j+1} \neq \emptyset$ puisque $F_{i_{j+1}}=F_{<j+1}$, il
admet un plus petit élément, et nous pouvons considérer
$$l:=(min \, I_{j+1})-1.$$
Avoir $l+1 \leq i_j$ conduirait à $F_{l+1} \leq F_{i_j}$ qui contredit
$F_{i_j}=F_{<j}<F_{<j+1}=F_{l+1}$. On a donc $0 \leq i_j < l+1$ et en
particulier $l \geq 0$. D'où l'existence de $F_l$. La minimalité de $l+1$ dans
$I_{j+1}$ interdit que $F_l=F_{<j+1}=F_{l+1}$. Donc
$$F_l < F_{l+1}=F_{<j+1} \:.$$
Comme par ailleurs $F_{i_j} \leq F_l$ (puisque $i_j < l+1 \ssi i_j \leq l$), on
a la tour
$$F_{<j}=F_{i_j} \leq F_l < F_{l+1}=F_{<j+1} \:.$$
Or, par définition, $(F)$ raffine trivialement $(F_<)$ et donc
$$\exists j' \in \{0, \dots, m' \} \quad F_l=F_{<j'} \:.$$
Montrons maintenant que nécessairement $j'=j$. Si tel n'était pas le cas, on
aurait : \\
\indent - Ou bien $j' < j$  et $F_l=F_{<j'} < F_{<j}$ ; d'où
$$F_l < F_{<j}=F_{i_j} < F_{<j+1} = F_{l+1} \:.$$
Mais ceci pose un insurmontable problème à $i_j$, qui ne peut ni être inférieur
ou égal à $l$ (par l'inégalité stricte de gauche), ni supérieur ou égal à $l+1$
(par celle de droite) : contradiction.\\
\indent - Ou bien $j+1 \leq j'$ et $F_{<j+1} \leq F_{<j'}= F_l$ qui contredit
directement
$$F_l < F_{l+1} = F_{<j+1} \:.$$
Il est donc prouvé que $j'=j$. Finalement
$$(F_{<j+1} / F_{<j} ) = (F_{l+1} / F_l)$$
et donc $(F_{<j+1} / F_{<j} )$ est bien une marche de $(F)$.
\end{proof}

\mespace
Nous montrerons au chapitre 4, corollaire 1.6,  
 que la réciproque est vraie : toute marche non triviale de $(F)$ est une
marche de $(F_<)$.

\gespace
Tirons ici du Fait \ref{fait:marchetourstricte} précédent l'important

\pespace
\begin{cor} [] \label{cor:tourgalstricte}
Soit $L/K$ une extension galtourable (Chap. 2, Déf. 1.4
). Pour toute tour galoisienne
$$(T) \qquad K=T_0 \unlhd \dots \unlhd T_i \dots \unlhd T_m=L$$
de $L/K$, la tour stricte $(T_<)$ associée à $(T)$ (Prop. \& Déf.
\ref{propdef:tourstricte}) est aussi galoisienne :
$$(T_<) \qquad K=T_0=T_{<0} \lhd \dots \lhd T_{<j} \lhd \dots \lhd T_{<m'}=T_m=L
\:.$$
\end{cor}

\pespace
\begin{proof}
Toute marche de $(T_<)$ est une extension galoisienne, en tant que marche de
$(T)$.
\end{proof}

\gespace
On a vu que la notion de tour stricte associée est compatible avec celle de
raffinement (Prop. \ref{prop:raffinementstrict}). Voyons qu'il en est de même
avec la notion de raffinement galoisien.

\pespace
\begin{prop}[] \label{prop:ragstrict}
Soient $L/K$ une extension quelconque, et
$$(F) \qquad K=F_0 \leq \dots \leq F_i \leq \dots \leq F_m=L$$
une tour de $L/K$. Pour tout raffinement galoisien (Déf.
\ref{def:raffinementsuite}.(4))
$$(E) \qquad K=E_0 \leq \dots \leq E_k \leq \dots \leq E_n=L$$
de $(F)$, $(E_<)$ est un raffinement galoisien de $(F_<)$.
\end{prop}

\begin{proof}
On sait déjà, par la proposition \ref{prop:raffinementstrict}, que $(E_<)$ est
un raffinement de 
$$(F_<) \qquad K=F_0=F_{<0} <\dots < F_{<j} < \dots < F_{<m'}=F_m=L \:.$$
Il reste à voir que 
$$(E_<) \qquad K=E_0=E_{<0} <\dots < E_{<l} < \dots < E_{<n'}=E_n=L$$ vérifie la condition
(RAFG) du (4) de la définition \ref{def:raffinementsuite}.
Notons que le Fait \ref{fait:marchetourstricte} assure que toutes les marches de
$(E_<)$ sont des marches de $(E)$ :
$$\forall l \in \{1, \dots, n' \} \quad \exists k_l \in \{1, \dots, n \}
\quad (E_{<l}/E_{<l-1})=(E_{k_l}/E_{k_l-1}) \:.$$
Supposons qu'il existe $l \in \{1, \dots, n'-1 \}$ tel que
$$\forall j \in \{0, \dots, m' \} \quad E_{<l} \neq F_{<j} \:.$$
(Si un tel $l$ n'existe pas, le raffinement est trivial, donc galoisien (Fait
\ref{fait:raffinement}.(1))).
En vertu de la démonstration  de la proposition \& définition \ref{propdef:tourstricte}, on
sait par ailleurs que
$$\{F_{<0}, \dots, F_{<m'} \}=\{F_0, \dots, F_m\} \:.$$
Dès lors on a
$$\forall i \in \{0, \dots, m \} \quad E_{<l} \neq F_i \:.$$
Comme $E_{<l}=E_{k_l}$, c'est donc que
$$\forall i \in \{0, \dots, m \} \quad E_{k_l} \neq F_i \:.$$
Par la condition (RAFG) du raffinement galoisien $(E)$ de $(F)$, on obtient
finalement que
$$(E_{<l}=E_{k_l} \gal E_{k_l-1}=E_{<l-1})$$
ce que l'on voulait.
\end{proof}

\mespace
\begin{cor} \label{cor:ragstrict} 
Pour toute tour stricte $(F)$ de $L/K$ et tout raffinement galoisien $(E)$ de
$(F)$, $(E_<)$ est encore un raffinement galoisien de $(F)$.
\end{cor}

\pespace
\begin{proof}
Celle du corollaire \ref{cor:raffinementstrict} mutatis mutandis.
\end{proof}

\mespace
L'introduction des tours strictes associées nous permet d'écrire la version stricte
suivante de la proposition \ref{prop:ragtourgaltou} :

\pespace
\begin{prop}[]
Toute tour galtourable stricte admet un raffinement galoisien qui est une tour
galoisienne stricte.
\end{prop}

\pespace
\begin{proof}
Soit $(F)$ une tour galtourable stricte
$$(F) \qquad K=F_0 \lessgtr \dots \lessgtr F_i \lessgtr \dots \lessgtr F_m=L \:.$$
La proposition \ref{prop:ragtourgaltou} fournit une tour galoisienne
$$(T) \qquad K=F_0=T_0 \unlhd \dots \unlhd T_j \unlhd \dots \unlhd T_n=L$$
qui est un raffinement (nécessairement galoisien par le Fait
\ref{fait:raffinement}.(2)) de $(F)$. La proposition \ref{prop:ragstrict} assure
que la tour stricte associée $(T_<)$ est un raffinement galoisien de $(F_<)=(F)$
(Remarque \ref{rem:tourstricte}.(1)). Et le corollaire \ref{cor:tourgalstricte}
certifie que $(T_<)$ est encore une tour galoisienne.
\end{proof}

\gespace
Notons enfin le fait suivant concernant les tours strictes.

\mespace
\begin{fait}[]
Soient $L/K$ une extension et
$$(F) \qquad K=F_0 < \dots < F_i < \dots < F_m=L $$
une tour stricte de $L/K$. Pour tout raffinement trivial strict
$$(E) \qquad K=E_0 =F_0= E_{j_0} < \dots < E_{j_i}=F_i < \dots < E_{j_m}=F_m=L$$
de $(F)$, on a nécessairement $(E)=(F)$ (Chap. 2, Déf. \& Conv. 1.1.(3)). 
\end{fait}

\pespace
\begin{proof}
La tour $(F)$ est stricte de raffinement trivial $(E)$. Par l'unicité de la
Prop. \& Déf. \ref{propdef:tourstricte}, c'est donc que $(F)=(E_<)$. Or $(E)$
étant une tour stricte, on a $(E_<)=(E)$ d'après la remarque
\ref{rem:tourstricte}.(1). D'où la conclusion.
\end{proof}


\gespace
\gespace
\section{Tour restreinte, tour ratio}
L'obtention de raffinements de tours de corps sera notre objet dans les
chapitres 4, 6 et 7. Nous voulons ici introduire une méthode de fragmentation
qui sera utilisée au chapitre 7 final. La définition suivante précise la notion
intuitive de suppression, à gauche ou à droite, des corps d'une tour donnée.

\mespace
\begin{defn} [] \label{def:resrat}
Soient $L/K$ une extension quelconque et
$$(F) \qquad K=F_0 \leq \dots \leq F_i \leq \dots \leq F_m=L $$ 
une tour de $L/K$. Pour tout indice fixé $r \in \{0, \dots, m \}$, nous appelons
:
\pespace
\noindent (1) "Tour restreinte de $(F)$ à l'indice $r$"\index{Tour!restreinte}, et nous notons
$$(res_r(F))$$
la tour obtenue en supprimant dans $(F)$ les $r$ premiers corps, i.e. $F_0,F_1,
\dots, F_{r-1}$ :\\
$$(res_r(F)) \qquad F_r \leq \dots \leq F_i \leq \dots \leq F_m=L \:.$$
\pespace \noindent (2) "Tour ratio de $(F)$ à l'indice $r$"\index{Tour!ratio}, et nous notons
$$(rat_r(F))$$
la tour obtenue en supprimant dans $(F)$ les $m-r$ derniers corps, i.e. $F_{r+1},
\dots, F_{m}$ :\\
$$(rat_r(F)) \qquad K=F_0 \leq \dots \leq F_i \leq \dots \leq F_r \:.$$
\pespace \noindent (3) "Tour inflatée à $L$ de $(rat_r(F))$"\index{Tour!inflatée}, et nous notons
$$\text{\rm \Large(}inf_L(rat_r(F))\text{\rm \Large)} \quad \text{ou} \quad (inf_{L,r}(F))$$
la tour obtenue en conservant tous les corps $F_i$ de $(rat_r(F))$ sauf le
dernier, que l'on remplace par $L$ :
$$(inf_{L,r}(F)) \qquad K=F_0 \leq \dots \leq F_i \leq \dots \leq F_{r-1} \leq L \:.$$
\pespace \noindent (4) Lorsque la tour $(F)$ est stricte (Chap. 2, Déf. \&
Conv. 1.1.(1)), 
on écrit par abus de notation $(res_{F_r}(F))$ (resp. $(rat_{F_r}(F))$) au lieu
de  $(res_r(F))$ (resp. $(rat_r(F))$).
\end{defn}

\gespace
Pour fixer les idées, donnons l'immédiat

\begin{fait}
Dans les notations de la définition \ref{def:resrat}, et en convenant d'écrire
$(K)$ (resp. $(L)$) la tour réduite au seul corps $K$ (resp. $L$), on a :\\
$$\begin{array}{lcl}
(res_0(F))=(F) \quad	&,	&\qquad (res_m(F))=(L) \;;\\
(rat_0(F))=(K) \quad	&,	&\qquad (rat_m(F))=(F) \;;\\
(inf_{L,0}(F))=(L)	&,	&\qquad (inf_{L,m}(F))=(F) \;.
\end{array}$$
\end{fait}

\gespace
La proposition suivante justifie les définitions qui précèdent.

\mespace
\begin{prop} [] \label{prop:raffinementresrat}
Soient $L/K$ une extension quelconque et
$$(F) \qquad K=F_0 \leq \dots \leq F_i \leq \dots \leq F_m=L $$
une tour de $L/K$. Soit $r$ un entier fixé quelconque dans $\{0, \dots, m \}$.
\pespace
\noindent (1) Pour tout raffinement
$$\begin{array}{r}
(E) \qquad K=E_0 \leq \dots \leq E_{j_0}=F_0=K \leq \dots \leq
E_{j_i}=F_i \leq  \dots \qquad\qquad\\ 
\dots \leq E_{j_r}=F_r \leq \dots \leq E_{j_m}=F_m=L \leq \dots \leq E_n=L 
\end{array} $$
 de $(F)$ (cf. Rem. \ref{rem:raffinement}.(2)), la tour restreinte (resp. ratio)
 à l'indice $j_r$ de $(E)$, i.e. $(res_{j_r}(E))$ {\rm \Large (}resp.
 $(rat_{j_r}(E))${\rm \Large )},  est un  raffinement de $(res_r(F))$ {\rm \Large
 (}resp. $(rat_r(F))${\rm \Large )}.
\pespace
\noindent (2) Réciproquement, pour tout raffinement $(S)$ de $(res_r(F))$ et
tout raffinement $(R)$ de $(rat_r(F))$, il existe un unique raffinement $(E)$ de
$(F)$ tel que l'on ait à la fois $(res_{j_r}(E))=(S)$ et $(rat_{j_r}(E))=(R)$.
\end{prop}

\pespace
\begin{proof}
(1) Puisque $(E)$ raffine $(F)$, on a par définition (Déf. \&
Conv. \ref{def:raffinement}.(1))
$$0 \leq j_0 < \dots < j_i < \dots < j_r < \dots < j_m \leq n $$
avec
$$ \forall i \in \{ 0, \dots, m \} \quad F_i=E_{j_i} \:.$$
Alors à l'évidence
$$j_r \leq j_r < \dots < j_m \leq n  \qquad \text{(resp.}\quad  0 \leq j_0 < \dots <
  j_r \leq j_r \text{ )}$$
avec
$$ \forall i \in \{ r, \dots, m \} \quad F_i=E_{j_i} \qquad \text{(resp.}\quad 
\forall i \in \{ 0, \dots, r \} \quad F_i=E_{j_i}\text{ )}\:.$$
Ceci établit directement que
$$\begin{array}{r}
(res_{j_r}(E)) \qquad F_r=E_{j_r} \leq \dots \leq E_{j_i}=F_i \leq \dots \leq E_j
\leq  \dots \qquad \qquad \\
\dots \leq E_{j_m}=F_m=L \leq \dots  \leq E_n=L 
\end{array} $$
$$ \begin{array}{r}
\text{{\LARGE (} resp. } \; (rat_{j_r}(E)) \qquad K=E_0=F_0 \leq \dots \leq
E_{j_0}=F_0=K \leq \dots \qquad   \\
\dots \leq E_{j_i}=F_i \leq \dots \leq E_j \leq \dots \leq E_{j_r}=F_r \quad\text{{ \LARGE )}}
\end{array} $$
est un raffinement de
$$(res_r(F)) \qquad F_r \leq \dots \leq F_i \leq \dots \leq F_m=L $$
$$\text{{\Large (} resp. } \;(rat_r(F)) \qquad K=F_0 \leq \dots \leq F_i \leq
\dots \leq F_r \quad\text{{ \Large )}}\:.$$
\pespace
\noindent (2) \'Ecrivons :
$$(S) \qquad F_r=S_0 \leq \dots \leq S_k \leq \dots \leq S_p=L \quad$$
$$(R) \qquad K=R_0 \leq \dots \leq R_l \leq \dots \leq R_q=F_r \:.$$
Comme $(S)$ (resp. $(R)$) raffine $(res_r(F))$ (resp. $(rat_r(F))$), on a par
définition
$$0 \leq k_r < \dots < k_m \leq p \qquad \text{(resp.}\quad  0 \leq l_0 < \dots <
  l_r \leq q \text{ )}$$
avec
$$ \forall i \in \{ r, \dots, m \} \quad F_i=S_{k_i} \qquad \text{(resp.}\quad 
\forall i \in \{ 0, \dots, r \} \quad F_i=R_{l_i}\text{ )}\:.$$
Posons :
$$\forall j \in \{ 0, \dots, q \} \quad E_j:=R_j \;,\quad \forall j \in \{ q+1,
\dots, q+p \} \quad E_j:=S_{j-q} \;;$$
$$ \forall i \in \{ 0, \dots, r-1 \} \quad j_i:=l_i \;, \quad j_r:=q \;, \quad \forall
i \in \{ r+1, \dots, m \} \quad j_i:=q+k_i\:.$$
On a la suite d'indices
$$\begin{array}{r}
0 \leq j_0=l_0 < \dots < j_{r-1}=l_{r-1} < j_r=q < j_{r+1}=q+k_{r+1} < \dots \\
\dots <j_m=q+k_m \leq q+p \;.
\end{array}$$
En effet
$$l_{r-1} < l_r \leq q \quad \implique \quad j_{r-1} < j_r$$
et
$$0 \leq k_r < k_{r+1} \quad \implique \quad j_r=q \leq q+k_r <
q+k_{r+1}=j_{r+1} \;.$$
Tout ceci avec
$$\forall i \in \{ 0, \dots, r-1 \} \quad F_i=R_{l_i}=E_{l_i}=E_{j_i}$$
$$F_r=R_q=E_q=E_{j_r}$$
$$\forall i \in \{ r+1, \dots, m \} \quad F_i=S_{k_i}=E_{q+k_i}=E_{j_i}$$
(ce dernier cas car
$$0 \leq k_r < k_{r+1} \leq k_i \leq p \; \implique \; 1 \leq k_i \leq p \;
\implique \; q+1 \leq q+k_i \leq q+p \;).$$
Finalement, il est prouvé que
$$\forall i \in \{ 0, \dots, m \} \quad F_i=E_{j_i}$$
ce qui exprime que $(E)$ est un raffinement de $(F)$.
\pespace
Montrons enfin que cette tour $(E)$ vérifie bien les conditions souhaitées.
Posons $n:=q+p$. D'après la définition \ref{def:resrat}
$$\begin{array}{rcccccccl}
(res_{j_r}(E))\qquad F_r=	&E_{j_r=q}	&\leq	&E_{q+1}	&\leq
\dots \leq	& E_j	& \leq \dots \leq	&E_n	&=L \\
	&\text{\rotatebox[origin=c]{90}{=}}	&
&\text{\rotatebox[origin=c]{90}{=}}	&
&\text{\rotatebox[origin=c]{90}{=}}	&
&\text{\rotatebox[origin=c]{90}{=}}	&\\
&S_0	&\leq 	&S_1	&\leq \dots \leq	& S_{k=j-q}	& \leq \dots
\leq	&S_p	& \\
\end{array}$$
ce qui exprime que $(res_{j_r}(E))=(S)$. De même, trivialement
$$\begin{array}{rcccccl}
(rat_{j_r}(E)) \qquad K=	&E_0	&\leq \dots \leq	& E_j	& \leq \dots
\leq	&E_{j_r=q}&=F_r \\
	&\text{\rotatebox[origin=c]{90}{=}}	&
&\text{\rotatebox[origin=c]{90}{=}}	&
&\text{\rotatebox[origin=c]{90}{=}}	&\\
&R_0	&\leq \dots \leq	& R_j	& \leq \dots \leq	&R_q	& \\
\end{array}$$
ce qui exprime que  $(rat_{j_r}(E))=(R)$.

Prouvons maintenant l'unicité de l'énoncé du (2). Soit $(E')$ une tour de corps
telle que $(res_{j_r}(E'))=(S)$ et $(rat_{j_r}(E'))=(R)$.
Si l'on désigne par $n'$ la hauteur de $(E')$, la première de ces deux égalités
implique directement que
$$n'=j_r+p=n \;.$$
Il faut prouver que $E'_j=E_j$ pour tout $j \in \{ 0, \dots, m \}$ (cf.
Chap. 2, Déf. \& Conv. 1.1.(3)).\\
- Si $j \in \{ 0, \dots, q=j_r \}$ 
$$\begin{array}{cccccccc}
(rat_{j_r}(E'))	&\qquad	&E'_0	&\leq \dots \leq	& E'_j	& \leq \dots
\leq	&E'_{j_r} &\qquad\qquad\\
\text{\rotatebox[origin=c]{90}{=}}	&	&\text{\rotatebox[origin=c]{90}{=}}	&
&\text{\rotatebox[origin=c]{90}{=}}	&
&\text{\rotatebox[origin=c]{90}{=}}	&\\
(R)	&	&R_0	&\leq \dots \leq	& R_j	& \leq \dots \leq	&R_{j_r}
&\\
\end{array}$$
d'où
$$\forall j \in \{ 0, \dots, q \} \quad E'_j=R_j=E_j \;.$$
- Si $j \in \{ q+1, \dots, q+p=n \}$ 
$$\begin{array}{ccccccccc}
(res_{j_r}(E')) &\qquad 	&E'_{j_r}	&\leq	&E'_{j_r+1}	&\leq \dots \leq
& E'_j	& \leq \dots \leq	&E'_n \\
\text{\rotatebox[origin=c]{90}{=}}	&	&\text{\rotatebox[origin=c]{90}{=}}	&
	&\text{\rotatebox[origin=c]{90}{=}}	&
&\text{\rotatebox[origin=c]{90}{=}}	&
&\text{\rotatebox[origin=c]{90}{=}}	\\
(S)	&	&S_0	&\leq	&S_1	&\leq \dots \leq	& S_{j-q}	& \leq
\dots \leq	&S_p	\\
\end{array}$$
d'où
$$\forall j \in \{ q+1, \dots, n \} \quad E'_j=S_{j-q}=E_j \;.$$
Ceci achève la démonstration de la proposition \ref{prop:raffinementresrat}.
\end{proof}

\gespace
Quelle est maintenant la version de la proposition \ref{prop:raffinementresrat}
pour les raffinements propres ou les raffinements galoisiens ? Nous aurons besoin
du

\pespace

\begin{lem} []\label{lem:raffinementresrat}
Soient $L/K$ une extension quelconque,
$$(F) \qquad K=F_0 \leq F_1 \leq \dots \leq F_i \leq \dots \leq F_m=L$$
une tour de $L/K$, et $r$ un entier fixé dans $\{ 0, \dots, m \}$. Pour tout
raffinement
$$\begin{array}{r}
(E) \qquad K=E_0 \leq \dots \leq E_{j_0}=F_0=K \leq \dots \leq
E_{j_i}=F_i \leq \dots \leq E_j \leq \dots \\ 
\dots \leq E_{j_m}=F_m=L \leq \dots \leq E_n=L 
\end{array}$$
de $(F)$, on a les implications suivantes :
$$\text{\rm (1)} \quad \forall j \in \{ 1, \dots, j_r-1 \} \quad 
(\exists \, i \in \{r+1, \dots, m \} \quad  E_j=F_i ) \quad \implique \quad E_j=F_r
\:. \quad \;$$
$\;\text{\rm (1-1)} \quad \forall j \in \{ 1, \dots, j_r-1 \}$,
$$(\forall i \in \{0, \dots, r \} \quad  E_j \neq F_i) \; \implique \;  (\forall
i \in \{0, \dots, m \} \quad  E_j \neq F_i) \;.$$
$$\begin{array}{rcl}
\text{\rm (1-2)}\qquad 	&\text{\rm \LARGE ( } \exists \,j \in \{ 1, \dots, j_r-1 \} \quad
\forall i \in \{0, \dots, r \} \quad  E_j \neq F_i \: \text{\rm \LARGE )}
&\qquad \qquad \qquad\\
&	\Downarrow &\\
&	\text{\rm \LARGE ( } \exists \,j \in \{ 1, \dots, n-1 \} \quad
\forall i \in \{0, \dots, m \} \quad  E_j \neq F_i \: \text{\rm \LARGE
)}\;.&
\end{array}$$

$$\text{\rm (2)} \quad \forall j \in \{ j_r+1, \dots, n-1 \} \quad 
(\exists \, i \in \{0, \dots, r-1 \} \quad  E_j=F_i ) \quad \implique \quad E_j=F_r
\:.$$
$\;\text{\rm (2-1)} \quad \forall j \in \{ j_r+1, \dots, n-1 \}$,
$$(\forall i \in \{r, \dots, m \} \quad  E_j \neq F_i) \; \implique \;  (\forall
i \in \{0, \dots, m \} \quad  E_j \neq F_i) \;.$$
$$\begin{array}{rcc}
\text{\rm (2-2)}\qquad 	&\text{\rm \LARGE ( } \exists \,j \in \{ j_r+1, \dots, n-1 \} \quad
\forall i \in \{r, \dots, m \} \quad  E_j \neq F_i \: \text{\rm \LARGE )}
&\qquad \qquad \qquad\\
&	\Downarrow &\\
&	\text{\rm \LARGE ( } \exists \,j \in \{ 1, \dots, n-1 \} \quad
\forall i \in \{0, \dots, m \} \quad  E_j \neq F_i \: \text{\rm \LARGE )}\;.&
\end{array}$$
\end{lem}

\pespace
\begin{proof}
(1) Par la croissance des suites $\{F_i\}_{ 0 \leq i \leq m  }$ et
$\{E_j\}_{ 0 \leq j \leq  n }$,
$$F_r \leq F_{r+1} \leq F_i=E_j \leq E_{j_r-1} \leq E_{j_r}=F_r \quad \implique
\quad E_j=F_r \:.$$
(1-1) Soit $j \in \{ 1, \dots, j_r-1 \}$ tel que $E_j \neq F_i$ pour tout
$i \in \{0, \dots, r \}$.  
Raisonnons par l'absurde en supposant
$$\exists \,i \in \{0, \dots, m \} \quad  E_j=F_i \;.$$
Comme $\{0, \dots, m \}=\{0, \dots, r \} \cup \,\{r+1, \dots, m \}$,\\
\noindent - ou bien
$$(\exists \, i \in \{r+1, \dots, m \} \quad  E_j=F_i ) \quad \text{et} \quad 
E_j=F_r \quad \text{d'après (1) : contradiction ;} $$
\noindent - ou bien
$$(\exists \, i \in \{0, \dots, r \} \quad  E_j=F_i ) \quad \text{qui contredit
notre hypothèse.}$$
D'où la conclusion voulue.\\
(1-2) Cela découle du (1-1) en prenant le même indice
$$ j \in \{ 1, \dots, j_r-1 \} \subseteq \{ 1, \dots, n-1 \} $$
des deux côtés de l'implication.
$$(2)\qquad \qquad F_r=E_{j_r}\leq E_{j_r+1} \leq E_j=F_i \leq F_{r-1} \leq F_r \quad \implique
\quad E_j=F_r \:. \qquad \qquad $$
(2-1) On raisonne par l'absurde comme dans la démonstration du (1-1) ci-dessus
en considérant $j \in \{j_r+1, \dots, n-1 \}$ tel que $E_j \neq F_i$ pour tout
$i \in \{r, \dots, m \}$. Ne pas avoir le résultat annoncé conduit à une
contradiction, ou bien par le (2) précédent, ou bien directement par hypothèse.\\
(2-2) Il découle quant à lui du (2-1) en prenant le même indice 
$$j \in \{j_r+1, \dots, n-1 \} \subseteq \{ 1, \dots, n-1 \}$$
des deux côtés de l'implication.
\end{proof}

\gespace
Avec la proposition \ref{prop:raffinementresrat} et le lemme
\ref{lem:raffinementresrat} précédent, nous pouvons énoncer la

\pespace
\begin{prop} [] \label{prop:rapresrat}
Soient $L/K$ une extension quelconque et
$$(F) \qquad K=F_0 \leq \dots \leq F_i  \leq \dots \leq F_m=L$$
une tour de $L/K$. Soit $r$ un entier fixé quelconque dans $\{0, \dots, m \}$.\\
(1) Pour tout raffinement propre (Déf. \& Conv.  \ref{def:raffinement}.(2)) $(E)$ de
$(F)$, \\
- ou la tour restreinte $(res_{j_r}(E))$ à l'indice $j_r$ de $(E)$ est un
raffinement propre de $(res_r(F))$ ;\\
- ou la tour ratio $(rat_{j_r}(E))$ à l'indice $j_r$ de $(E)$ est un
raffinement propre de $(rat_r(F))$.
\pespace
\noindent (2) Réciproquement, pour tout raffinement $(S)$ de $(res_r(F))$ et
tout raffinement $(R)$ de $(rat_r(F))$, tels que $(R)$ ou $(S)$ soit un
raffinement propre, l'unique raffinement $(E)$ induit par $(R)$ et $(S)$ (cf. Prop.
\ref{prop:raffinementresrat}.(2)) est un raffinement propre de $(F)$.
\end{prop}

\mespace
\begin{proof}
(1) On se place dans les notations de la démonstration de la proposition
\ref{prop:raffinementresrat}. Notre hypothèse signifie que (Déf. \& Conv. 
\ref{def:raffinement}.(2))
$$\exists \,j \in \{ 1, \dots, n-1 \} \quad \forall i \in \{0, \dots, m \} \quad
E_j \neq F_i \;;$$
donc $E_j \neq F_i$ pour tout $i$ dans $\{0, \dots, r \}$ ou $\{r, \dots, m
\}$. Or $j\neq j_r$ car sinon $E_j=E_{j_r}=F_r$ : contradiction. Donc\\
- ou bien 
$$\exists \,j \in \{ 1, \dots, j_r-1 \} \quad \forall i \in \{0, \dots, r \}
\quad  E_j \neq F_i $$
ce qui exprime que la tour ratio $(rat_{j_r}(E))$ est un raffinement propre de 
$(rat_r(F))$ ;\\
- ou bien 
$$\exists \,j \in \{ j_r+1, \dots, n-1 \} \quad \forall i \in \{r, \dots, m \}
\quad  E_j \neq F_i $$
ce qui exprime que la tour restreinte $(res_{j_r}(E))$ est un raffinement propre de 
$(res_r(F))$.

\pespace
\noindent (2) Les notations sont celles de la démonstration du (2) de la
proposition \ref{prop:raffinementresrat} :
$$(S) \qquad F_r=S_0 \leq \dots \leq S_k \leq \dots \leq S_p=L$$
$$(R) \qquad K=R_0 \leq \dots \leq R_l \leq \dots \leq R_q=F_r \:.$$
- Dire que (S) est un raffinement propre de $(res_r(F))$ signifie
$$\exists \, k \in \{1, \dots , p-1 \} \quad \forall i \in \{r, \dots, m \} \quad
S_k \neq F_i \;.$$
Donc pour $j:=q+k$,
$$\exists \, j \in \{q+1=j_r+1, \dots, q+p-1=n-1 \} \quad \forall i \in \{r, \dots, m \} \quad
S_{j-q} \neq F_i \;.$$
Mais, pour un tel $j$, on a posé $E_j:=S_{j-q}$. Ainsi
$$\exists \, j \in \{j_r+1, \dots, n-1 \} \quad \forall i \in \{r, \dots, m \} \quad
E_j \neq F_i \;.$$
D'après le (2-2) du lemme \ref{lem:raffinementresrat}, on en déduit que
$$\exists \, j \in \{1, \dots, n-1 \} \quad \forall i \in \{0, \dots, m \} \quad
E_j \neq F_i $$
ce qui exprime exactement que $(E)$ est un raffinement propre de $(F)$.
\pespace
\noindent - De la même façon, dire que $(R)$ est un raffinement propre de
$(rat_r(F))$ signifie
$$\exists \, l \in \{1, \dots , q-1 \} \quad \forall i \in \{0, \dots, r \} \quad
R_l \neq F_i \;.$$
Donc pour $j:=l$,
$$\exists \, j \in \{1, \dots , q-1=j_r-1 \} \quad \forall i \in \{0, \dots, r \} \quad
E_j=R_j \neq F_i \;.$$
D'après le (1-2) du lemme \ref{lem:raffinementresrat}, on en déduit que
$$\exists \, j \in \{1, \dots , n-1 \} \quad \forall i \in \{0, \dots, m \} \quad
E_j \neq F_i $$
ce qui exprime encore une fois que $(E)$ est un raffinement propre de $(F)$. 
\end{proof}

\gespace
Voici l'analogue galoisien de la proposition \ref{prop:raffinementresrat}.

\pespace
\begin{prop}[]\label{prop:ragresrat}
Dans les notations respectives de la proposition \ref{prop:raffinementresrat}:\\
(1) Pour tout raffinement galoisien $(E)$ de $(F)$, la tour restreinte $(res_{j_r}(E))$
{\rm \Large (}resp. la tour ratio $(rat_{j_r}(E))${\rm \Large )} est un raffinement galoisien de
$(res_r(F))$ {\rm \Large (}resp. $(rat_r(F))${\rm \Large )}.
\pespace
\noindent (2) Réciproquement, pour tout raffinement galoisien $(S)$ de
$(res_r(F))$ et tout raffinement galoisien $(R)$ de $(rat_r(F))$, il existe un
unique raffinement galoisien $(E)$ de $(F)$ tel que l'on ait à la fois $(res_{j_r}(E))=(S)$ et
$(rat_{j_r}(E))=(R)$. Ce raffinement est celui de la proposition
\ref{prop:raffinementresrat}.
\end{prop}

\pespace
\begin{proof}
(1) Par définition d'un raffinement galoisien (Déf.
\ref{def:raffinementsuite}.(4)), la tour $(E)$ vérifie la condition
$$\text{(RAFG)} \quad  \forall j \in \{1, \dots, n-1\} \; \text{ \Large ( }
\forall i \in \{0, \dots, m\} \quad E_j\neq F_i \text{ \Large ) } \; \implique
\; E_{j-1} \unlhd E_j   \:.$$
- Considérons $j \in \{1, \dots, j_r-1\}$ tel que $E_j \neq F_i$ pour tout
$i \in \{0, \dots, r\}$.
D'après le (1-1) du lemme \ref{lem:raffinementresrat},
$$\forall i \in \{0, \dots, m\} \quad E_j\neq F_i \;,$$
d'où $E_{j-1} \unlhd E_j$ par la condition (RAFG) ci-dessus. Ceci prouve que
$(rat_{j_r}(E))$ est un raffinement galoisien de $(rat_r(F))$.
\pespace
\noindent - Considérons $j \in \{j_r+1, \dots, n-1\}$ tel que $E_j \neq F_i$ pour
tout $i \in \{r, \dots, m\}$.
D'après le (2-1) du lemme \ref{lem:raffinementresrat},
$$\forall i \in \{0, \dots, m\} \quad E_j\neq F_i \;,$$
d'où $E_{j-1} \unlhd E_j$ par la condition (RAFG). Ceci prouve que
$(res_{j_r}(E))$ est un raffinement galoisien de $(res_r(F))$.
\pespace
\noindent (2) Les notations sont celles du (2) de la proposition
\ref{prop:raffinementresrat} :
$$(S) \qquad F_r=S_0 \leq \dots \leq S_k \leq \dots \leq S_p=L \quad $$
$$(R) \qquad K=R_0 \leq \dots \leq R_l \leq \dots \leq R_q=F_r \:.$$
Prouvons que le raffinement $(E)$ qui y est construit est nécessairement
galoisien. Soit $j \in \{1, \dots, n-1\}=\{1, \dots, j_r-1\}\cup \{j_r\}
\cup\{j_r+1, \dots, n-1\}$ tel que $E_j \neq F_i$ pour tout $i \in \{0, \dots,
m\}$. On ne peut pas avoir $j=j_r$ puisque $E_{j_r}=F_r$. Par conséquent :\\
- ou bien $j \in \{1, \dots, j_r-1=q-1\}$ et l'on a en particulier
$$\forall i \in \{0, \dots, r\} \quad R_j=E_j\neq F_i \;.$$
La condition (RAFG) vérifiée par $(R)$ assure alors que
$$E_{j-1}=R_{j-1} \unlhd R_j=E_j \;;$$
\pespace
\noindent - ou bien $j \in \{j_r+1=q+1, \dots, n-1=q+p-1\}$ et pour $k:=j-q$
$$\forall i \in \{r, \dots, m\} \quad S_k=E_j\neq F_i \;.$$
La condition (RAFG) vérifiée par $(S)$ assure alors que
$$E_{j-1}=S_{k-1} \unlhd S_k=E_j \;.$$

Dans tous les cas donc, $E_{j-1}\unlhd E_j$ ce qui exprime que $(E)$ est un
raffinement galoisien de $(F)$.
\end{proof}

\gespace
Les propositions \ref{prop:rapresrat} et \ref{prop:ragresrat} se combinent en la
proposition suivante pour les raffinements galoisiens propres.

\pespace
\begin{prop}[] \label{prop:ragpresrat}
Soient $L/K$ une extension quelconque et
$$(F) \qquad K=F_0  \leq \dots \leq F_i \leq \dots \leq F_m=L$$
une tour de $L/K$. Soit $r$ un entier fixé quelconque dans $\{0, \dots, m \}$.\\
(1) Pour tout raffinement galoisien propre $(E)$ de $(F)$ (Déf. \& Conv. 
\ref{def:raffinement}.(2)) \\
- ou la tour restreinte $(res_{j_r}(E))$ à l'indice $j_r$ de $(E)$ est un
raffinement galoisien propre de $(res_r(F))$ ;\\
- ou la tour ratio $(rat_{j_r}(E))$ à l'indice $j_r$ de $(E)$ est un
raffinement galoisien propre de $(rat_r(F))$.
\pespace
\noindent (2) Réciproquement, pour tout raffinement galoisien $(S)$ de $(res_r(F))$ et
tout raffinement galoisien $(R)$ de $(rat_r(F))$, tels que $(R)$ ou $(S)$ soit un
raffinement propre, l'unique raffinement $(E)$ induit par $(R)$ et $(S)$ (cf. Prop.
\ref{prop:raffinementresrat}.(2)) est un raffinement galoisien propre de $(F)$.
\end{prop}

\pespace
\begin{proof}
(1) Par le (1) de la proposition \ref{prop:ragresrat}, $(res_{j_r}(E))$ et
$(rat_{j_r}(E))$ sont des raffinements galoisiens. Le (1) de la proposition
\ref{prop:rapresrat} assure que l'un ou l'autre est un raffinement propre. C'est
donc que $(res_{j_r}(E))$ ou $(rat_{j_r}(E))$ est un raffinement galoisien
propre.
\pespace
\noindent (2) Par le (2) de la proposition \ref{prop:ragresrat}, $(E)$ est un raffinement
galoisien de $(F)$. Le (2) de la proposition \ref{prop:rapresrat} assure que
$(E)$ est un raffinement propre de $(F)$. Autrement dit, $(E)$ est un
raffinement galoisien propre de $(F)$ (cf. scholie de la Déf.
\ref{def:raffinementsuite}).
\end{proof}

\addtocontents{toc}{\mespace\pespace}
\chapter{PREMIERS TH\'EOR\`EMES DE DISSOCIATION}
\addtocontents{lof}{\mespace\pespace}
\addtocontents{lof}{\noindent Chapitre \thechapter}
\addtocontents{toc}{\pespace}

\gespace
En théorie des groupes, on connaît les deux célèbres théorèmes :
\pespace
{\bf Théorème de Schreier} (\cite{Schr}, \cite{Hu})\index{Théorème!de Schreier}\\
{\it Deux suites normales d'un même groupe admettent des raffinements
équivalents.}

\mespace
{\bf Théorème de Jordan-Hölder} (\cite{Jo}, \cite{Hö},
\cite{Hu})\index{Théorème!de Jordan-Hölder}\\
{\it Soit $G$ un groupe admettant une suite de composition.\\
(i) Toute suite normale stricte de $G$ admet un raffinement qui est une suite de
composition de $G$.\\
(ii) Deux suites de composition de $G$ sont équivalentes.}

\gespace
En lieu et place d'un groupe $G$, nous considérons ici une extension galtourable
$L/K$. Nous remplaçons les suites normales de $G$ et ses suites de composition
par les tours galoisiennes de $L/K$ et ses "tours de composition". Notre but
dans ce chapitre 4 est d'établir un analogue galoisien aux théorèmes de Schreier
et de Jordan-Hölder. Pour cela, nous dévissons, nous dissocions les tours
galoisiennes de $L/K$ autant que nécessaire de façon à obtenir des "tours
équivalentes" (de marches à groupes de Galois isomorphes à l'ordre près) qui
n'admettent aucun raffinement galoisien propre. Nous appelons "théorèmes de
dissociation" les théorèmes ainsi obtenus, et nous les généraliserons dans le
chapitre 7 final à toutes les extensions algébriques finies.\index{Dissociation}

\gespace
\gespace
\section[Tours de composition galoisiennes]{Tours de composition galoisiennes,\newline tours galoisiennes
équivalentes}
\gespace
\begin{defn} \label{def:tourgalcomp} 
Soient $L/K$ une extension galtourable et
$$(F) \qquad K=F_0 \unlhd \dots \unlhd F_i \unlhd \dots \unlhd F_m=L$$
une tour galoisienne de $L/K$.
\pespace
\noindent
(1) Nous disons que $(F)$ est "une tour de composition galoisienne de
$L/K$"\index{Tour!de composition} si
et seulement si elle est stricte et n'admet aucun raffinement galoisien propre.
\pespace
\noindent
(2) Soit
$$\begin{array}{c}
(E) \qquad K=E_0 \unlhd \dots \unlhd E_j \unlhd \dots
\unlhd E_n=L\\
\end{array}\\$$
une autre tour galoisienne de $L/K$. Nous disons que $(E)$ et $(F)$ sont
"équivalentes"\index{Tour!equivalente@équivalente}, et nous notons $(E) \sim
(F)$, si et seulement si elles ont même
nombre de marches : $m=n$, et si, à permutation près, les groupes de Galois de
ces marches sont isomorphes (topologiquement en degrés infinis) :
$$\exists \sigma \in S_m \quad \forall i \in \{1, \dots, m=n \} \quad
Gal(F_i/F_{i-1}) \isomto Gal(E_{\sigma(i)}/E_{\sigma(i)-1}) \;.$$
\end{defn}

\mespace
Un cas très particulier de cette définition est fourni par le

\mespace
\begin{fait}  \label{fait:tourcompositiontriviale}
L'extension triviale $L=K$ admet une tour de composition galoisienne et une seule, celle à zéro
marche :
$$(C) \qquad K=F_0=L \;.$$
\end{fait}

\mespace
\begin{proof}
D'après le (1) de la définition \& convention 1.1 du chapitre 2, 
la tour triviale $(C)$ est stricte. Raisonnons par l'absurde en supposant que
$(C)$ admette un raffinement galoisien propre
$$(E) \qquad K=E_0 \unlhd \dots \unlhd E_{j_0}=F_0 \unlhd \dots \unlhd E_j
\unlhd \dots \unlhd E_n=L \;.$$
Par la condition (RAF3) de la définition \& convention 1.1 du
chapitre 3, 
$$\exists j \in \{1,\dots,n-1\} \quad \forall i \in \{0,\dots,m\} \quad E_j \neq
F_i \;.$$
En particulier $E_j \neq F_0 =K$. Dès lors, par la croissance de $(E)$, 
$$K=E_0 \leq E_j \leq E_n = L=K \;;$$
d'où $E_j=K$ : contradiction. Il est donc établi que $(C)$ est une tour de
composition galoisienne. Montrons que $(C)$ est l'unique tour de composition de $L/K$. Soit
$(C')$ une tour de composition galoisienne de $L/K$ de hauteur $m'$. Comme
$[L:K]=1$, on déduit du Fait 1.3 du chapitre 2 
que nécessairement $m'=0$.
Donc $(C')$ est la tour triviale et par conséquent $(C')=(C)$.
\end{proof}

\mespace
La définition \ref{def:tourgalcomp} précédente sera étendue au chapitre 7 aux
tours quelconques d'une extension finie via la notion de "tour d'élévation". En
particulier, il ne suffira pas d'y enlever les qualificatifs "galoisiens".

\gespace
En théorie des groupes, on connaît la 

\mespace
{\bf Proposition.}
{\it Pour qu'une suite normale soit de composition, il faut et il suffit que
chacun de ses facteurs soit simple.}

\mespace
Voici son analogue galoisien :

\mespace
\begin{prop} \label{prop:compssigalgalsimple}
Soit $L/K$ une extension galtourable quelconque. Pour qu'une tour galoisienne de
$L/K$ soit de composition (cf. Déf. \ref{def:tourgalcomp}.(1)), il faut et il
suffit que chacune de ses marches soit galsimple (Chap.2, Déf. 1.6.(2)). 
\end{prop}

\mespace
\begin{proof}
Montrons d'abord que l'équivalence est vraie pour l'extension triviale $L=K$. Par
le Fait \ref{fait:tourcompositiontriviale}, celle-ci admet une tour de
composition et une seule, celle à zéro marche :
$$(C) \qquad K=F_0=L \;.$$
Soit donc $(F)$ une tour de composition galoisienne de $L=K$. Comme $(F)=(C)$, la
galsimplicité des marches de $(F)$ est vérifiée puisqu'il n'y en a pas.
\pespace
Inversement, soit $(F)$ une tour de $L/K$ dont toutes les marches sont
galsimples. Si $(F)$ était de hauteur non nulle, elle admettrait donc une marche
stricte : contradiction.
Finalement $(F)$ est de hauteur nulle ; c'est la tour de composition galoisienne $(C)$.

\pespace
Supposons maintenant l'extension $L/K$ non triviale, et soit 
$$(F) \qquad K=F_0 \unlhd \dots \unlhd F_i \unlhd F_{i+1} \unlhd \dots \unlhd
F_m=L$$
une tour galoisienne de $L/K$. Supposons que $(F)$ soit une tour de composition
galoisienne et que l'une de ses marches $F_{i_0+1} \gal F_{i_0}$ ne soit pas galsimple.
Par définition de la galsimplicité et le fait que la marche est galoisienne, il existe un corps $F$ tel que
$$ F_{i_0} \lhd F \lhd F_{i_0+1} \;.$$
Soit alors $(E)$ la tour définie par
$$\left\{\begin{array}{cl}
\forall j \in \{0, \dots, i_0\} \quad	&E_j=F_j \\
					&E_{i_0+1}=F\\
\forall j \in \{i_0+2, \dots, m+1 \} \quad	&E_j=F_{j-1} \, \\
\end{array}\right.$$
i.e.
$$(E) \quad K=E_0 \unlhd \dots \unlhd E_{i_0}=F_{i_0} \unlhd E_{i_0+1}=F 
\unlhd E_{i_0+2}=F_{i_0+1} \unlhd \dots \unlhd E_{m+1}=F_m=L \;.$$
Comme elle est galoisienne, cette tour $(E)$ est un raffinement galoisien de
$(F)$ (cf. Chap. 3, Fait 1.5.(2)). 
De plus
$$\left\{\begin{array}{ccc}
\forall i \in \{0, \dots , i_0\}	&(F_{i} \leq F_{i_0} < F)
&\implique \quad F_i \neq F \\
\forall i \in \{i_0+1, \dots , m\}	&(F < F_{i_0+1} \leq F_{i})
&\;\;\implique \quad F_i \neq F \;. \\
\end{array}\right.$$
Donc $(E)$ est un raffinement propre de $(F)$.
On a ainsi construit un raffinement galoisien propre de $(F)$ : contradiction,
puisque $(F)$ est de composition.

\pespace
Inversement, supposons que toutes les marches de $(F)$ soient galsimples. Elles
sont donc en particulier toutes non triviales, et la tour $(F)$ est stricte.
Raisonnons par l'absurde en supposant l'existence d'un raffinement galoisien
propre $(E)$ de $(F)$. D'après la proposition 1.7 du chapitre 3, 
c'est une tour galoisienne :
$$(E) \qquad K=E_0=F_0 \unlhd \dots \unlhd E_{j_i}=F_i \unlhd \dots \unlhd E_{j}
\unlhd \dots \unlhd E_n=F_m=L \;.$$
De plus, par la définition et convention 1.1.(2) du chapitre 3, l'ensemble
$$\{j \in \{1, \dots , n-1\} \quad | \quad \forall i \in \{0,
\dots,m\} \quad E_j \neq F_i  \} $$
est non vide. Notons $l$ son plus petit élément. Deux cas :
\begin{itemize}
\item ou bien $l-1=0$, $E_{l-1}=F_0$ ;
\item ou bien $l-1 \geq 1$ et par minimalité de $l$, il existe $k$ dans $\{0,
\dots,m\}$ tel que  $E_{l-1}=F_k$.
\end{itemize}
Donc dans tous les cas
$$ \exists k \in \{0, \dots,m\} \quad E_{l-1} = F_k \;.$$
On n'a pas $E_{l-1}=F_m$ car alors
$$F_m=E_{l-1} \leq E_l \leq L = F_m \quad \implique \quad E_l=F_m 
\text{ : contradiction.}$$
Donc
$$ \exists k \in \{0, \dots,m-1\} \quad E_{l-1} = F_k \;.$$
Rappelons que $E_{j_{k+1}}=F_{k+1}$, et minorons à partir de là l'indice
$j_{k+1}$ : 
\pespace
\begin{itemize}
\item si $j_{k+1} \leq l-1$, 
$$F_{k+1}=E_{j_{k+1}} \leq E_{l-1}=F_k < F_{k+1} \text{ (car }(F)\text{ est
stricte) : absurde ;}$$
\item si $j_{k+1}=l$, $F_{k+1}=E_l$ : contradiction par définition de
$l$. 
\end{itemize}
\pespace
On a donc nécessairement $j_{k+1} \geq l+1$, de sorte que
$$F_k=E_{l-1} < E_l \leq E_{l+1} \leq \dots \leq E_{j_{k+1}}=F_{k+1} \;.$$
Retenons en particulier que
$$F_k=E_{l-1} < E_l < F_{k+1} \;.$$
Mais la marche $E_l / E_{l-1}$ est galoisienne ($(E)$ est une tour galoisienne)
; donc l'extension $F_{k+1}/F_k$ n'est pas galsimple : contradiction.
\end{proof}

\gespace
Le corollaire suivant prouve la compatibilité de la notion de tour de
composition avec l'équivalence des tours galoisiennes.

\mespace
\begin{cor} \label{cor:equivtourcomp}
Soit $L/K$ une extension galtourable et $(T)$, $(T')$ deux tours galoisiennes de
$L/K$. On suppose que $(T)$ et $(T')$ sont équivalentes (cf. Déf.
\ref{def:tourgalcomp}). Alors $(T)$ est de composition si et seulement si $(T')$
est de composition.
\end{cor}

\mespace \noindent
{\it Scholie. }Nous montrerons qu'une extension galtourable n'admet une tour de
composition galoisienne que lorsqu'elle est finie (cf. $2^{\text{ème}}$ théorème de dissociation :
Th. \ref{th:2diss}).

\mespace
\begin{proof}
Posons 
$$\begin{array}{c}
(T) \qquad K=T_0 \unlhd \dots \unlhd T_i \unlhd T_{i+1} \unlhd \dots \unlhd
T_m=L \;,\\
\\
(T') \qquad K=T'_0 \unlhd \dots \unlhd T'_j \unlhd T'_{j+1} \unlhd \dots \unlhd
T'_m=L \;.
\end{array}$$
Supposons que $(T)$ ne soit pas de composition et prouvons qu'il en est ainsi de
$(T')$. Par la proposition
\ref{prop:compssigalgalsimple} précédente, au moins
une marche de $(T)$ n'est pas galsimple :
$$ \exists i \in \{1, \dots , m\} \qquad T_i \nongalsimple T_{i-1} \;.$$
Ceci signifie, par définition de la galsimplicité (Chap. 2, Déf. 1.6.(2)), 
\begin{itemize}
\item ou bien que $T_i = T_{i-1}$ ;
\item ou bien qu'il existe un corps intermédiaire $F$ tel que
$$T_{i-1} \lhd F < T_i \;,$$
et donc, par la bijection de Krull classique,
$$Gal(T_i / T_{i-1}) \rhd_f Gal(T_i / F) \rhd \iit \;.$$
\end{itemize}
On sait qu'il existe $\sigma \in S_m$ (Déf. \ref{def:tourgalcomp}) tel que
$$Gal(T_i / T_{i-1}) \isomto Gal(T'_{\sigma(i)} / T'_{\sigma(i)-1}) \;.$$
Donc
\begin{itemize}
\item ou bien, dans le premier cas,  $T'_{\sigma(i)} = T'_{\sigma(i)-1}$ ;
\item ou bien, dans le deuxième, l'image $H$ de $Gal(T_i / F)$ par cet isomorphisme
topologique est telle que
$$Gal(T'_{\sigma(i)} / T'_{\sigma(i)-1}) \rhd_f H \rhd \iit \;.$$
On en déduit, par la réciproque de la bijection de Krull classique, la tour
$$T'_{\sigma(i)-1} \lhd T'^H_{\sigma(i)}  < T'_{\sigma(i)} \;.$$
\end{itemize}
Dans les deux cas, $(T')$ a une marche qui n'est pas galsimple, et par la proposition
\ref{prop:compssigalgalsimple} précédente, $(T')$ n'est pas de composition.
\end{proof}

\gespace
Il arrive que l'on ne sache pas décider si une tour est stricte ou pas. Dans la
quête de tours équivalentes strictes, la proposition suivante y remédie.

\mespace
\begin{prop} \label{prop:tourgalstrictequiv}
Soit $L/K$ une extension galtourable. Pour toutes tours galoisiennes de $L/K$
$$(F) \qquad K=F_0 \unlhd \dots \unlhd F_i \unlhd \dots \unlhd F_m=L$$
et
$$(E) \qquad K=E_0 \unlhd \dots \unlhd E_k \unlhd \dots \unlhd E_n=L$$
équivalentes, les tours strictes associées (Chap. 3, Prop. \& Déf. 2.1
)
$$(F_<) \qquad K=F_0=F_{<0} \lhd \dots \lhd F_{<j} \lhd \dots \lhd F_{<m'}=F_m=L$$
et
$$(E_<) \qquad K=E_0=E_{<0} \lhd \dots \lhd E_{<l} \lhd \dots \lhd E_{<n'}=E_n=L$$
sont équivalentes.
\end{prop}

\mespace
\begin{proof}
On sait déjà, par le corollaire 2.6 du chapitre 3, 
que $(E_<)$ et $(F_<)$ sont bien des tours galoisiennes (ce qui justifie leur
notation dans l'énoncé). Il nous reste ainsi à montrer qu'elles sont
équivalentes. Pour cela, prouvons que pour tout $j$ dans $\{1, \dots, m'\}$, il
existe un unique entier $i_j$ dans $\{1, \dots, m\}$ tel que l'on ait
$$F_{<j}=F_{i_j} \qquad \text{et} \qquad F_{<j-1}=F_{i_j-1} \;.$$
\pespace
\noindent
L'existence d'entiers vérifiant l'une ou l'autre condition est claire.
Examinons tout d'abord l'existence et l'unicité d'un entier vérifiant les deux
conditions simultanément.

Soit $j \in \{1, \dots, m'\}$ quelconque mais fixé.
Dans la démonstration du Fait 2.5 du chapitre 3, 
on a vu que l'entier
$$i_j := min \{i \in \{0, \dots, m\} \; \mid \; F_i=F_{<j} \}$$
(toujours $\geq 1$) convient.
Prenons un entier $i' \in \{1, \dots, m\}$ convenant également. La première
condition assure que $F_{i'}=F_{<j} $, et la minimalité de $i_j$ fournit alors
$$i_j \leq i' \;.$$
Dès lors, si $i_j \neq i' $, on a $i_j \leq i'-1$ ; et par croissance de $(F)$ et
stricte croissance de $(F_<)$ 
$$F_{<j-1} < F_{<j} = F_{i_j} \leq F_{i'-1} \;.$$ 
Mais $F_{<j-1} < F_{i'-1}$ signifie en particulier que $i'$ ne convient pas.
C'est donc que $i_j=i'$. Comme par définition $F_{i_j-1}=F_{<j-1} < F_{<j} =
F_{i_j}$ , nous pouvons considérer l'application
$$ \begin{array}{cccc}
\Phi_F \;:	&\{1, \dots, m'\}	&\longrightarrow	&\{i \in \{1, \dots, m\}
\; \mid \; F_i \neq F_{i-1} \}\\
		&j			&\longmapsto		&\Phi_F(j):=i_j\;.\\
\end{array}$$
On a
$$\forall j \in \{1, \dots, m'\} \quad F_{\Phi_F(j)}=F_{<j} \;, \quad
F_{\Phi_F(j)-1}=F_{<j-1} \;.$$
Montrons maintenant que $\Phi_F$ est une bijection.

{\emphase $\circ \; \Phi_F$ est injective :}\\
Comme l'ensemble de départ de $\Phi_F$ est totalement ordonné, il nous suffit de
prouver que $\Phi_F$  est strictement croissante. Puisque la tour $(F_<)$ est
stricte
$$\begin{array}{ccccc}
j < j'	&\implique	&F_{<j}					&<	&F_{<j'}\\
	&		&\text{\rotatebox[origin=c]{90}{=}}		&
	&\text{\rotatebox[origin=c]{90}{=}}	\\
	&\ssi		&F_{\Phi_F(j)}				&<	&F_{\Phi_F(j')}.
\end{array}$$
Ceci implique 
$$\Phi_F(j) < \Phi_F(j') \;.$$
En effet, si $\Phi_F(j') \leq \Phi_F(j)$, on aurait par la croissance de
$\{F_i\}_{\{0 \leq i \leq m \}}$,
$$F_{\Phi_F(j')} \leq F_{\Phi_F(j)} \quad : \quad \text{contradiction}.$$
En résumé
$$j < j' \quad \implique \quad \Phi_F(j) < \Phi_F(j') \;,$$
ce qui prouve l'injectivité de $\Phi_F$.

\pespace
{\emphase $\circ \; \Phi_F$ est surjective :}\\
Soit un entier $i$ quelconque mais fixé dans $\{1, \dots, m\}$ tel que $F_i \neq
F_{i-1}$. Comme $(F)$ raffine $(F_<)$ trivialement 
(Chap. 3, Prop. \& Déf. 2.1  
), on a, par la condition (RAFT) (Chap. 3, Déf 1.3.(2)
)
$$\exists j \in \{0, \dots, m'\} \quad F_i=F_{<j} \;.$$
Observons tout d'abord que $j$ ne peut être nul. En effet
$$j=0 \quad \implique \quad K=F_0 \leq F_{i-1} < F_i=F_{<0}=K \quad \implique
\quad K < K \quad : \quad \text{contradiction.}$$
C'est donc que $j \in \{1, \dots, m'\}$. 
On peut ainsi considérer son image par $\Phi_F$ : soit $i_j:=\Phi_F(j)$.
Par définition de $\Phi_F(j)$, $F_{i_j} \neq F_{i_j-1}$ et $F_{i_j}=F_{<j}=F_i$.
Supposons que les entiers $i_j$ et $i$ soient différents. Deux cas se
présentent :
\begin{itemize}
\item ou bien
$$ i_j < i \quad \ssi \quad i_j \leq i-1 \quad \implique \quad F_{i_j} \leq
F_{i-1} < F_i =F_{i_j} \; : \; \text{absurde ;}$$
\item ou bien
$$ i < i_j \quad \ssi \quad i \leq i_j-1 \quad \implique \quad F_{i} \leq
F_{i_j-1} < F_{i_j} =F_{i} \; : \; \text{absurde.}$$
\end{itemize}
On vient bien de montrer que $i=i_j=\Phi_F(j)$. Ceci étant vrai pour tout $i$
dans $\{1, \dots, m\}$ tel que $F_i \neq F_{i-1}$, on a prouvé la surjectivité
de $\Phi_F$.

\pespace
Nous avons besoin pour la suite de l'analogue pour $E$ de $\Phi_F$, à savoir la
bijection $\Phi_E$ définie par
$$ \begin{array}{cccc}
\Phi_E \;:	&\{1, \dots, n'\}	&\longrightarrow	&\{k \in \{1,
\dots, n\} \; \mid \; E_k \neq E_{k-1} \}\\
		&l			&\longmapsto		&\Phi_E(l):=k_l\\
\end{array}$$
où $k_l$ est l'unique entier vérifiant
$$E_{k_l}=E_{<l} \qquad \text{et} \qquad E_{k_l-1}=E_{<l-1} \;.$$
Nous avons fait l'hypothèse que les tours galoisiennes $(E)$ et $(F)$ sont
équivalentes : $(E) \sim (F)$, ce qui signifie (Déf. \ref{def:tourgalcomp}.(2))
que $m=n$ et
$$\exists \sigma \in S_m \quad \forall i \in \{1, \dots, m=n \} \quad
Gal(F_i/F_{i-1}) \isomto Gal(E_{\sigma(i)}/E_{\sigma(i)-1}) \;.$$
Pour prouver que $(E_<) \sim (F_<)$, notons tout d'abord le banal, mais décisif,
argument de théorie de Galois :
$$\begin{array}{ccc}
F_i / F_{i-1} \text{ non triviale}	&\ssi	&Gal(F_i / F_{i-1}) \neq
\iit\\
					&\ssi	&Gal(E_{\sigma(i)}/E_{\sigma(i)-1})\neq
\iit\\
					&\ssi	&E_{\sigma(i)} / E_{\sigma(i)-1}
					\text{ non triviale.}
\end{array}$$
Nous pouvons donc considérer la restriction $\sigma_{\mid}$ de $\sigma$ aux
marches non triviales :
$$ \begin{array}{cccc}
\sigma_{\mid} \;:	&\{i \in \{1, \dots, m=n\} \; \mid \; F_i \neq F_{i-1} \}	
&\longrightarrow	&\{k \in \{1,\dots, n\} \; \mid \; E_k \neq E_{k-1} \}\\
		&i			&\longmapsto		&\sigma(i) \;.\\
\end{array}$$
Puisque $\sigma$ est surjective, l'équivalence ci-dessus assure que
$\sigma_{\mid}$ est également surjective. L'injectivité de $\sigma_{\mid}$ est
une conséquence directe de celle de $\sigma$. Par conséquent, $\sigma_{\mid}$
est aussi une bijection.

\pespace
\noindent
Considérons dès lors le composé des bijections
$$ \begin{array}{cccc}
\tau := \Phi_E^{-1} \circ \sigma_{\mid} \circ \Phi_F \;:	&\{1, \dots, m'\}	
&\longrightarrow	&\{1, \dots, n'\}\\
						&j	&\longmapsto		
&(\Phi_E^{-1} \circ \sigma_{\mid} \circ \Phi_F)(j)\;.\\
\end{array}$$
Comme $\tau$ est encore une bijection, on a en particulier $m'=n'$ et $\tau \in
S_{m'}$.\\
Ainsi, par les conditions que vérifie $\Phi_F(j)$ et la propriété de définition
de $\sigma$, on a pour tout $j \in \{1, \dots, m'\}$
$$Gal(F_{<j}=F_{\Phi_F(j)} / F_{<j-1}=F_{\Phi_F(j)-1})	\isomto	
Gal(E_{\sigma(\Phi_F(j))} / E_{\sigma(\Phi_F(j))-1}) \;,$$
et la marche étant non triviale
$$\begin{array}{rcl}
&\isomto	&Gal(E_{\sigma_{\mid}(\Phi_F(j))} / E_{\sigma_{\mid}(\Phi_F(j))-1})\\
&	&\\
&=	&Gal(E_{\Phi_E(\Phi_E^{-1} \circ \sigma_{\mid} \circ \Phi_F(j))} /
E_{\Phi_E(\Phi_E^{-1} \circ \sigma_{\mid} \circ \Phi_F(j))-1}).\\
&	&\\
\text{Enfin par les deux}\!	&\!\!\text{conditions}\!&\!\text{que vérifie } \Phi_E(\Phi_E^{-1} \circ
\sigma_{\mid} \circ \Phi_F(j)) \\
&	&\\
Gal(F_{<j} / F_{<j-1})	&\isomto	
&Gal(E_{<\,\Phi_E^{-1} \circ \sigma_{\mid} \circ \Phi_F(j)} /
E_{<\,\Phi_E^{-1} \circ \sigma_{\mid} \circ \Phi_F(j)-1})\\
&	&\\
&\isomto	&Gal(E_{<\,\tau(j)} / E_{<\,\tau(j)-1}).
\end{array}$$
Ceci prouve exactement que $(E_<) \sim (F_<)$.
\end{proof}

\mespace
\begin{cor} 
Dans les notations de la proposition \ref{prop:tourgalstrictequiv}, toute marche
non triviale de $(F)$ est une marche de $(F_<)$.
\end{cor}

\pespace
\begin{proof}
C'est ce qu'exprime la surjectivité de l'application $\Phi_F$ de la
démonstration de la proposition \ref{prop:tourgalstrictequiv}.
\end{proof}

\gespace
\mespace
\section{Le cas des extensions galoisiennes}


Dans le cas particulier des extensions galoisiennes, les résultats auxquels nous
voulons aboutir (cf. introduction du présent chapitre) sont des conséquences
directes de leurs analogues pour les groupes. Les quatre propositions suivantes
ne se limitent pas à des extensions finies.

\pespace
\begin{prop} \label{prop:tourinduite}
Toute suite normale du groupe de Galois d'une extension galoisienne quelconque
induit une tour galoisienne de cette extension.
Précisément, soit $L \gal K$ une extension galoisienne de groupe $G:=Gal(L/K)$. Pour
tout suite normale
$$(S) \qquad G=G_0 \unrhd \dots \unrhd G_i \unrhd G_{i+1} \unrhd \dots \unrhd
G_m=\iit \;, $$
posons $T_i:=L^{G_i} \; (i=0,\dots,m)$. On a alors la tour galoisienne de $L
\gal K$ 
$$(T) \qquad K=T_0 \unlhd \dots \unlhd T_i \unlhd T_{i+1} \unlhd \dots \unlhd
T_m=L \;.$$
Nous disons que la tour galoisienne $(T)$ est induite par la suite normale
$(S)$.
\end{prop}

\begin{proof}
Que $(T)$ soit une tour de $L/K$ est trivial car
$$G_i \geq G_{i+1} \quad \implique \quad T_i=L^{G_i} \leq L^{G_{i+1}}=T_{i+1}\;.$$
Prouvons qu'elle est galoisienne. Par le (1) de la proposition 3.1
du chapitre 1 
$$G_i \unrhd G_{i+1} \quad \implique \quad \overline{G_i} \unrhd \overline{G_{i+1}} $$
au sens de la topologie de Krull de $G$ ; d'où
$$\overline{G_i}=Gal(L/L^{G_i}) \unrhd Gal(L/L^{G_{i+1}})=\overline{G_{i+1}}$$
i.e. $Gal(L/T_i) \unrhd Gal(L/T_{i+1})$
ce qui signifie que l'extension $T_{i+1}/T_i$ est galoisienne $(i=0,\dots,m-1)$.
Toutes les marches de $(T)$  étant galoisiennes, la tour $(T)$ est galoisienne
par définition (Chap. 2, Déf. \& Conv. 1.1.(2)). 
\end{proof}

\gespace
\begin{prop}  \label{prop:raftourinduite}
Tout raffinement d'une suite normale du groupe de Galois d'une extension
galoisienne induit un raffinement galoisien de la tour galoisienne
correspondante de l'extension.\\
Précisément, soient $L \gal K$ une extension galoisienne quelconque de groupe
$G:=Gal(L/K)$,
$$(S) \qquad G=G_0 \unrhd \dots \unrhd G_i \unrhd G_{i+1} \unrhd \dots \unrhd
G_m=\iit $$
une suite normale de $G$, et
$$(T) \qquad K=T_0 \unlhd \dots \unlhd T_i \unlhd T_{i+1} \unlhd \dots \unlhd
T_m=L $$
la tour galoisienne de $L/K$ induite par $(S)$ (Prop. \ref{prop:tourinduite}). Pour tout
raffinement  
$$(S') \qquad G=G'_0 \unrhd \dots \unrhd G'_j \unrhd G'_{j+1} \unrhd \dots \unrhd
G'_{m'}=\iit $$
de $(S)$ (cf. \cite[p.120]{Ro}), la tour galoisienne
$$(T') \qquad K=T'_0 \unlhd \dots \unlhd T'_j \unlhd T'_{j+1} \unlhd \dots \unlhd
T'_{m'}=L $$
induite par $(S')$ est un raffinement galoisien (cf. Chap. 3, Déf. 1.3.(4)) 
de $(T)$.
\end{prop}

\mespace
\begin{proof}
Par définition d'un raffinement d'une suite normale de groupe, on a $m \leq m'$
et l'existence d'une suite d'entiers
$$0 \leq  j_0 < j_1 < \dots < j_m \leq m'$$
telle que
$$\forall i \in \{0,\dots,m\} \qquad G_i=G'_{j_i} \;.$$
Dès lors, on a directement
$$\begin{array}{cccc}
\forall i \in \{0,\dots,m\} \quad	&L^{G_i}	&=	&L^{G'_{j_i}}\\
&\text{\rotatebox[origin=c]{90}{=}}	&
&\text{\rotatebox[origin=c]{90}{=}}\\
					&T_i		&{\scriptscriptstyle =\,=\,=\,=}
&T'_{j_i} \:.
\end{array}$$
Donc $(T')$ raffine $(T)$. Le Fait 1.5.(2) du chapitre 3 
nous permet de conclure : la tour galoisienne $(T')$ est un raffinement
galoisien de $(T)$.
\end{proof}

\gespace
Le résultat précédent vient de ce que l'on a posé la définition d'un raffinement
de tour de corps (au chapitre 3) de manière compatible, dans le cas d'une
extension galoisienne, avec la définition bien connue d'un raffinement d'une
suite normale de groupes. Cette compatibilité permet d'énoncer les réciproques
des propositions \ref{prop:tourinduite} et \ref{prop:raftourinduite}.

\gespace
\begin{prop} \label{prop:suiteinduite}

(1) Toute tour galoisienne de corps d'une extension galoisienne induit une suite
normale du groupe de Galois de cette extension.\\
Précisément, soit $L \gal K$ une extension galoisienne quelconque. Pour toute tour
galoisienne
$$(T) \qquad K=T_0 \unlhd \dots \unlhd T_i \unlhd T_{i+1} \unlhd \dots \unlhd
T_m=L $$
de $L/K$, on a, en posant $G_i:=Gal(L/T_i) \quad (i=0,\dots,m)$, la suite
normale de groupes
$$(S) \qquad Gal(L/K)=G_0 \unrhd \dots \unrhd G_i \unrhd G_{i+1} \unrhd \dots \unrhd
G_m=\iit \;.$$
Nous disons que la suite normale $(S)$ est "induite" par la tour galoisienne
$(T)$.

(2) La suite normale $(S)$ induite par le (1)  induit à son tour une tour
galoisienne par la proposition \ref{prop:tourinduite}, qui n'est autre que la
tour $(T)$ initiale. 
\end{prop}

\pespace
\begin{proof}
(1) Le résultat est une conséquence immédiate de la théorie de Galois générale :
$$\forall i \in \{0,\dots,m-1\}  \qquad T_i \leq T_{i+1} \quad \ssi \quad
Gal(L/T_{i+1}) \leq Gal(L/T_i)$$
et
$$T_i \unlhd T_{i+1} \quad \ssi \quad Gal(L/T_{i+1}) \unlhd Gal(L/T_i)  \quad
\ssi \quad G_{i+1} \unlhd G_i \;.$$

(2) Pour tout entier $i$ dans $\{0,\dots,m\}$, $L/T_i$ est galoisienne, comme
sous-extension de $L \gal K$ ; donc 
$$T_i=L^{Gal(L/T_i)}=L^{G_i} \;.$$
\end{proof}

\gespace
\begin{prop} \label{prop:rafsuiteinduite}
Soit $L \gal K$ une extension galoisienne de groupe $G:=Gal(L/K)$. Tout
raffinement galoisien de $L \gal K$ induit un raffinement de la suite normale
correspondante.\\
Précisément : soient
$$(T) \qquad K=T_0 \unlhd \dots \unlhd T_i \unlhd T_{i+1} \unlhd \dots \unlhd
T_m=L $$
une tour galoisienne de $L \gal K$ et
$$(T') \qquad K=T'_0 \unlhd \dots \unlhd T'_j \unlhd T'_{j+1} \unlhd \dots \unlhd
T'_{m'}=L $$
un raffinement galoisien de $(T)$. Alors la suite normale
$$(S') \qquad G=G'_0 \unrhd \dots \unrhd G'_j \unrhd G'_{j+1} \unrhd \dots \unrhd
G'_{m'}=\iit $$
induite par $(T')$ (cf. Prop. \ref{prop:suiteinduite}.(1)) est un raffinement de
$$(S) \qquad G=G_0 \unrhd \dots \unrhd G_i \unrhd G_{i+1} \unrhd \dots \unrhd
G_m=\iit $$
induite par $(T)$.
\end{prop}

\mespace
\begin{proof}
D'après la proposition 1.7 du chapitre 3, 
$(T')$ est une tour galoisienne de $L \gal K$. Par définition d'un raffinement d'une
tour de corps (Chap. 3, Déf. \& Conv. 1.1.(1)), 
 on a $m \leq m'$ et l'existence d'une suite d'indices
$$0 \leq j_0 < j_1 < \dots < j_m \leq m'$$
telle que
$$\forall i \in \{0,\dots,m\}  \qquad T_i = T'_{j_i} \;.$$
D'où par définition
$$\begin{array}{cccc}
\forall i \in \{0,\dots,m\} \quad 	&Gal(L/T_i)	&=
&Gal(L/T'_{j_i})\\
&\text{\rotatebox[origin=c]{90}{=}}	&&\text{\rotatebox[origin=c]{90}{=}}\\
&G_i	&{\scriptscriptstyle =\,=\,=\,=}	&G'_{j_i} \;.
\end{array}$$
\end{proof}

\mespace
\begin{cor} \label{cor:rapsuiteinduite}
Soit $L \gal K$ une extension galoisienne de groupe $G:=Gal(L/K)$. Pour tout
raffinement galoisien propre $(T')$ d'une tour galoisienne $(T)$ de $L \gal K$, la
suite normale $(S')$ de $G$ induite par $(T')$ (cf. Prop.
\ref{prop:suiteinduite}.) est un raffinement propre de la suite normale $(S)$
induite par $(T)$.
\end{cor}

\pespace
\begin{proof}
La proposition \ref{prop:rafsuiteinduite} assure déjà que $(S')$ est un
raffinement de $(S)$. De plus, dans les mêmes notations, on a, par définition
d'un raffinement propre de tour de corps 
(Chap. 3, Déf. \& Conv. 1.1.(2)), 
la condition
$$\exists j \in \{1,\dots,m'-1\} \quad \forall i \in \{0,\dots,m\} \quad T'_j
\neq T_i\;.$$
Or d'après le (1) de la proposition \ref{prop:suiteinduite}, $G_i:=Gal(L/T_i)$,
d'où
$$T_i=L^{G_i} \quad (i=0,\dots,m) \;.$$
De même
$$T'_j=L^{G'_j} \quad (j=0,\dots,m') \;.$$
Donc directement
$$\exists j \in \{1,\dots,m'-1\} \quad \forall i \in \{0,\dots,m\} \quad G'_j
\neq G_i\;.$$
\end{proof}

\gespace
Dans le cas d'une extension galoisienne finie, on peut compléter les résultats
précédents grâce à la bijection de Galois.

\mespace
\begin{prop} \label{prop:réciproqueinduite}
Soit $L \gal K$ une extension galoisienne finie de groupe $G:=Gal(L/K)$.

(1) Pour toute suite normale stricte
$$(S) \qquad G=G_0 \rhd \dots \rhd G_i \rhd G_{i+1} \rhd \dots \rhd
G_m=\iit \;,$$
la tour galoisienne $(T)$ induite par $(S)$ (cf. Prop. \ref{prop:tourinduite}) est
stricte.

(2) Pour tout raffinement propre $(S')$ d'une suite normale $(S)$ de
$G$, la tour $(T')$ induite par $(S')$ est un raffinement galoisien
propre de la tour $(T)$ induite par $(S)$.

(3) Pour toute suite normale $(S)$ de $(G)$, la suite normale de $G$ induite par
la tour galoisienne de $L \gal K$ induite par $(S)$ est égale à $(S)$.
\end{prop}

\mespace
\begin{proof}
(1) Par la bijection classique de Galois :
$$G_i \neq G_{i+1} \quad \implique \quad T_i=L^{G_i} \neq L^{G_{i+1}} =T_{i+1}
\qquad (i=0,\dots,m-1) \;.$$

(2) D'après la proposition \ref{prop:raftourinduite}, $(T')$ est un raffinement
de $(T)$. On a, par définition d'un raffinement propre d'une suite normale, la
condition
$$\exists j \in \{1,\dots,n-1\} \quad \forall i \in \{0,\dots,m\} \quad G'_j
\neq G_i\;.$$
Et par la bijection classique de Galois,
$$\begin{array}{cccc}
\exists j \in \{1,\dots,n-1\} \quad \forall i \in \{0,\dots,m\} \quad 
&L^{G'_j}	&\neq	&L^{G_i}\\
&\text{\rotatebox[origin=c]{90}{=}}&	&\text{\rotatebox[origin=c]{90}{=}}\\
&T'_j	&	&T_i \;.
\end{array}$$
Par définition et le théorème d'Artin dans l'extension finie $L \gal K$, on a
$$Gal(L/T_i)=Gal(L/L^{G_i})=G_i \qquad (i=0,\dots,m) \;.$$
\end{proof}

\gespace
Nous faisons maintenant le lien entre la notion de groupe simple et la notion
d'extension galsimple, toujours dans le cas d'une extension galoisienne finie.

\mespace
\begin{prop}
Pour qu'une extension galoisienne finie soit galsimple (cf. Chap. 2, Déf.
1.6.(2)), 
il faut et il suffit que son groupe de Galois soit simple.
\end{prop}

\mespace
\begin{proof}
Soit $L \gal K$ une extension galoisienne finie de groupe de Galois $G:=Gal(L/K)$.
Supposons $L \gal K$ galsimple. Par définition, elle n'est pas triviale et $G \neq
\iit$ car sinon
$$K=L^{Gal(L/K)}=L^G=L^{\petitiit}=L \;:\; \text{contradiction}\;.$$
Raisonnons par l'absurde en supposant que $G$ ne soit pas simple, i.e. qu'il
existe un sous-groupe $H$ de $G$ tel que l'on ait la suite normale stricte $G
\rhd H \rhd \iit$. Le (1) de la proposition \ref{prop:réciproqueinduite}
ci-dessus nous assure alors que l'on a la tour galoisienne stricte
$$K \lhd L^H \lhd L$$
qui contredit la galsimplicité de $L/K$.

\pespace
Inversement, si $G$ est simple, on sait que $G \neq \iit$ et donc que $L \neq
K$. Raisonnons à nouveau par l'absurde en supposant que $L/K$ ne soit pas
galsimple. C'est qu'il existe un corps intermédiaire strict $K < M < L$ avec
$M/K$ galoisienne. Alors, par la théorie de Galois générale et la
simplicité de $G$,
$$Gal(L/M) \unlhd G \quad \implique \quad \text{\Large ( }Gal(L/M)=G \,\text{ ou
}\, Gal(L/M)=\iit \text{ \Large )} \;.$$
Mais, si $Gal(L/M)=G$, on a
$$M=L^{Gal(L/M)}=L^G=L^{Gal(L/K)}=K \; : \; \text{contradiction} \;;$$
et si $Gal(L/M)=\iit$,
$$M=L^{Gal(L/M)}=L^{\petitiit}=L \; : \; \text{contradiction} \;.$$
C'est donc que l'extension $L/K$ est galsimple.
\end{proof}

\gespace
Dans le cas d'une extension galoisienne, les propositions précédentes font le lien
entre tours galoisiennes de corps et suites normales de groupes, avec leurs
raffinements. La correspondance se poursuit, dans le cas fini, entre tours de
composition galoisiennes et suites de composition.

\gespace
\begin{prop} \label{prop:tourcompositiongalfinie}
Toute extension galoisienne finie admet une tour de composition galoisienne.
\end{prop}

\mespace
\begin{proof}
C'est clair pour l'extension triviale par le Fait
\ref{fait:tourcompositiontriviale}. Soit $L/K$ une extension
galoisienne finie non triviale de
groupe $G:=Gal(L/K)$. En tant que groupe fini non trivial, il admet une suite
normale de composition :
$$(C) \qquad G=G_0 \rhd \dots \rhd G_i \rhd G_{i+1} \rhd \dots \rhd
G_m=\iit \;.$$
Montrons que la tour galoisienne induite par $(C)$ (Prop.
\ref{prop:tourinduite})
$$(T) \qquad K=T_0 \unlhd \dots \unlhd T_i \unlhd T_{i+1} \unlhd \dots \unlhd
T_m=L $$
est une tour de composition de $L/K$.
Le (1) de la proposition \ref{prop:réciproqueinduite} nous assure que la tour
$(T)$ est stricte. Raisonnons par l'absurde en supposant l'existence d'un
raffinement galoisien propre
$$(T') \qquad K=T'_0 \unlhd \dots \unlhd T'_j \unlhd T'_{j+1} \unlhd \dots
\unlhd T'_{m'}=L $$
de $(T)$. Par le corollaire \ref{cor:rapsuiteinduite}, il induit un raffinement
propre $(C')$ de $(C)$, ce qui contredit le fait que $(C)$ soit une suite
normale de composition.
\end{proof}

\gespace
Les propositions précédentes conduisent directement à un analogue galoisien du
théorème de Schreier (cf. introduction de ce chapitre) pour une extension
galoisienne finie. Cependant, comme cet analogue sera généralisé à toutes les
extensions galtourables (cf. Th. \ref{th:1diss} ), nous omettons ici sa
démonstration dans ce cas particulier.

\mespace
\begin{prop} \label{prop:galschreiergaloisien}
Deux tours galoisiennes d'une même extension galoisienne finie admettent des
raffinements galoisiens équivalents.
\end{prop}

\gespace
\begin{rem}
Dans le cas particulier d'une extension abélienne, cette proposition
\ref{prop:galschreiergaloisien} permet de retrouver la proposition 2.2 de
\cite{Do-M} qui fournit en outre la hauteur des tours de composition.
\end{rem}


\gespace
\gespace
\section{Premier théorème de dissociation}

Venons-en maintenant, pour les extensions galtourables, à l'analogue galoisien
du théorème de Schreier (cf. introduction du présent chapitre). La
démonstration moderne de l'équivalence des raffinements de Schreier utilise le
lemme de Zassenhaus ("Butterfly lemma"). Nous allons montrer que ce sont les
parallélogrammes galoisiens qui jouent, pour les extensions galtourables, le
rôle du lemme de Zassenhaus pour les groupes. Ces parallélogrammes ont de plus
l'avantage de rendre visuel l'esprit de la démonstration. Pour une meilleure
lisibilité de celle-ci, nous avons sorti la proposition suivante qui est en
fait un lemme technique.

\mespace
\begin{prop} \label{prop:butterfly}
Soient $L/K$ une extension galtourable de degré quelconque, et
$$(T^1) \qquad K=T^1_0 \unlhd \dots \unlhd T^1_i \unlhd T^1_{i+1} \unlhd \dots
\unlhd T^1_m=L $$
$$(T^2) \qquad K=T^2_0 \unlhd \dots \unlhd T^2_j \unlhd T^2_{j+1} \unlhd \dots
\unlhd T^2_n=L $$
deux tours galoisiennes de $L/K$. Alors nécessairement :\\
(1) Pour tous indices $i,j,k$ tels que $0 \leq i \leq m-1$ et $0 \leq k < j \leq
n-1$, on a le parallélogramme galoisien
$$[T^1_{i+1} T^2_k \cap T^1_i T^2_j, T^1_{i+1} T^2_{k+1} \cap T^1_i
T^2_j, T^1_{i+1} T^2_{k+1} \cap T^1_i T^2_{j+1}, T^1_{i+1} T^2_k \cap T^1_i
T^2_{j+1}]  \;.$$
En particulier, on a les isomorphismes de restriction
$$ \begin{array}{c}
Gal(T^1_{i+1} T^2_{k+1} \cap T^1_i T^2_{j+1} / T^1_{i+1} T^2_{k+1} \cap T^1_i
T^2_j)	\\
	\\
\text{\rotatebox[origin=c]{270}{$\isomto$}}	\\
	\\
Gal(T^1_{i+1} T^2_k \cap T^1_i T^2_{j+1} / T^1_{i+1} T^2_k \cap T^1_i
T^2_j) \:.
\end{array}$$
(2) Pour tous $i,j,k$ tels que $0 \leq k < i \leq m-1$ et $0 \leq j \leq n-1$,
on a le parallélogramme 
$$[T^1_k T^2_{j+1} \cap T^1_i T^2_j, T^1_{k+1} T^2_{j+1} \cap T^1_i
T^2_j, T^1_{k+1} T^2_{j+1} \cap T^1_{i+1} T^2_j, T^1_k T^2_{j+1} \cap T^1_{i+1}
T^2_j]  \;.$$
En particulier on a les isomorphismes de restriction
$$ \begin{array}{c}
Gal(T^1_{k+1} T^2_{j+1} \cap T^1_{i+1} T^2_j / T^1_{k+1} T^2_{j+1} \cap T^1_i
T^2_j)	\\
	\\
\text{\rotatebox[origin=c]{270}{$\isomto$}}	\\
	\\
Gal(T^1_k T^2_{j+1} \cap T^1_{i+1} T^2_j / T^1_k T^2_{j+1} \cap T^1_i
T^2_j) \:.
\end{array}$$
\end{prop}

\mespace
\begin{proof}
(1) De l'extension galoisienne $T^1_{i+1} \gal T^1_i$, on déduit du corollaire
2.2 
du chapitre 2 l'extension galoisienne $(T^1_{i+1} T^2_k \gal T^1_i T^2_k)$. Or 
$$T^1_i T^2_k \leq T^1_{i+1} T^2_k \cap T^1_i T^2_{k+1} \leq T^1_{i+1} T^2_k 
\; ;$$
d'où la sous-extension galoisienne $(T^1_{i+1} T^2_k \gal T^1_{i+1} T^2_k \cap
T^1_i T^2_{k+1})$. De même, on a les implications
$$(T^2_{k+1} \gal  T^2_k) \; \implique \; (T^1_i T^2_{k+1} \gal  T^1_i T^2_k )
\; \implique \;(T^1_i T^2_{k+1} \gal T^1_{i+1} T^2_k \cap T^1_i T^2_{k+1}) \;.$$
Comme $T^1_{i+1} T^2_k T^1_i T^2_{k+1}=T^1_{i+1} T^2_{k+1}$, on a donc le
parallélogramme galoisien
$$P(i,k):=[T^1_{i+1} T^2_k \cap T^1_i T^2_{k+1} \,,\, T^1_i T^2_{k+1} \,,\,
T^1_{i+1} T^2_{k+1} \,,\, T^1_{i+1} T^2_k] \;.$$
Par ailleurs, de l'hypothèse $k < j$, i.e. $k+1 \leq j$, suit
$T^2_{k+1} \leq T^2_j$, d'où $T^1_i T^2_{k+1} \leq T^1_i T^2_j$. Donc également
$$T^1_{i+1} T^2_k \cap T^1_i T^2_{k+1} \leq T^1_{i+1} T^2_k \cap T^1_i T^2_j 
\;.$$
Comme $T^1_i T^2_{k+1} \leq T^1_{i+1} T^2_{k+1}$, on en tire aussi
$$T^1_i T^2_{k+1} \leq T^1_{i+1} T^2_{k+1} \cap T^1_i T^2_j \;.$$

\begin{figure}[!h]
\begin{center}
\vskip -6mm
\includegraphics[width=9.5cm]{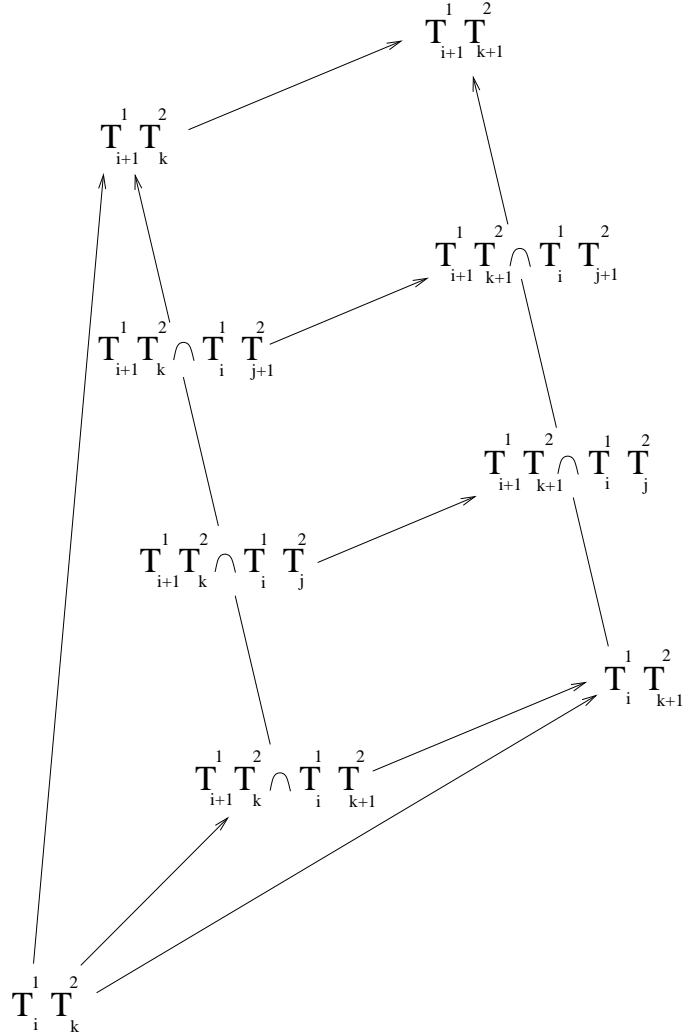}
\end{center}
\vskip -5mm
\rm{\caption{\label{fig:17} \leg Parallélogrammes $P(i,j,k)$ et $Q(i,j,k)$ }}
\vskip -3mm
\end{figure}

De plus
$$\begin{array}{c}
(T^1_{i+1} T^2_{k+1} \cap T^1_i T^2_j) \cap T^1_{i+1} T^2_k = T^1_{i+1} T^2_k
\cap T^1_i T^2_j\\
	\\
\text{{\Large ( }resp. } (T^1_{i+1} T^2_{k+1} \cap T^1_i T^2_{j+1}) \cap
T^1_{i+1} T^2_k = T^1_{i+1} T^2_k \cap T^1_i T^2_{j+1}   \text{ \Large ).}
\end{array}$$ 
On déduit alors du corollaire 4.3
.(2-1) du chapitre 1 que l'on a le sous-parallélogramme galoisien de $P(i,k)$ :
$$\begin{array}{cc}
&P(i,j,k) \\
&\text{\rotatebox[origin=c]{90}{=:}} \\ 
&\left[ T^1_{i+1} T^2_k \cap T^1_i T^2_j,T^1_{i+1} T^2_{k+1} \cap T^1_i T^2_j,
T^1_{i+1} T^2_{k+1},T^1_{i+1} T^2_k \right]\\
\\
\text{{\LARGE ( }resp. } 	&P(i,j+1,k) \\
&\text{\rotatebox[origin=c]{90}{=:}} \\ 
&\left[ T^1_{i+1} T^2_k \cap T^1_i T^2_{j+1},T^1_{i+1} T^2_{k+1} \cap T^1_i
T^2_{j+1},T^1_{i+1} T^2_{k+1},T^1_{i+1} T^2_k \right] \text{\LARGE ).}
\end{array}$$

Prouvons maintenant, par récurrence descendante sur $k \in \{j-1,\dots, 0\}$
que l'on a le parallégramme galoisien quotient
$$\begin{array}{c}
Q(i,j,k) \\
\text{\rotatebox[origin=c]{90}{=:}} \\ 
\left[ T^1_{i+1} T^2_k \cap T^1_i T^2_j,T^1_{i+1} T^2_{k+1} \cap T^1_i
T^2_j,T^1_{i+1} T^2_{k+1} \cap T^1_i T^2_{j+1},T^1_{i+1} T^2_k \cap T^1_i
T^2_{j+1}\right]\;. \end{array}$$

\pespace
{\it Rang $k=j-1$.}\\ Par le corollaire 2.2 du chapitre 2, 
on a l'implication
$$\text{\Large (  }(T^2_{j+1}=T^2_{k+2}) \gal T^2_{k+1} \text{\Large )} \quad
\implique \quad (T^1_i T^2_{j+1} \gal T^1_i T^2_{k+1}) \;.$$
De $T^1_{i+1} T^2_{k+1} \gal T^1_i T^2_{k+1}$ suit donc par intersection
$$(T^1_{i+1} T^2_{k+1} \cap T^1_i T^2_{j+1}) \gal (T^1_i T^2_{k+1}= T^1_{i+1} T^2_{k+1}
\cap T^1_i T^2_j) \;.$$
Par le corollaire 4.3.(2-2) du chapitre 1, 
appliqué dans le sous-parallélogramme $P(i,j,k)$, on en déduit alors
l'existence du parallélogramme galoisien quotient $Q(i,j,k)$.

\pespace
{\it Rang $k-1$.}\\ Supposons l'existence du parallélogramme $Q(i,j,k)$ (avec $k
\geq 1$) et prouvons celle de $Q(i,j,k-1)$. De la donnée de $Q(i,j,k)$ suit en
particulier l'extension galoisienne
$$(T^1_{i+1} T^2_k \cap T^1_i T^2_{j+1}) \gal (T^1_{i+1} T^2_k \cap T^1_i T^2_j)
\;.$$
L'existence de $Q(i,j,k-1)$ se déduit alors du corollaire 4.3.(2-2) du chapitre
1 
appliqué cette fois dans le sous parallélogramme $P(i,j,k-1)$ de $P(i,k-1)$.

La récurrence établissant l'existence des parallélogrammes quotients $Q(i,j,k)$
est donc prouvée. On en déduit en particulier l'isomorphisme annoncé.

\pespace
(2) Il s'agit du (1) mutatis mutandis, par la permutation
$$\left(\begin{array}{ccccccc}
T^1	&T^2	&i	&j	&k	&m	&n\\
T^2	&T^1	&j	&i	&k	&n	&m
\end{array}\right)\;.$$
\end{proof}

\gespace
Nous sommes maintenant en mesure de prouver l'analogue galoisien suivant au
théorème de Schreier pour les suites normales de groupes (\cite[p.124]{Ro},
\cite{Schr}).

\mespace
\begin{Th}  ($1^{\text{er}}$ théorème de dissociation, dit "de Galschreier")
\label{th:1diss}\index{Dissociation!$1^{\text{er}}$ Théorème (de)}\index{Théorème!de dissociation}\index{Théorème!de Galschreier}\\
Deux tours galoisiennes d'une même extension galtourable admettent des
raffinements qui sont des tours galoisiennes équivalentes.
\end{Th}

\mespace \noindent
{\it Scholies.} (1) Les extensions galtourables sont ici quelconques, de degré fini ou
infini.\\
(2) Ces raffinements sont nécessairement galoisiens en vertu du Fait 1.5.(2) du
chapitre 3. 

\mespace
\begin{proof}
Nous allons prouver que les tours galoisiennes $(T^1)$ et $(T^2)$ de la
proposition \ref{prop:butterfly} admettent des raffinements équivalents. En
faisant varier $k$ de $j-1$ à $0$ dans le (1) de cette proposition, on a 
$$\forall i \in \{0, \dots , m-1 \} \quad \forall j \in \{1, \dots, n-1 \}$$
$$\begin{array}{ccc}
Gal(T^1_{i+1} T^2_j \cap T^1_i T^2_{j+1} / T^1_{i+1} T^2_j \cap T^1_i
T^2_j)	&=	&Gal(T^1_{i+1} T^2_j \cap T^1_i T^2_{j+1} / T^1_i T^2_j)\\
\\
\text{\rotatebox[origin=c]{270}{$\isomto$}}	\\
\\
Gal(T^1_{i+1} T^2_{j-1} \cap T^1_i T^2_{j+1} / T^1_{i+1} T^2_{j-1} \cap T^1_i
T^2_j)\\
\\
\text{\rotatebox[origin=c]{270}{$\isomto$}}	\\
\vdots \\
\text{\rotatebox[origin=c]{270}{$\isomto$}}	\\
\\
Gal(T^1_{i+1} T^2_0 \cap T^1_i T^2_{j+1} / T^1_{i+1} T^2_0 \cap T^1_i
T^2_j)	&=	&Gal(T^1_{i+1} \cap T^1_i T^2_{j+1} / T^1_{i+1} \cap T^1_i
T^2_j) \;.\\
\end{array}$$
\pespace
Retenons que pour tous $i \in \{0, \dots, m-1\}$ et $j \in \{1, \dots, n-1\}$,
on a l'isomorphisme
$$isom^1_{(i,j)} : Gal(T^1_{i+1} \cap T^1_i T^2_{j+1} / T^1_{i+1} \cap T^1_i
T^2_j) \isomto Gal(T^1_{i+1} T^2_j \cap T^1_i T^2_{j+1} / T^1_i T^2_j) \;.$$

\pespace \noindent
De même, en faisant varier $k$ de $i-1$ à $0$ dans le (2) de la proposition
\ref{prop:butterfly}, on a
$$\forall i \in \{1, \dots , m-1 \} \quad \forall j \in \{0, \dots, n-1 \}$$
$$\begin{array}{ccc}
Gal(T^1_i T^2_{j+1} \cap T^1_{i+1} T^2_j / T^1_i T^2_{j+1} \cap T^1_i
T^2_j)	&=	&Gal(T^1_i T^2_{j+1} \cap T^1_{i+1} T^2_j / T^1_i T^2_j)\\
\\
\text{\rotatebox[origin=c]{270}{$\isomto$}}	\\
\\
Gal(T^1_{i-1} T^2_{j+1} \cap T^1_{i+1} T^2_j / T^1_{i-1} T^2_{j+1} \cap T^1_i
T^2_j)\\
\\
\text{\rotatebox[origin=c]{270}{$\isomto$}}	\\
\vdots \\
\text{\rotatebox[origin=c]{270}{$\isomto$}}	\\
\\
Gal(T^1_0 T^2_{j+1} \cap T^1_{i+1} T^2_j / T^1_0 T^2_{j+1} \cap T^1_i
T^2_j)	&=	&Gal(T^2_{j+1} \cap T^1_{i+1} T^2_j / T^2_{j+1} \cap T^1_i
T^2_j) \;.\\
\end{array}$$
\pespace
Retenons que pour tous $i \in \{1, \dots, m-1\}$ et $j \in \{0, \dots, n-1\}$,
on a l'isomorphisme
$$isom^2_{(i,j)} : Gal(T^2_{j+1} \cap T^1_{i+1} T^2_j / T^2_{j+1} \cap T^1_i
T^2_j) \isomto Gal(T^1_i T^2_{j+1} \cap T^1_{i+1} T^2_j / T^1_i T^2_j) \;.$$
D'où l'isomorphisme composé
$$isom_{(i,j)}:=\left\{\begin{array}{ll}
(isom^2_{(i,j)})^{-1} \circ isom^1_{(i,j)}	&\text{si } (i,j) \neq (0,0)\\
\\
isom^1_{(i,j)}					&\text{si } i=0\, ,\;j \neq 0\\

\\
(isom^2_{(i,j)})^{-1} 				&\text{si } i \neq 0\, ,\;j=0\\
\\
id_{\,Gal(T^1_1 \cap T^2_1/K)}			&\text{si } (i,j) = (0,0)\\
\end{array}\right.$$
Dans tous les cas, on a
$$isom_{(i,j)}\;: \quad Gal(T^1_{i+1} \cap T^1_i T^2_{j+1} / T^1_{i+1} \cap T^1_i
T^2_j) \isomto Gal(T^2_{j+1} \cap T^1_{i+1} T^2_j / T^2_{j+1} \cap T^1_i
T^2_j) \;.$$

\pespace \noindent
Remarquons par division euclidienne que, pour tout $l \in \{0, \dots ,mn-1\}$,
il existe un unique couple $(q^1_l,r^1_l) \in  \{0, \dots ,m-1\} \times  \{0,
\dots ,n-1\}$ et un unique $(q^2_l,r^2_l) \in  \{0, \dots ,n-1\} \times  \{0,
\dots ,m-1\}$ tels que
$$l=q^1_l  n + r^1_l =q^2_l  m + r^2_l \;.$$
Soient maintenant $(T'^1)$ et $(T'^2)$ les deux tours de $L/K$ définies par les
formules
$$(\frak{F}) \; \left\{\begin{array}{cc}
\forall l \in \{0, \dots ,mn-1\}	&T'^1_l:=T^1_{q^1_l+1} \cap
T^1_{q^1_l} T^2_{r^1_l} \;;\\
\\
					&T'^2_l:=T^2_{q^2_l+1} \cap
T^1_{r^2_l} T^2_{q^2_l}  \;;\\
\\
\\
T'^1_{mn}:=T^1_m \cap T^1_{m-1} T^2_n=L  \;;	&T'^2_{mn}:=T^2_n \cap T^1_m
T^2_{n-1}=L \;.
\end{array}\right.$$
Notons que
$$\left\{\begin{array}{l}
\forall i \in  \{0, \dots ,m-1\} \qquad T^1_i=T^1_{i+1} \cap T^1_i K=T^1_{i+1}
\cap T^1_i T^2_0=T'^1_{in} \;;\\
\\
T^1_m=L=T'^1_{mn} \;.
\end{array}\right.$$
La suite d'indices
$$0 \leq l^1_0:=0 < l^1_1:=n < \dots < l^1_i:=in < \dots < l^1_{m-1}:=(m-1)n <
l^1_m:=mn \leq mn$$
est donc telle que
$$\forall i \in  \{0, \dots ,m\} \qquad T'^1_{l_i}=T^1_i \;;$$ et 
$$(T'^1) \qquad K=T'^1_0 \leq T'^1_1 \leq \dots \leq T'^1_l \leq \dots
\leq T'^1_{mn-1} \leq \dots \leq T'^1_{mn}=L$$
est un raffinement de la tour $(T^1)$ de l'énoncé au sens de la définition \&
convention 1.1 du chapitre 3. 
De même, notons que
$$\left\{\begin{array}{l}
\forall j \in  \{0, \dots ,n-1\} \qquad T^2_j=T^2_{j+1} \cap K T^2_j=T^2_{j+1}
\cap T^1_0 T^2_j=T'^2_{jm} \;;\\
\\
T^2_n=L=T^2_{nm} \;.
\end{array}\right.$$
La suite d'indices
$$0 \leq l^2_0:=0 < l^2_1:=m < \dots < l^2_j:=jm < \dots < l^2_{n-1}:=(n-1)m <
l^2_n:=nm \leq nm$$
est donc telle que
$$\forall j \in  \{0, \dots ,n\} \qquad T'^2_{l_j}=T^2_j \;. $$
Et on vient de montrer que $(T'^2)$ raffine $(T^2)$.\\

Par le (1) (resp. le (2)) de la proposition \ref{prop:butterfly} pour $j \neq 0$
(resp. $j = 0$), on a la tour galoisienne
$$ \begin{array}{rcl}
(T'^1)\quad K=T'^1_0=	&T^1_1 \cap T^1_0 T^2_0	&\unlhd \, T'^1_1=T^1_1 \cap
T^1_0 T^2_1\unlhd \dots \unlhd T'^1_{n-1}=T^1_1 \cap T^1_0 T^2_{n-1} \\
\\
		&\unlhd \, T^1_1 \cap T^1_0 T^2_n	&\\
		&\text{\rotatebox[origin=c]{90}{=}} &\\ 
		&T^1_1	&\\
		&\text{\rotatebox[origin=c]{90}{=}} &\\ 
	T'^1_n=	&T^1_2 \cap T^1_1 T^2_0	&\unlhd \, T'^1_{n+1}=T^1_2 \cap
	T^1_1 T^2_1 \unlhd \dots \unlhd
T'^1_{2n-1}=T^1_2 \cap T^1_1 T^2_{n-1} \\
\\
		&\unlhd \, T^1_2 \cap T^1_1 T^2_n	&\\
		&\text{\rotatebox[origin=c]{90}{=}} &\\ 
		&T^1_2	&\\
		&\text{\rotatebox[origin=c]{90}{=}} &\\ 
	T'^1_{2n}=	&T^1_3 \cap T^1_2 T^2_0	&\unlhd \, T'^1_{2n+1}=T^1_3 \cap
	T^1_2 T^2_1 \unlhd \dots \unlhd
T'^1_{3n-1}=T^1_3 \cap T^1_2 T^2_{n-1} \\

		&\vdots					&\\
		&\vdots					&\\

		&\unlhd \, T^1_i \cap T^1_{i-1} T^2_n	&\\
		&\text{\rotatebox[origin=c]{90}{=}} &\\ 
		&T^1_i	&\\
		&\text{\rotatebox[origin=c]{90}{=}} &\\ 
	T'^1_{in}=	&T^1_{i+1} \cap T^1_i T^2_0	&\unlhd \,
	T'^1_{in+1}=T^1_{i+1} \cap	T^1_i T^2_1 \unlhd \dots \\
\\
		&		&\dots \unlhd
T'^1_{(i+1)n-1}=T^1_{i+1} \cap T^1_i T^2_{n-1} \\

		&\vdots					&\\

		&\text{\rotatebox[origin=c]{90}{=}} &\\ 
		&T^1_{m-1}	&\\
		&\text{\rotatebox[origin=c]{90}{=}} &\\ 
	T'^1_{(m-1)n}=	&T^1_m \cap T^1_{m-1} T^2_0	&\unlhd \,
	T'^1_{(m-1)n+1}=T^1_m \cap T^1_{m-1} T^2_1 \unlhd \dots \\
\\
	&	&\dots \unlhd T'^1_{mn-1}=T^1_m \cap T^1_{m-1} T^2_{n-1} \\
\\
		&\unlhd \, T^1_m \cap T^1_{m-1} T^2_n 	&=T'^1_{mn}=L \;.\\
\end{array}$$

De même par le (2) (resp. le (1)) de la proposition \ref{prop:butterfly} pour $i
\neq 0$ (resp. $i = 0$), on a la tour galoisienne

$$ \begin{array}{rcl}
(T'^2)\quad K=T'^2_0=	&T^2_1 \cap T^1_0 T^2_0	&\unlhd \, T'^2_1=T^2_1 \cap
T^1_1 T^2_0\unlhd \dots \unlhd T'^2_{m-1}=T^2_1 \cap T^1_{m-1} T^2_0 \\
\\
		&\unlhd \, T^2_1 \cap T^1_m T^2_0	&\\
		&\text{\rotatebox[origin=c]{90}{=}} &\\ 
		&T^2_1	&\\
		&\text{\rotatebox[origin=c]{90}{=}} &\\ 
	T'^2_m=	&T^2_2 \cap T^1_0 T^2_1	&\unlhd \, T'^2_{m+1}=T^2_2 \cap
	T^1_1 T^2_1 \unlhd \dots \unlhd
T'^2_{2m-1}=T^2_2 \cap T^1_{m-1} T^2_1 \\
\\
		&\unlhd \, T^2_2 \cap T^1_m T^2_1	&\\
		&\text{\rotatebox[origin=c]{90}{=}} &\\ 
		&T^2_2	&\\
		&\text{\rotatebox[origin=c]{90}{=}} &\\ 
	T'^2_{2m}=	&T^2_3 \cap T^1_0 T^2_2	&\unlhd \, T'^2_{2m+1}=T^2_3 \cap
	T^1_1 T^2_2 \unlhd \dots \unlhd
T'^2_{3m-1}=T^2_3 \cap T^1_{m-1} T^2_2 \\

		&\vdots					&\\
		&\vdots					&\\

		&\unlhd \, T^2_j \cap T^1_m T^2_{j-1}	&\\
		&\text{\rotatebox[origin=c]{90}{=}} &\\ 
		&T^2_j	&\\
		&\text{\rotatebox[origin=c]{90}{=}} &\\ 
	T'^2_{jm}=	&T^2_{j+1} \cap T^1_0 T^2_j	&\unlhd \,
	T'^2_{jm+1}=T^2_{j+1} \cap	T^1_1 T^2_j \unlhd \dots \\
\\
		&		&\dots \unlhd
T'^2_{(j+1)m-1}=T^2_{j+1} \cap T^1_{m-1} T^2_j \\

		&\vdots					&\\

		&\text{\rotatebox[origin=c]{90}{=}} &\\ 
		&T^2_{n-1}	&\\
		&\text{\rotatebox[origin=c]{90}{=}} &\\ 
	T'^2_{(n-1)m}=	&T^2_n \cap T^1_0 T^2_{n-1}	&\unlhd \,
	T'^2_{(n-1)m+1}=T^2_n \cap T^1_1 T^2_{n-1} \unlhd \dots \\
\\
	&	&\dots \unlhd T'^2_{nm-1}=T^2_n \cap T^1_{m-1} T^2_{n-1} \\
\\
		&\unlhd \, T^2_n \cap T^1_m T^2_{n-1} 	&=T'^2_{mn}=L \;.\\
\end{array}$$
\pespace

Il reste à montrer que les tours $(T'^1)$ et $(T'^2)$ ainsi construites sont
équivalentes (cf. Déf. \ref{def:tourgalcomp}). Elles ont même nombre de marches
$mn$. Considérons ensuite l'application
$$\begin{array}{rccl}
\sigma :	& \{1, \dots ,mn\} 	&\longrightarrow	& \{1, \dots
,mn\} \\
		&l			&\longmapsto		&\sigma(l):=
r^1_{l-1} m + q^1_{l-1} + 1 
\end{array}$$
où $q^1_{l-1}$ et $r^1_{l-1}$ sont définis, comme précédemment, par division
euclidienne de $l-1$ par $n$ :
$$l-1=q^1_{l-1}  n + r^1_{l-1} \qquad 0 \leq r^1_{l-1} \leq n-1 \;.$$
En particulier donc, $0 \leq q^1_{l-1} \leq m-1$, en sorte que $r^1_{l-1} m +
q^1_{l-1}$ est une division euclidienne par $m$, ce qui assure l'injectivité,
donc la bijectivité de l'application $\sigma$. Autrement dit, $\sigma$ est un
élément du groupe symétrique $S_{mn}$. Explicitement, $\sigma$ est l'identité si
$m=1$ ou $n=1$, et pour $m \geq 2, \: n \geq 2$ :
$$\sigma = \left(\begin{array}{*{10}{c}}
1	&2	&\dots	&n	&n+1	&n+2	&\dots	&2n	&\dots	&mn\\
1	&m+1	&\dots	&(n-1)m+1&2	&m+2	&\dots	&(n-1)m+2&\dots	&mn
\end{array}\right) .$$
Notons tout d'abord que $Gal(T'^1_{mn}=L / T'^1_{mn-1}=T^1_{m-1} T^2_{n-1}) =
Gal(T'^2_{\sigma(mn)} /T'^2_{\sigma(mn)-1})$
Fixons-nous un $l \in \{1, \dots ,mn-1\}$ avec toujours
$$l=q^1_l  n + r^1_l \qquad (q^1_l , r^1_l) \in  \{0, \dots ,m-1\} \times \{0,
\dots ,n-1\} \;.$$
\pespace
Envisageons deux cas.\\
{\it $1^{er}$ cas : }$r^1_l=0$.\\
Nécessairement $q^1_l \geq 1$ (sinon $l=0$ : absurde). Alors de $l-1=(q^1_l-1)n +
n-1$ suit
$$q^1_{l-1}=q^1_l-1 ,\quad r^1_{l-1}=n-1 \;.$$
d'où
$$\sigma(l)= (n-1) m + q^1_l \;.$$
On a déjà vu que
$$ \forall i \in \{0, \dots ,m-1\} \quad T'^1_{in}=T^1_i \;.$$
Donc, comme $T^2_n=L$, 
$$T'^1_l=T'^1_{q^1_ln}=T^1_{q^1_l}=T^1_{q^1_l} \cap T^1_{q^1_l-1} T^2_n \;.$$
Dès lors, par l'isomorphisme $isom_{(q^1_{l-1},n-1)}$,
$$\begin{array}{rcl}
Gal(T'^1_l/T'^1_{l-1}) 	&=		&Gal(T^1_{q^1_l} \cap T^1_{q^1_l-1}
T^2_n / T^1_{q^1_l} \cap T^1_{q^1_l-1} T^2_{n-1})\\
\\
			&\isomto	&Gal(T^2_n \cap T^1_{q^1_l} T^2_{n-1}
/ T^2_n \cap T^1_{q^1_l-1} T^2_{n-1})\\
\\
			&=		&Gal(T'^2_{(n-1)m+q^1_l} /
			T'^2_{(n-1)m+q^1_l-1})\\
			\\
			&=		&Gal(T'^2_{\sigma(l)} /
			T'^2_{\sigma(l)-1}) \;.\\	
\end{array}$$

\noindent{\it $2^{\text{nd}}$ cas : }$r^1_l \geq 1$.
\pespace
\noindent De $l-1=q^1_l n + r^1_l-1$ suit alors
$$q^1_{l-1}=q^1_l \;, \quad r^1_{l-1}=r^1_l - 1 \;;$$
d'où
$$\sigma(l)= (r^1_l - 1) m + q^1_l + 1 \;.$$
On en déduit par l'isomorphisme $isom_{(q^1_{l-1},r^1_{l-1})}$ que

$$\begin{array}{rcl}
Gal(T'^1_l/T'^1_{l-1}) 	&=		&Gal(T^1_{q^1_l+1} \cap T^1_{q^1_l}
T^2_{r^1_l} / T^1_{q^1_l+1} \cap T^1_{q^1_l} T^2_{r^1_l-1})\\
\\
			&\isomto	&Gal(T^2_{r^1_l} \cap T^1_{q^1_l+1}
			T^2_{r^1_l-1} / T^2_{r^1_l} \cap T^1_{q^1_l}
			T^2_{r^1_l-1}) \;.\\
\end{array}$$
Pour $q^1_l+1 \leq m-1$, ce dernier groupe est, par définition de la tour
$(T'^2)$, égal à
$$Gal(T'^2_{(r^1_l-1)m+q^1_l+1} / T'^2_{(r^1_l-1)m+q^1_l}) = 
Gal(T'^2_{\sigma(l)} /T'^2_{\sigma(l)-1}) \;.$$
Enfin si $q^1_l+1 = m$, on a $\sigma(l)=r^1_l m$, d'où
$$T^2_{r^1_l} \cap T^1_{q^1_l+1} T^2_{r^1_l-1} = T^2_{r^1_l} \cap L
T^2_{r^1_l-1} = T^2_{r^1_l} \;.$$
Et l'on a déjà noté que $T^2_{r^1_l}=T'^2_{r^1_l m}=T'^2_{\sigma(l)}$.

Ainsi dans tous les cas, on a prouvé que
$$\forall l \in \{1, \dots ,mn\} \quad Gal(T'^1_l/T'^1_{l-1}) \isomto
Gal(T'^2_{\sigma(l)} /T'^2_{\sigma(l)-1}) \;,$$
i.e. que $(T'^2) \sim (T'^1)$.
Ceci achève la démonstration du $1^{er}$ théorème de dissociation.
\end{proof}

\gespace
Dans le cas de tours strictes, le théorème \ref{th:1diss} précédent peut être
amélioré en le suivant :

\mespace
\begin{cor} \label{cor:1dissstrict}
Deux tours galoisiennes strictes d'une même extension galtourable admettent des
raffinements qui sont des tours galoisiennes strictes équivalentes.
\end{cor}

\mespace
\begin{proof}
Dans les notations des proposition \ref{prop:butterfly} et théorème
\ref{th:1diss}, les tours $(T^1)$ et $(T^2)$, supposées strictes, admettent les
raffinements équivalents $(T'^1)$ et $(T'^2)$, dont rien ne permet d'affirmer
qu'ils sont stricts.
La proposition \ref{prop:tourgalstrictequiv} nous montre quant à elle que les
tours galoisiennes strictes associées $(T'^1_<)$ et $(T'^2_<)$ sont équivalentes.
Or le corollaire 2.4 du chapitre 3 
conclut que $(T'^1_<)$ (resp. $(T'^2_<)$)
est encore un raffinement de $(T^1)$ (resp. $(T^2)$).
\end{proof}

\gespace
Ce théorème \ref{th:1diss} admet encore une généralisation aux tours
galtourables :

\mespace
\begin{Th} [] ($1^{er}$ théorème de dissociation bis)
\label{th:1dissbis}\index{Théorème!de
dissociation}\index{Dissociation!$1^{\text{er}}$ Théorème bis (de)}\\
Deux tours galtourables d'une même extension galtourable admettent des
raffinements qui sont des tours galoisiennes équivalentes.
\end{Th}

\mespace
\begin{proof}
Soient $L/K$ une extension galtourable et $(E^1)$, $(E^2)$ deux tours
galtourables de $L/K$. Par la proposition 1.9 du chapitre 3, 
$(E^1)$ (resp. $(E^2)$) admet un raffinement qui est une tour galoisienne $(T^1)$
(resp. $(T^2)$).
Le théorème \ref{th:1diss} assure alors que les tours galoisiennes $(T^1)$ et
$(T^2)$ de $L/K$ admettent des raffinements équivalents $(T'^1)$ et $(T'^2)$.
Or, par le Fait 1.6 du chapitre 3, 
$(T'^1)$ (resp. $(T'^2)$) est encore un raffinement de $(E^1)$ (resp. $(E^2)$).
\end{proof}

\mespace
De même, le théorème précédent admet une version pour les tours strictes :

\mespace
\begin{cor} \label{cor:1dissbisstrict}
Deux tours galtourables strictes d'une même extension galtourable  admettent des
raffinements qui sont des tours galoisiennes strictes équivalentes.
\end{cor}

\pespace
\begin{proof}
Celle du corollaire \ref{cor:1dissstrict} mutatis mutandis.
\end{proof}


\mespace
\mespace
\section{Deuxième et troisième théorèmes de dissociation}

Notons la proposition suivante, essentielle pour le théorème général
\ref{th:2diss} à suivre :

\mespace
\begin{prop} \label{prop:galinfpasgalsimple}
Une extension galoisienne infinie n'est jamais galsimple.
\end{prop}

\mespace
\begin{proof}
Soit $L/K$ une extension galoisienne infinie. Considérons $\{E_i\}_{i \in I}$
l'ensemble des corps intermédiaires entre $K$ et $L$ tels que l'extension
quotient $E_i/K$ soit galoisienne finie :
$$\forall i \in I \quad K \unlhd E_i \leq L \qquad [E_i:K] < \infty \;.$$
On sait \cite[p.16, Satz 2.3]{Ko1} (ou \cite{Ko}) que
$$L = \bigcup_{i \in I} E_i \;.$$
Si l'on avait $E_i=K$ pour tout $i$ dans $I$, on aurait donc $L=K$, ce qui
contredirait l'hypothèse $[L:K]=\infty$.
Par conséquent
$$\exists i_0 \in I \qquad E_{i_0} \neq K \;.$$
De plus
$$([E_{i_0}:K] < \infty \;, \quad [L:K] = \infty) \qquad \implique \qquad E_{i_0} \neq L
\;.$$
Ainsi $K \lhd E_{i_0} < L$, et l'extension $L/K$ n'est pas galsimple.
\end{proof}

\gespace
En général, un groupe n'admet pas de suite de composition. Les groupes finis en
admettent une, mais ce ne sont pas les seuls. Nous avons prouvé que les
extensions de corps, même séparables, n'admettent pas nécessairement de tour
galoisienne (Chap. 2, Ex. 1.10.(i) ). 
Avec la notion d'extension galtourable, le théorème suivant répond à la question
de savoir quelles sont exactement les extensions de corps admettant une tour de
composition galoisienne.

\mespace
\begin{Th}[] ($2^{\text{ème}}$ théorème de dissociation) 
\label{th:2diss}\index{Théorème!de
dissociation}\index{Dissociation!$2^{\text{ème}}$ Théorème (de)}\\
Une extension de corps admet une tour de composition galoisienne si et seulement
si elle est galtourable finie.
\end{Th}

\mespace \noindent
{\it Scholie.} En particulier, la proposition \ref{prop:tourcompositiongalfinie} se généralise donc aux
extensions galtourables.

\mespace
\begin{proof}
Soit $L/K$ une extension admettant une tour de composition galoisienne 
$$(F) \qquad K=F_0 \lhd \dots \lhd F_i \lhd \dots \lhd F_m=L \;.$$
Chacune des marches de $(F)$ est galsimple en vertu de la proposition
\ref{prop:compssigalgalsimple}. Par définition (Chap. 2, Déf. 1.4), 
$L/K$ est en particulier galtourable. Supposons qu'elle soit de degré infini.
De $\infty = [L:K]={\displaystyle\prod^{m-1}_{i=0}}[F_{i+1}:F_i]$ suit qu'au
moins l'une des marches est infinie :
$$\exists i_0 \in \{0, \dots,m-1\} \quad [F_{i_0+1}:F_{i_0}]=\infty \;.$$
Mais alors $F_{i_0+1} / F_{i_0}$ est galoisienne infinie et galsimple, ce qui
contredit la proposition \ref{prop:galinfpasgalsimple} précédente.

\pespace
Réciproquement, supposons l'extension $L/K$ galtourable finie non triviale (le
résultat est vrai pour l'extension triviale d'après le Fait
\ref{fait:tourcompositiontriviale}). Soit 
$$(F) \qquad K=F_0 \unlhd \dots \unlhd F_i \unlhd \dots \unlhd F_m=L$$
une tour galoisienne de $L/K$. Quitte à prendre sa tour stricte associée
$(F_<)$, on peut supposer que $(F)$ est une tour stricte en vertu du corollaire
2.6 du chapitre 3. 
L'extension $L/K$ étant de degré fini, il en est de même de chacune de ses marches
$F_{i+1} \gal F_i \quad (i=0, \dots, m-1)$. En tant qu'extensions galoisiennes
finies, celles-ci admettent des tours de composition galoisiennes d'après la
proposition \ref{prop:tourcompositiongalfinie} :
$$(T_i) \qquad F_i=T_{j_i} \lhd \dots \lhd T_{j_{i+1}}=F_{i+1}$$
où $j_i < j_{i+1}$ puisque $(F)$ est stricte.
De même que dans la démonstration de la proposition 1.9 du chapitre 3, 
la juxtaposition des tours galoisiennes $(T_i) \quad(i=0, \dots, m-1)$ donne la
tour galoisienne stricte
$$(T) \qquad K=F_0=T_{j_0} \lhd \dots \lhd T_{j_1}=F_1 \lhd \dots \lhd
T_{j_m}=F_m=L$$
de $L/K$. Chacune des marches de $(T)$ est une marche de l'une des tours
$(T_i)$, donc est galsimple (cf. Prop. \ref{prop:compssigalgalsimple}).
Finalement la tour galoisienne $(T)$ est elle-même de composition.
\end{proof}

\mespace
Nous sommes maintenant en mesure de démontrer l'analogue galoisien au théorème
de Jordan-Hölder annoncé dans l'introduction du présent chapitre.

\mespace
\begin{Th}[] ($3^{\text{ème}}$ théorème de dissociation, dit "de
Galjordanhölder") \label{th:3diss}\index{Théorème!de
dissociation}\index{Théorème!de
Galjordanhölder}\index{Dissociation!$3^{\text{ème}}$ Théorème (de)}\\
Soit $L/K$ une extension finie galtourable.\\
(1) Toute tour galoisienne stricte de $L/K$ admet un raffinement qui est une
tour de composition galoisienne de $L/K$.\\
(2) Deux tours de composition galoisiennes de $L/K$ sont équivalentes.
\end{Th}

\mespace \noindent
{\it Scholie. }Le raffinement du (1) est nécessairement un raffinement
galoisien en vertu du Fait 1.5.(2) du chapitre 3. 

\mespace
\begin{proof}
(1) D'après le $2^{\text{ème}}$ théorème de dissociation (Th. \ref{th:2diss}), $L/K$
admet une tour de composition galoisienne $(C)$. Soit $(T)$ une tour galoisienne stricte de
$L/K$. D'après le théorème de Galschreier (Th. \ref{th:1diss}), $(C)$ et $(T)$
admettent des raffinements $(C')$ et $(T')$ qui sont des tours galoisiennes
équivalentes. Ces raffinements sont galoisiens (Chap. 3, Fait 1.5.(2))
. Comme la tour de composition $(C)$ n'admet aucun raffinement galoisien
propre, $(C')$ est nécessairement un raffinement trivial de $(C)$. Or $(C)$
est stricte par définition. C'est donc la tour stricte associée à $(C')$
(Chap. 3, Prop. \& Déf. 2.1), 
d'où
$$(C)=(C'_<) \;.$$
Par la proposition \ref{prop:tourgalstrictequiv}, on en déduit que les tours
galoisiennes $(C)=(C'_<)$ et $(T'_<)$ sont équivalentes. Comme $(C)$ est de
composition, il résulte alors du corollaire \ref{cor:equivtourcomp} 
que le raffinement $(T'_<)$ de $(T)$ est une tour de composition galoisienne.
De plus d'après le corollaire 2.8 du chapitre 3, 
$(T'_<)$ est encore un raffinement de $(T)$.

\pespace \noindent
(2) Dans le (1) précédent, lorsque $(T)$ est une tour de composition
galoisienne, le même
argument que pour $(C)$ conduit à $(T'_<)=(T)$.
Et finalement $(T'_<) \sim (C'_<)$ signifie $(T) \sim (C)$. 
\end{proof}

\gespace
Comme le $1^{\text{er}}$ théorème de dissociation, ce théorème \ref{th:3diss}
admet une généralisation aux tours galtourables.

\mespace
\begin{Th}[] ($3^{\text{ème}}$ théorème de dissociation bis)\index{Théorème!de
dissociation}\index{Dissociation!$3^{\text{ème}}$ Théorème bis (de)} \\
Soit $L/K$ une extension finie galtourable.\\
(1) Toute tour galtourable stricte de $L/K$ admet un raffinement qui est une
tour de composition galoisienne de $L/K$.\\
(2) Deux tours de composition galoisiennes de $L/K$ sont équivalentes.
\end{Th}

\mespace
\begin{proof}
\noindent (1) Soit $(E)$ une tour galtourable stricte de $L/K$. Par la
proposition 2.9 du chapitre 3, 
elle admet un raffinement $(T)$ qui est une tour galoisienne stricte de $L/K$.
Et d'après le théorème \ref{th:3diss}.(1) précédent, $(T)$ admet un raffinement
$(C)$ qui est une tour de composition galoisienne de $L/K$. On déduit alors de
la transitivité de la notion de raffinement (Chap. 3, Fait 1.6
) que $(C)$ est un raffinement de $(E)$.

\noindent (2) Identique à celui du $3^{\text{ème}}$ théorème de dissociation.
\end{proof}

\addtocontents{toc}{\mespace\pespace}
\chapter{ILLUSTRATIONS ARITHMÉTIQUES ET GALSIMPLICITÉ}
\addtocontents{lof}{\gespace}
\addtocontents{lof}{\noindent Chapitre \thechapter}
\addtocontents{lof}{\pespace}
\addtocontents{toc}{\pespace}
\gespace
Ce chapitre 5 est un chapitre de transition, une respiration arithmétique avant
les tours d'élévation qui nous permettrons de dissocier toutes les extensions
finies et pas seulement les galtourables. La section 1 illustre les théorèmes de
Galschreier et Galjordanhölder du chapitre 4 par une extension galtourable de
degré 480 dont nous donnons deux tours galoisiennes différentes que l'on
raffine en deux tours de composition galoisiennes équivalentes. La section 2
fournit, via un résultat de Selmer-Serre, une classe infinie d'extensions
simples (au sens du (1) de la définition 1.6 du chapitre 2) 
mais non galoisiennes. La section 3 montre, via un contre-exemple, que le
"Théorème $M$" du chapitre 6 suivant ne s'étend pas au cas d'une extension
infinie. Ce contre-exemple justifie la finitude des extensions du chapitre 7
final. Enfin, la section 4 donne quelques propriétés des extensions galsimples
non galoisiennes, notamment leur transitivité qui nous sera utile pour la
maximalité des sous-extensions d'intourabilité.

\gespace
\gespace
\section[Illustrations des théorèmes par une extension de degré 480]
{Illustrations des théorèmes de Galschreier et \newline Galjordanhölder par une
extension de degré 480}
\mespace
On note génériquement $\zeta_n := e^{2i \pi /n} \quad (n \in {\mathbb N} 
\setminus \{ 0 \})$. 
Dans toute cette section, on considère les tours
$$(T^1) \qquad K=T^1_0:={\mathbb Q} < T^1_1:=T^1_0(i,\sqrt[4]{5}) <
T^1_2:=T^1_1(\zeta_{15},Y^{1/5},Z^{1/3})=L $$
et
\vskip -8mm
$$(T^2) \qquad K=T^2_0:={\mathbb Q} < T^2_1:=T^2_0(\zeta_{15}) <
T^2_2:=T^2_1(i,Y^{1/5}) < T^2_3=L \;,$$
où l'on désigne par :
$$\left\{\begin{array}{l}
- \, Y^{1/5} \text{ l'une quelconque des racines cinquièmes complexes de} \\ 
\qquad \qquad Y:=(2-\zeta_5)^3(2-\zeta_5^4)^2 \;;\\
- \, Z^{1/3} \text{ l'une quelconque des racines troisièmes complexes de} \\
\qquad \qquad Z:=6 - \sqrt{5} \;.
\end{array}\right.$$

\gespace
\begin{fait} \label{fait:5A} 
L'extension ${\mathbb Q}(\zeta_5)(Y^{1/5}) / {\mathbb Q}(\zeta_5)$ est cyclique
de degré 5.
\end{fait}

\mespace
\begin{proof}
Sinon, modulo ${\mathbb Q}(\zeta_5)^{\times 5}$,
$$\overline{Y}=\overline{\iit} \qquad \SSI \qquad
\overline{2-\zeta_5^4}=\overline{2-\zeta_5} \;.$$
En particulier, pour le corps local $C:={\mathbb Q}_5(\zeta_5)$, on a alors
$$\frac{2-\zeta_5^4}{2-\zeta_5} \in C^{\times 5} \;.$$
Prouvons que ceci n'est pas, grâce aux méthodes (et notations) de \cite{Se} et
\cite{Wy}. Tout d'abord, $2-\zeta_5^4$ est une unité principale de $C$ :
$$2-\zeta_5^4 \in U_C^1 \;,$$
car pour l'uniformisante $1- \zeta_5$, on a les valuations
$$\text{ord}(1-(2-\zeta_5^4))=\text{ord}(\zeta_5^4(1-\zeta_5))=1 \;.$$
Comme il en est de même de $2-\zeta_5$, on a donc 
$$\frac{2-\zeta_5^4}{2-\zeta_5} \in U_C^1 \;.$$
Calculons le défaut de cette unité principale :
$$\begin{array}{rcl}
\text{ord}(1-\frac{2-\zeta_5^4}{2-\zeta_5})	&=
&\text{ord}(-\zeta_5+\zeta_5^4)=\text{ord}(\zeta_5^4(1-\zeta_5^2))\\
					&=
&\text{ord}(1-\zeta_5)+\text{ord}(1+\zeta_5)\\
					&=	&1+\text{ord}(2-(1-\zeta_5))\\
					&=	&1 \;;
				
\end{array}$$
on en déduit que 
$$\text{def}\left(\frac{2-\zeta_5^4}{2-\zeta_5}\right)=1 \qquad \qquad
\text{(cf. \cite{Wy})} \;.$$
D'où la conclusion puisque
$$\frac{2-\zeta_5^4}{2-\zeta_5} \in C^{\times 5} \quad \SSI \quad
\text{def}(\frac{2-\zeta_5^4}{2-\zeta_5})=+ \infty\;.$$
\end{proof}

\gespace
\begin{fait} \label{fait:5B}
(0) La tour $(T^2)$ est galoisienne ;\\

\noindent (1) $\sqrt[4]{5} \notin {\mathbb Q}(i,\zeta_3,\zeta_5)$ ;\\

\noindent (2) $T_1^1 \cap T_1^2 = {\mathbb Q}(\sqrt{5})$ ;\\

\noindent (3) $T_1^1 \cap {\mathbb Q}(\zeta_5) = {\mathbb Q}(\sqrt{5})$ ;\\

\noindent (4) $Gal({\mathbb Q}(\zeta_{15}) / {\mathbb Q}(\sqrt{5})) \isomto
({\mathbb Z} / 2 {\mathbb Z})^2$.
\end{fait}

\mespace
\begin{proof}
(0) La marche $L=T_2^2(\sqrt[4]{5}, Z^{1/3})/T_2^2$ est le compositum de deux
extensions kummériennes.

\noindent (1) L'extension ${\mathbb Q}(i, \zeta_3, \zeta_{5}) \gal {\mathbb Q}$
est cyclotomique, donc abélienne. Si l'on avait $\sqrt[4]{5} \in {\mathbb Q}(i,
\zeta_3, \zeta_{5})$, l'extension ${\mathbb Q}(\sqrt[4]{5}) / {\mathbb Q}$
serait abélienne, comme quotient d'une extension abélienne, ce qui n'est pas. \\

\noindent (2) D'après \cite[p.74, Th.2]{La2}, $i \notin T_1^2= {\mathbb
Q}(\zeta_{15})$, d'où 
$$[T_1^2(i) : T_1^2]=2 \;.$$
Par ailleurs
$$\zeta_5 + \zeta_5^4 = \frac{\sqrt{5}-1}{2}$$
(explicitement écrit dans \cite[p.67-68]{Esc} !). Donc
$${\mathbb Q}(\sqrt{5}) <  {\mathbb Q}(\zeta_{5}) < T_1^2= {\mathbb Q}(\zeta_3,
\zeta_{5})\;.$$
En particulier, $\sqrt{5} \in T_1^2(i)$ ; tandis que par le (1), $\sqrt[4]{5}
\notin T_1^2(i)$. Dès lors,
$$[T_1^1 T_1^2 = {\mathbb Q}(i, \zeta_3, \zeta_{5}, \sqrt[4]{5}) :  T_1^2(i)]=2 \;.$$
On a ainsi la figure

\begin{figure}[!h]
\begin{center}
\vskip -6mm
\includegraphics[width=10.5cm]{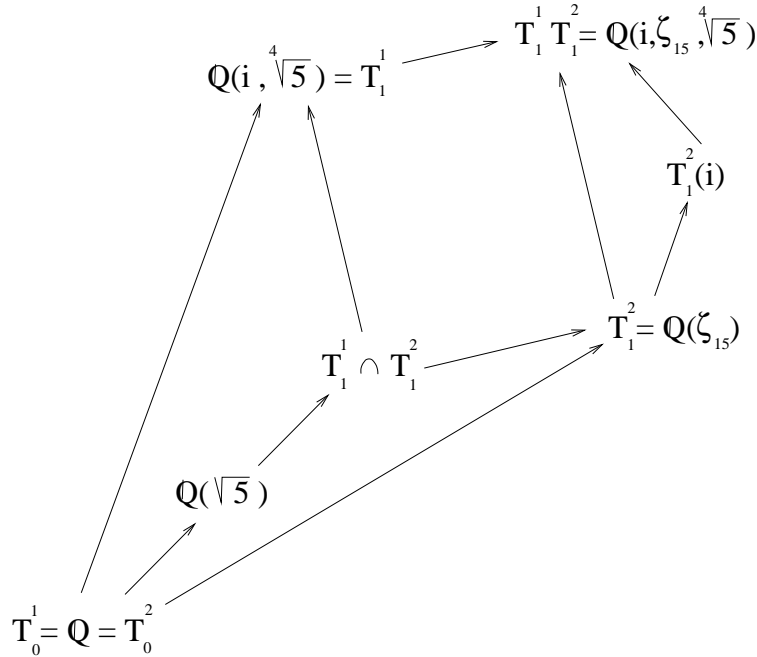}
\end{center}
\vskip -5mm
\rm{\caption{\label{fig:1} \leg Parallélogramme $[T_1^1 \cap T_1^2, T_1^2, T_1^1
T_1^2,T_1^1] $}}
\vskip -1mm
\end{figure}

\noindent dans laquelle le quadrilatère $(T_1^1 \cap T_1^2, T_1^2, T_1^1 T_1^2,T_1^1 )$
est un parallélogramme galoisien. En effet, l'extension $T_1^1 / T_1^1 \cap
T_1^2$ {\large (}resp. $T_1^2 / T_1^1 \cap T_1^2$ {\large )} est galoisienne comme sous-extension
de $T_1^1 / {\mathbb Q}$ {\large (}resp. $T_1^2 / {\mathbb Q}$ {\large )}
 qui est galoisienne de degré 8 car $T_1^1={\mathbb Q}(i,\sqrt[4]{5})$ est le
 corps de décomposition du polynôme $X^4-5$  {\large (}resp. car $[T_1^2={\mathbb
 Q}(\zeta_{15}): {\mathbb Q}]=\varphi(15)=8$ {\large )}. Par conséquent :
$$\begin{array}{rcl}
 8=[T_1^1:{\mathbb Q}]	&=	&[{\mathbb  Q}(\sqrt{5}): {\mathbb Q}] [T_1^1
 \cap T_1^2 : {\mathbb  Q}(\sqrt{5})] [T_1^1 : T_1^1 \cap T_1^2 ]\\
 			&=	&2 [T_1^1 \cap T_1^2 : {\mathbb  Q}(\sqrt{5})]
			[T_1^1 T_1^2: T_1^2]\\
			&=	&2 [T_1^1 \cap T_1^2 : {\mathbb  Q}(\sqrt{5})]
			[T_1^2(i): T_1^2] [T_1^1 T_1^2: T_1^2(i)]\\
			&=	&8 [T_1^1 \cap T_1^2 : {\mathbb  Q}(\sqrt{5})]\\
\ssi T_1^1 \cap T_1^2 	&= 	&{\mathbb  Q}(\sqrt{5})\;.
\end{array} $$

\noindent (3) Comme $\sqrt{5} \in {\mathbb Q}(\zeta_{5})$, on a par le (2) 
$${\mathbb Q}(\sqrt{5}) \leq T_1^1 \cap {\mathbb Q}(\zeta_{5}) \leq T_1^1 \cap
{\mathbb Q}(\zeta_{15}) = T_1^1 \cap T_1^2 = {\mathbb Q}(\sqrt{5}) \;.$$

\noindent (4) Comme $T_1^2={\mathbb Q}(\zeta_{15})$, on a le parallélogramme
$[{\mathbb Q}, {\mathbb Q}(\zeta_{3}), T_1^2, {\mathbb Q}(\zeta_{5})]$.

\begin{figure}[!h]
\begin{center}
\vskip -6mm
\includegraphics[width=8cm]{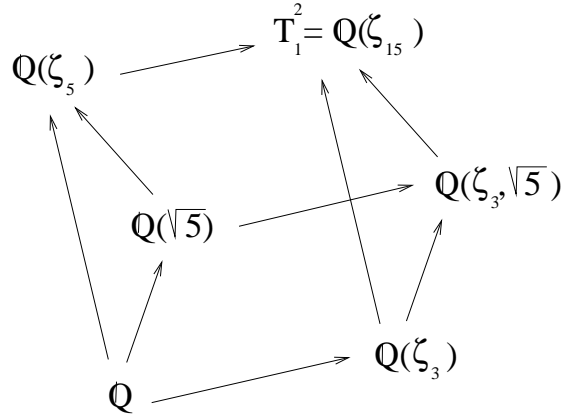}
\end{center}
\vskip -5mm
\rm{\caption{\label{fig:2} \leg Un sous-parallélogramme de $[{\mathbb Q}, {\mathbb
Q}(\zeta_{3}), T_1^2, {\mathbb Q}(\zeta_{5}) ]$}}
\vskip 1mm
\end{figure}

D'après le corollaire 4.3.(1-1) du chapitre 1, 
on a le sous-parallélogramme \\$[{\mathbb Q}(\sqrt{5}), {\mathbb Q}(\zeta_{3},
\sqrt{5}), T_1^2, {\mathbb Q}(\zeta_{5})]$. On en déduit, par scindement de la
diagonale (Chap. 1, Prop. 4.1
) que 
$$Gal(T_1^2 / {\mathbb Q}(\sqrt{5})) \isomto Gal(T_1^2 / {\mathbb Q}(\zeta_3,
\sqrt{5})) \times Gal(T_1^2 / {\mathbb Q}(\zeta_5)) \;.$$
Comme $[T_1^2 : {\mathbb Q}(\zeta_3, \sqrt{5})] = [T_1^2 : {\mathbb
Q}(\zeta_5)]=2$, on a bien le résultat annoncé :
$$Gal(T_1^2 / {\mathbb Q}(\sqrt{5})) \isomto
({\mathbb Z} / 2 {\mathbb Z})^2 \;.$$
\end{proof}

\gespace
Nous avons déjà dit que la tour $(T^2)$ est galoisienne {\large (}(0) du Fait
précédent{\large )}, le lecteur aura deviné qu'il en est de même de la tour
$(T^1)$. Nous le démontrons maintenant, en utilisant deux fois un puissant outil
introduit par Richard Massy dans \cite{M2} sous le nom de "moyenne galoisienne"
("Galois average").
\mespace

\begin{prop} \label{prop:5C}
L'extension $L=T_2^1 / T_1^1$ est galoisienne, non abélienne, de degré 60.
\end{prop}

\mespace
\begin{proof}
Soient $v$ et $w$ les générateurs de 
$$Gal(T_1^2={\mathbb Q}(\zeta_{15}) / {\mathbb Q}) = Gal({\mathbb
Q}(\zeta_{3}) / {\mathbb Q}) \times Gal({\mathbb Q}(\zeta_{5}) /
{\mathbb Q})$$
définis par
$$v(\zeta_3)=\zeta_3^2 \:, \; v(\zeta_5)=\zeta_5 \;; \quad w(\zeta_5)=\zeta_5^2
\:, \;w(\zeta_3)=\zeta_3 \;.$$
On a 
$$ w^2(\zeta_5+\zeta_5^4)=\zeta_5^4+\zeta_5=\frac{\sqrt{5}-1}{2} \;;$$
donc $w^2$ laisse fixe $\sqrt{5}$ et
$$Gal(T_1^2/{\mathbb Q}(\sqrt{5}))=\{1,v\} \times \{1,w^2\} \;.$$
Or, d'après le (2) du Fait \ref{fait:5B}, on a le parallélogramme galoisien
$$[{\mathbb Q}(\sqrt{5}), T_1^2, T_1^1 T_1^2, T_1^1] \;.$$ 
En notant $\widetilde{v}$ et $\widetilde{w^2}$ les prolongements respectifs de
$v$ et $w^2$ à $T_1^1 T_1^2$, on en déduit que
$$Gal(T_1^1 T_1^2 / T_1^1)=\{1,\widetilde{v}\} \times \{1,\widetilde{w^2}\} \;.$$
Soit $\eta$ "l'homomorphisme cyclotomique" (\cite{M2}) de $Gal(T_1^1 T_1^2/
T_1^1)$ dans le groupe multiplicatif ${\mathbb F}_5^{\times}$ du corps à cinq
éléments, défini par
$$\widetilde{w^2}(\zeta_5)=\zeta_5^4 \;\ssi \; \eta(\widetilde{w^2})=4
\;; \qquad \widetilde{v}(\zeta_5)=\zeta_5 \; \ssi \; \eta(\widetilde{v})=1
\;.$$
Soient $y:=2 - \zeta_5$ et $\overline{y}$ sa classe dans $(T_1^1 T_1^2)^{\times}
/ (T_1^1 T_1^2)^{\times 5}$. La moyenne galoisienne de $T_1^1 T_1^2$ sur $T_1^1$
pour $\eta$ en $\overline{y}$ est par définition \cite[Sect. 2]{M2}
$$\begin{array}{rcl}
ga_{T_1^1 T_1^2 / T_1^1}^{\eta}(\overline{y})	&=	&\text{\Large (}
{\displaystyle \prod_{\gamma \in Gal(T_1^1 T_1^2 / T_1^1)}}
\gamma(\overline{y})^{\eta(\gamma)^{-1}} \text{\Large )}^{4^{-1}} \\
\\
						&=	&\text{\large (}
\overline{y} \, \widetilde{v}(\overline{y}) \, \widetilde{w^2}(\overline{y})^4
\, \widetilde{v}\widetilde{w^2}(\overline{y})^4 \text{\Large )}^{4^{-1}} \\
\\
						&=	&\overline{y}^4 \,
\widetilde{v}(\overline{y})^4 \, \widetilde{w^2}(\overline{y}) \,
\widetilde{v}\widetilde{w^2}(\overline{y}) \\
\\
						&=
&\overline{(2-\zeta_5)^4 \, (2-\zeta_5)^4\, (2-\zeta_5^4)\, (2-\zeta_5^4)} \\
\\
						&=
&\overline{(2-\zeta_5)^3 \, (2-\zeta_5^4)^2}= \overline{Y} \;.
\end{array}$$
D'après \cite{M2}, l'extension $T_1^1 T_1^2(Y^{1/5}) / T_1^1$ est galoisienne.
On a $T_1^2(Y^{1/5}) \cap T_1^1 T_1^2= T_1^2$ par primalité des degrés car $[
T_1^1 T_1^2: T_1^2]=4$ (cf. (2) de la démonstration du Fait \ref{fait:5B}) et
$[T_1^2(Y^{1/5}) : T_1^2] \in \{1,5\}$. On en déduit le parallélogramme
galoisien $[T_1^2, T_1^1 T_1^2, T_1^1 T_1^2(Y^{1/5}), T_1^2(Y^{1/5})]$.

\begin{figure}[!h]
\begin{center}
\vskip -6mm
\includegraphics[width=10cm]{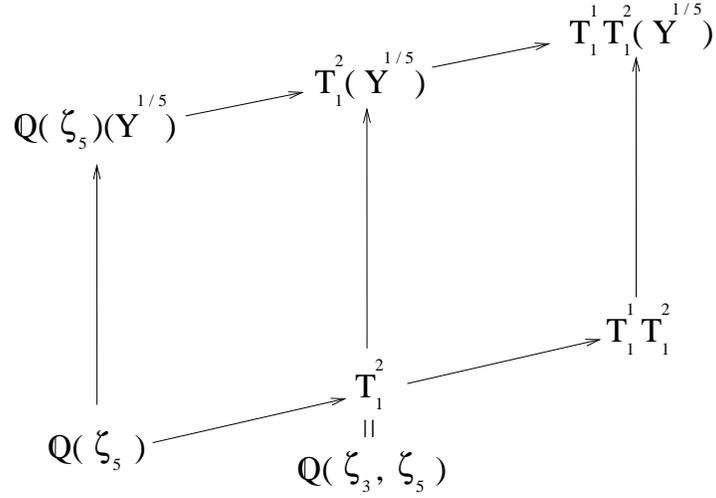}
\end{center}
\vskip -5mm
\rm{\caption{\label{fig:3} \leg Deux parallélogrammes adjacents}}
\vskip -3mm
\end{figure}

Or d'après le Fait \ref{fait:5A}, on a $[{\mathbb Q}(\zeta_{5})(Y^{1/5}) :
{\mathbb Q}(\zeta_{5})]=5$. A nouveau par primalité des degrés, on en déduit le
parallélogramme
$$[{\mathbb Q}(\zeta_{5}),T_1^2, T_1^2(Y^{1/5}), {\mathbb
Q}(\zeta_{5})(Y^{1/5})] \;.$$
Par conséquent
$$5=[T_1^2(Y^{1/5}) : T_1^2] = [T_1^1 T_1^2(Y^{1/5}) : T_1^1 T_1^2]\:.$$
On obtient ainsi le degré de l'extension galoisienne $T_1^1 T_1^2(Y^{1/5})/
T_1^1$ : 
$$[T_1^1 T_1^2(Y^{1/5}) : T_1^1] = 5 [T_1^1 T_1^2 : T_1^1] = 20 \;.$$
\pespace
Montrons que l'extension $T_2^1 / T_1^1$ est obtenue par une $3$-moyenne au
dessus de l'extension $T_1^1 T_1^2(Y^{1/5}) / T_1^1$. Pour l'homomorphisme
trivial 
$$\begin{array}{crccc} 
\iit :	&G:=	&Gal(T_1^1 T_1^2(Y^{1/5}) / T_1^1)	&\rightarrow
&{\mathbb F}_3^{\times} \;,\\
	&	&\gamma					&\mapsto
&1
\end{array}$$
on a
$$\begin{array}{rcl}
ga_{T_1^1 T_1^2(Y^{1/5}) / T_1^1}^{\iit}(\overline{Y}^{\,-1/5})	&=	&\text{\Large (}
{\displaystyle \prod_{\gamma \in G} \gamma(\overline{Y}^{\,-1/5}) \text{\Large
)}^{20^{-1}} =
\text{\Large (} \prod_{\gamma \in G} \gamma(\overline{Y}^{\,-1/5}) \text{\Large
)}^{-1} }\\
\\
	&=	&\overline{N_{T_1^1 T_1^2(Y^{1/5}) / T_1^1}(Y^{\,1/5})} \;.
\end{array}$$
Notons que $T_1^1 T_1^2(Y^{1/5})=T_1^1 T_2^2$, et
$$\begin{array}{rcl}
N_{T_1^1 T_2^2 / T_1^1}(Y^{\,1/5})	&=	&N_{T_1^1 T_1^2 /
T_1^1} \text{\Large (}N_{T_1^1 T_2^2 / T_1^1 T_1^2}(Y^{\,1/5})\text{\Large
)}\\
\\
						&=	&N_{T_1^1 T_1^2 /
T_1^1} ( Y^{\,1/5} \, \zeta _5 Y^{\,1/5} \, \zeta _5 ^2 Y^{\,1/5} \, \zeta^3 _5
Y^{\,1/5} \, \zeta _5^4 Y^{\,1/5})\\
\\
						&=	&N_{T_1^1 T_1^2 /
T_1^1} (Y) \\
\\
						&=	&Y \,\widetilde{v}(Y)
\,\widetilde{w^2}(Y) \,\widetilde{v}\widetilde{w^2}(Y) \;.
\end{array}$$
Comme $\widetilde{v}(\zeta_5)=\zeta_5$, on a $\widetilde{v}(Y)=Y$ et
$$\begin{array}{rcl}
ga_{T_1^1 T_2^2 / T_1^1}^{\iit}(\overline{Y}^{\,-1/5})	&=	&\overline{
Y^2 \,\widetilde{w^2}(Y)^2 } \\
\\
							&=
&\overline{(2-\zeta_5)^6 \, (2-\zeta_5^4)^4 \, (2-\zeta_5^4)^6 \,
(2-\zeta_5)^4} \\
\\
							&=
&\overline{(2-\zeta_5) \, (2-\zeta_5^4)} \\
\\
							&=
&\overline{4 - 2(\zeta_5 + \zeta_5^4) + 1} \\
\\
 							&=
&\overline{4 - (\sqrt{5} - 1) + 1} \\
\\
 							&=
&\overline{6 - \sqrt{5}} \;.\\
\end{array}$$
On en déduit, toujours par \cite{M2}, que l'extension $L/T_1^1$ est galoisienne,
car on a :
$$\begin{array}{rcl}
L	&=	&T^1_2=T^1_1(\zeta_{15},Y^{1/5},Z^{1/3})=T^1_1
T^2_1(Y^{1/5},Z^{1/3})\\
\\
	&=	&T^1_1 T^2_1(Y^{1/5})(Z^{1/3})=T^1_1 T^2_2(Z^{1/3})\;.\\
\end{array}$$
Elle est de degré 
$$[L:T_1^1]=[T_1^1 T_2^2 (Z^{1/3}) : T_1^1 T_2^2] [T_1^1 T_2^2 :
T_1^1] =3 \times 20  = 60$$
car $Z=6-\sqrt{5} \notin (T_1^1 T_2^2)^{\times 3}$. En effet sinon on aurait
$N_{T_1^1 T_2^2 / {\mathbb Q}}(6 - \sqrt{5}) \in {\mathbb Q}^{\times 3} $,
ce qui n'est pas, puisque
$$\begin{array}{rcl}
N_{T_1^1 T_2^2 / {\mathbb Q}}(6 - \sqrt{5})	&=	&N_{{\mathbb
Q}(\sqrt{5})/ {\mathbb Q}}(N_{T_1^1 T_2^2 /
{\mathbb Q}(\sqrt{5})}(6 - \sqrt{5}))\\
\\
						&=	&N_{{\mathbb
Q}(\sqrt{5})/ {\mathbb Q}}(6 - \sqrt{5})^{80}\\
\\
						&=	&31^{80} \notin {\mathbb
Q}^{\times 3} \;.\\
\end{array}$$
Enfin $L/T_1^1$ est non abélienne, car soit $\check{v}$ un prolongement de
$\widetilde{v}$ à $T_1^1 T_2^2$ :

\begin{figure}[!h]
\begin{center}
\vskip -6mm
\includegraphics[width=8cm]{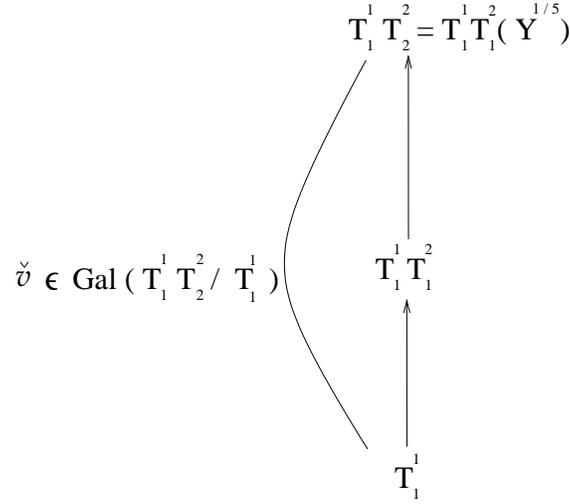}
\end{center}
\vskip -5mm
\rm{\caption{\label{fig:4} \leg Extension galoisienne $T_1^1 T_2^2 / T_1^1$}}
\vskip -3mm
\end{figure}

\newpage
\noindent Clairement
$$\check{v}(\zeta_3)=\widetilde{v}(\zeta_3)=\zeta_3^2 \;.$$
L'homomorphisme cyclotomique de $Gal(T_1^1 T_2^2 /T_1^1)$ dans ${\mathbb
F}^{\times}_3$ n'est donc pas trivial. Comme
$\overline{Z}=\overline{6-\sqrt{5}}$ est dans l'image de la moyenne galoisienne
$ga^1_{T_1^1 T_2^2 / T_1^1}$ pour cet homomorphisme trivial, il résulte du
\cite[Th.1.2.(3.2)]{M2} que l'extension $L/T_1^1$ ne peut être abélienne.
\end{proof}

\mespace
\begin{cor} \label{cor:5D}
L'extension $L/{\mathbb Q}$ est de degré 480.
\end{cor}

\mespace
\begin{proof}
C'est clair puisque
$$[L:{\mathbb Q}]=[L:T_1^1] [T_1^1:{\mathbb Q}]$$
où $[L:T_1^1]=60$ (Prop. \ref{prop:5C}) et $[T_1^1:{\mathbb Q}]=8$ d'après la
démonstration du (2) du Fait \ref{fait:5B}.
\end{proof}

\gespace
La proposition suivante illustre le théorème de Galschreier ($1^{\text{er}}$
théorème de dissociation) pour une extension de degré 480.

\mespace
\begin{prop} \label{prop:5E}
On a les tours galoisiennes (strictes) équivalentes de $L/K$ (Chap. 4, Déf. 1.1.(2)
) $(T'_1) \,\sim\, (T'_2)$ où
$$\begin{array}{rr}
(T'^1) 	&K={\mathbb Q}=:T'^1_0 \; \lhd \; T'^1_1:={\mathbb Q}(\sqrt{5}) \; \lhd \;
T'^1_2:={\mathbb Q}(i,\sqrt{5}) \; \lhd \; T'^1_3:={\mathbb Q}(i,\sqrt[4]{5})  \\
\\
& \!\!\!\!\lhd \; T'^1_4:={\mathbb Q}(i,\sqrt[4]{5},\zeta_{15}) \; \lhd \;
T'^1_5:={\mathbb Q}(i,\sqrt[4]{5},\zeta_{15},Y^{1/5}) \; \lhd \;
T'^1_6:=T'^1_5(Z^{1/3})=L
\end{array}$$
et
$$\begin{array}{rr}
(T'^2) \qquad 	&K={\mathbb Q}=:T'^2_0 \; \lhd \; T'^2_1:={\mathbb Q}(\sqrt{5}) \; \lhd \;
T'^2_2:={\mathbb Q}(\zeta_{15}) \; \lhd \; T'^2_3:={\mathbb Q}(i,\zeta_{15})  \\
\\
&\; \lhd \; T'^2_4:={\mathbb Q}(i,\zeta_{15},Y^{1/5}) \; \lhd \;
T'^2_5:={\mathbb Q}(i,\sqrt[4]{5},\zeta_{15},Y^{1/5}) \; \lhd \;
T'^2_6:=L \;.
\end{array}$$
\end{prop}

\mespace
\begin{proof}
Prouvons qu'elles sont obtenues à partir des tours $(T^1)$ et $(T^2)$ (cf. début
de la présente section) par les formules $({\mathcal F})$ de la démonstration du
$1^{\text{er}}$ théorème de dissociation  (Chap. 4, Th. 3.2
). On a :
$${\mathbb Q}= T_1^1 \cap T_0^1 T_0^2 = T_1^2 \cap T_0^1 T_0^2 \;;$$
$${\mathbb Q}(\sqrt{5})\; \stackrel{\ref{fait:5B}.(2)}{\displaystyle =} T_1^1
\cap T_1^2 = T_1^1 \cap T_0^1 T_1^2 = T_1^2 \cap T_1^1 T_0^2 \;;$$
$${\mathbb Q}(\zeta_{15}) = T_1^2 = T_2^2 \cap T_0^1 T_1^2 \;;$$
$${\mathbb Q}(i,\sqrt[4]{5}) = T_1^1 = T_2^1 \cap T_1^1 T_0^2 \;.$$
D'après la démonstration de la proposition \ref{prop:5C}, on dispose du
parallélogramme galoisien
$$[T_1^2, T_1^2(Y^{1/5}), T_1^1 T_1^2(Y^{1/5}), T_1^1 T_1^2]$$
dans lequel on peut utiliser le (1-1) du corollaire 4.3 
du chapitre 1 : le corps intermédiaire $T_1^2(i)$ induit le sous-parallélogramme
$$[T_1^2(i), T_1^2(i,Y^{1/5}), T_1^1 T_1^2(Y^{1/5}), T_1^1 T_1^2]\;.$$

\begin{figure}[!h]
\begin{center}
\vskip -6mm
\includegraphics[width=7.5cm]{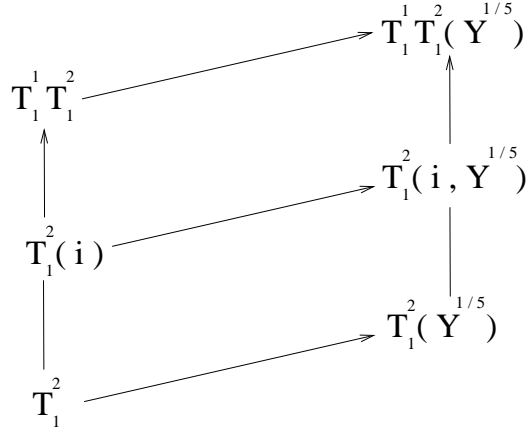}
\end{center}
\vskip -4mm
\rm{\caption{\label{fig:5} \leg Sous-parallélogramme $[T_1^2(i), T_1^2(i,Y^{1/5}),
T_1^1 T_1^2(Y^{1/5}), T_1^1 T_1^2]$ }}
\vskip -1mm
\end{figure}

\noindent On en déduit en particulier que
$${\mathbb Q}(i,\zeta_{15}) = T_1^2(i)= T_1^2(i,Y^{1/5}) \cap T_1^1 T_1^2 =
T_2^2 \cap T_1^1 T_1^2\;.$$
Par ailleurs
$${\mathbb Q}(i,\sqrt[4]{5}, \zeta_{15}) = T_1^1 T_1^2 = T_2^1 \cap T_1^1 T_1^2 \;;$$
$${\mathbb Q}(i, \zeta_{15}, Y^{1/5} ) = T_2^2 = T_3^2 \cap T_0^1 T_2^2 \;;$$
$${\mathbb Q}(i, \sqrt[4]{5},\zeta_{15}, Y^{1/5} ) = T_1^1 T_2^2 = T_2^1 \cap
T_1^1 T_2^2 =T_3^2 \cap T_1^1 T_2^2 \;.$$
Enfin dans le parallélogramme galoisien $[T_1^1 \cap T_1^2 \;
\stackrel{\ref{fait:5B}.(2)}{\displaystyle =} \;{\mathbb Q}(\sqrt{5}), T_1^2
, T_1^1 T_1^2, T_1^1 ]$, on déduit du (1-1) du corollaire 4.3 
 du chapitre 1 appliqué au corps intermédiaire ${\mathbb Q}(i,\sqrt{5}) \leq
T_1^1$, le sous-parallélogramme
$$[{\mathbb Q}(i, \sqrt{5}), T_1^2(i), T_1^1 T_1^2, T_1^1 ] \;.$$
En particulier
$$\begin{array}{rcl}
{\mathbb Q}(i, \sqrt{5})	&=	&T_1^1 \cap T_1^2(i) = T_1^1 \cap
{\mathbb Q}(i, \zeta_{15})\\
\\				&=	&T_1^1 \cap (T_2^2 \cap T_1^1 T_1^2 ) =
T_1^1 \cap T_2^2\\
\\
				&=	&T_1^1 \cap T_0^1 T_2^2
\end{array}$$
\end{proof}

\gespace

Le corollaire suivant illustre le théorème de Galjordanhölder ($3^{\text{ème}}$
théorème de dissociation : Chap. 4, Th. 4.3
) pour la même extension de degré 480.

\mespace
\begin{cor} \label{cor:5F}
On a les tours de composition galoisiennes équivalentes de $L/K$ $(T"_1)
\,\sim\, (T"_2)$ où
$$\begin{array}{rr}
(T"^1) \qquad 	&K={\mathbb Q}=:T"_0^1 \; \lhd \; T"_1^1:={\mathbb Q}(\sqrt{5}) \; \lhd \;
T"_2^1:={\mathbb Q}(i,\sqrt{5})  \qquad\qquad\qquad\qquad\qquad\\
\\
&\; \lhd \; T"_3^1:={\mathbb Q}(i,\sqrt[4]{5})
\; \lhd \; T"_4^1:={\mathbb Q}(i,\sqrt[4]{5},\zeta_{5})
\; \lhd \; T"_5^1:={\mathbb Q}(i,\sqrt[4]{5},\zeta_{15})  \\
\\
&\; \lhd \; T"_6^1:={\mathbb Q}(i,\sqrt[4]{5},\zeta_{15},Y^{1/5}) \; \lhd \;
T"_7^1:={\mathbb Q}(i,\sqrt[4]{5},\zeta_{15},Y^{1/5}, Z^{1/3})=L
\end{array}$$
et
$$\begin{array}{rr}
(T"^2) \qquad 	&K={\mathbb Q}=:T"^2_0 \; \lhd \; T"^2_1:={\mathbb Q}(\sqrt{5}) \; \lhd \;
T"^2_2:={\mathbb Q}(\zeta_{5})  \qquad\qquad\qquad\qquad\qquad\\
\\
&\; \lhd \; T"^2_3:={\mathbb Q}(\zeta_{15})
\; \lhd \; T"^2_4:={\mathbb Q}(i,\zeta_{15})
\; \lhd \; T"^2_5:={\mathbb Q}(i,\zeta_{15},Y^{1/5})  \\
\\
&\; \lhd \; T"^2_6:={\mathbb Q}(i,\sqrt[4]{5},\zeta_{15},Y^{1/5}) \; \lhd \;
T"^2_7:={\mathbb Q}(i,\sqrt[4]{5},\zeta_{15},Y^{1/5}, Z^{1/3})=L \;.
\end{array}$$
\end{cor}

\mespace
\begin{proof}
La tour $(T"^1)$ est un raffinement galoisien de la tour galoisienne $(T'^1)$ de
la proposition \ref{prop:5E} précédente, avec un seul nouveau corps $T"_4^1$.
Par la proposition 1.7 
 du chapitre 3, $(T"^1)$ est donc une tour galoisienne. Puisque $L/{\mathbb Q}$
est de degré 480 (Cor. \ref{cor:5D}), on a nécessairement :
$$\begin{array}{ll}
[T"^1_7 : T"^1_6]=3 \;,\quad 	& [T"^1_6 : T"^1_5]=5 \\ { } 
[T"^1_5 : T"^1_4]=2 \;,\quad 	& [T"^1_4 : T"^1_3]=2 \\ { }
[T"^1_3 : T"^1_2]=2 \;,\quad 	& [T"^1_2 : T"^1_1]=2 \\ { }
[T"^1_1 : T"^1_0]=2 \;.\quad 	 
\end{array}$$
Toutes les marches de $(T"^1)$ sont donc simples (Chap. 2, Ex. 1.10
.(iii)) et a fortiori galsimples (Chap. 2, Prop. 1.8
.(1)). On déduit alors de la proposition 1.3 
 du chapitre 4 que $(T"^1)$ est une tour de composition galoisienne. Par un
 raisonnement identique, on établit que $(T"^2)$ est aussi une tour de
 composition galoisienne.\\
 \indent Remarquons également que $(T"^2)$ est un raffinement galoisien de $(T'^2)$ avec
 un seul nouveau corps : $(T"_2^2)$. D'après la proposition \ref{prop:5E}, les
 $1^{\text{ère}}$, $4^{\text{ème}}$, $5^{\text{ème}}$, $6^{\text{ème}}$,
 $7^{\text{ème}}$ marches de $(T"^2)$ sont équivalentes respectivement aux
 $1^{\text{ère}}$, $2^{\text{ème}}$, $6^{\text{ème}}$, $3^{\text{ème}}$,
 $7^{\text{ème}}$ de $(T"^1)$.
 De plus, on a le parallélogramme galoisien 
 $$[{\mathbb Q}(\sqrt{5}), {\mathbb
 Q}(\zeta_{15}), {\mathbb  Q}(i, \zeta_{15}, \sqrt[4]{5}), {\mathbb 
 Q}(i,\sqrt[4]{5})]$$
et la figure suivante :

\begin{figure}[!h]
\begin{center}
\vskip -6mm
\includegraphics[width=14.5cm]{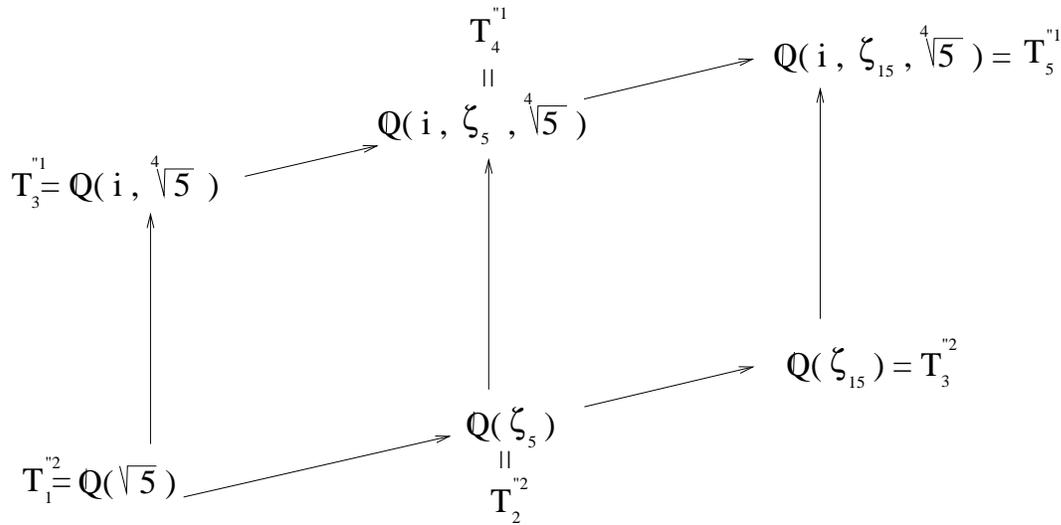}
\end{center}
\vskip -5mm
\rm{\caption{\label{fig:6} \leg Parallélogramme $[T"^2_1, T"^2_3, T"^1_5, T"^1_3]$}}
\vskip 1mm
\end{figure}
 
L'extension ${\mathbb Q}(\zeta_{5}) / {\mathbb Q}(\sqrt{5})$ est quadratique,
donc galoisienne, et par le corollaire 4.3 du chapitre 1 
on a les parallélogrammes
$$[{\mathbb Q}(\sqrt{5}), {\mathbb
 Q}(\zeta_{5}), {\mathbb  Q}(i, \zeta_{5}, \sqrt[4]{5}), {\mathbb 
 Q}(i,\sqrt[4]{5})]$$
et
$$[{\mathbb Q}(\zeta_{5}), {\mathbb
 Q}(\zeta_{15}), {\mathbb  Q}(i, \zeta_{15}, \sqrt[4]{5}), {\mathbb 
 Q}(i, \zeta_{5}, \sqrt[4]{5})] \;.$$
Donc en particulier
$$Gal(T"^1_4 / T"^1_3) \isomto Gal(T"^2_2 / T"^2_1)$$
et
$$Gal(T"^1_5 / T"^1_4) \isomto Gal(T"^2_3 / T"^2_2) \;.$$

Ceci achêve de prouver que les tours de composition galoisiennes $(T"^1)$ et $(T"^2)$ sont
équivalentes.
\end{proof}

\gespace
\gespace
\section{Exemple d'extensions simples}
\mespace
Nous exhibons ici une classe infnie d'extensions simples non galoisiennes :
celles de l'exemple 1.10.(iii) 
du chapitre 2 : l'extension ${\mathbb Q}(\theta) / {\mathbb Q}$ où $\theta^n -
\theta -1=0 \; (n \geq3 )$ est simple mais non galoisienne.
\pespace
Nous utilisons la propriété classique suivante des groupes de permutations :

\mespace
\begin{fait} \label{fait:5G}
Pour tout entier $n \geq 3$, le groupe symétrique $S_{n-1}$ s'identifie au stabilisateur d'un
point quelconque dans $S_n$. Alors, $S_{n-1}$ est un
sous-groupe maximal (intransitif) de $S_n$.
\end{fait}

\mespace
\begin{proof}
Cf. \cite[p.268]{D-M}.
\end{proof}

\gespace
L'argument principal de notre démonstration est un résultat de
Selmer-Serre\footnote{\,Remerciements au Professeur C.U. Jensen de l'Université de
Copenhague qui me le fit connaître lors d'un cours de D.E.A à Valenciennes en
1999.} :
\pespace
\begin{prop} \label{prop:5H}
Pour tout entier $n \geq 2$, le polynôme
$$P_n(X):=X^n-X-1$$
est irréductible sur ${\mathbb Q}$, et l'on a 
$$Gal(L_n/{\mathbb Q}) \isomto S_n$$
où $L_n$ désigne le corps de décomposition dans ${\mathbb C}$ de $P_n(X)$.
\end{prop}

\mespace
\begin{proof}
Cf. \cite{Sel} pour l'irréductibilité et \cite{Se} pour le reste.
\end{proof} 

\gespace
\begin{cor} \label{cor:5I}
Pour tout entier $n \geq 3$, l'extension ${\mathbb Q}(\theta) / {\mathbb Q}$ où
$$\theta^n - \theta -1=0$$
est simple et non galoisienne : $({\mathbb Q}(\theta) \,\simple \nongal {\mathbb
Q})$.
\end{cor}

\mespace
\begin{proof}
Avec les notations de la proposition \ref{prop:5H}, soient d'une part ${\mathcal
R}:=\{\theta_1, \dots , \theta_n \}$ l'ensemble des
racines (nécessairement distinctes) de $P_n(X)$ dans $L_n$, et d'autre part
$E_n:={\mathbb Q}(\theta_n)$, $H_n:=Gal(L_n/E_n)$. La numérotation des racines
de $P_n(X)$ induit un homomorphisme explicite
$$\begin{array}{rccc}
\psi_n \;:	&Gal(L_n/{\mathbb Q})	&\rightarrow	&S_n \\
		&\sigma			&\mapsto	&\psi_n(\sigma)
\end{array}$$
défini par $\text{\large (}\psi_n(\sigma)\text{\large )}(i)=j$ quand
$\sigma(\theta_i)=\theta_j \:(1 \leq i,j \leq
n)$. En vertu de la proposition \ref{prop:5H}, $\psi_n$ est un isomorphisme. Il
est clair que $\gamma(\theta_n)=\theta_n$ pour tout $\gamma \in H_n$, de sorte que 
$$\forall \gamma \in H_n \quad \text{\large (}\psi_n(\gamma)\text{\large )}(n)=n \;,$$
ce qui exprime que $\psi_n(H_n)$ est inclus dans le stabilisateur de $n$ dans
$S_n$, lui même isomorphe à $S_{n-1}$ par le Fait \ref{fait:5G} :
$$\psi_n(H_n) \leq Stab_{S_n}(n)=S_{n-1} \;.$$
Ainsi :
$$|\psi_n(H_n)| = |H_n| = [L_n:E_n] = \frac{[L_n:{\mathbb Q}]}{[E_n:{\mathbb
Q}]}=\frac{n!}{n}=|S_{n-1}| \;,$$
d'où l'on déduit que
$$\psi_n(H_n)=S_{n-1}$$
est un sous-groupe maximal (Fait \ref{fait:5G}) de $\psi_n\text{\large (}
Gal(L_n / {\mathbb Q})\text{\large )}=S_n$ (Prop. \ref{prop:5H}). Par
conséquent, $H_n=Gal(L_n/E_n)$ est un sous-groupe maximal de $Gal(L_n / {\mathbb
Q})$. Il résulte alors de la bijection de Galois que $E_n / {\mathbb Q}$ est une
extension simple.
\pespace
Par ailleurs, l'hypothèse $n \geq 3$ implique quant à elle que $S_{n-1}$ n'est
pas normal dans $S_n$ (par exemple $(1 \: 2)^{(1 \: n)}=(2 \: n) \notin
S_{n-1}$) ; donc $E_n / {\mathbb Q}$ n'est pas galoisienne.
\end{proof}

\gespace
\gespace
\section{Un contre-exemple au "Théorème $M$" pour une extension infinie}
\gespace
Nous montrerons au chapitre 6 suivant, via le "Théorème $M$", que toute extension
finie se dissocie canoniquement en une "extension quotient galtourable maximale"
et une "sous-extension d'intourabilité maximale". Il est alors naturel de se
demander si cette dissociation admet un analogue dans
le cas d'une extension infinie. Le contre-exemple suivant de cette section,
inséré dans le présent chapitre car indépendant de ce qui suit, nous prouve en
particulier que, dans nos définitions, cet analogue n'existe pas.

\gespace
\begin{notas} \label{notas:5J}
Dans toute cette section, notons $(E)$ la suite infinie de corps emboîtés
$$(E) \qquad E_0 \unlhd E_1 \unlhd \dots \unlhd E_n \unlhd E_{n+1} \unlhd \dots
\leq E_{\infty}$$
avec
$$\left\{\begin{array}{l}
E_0 := {\mathbb Q} \;, \;x_0:=2\;;\\
\\
E_{n} := {\mathbb Q}(x_{n}) = E_{n-1}(x_n) \;,\; x_n := \sqrt[2^n]{2} \qquad (n
\in {\mathbb N} \setminus \{0\}) \;;\\
\\
E_{\infty} := {\displaystyle \bigcup_{n \in {\mathbb N}} E_n }\;.$$
\end{array}\right.$$
\end{notas}

\pespace
{\it Scholie. } Il ne s'agit pas d'une tour de corps au sens de notre définition
\& convention 1.1 
du chapitre 2 , puisqu'il y a une infinité de corps.\\

Pour la démonstration du Fait \ref{fait:5M} à suivre, nous aurons besoin des lemmes
suivants :

\mespace
\begin{lem} \label{lem:5K}
Tout corps intermédiaire $M$ entre $E_0$ et $E_{\infty}$ : $E_0 \leq M \leq
E_{\infty}$ induit l'ensemble non-vide
$$I_M:=\{n \in {\mathbb N} \;|\; x_n \in M \} = \{n \in {\mathbb N} \;|\; E_n \leq M
\}  \;,$$
qui est infini si et seulement si $M=E_{\infty}$.
\end{lem}

\mespace
\begin{proof}
Notons tout d'abord que les deux définitions de $I_M$ coïncident, puisque $x_n$
est un élément primitif de $E_n$ sur ${\mathbb Q}$ pour tout $n \in {\mathbb
N}$. 
\pespace
Comme $x_0=2$, $I_M$ contient toujours $0$, et n'est donc jamais vide. Soit
maintenant $M$ un corps intermédiaire tel que $I_M$ soit infini. Alors :
$$\forall n \in {\mathbb N} \quad \exists m \in {\mathbb N} \quad m \geq n
\qquad m \in I_M \;;$$
i.e.
$$\forall n \in {\mathbb N} \quad \exists m \geq n \quad E_m \leq M \;.$$
Par la croissance de la suite $(E_n)_{n \in {\mathbb N}}$ , on a donc
$$\forall n \in {\mathbb N} \quad \exists m \geq n \quad E_n \leq E_m
\leq M \;.$$
Ainsi
$$E_{\infty} = {\displaystyle \bigcup_{n \in {\mathbb N}}} E_n \leq M \leq
E_{\infty} \qquad \text{i.e. } E_{\infty} = M \;.$$
La réciproque est immédiate puisque $I_{E_{\infty}}= {\mathbb N}$.
\end{proof}

\gespace
\begin{lem} \label{lem:5L}
Pour tout corps intermédiaire $M$ entre $E_0$ et $E_{\infty}$ : $E_0 \leq M \leq
E_{\infty}$, vérifiant la propriété suivante
$$\exists m \in {\mathbb N} \quad x_m \in M \quad x_{m+1} \notin M \;,$$
le polynôme minimal de $x_{m+2}$ sur $M$ est
$$Irr(x_{m+2},M,X) = X^4-x_m \;.$$
\end{lem}

\mespace
\begin{proof}
Comme $x_m \in M$, $P_m(X):=X^4-x_m \in M[X]$ ; et on a $P_m(x_{m+2})=0$ par
définition des $x_n$. Donc $Irr(x_{m+2},M,X)$ divise
$$P_m(X)=(X-x_{m+2})(X+x_{m+2})(X-i \, x_{m+2})(X+i \, x_{m+2})\;.$$
Par ailleurs, l'extension $M(x_{m+1})/M$ est quadratique puisque
$x_{m+1}=\sqrt{x_m} \notin M$. On en déduit que
$$[M(x_{m+2}):M]=2[M(\sqrt{x_{m+1}}):M(x_{m+1})] \in \{2,4\} \;.$$
Mais si le polynôme minimal de $x_{m+2}$ sur $M$ était de degré 2, il existerait
un $l \in \{1,2,3\}$ tel que
$$Irr(x_{m+2},M,X) = (X-x_{m+2})(X-i^l \, x_{m+2}) \;;$$
et son terme constant serait $i^l \, x_{m+2}^2=i^l \, x_{m+1} \in M$.
Or ceci est imposible pour $l=2$ car $x_{m+1} \notin M$ ; et pour $l=1$ ou $3$
car $M$ est un corps réel.\\
Finalement $[M(x_{m+2}):M]=4$ et $P_m(X)=Irr(x_{m+2},M,X)$.
\end{proof}

\gespace
\begin{fait} \label{fait:5M}
L'extension $E_{\infty} / E_0$ n'est pas galtourable : $E_{\infty} \nongaltou
E_0$, et n'admet aucune sous-extension galsimple.
\end{fait}

\mespace
\begin{proof}
Raisonnons par l'absurde en supposant que l'extension $E_{\infty} / E_0$ soit
galtourable, i.e. qu'elle admet une tour galoisienne
$$(T) \qquad E_0=T_0 \unlhd \dots \unlhd T_j \unlhd \dots \unlhd
T_p=E_{\infty} \;.$$
Comme $E_{\infty} / E_0$ est clairement de degré infini, l'une au moins des
marches de $(T)$ est infinie :
$$\exists j \in \{0, \dots , p-1 \} \quad [T_{j+1}:T_j]=\infty.$$
Notons $j$ le plus grand entier tel que la marche $T_{j+1}/ T_j$ soit infinie :
$$\forall k \in \{0, \dots , p-1 \} \quad k \geq j+1 \quad \implique \quad
[T_{k+1}:T_k] < \infty \;.$$
Par transitivité du degré, l'extension $E_{\infty} / T_{j+1}$ est finie, et elle
est séparable puisque nous sommes en caractéristique nulle. On peut donc lui
appliquer le théorème de l'élément primitif : 
$$\exists x \in E_{\infty}  \quad E_{\infty}=T_{j+1}(x) \;.$$
Ainsi :

\begin{figure}[!h]
\begin{center}
\vskip -6mm
\includegraphics[width=6.5cm]{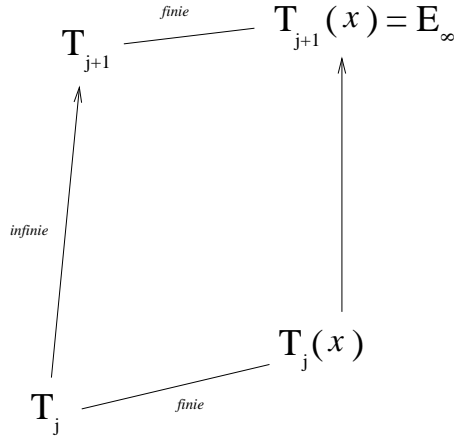}
\end{center}
\vskip -5mm
\rm{\caption{\label{fig:7} \leg L'extension infinie $E_{\infty} / T_{j}(x)$}}
\vskip -1mm
\end{figure}

\noindent d'où il  résulte que $[E_{\infty}:T_{j}(x)]=\infty$. En particulier $T_j(x)$ est
distinct de $E_{\infty}$. On déduit alors du lemme \ref{lem:5K} que l'ensemble
$I_{T_j(x)}$ est fini. Notons $m$ son plus grand élément :
$$m:=max \;I_{T_j(x)} \;.$$
On a $m \in I_{T_j(x)}$ et $m+1 \notin I_{T_j(x)}$, i.e.  $x_m \in
{T_j(x)}$ et  $x_{m+1} \notin {T_j(x)}$.
Le lemme \ref{lem:5L} nous assure alors que
$$Irr(x_{m+2},{T_j(x)},X) = X^4-x_m \;.$$
Mais par ailleurs, l'extension $E_{\infty} / T_j(x)$ est galoisienne en tant que
translatée de la marche galoisienne $T_{j+1} \gal T_j$ de $(T)$. Comme $x_{m+2}
\in E_{\infty}$, toutes les racines du polynôme irréductible $X^4-x_m$ doivent
être dans $E_{\infty}$ : absurde, puisque $E_{\infty}$ est un corps réel tandis que deux
des racines de $X^4-x_m$ sont complexes non-réelles.Il est donc prouvé que
$E_{\infty} / E_0$ n'est pas galtourable.

Raisonnons encore une fois par l'absurde en supposant l'existence d'une
sous-extension galsimple $E_{\infty} \galsimple M$ de $E_{\infty}/ E_0$. Par
définition (cf. Chap. 2, Déf. 1.6), $E_{\infty} / M$ est non-triviale, et on
déduit du lemme \ref{lem:5K} que $I_M$ est fini. Notons $m:=max \;I_M$. On a :
$$x_m \in M \quad \text{et} \quad x_{m+1} \notin M \;;$$
d'où par le lemme \ref{lem:5L}
$$Irr(x_{m+2},M,X) = X^4-x_m \;.$$
Ceci implique en particulier que $[M(x_{m+2}):M]=4$ et l'on a la tour
$$M \lhd M(x_{m+1}) < M(x_{m+2}) \leq E_{\infty}$$
qui contredit la galsimplicité de $E_{\infty} \galsimple M$.
\end{proof}

\mespace
\mespace
\section{Extensions galsimples non galoisiennes}
\mespace
Nous utiliserons les résultats de cette section pour mettre en évidence le rôle
central du corps d'intourabilité de toute extension finie.
\pespace
La proposition suivante, vraie quant à elle pour un degré quelconque, montre que
la galsimplicité passe toujours au quotient.

\mespace
\begin{prop} \label{prop:quotientgalsimple}
Toute extension quotient propre d'une extension galsimple est galsimple non
galoisienne. Précisément :
$$\forall (L \galsimple K) \qquad \forall M \qquad K < M < L \quad \implique
\quad (M \galsimple \nongal K) \;.$$
\end{prop}

\begin{proof}
La galsimplicité de $L/K$ interdit que l'extension quotient propre $M/K$ soit
galoisienne. Et elle est nécessairement galsimple car s'il existait un corps $F$
tel que $K \lhd F < M$, on en déduirait $K \lhd F < L$ : contradiction.
\end{proof}

\mespace
Prouvons maintenant qu'il y a transitivité de la galsimplicité non galoisienne.

\pespace
\begin{prop}
Pour toute tour de corps $F_0 \leq F_1 \leq F_2$, le fait que les extensions
$F_1 / F_0$  et $F_2 / F_1$ soient galsimples non galoisiennes implique qu'il en
est de même de l'extension $F_2 / F_0$. Précisément, dans nos notations (cf. Chap.
2, Not. 1.9) : 
$$\text{\rm \large (} \;(F_1 \galsimple \nongal F_0) \;, \: (F_2 \galsimple \nongal
F_1) \; \text{\rm \large )} \quad \implique \quad (F_2 \galsimple \nongal F_0) \;.$$
\end{prop}

\mespace
\begin{proof}
On sait déjà que l'extension $F_2 / F_0$ n'est pas galoisienne, sans quoi
l'on aurait la sous-extension galoisienne $F_2 \gal F_1$.
Raisonnons par l'absurde en supposant que l'extension $F_2 / F_0$ ne soit pas
galsimple. Comme elle n'est pas triviale puisque $F_1 / F_0$  (et $F_2 / F_1$)
ne l'est pas, cela signifie qu'il existe un corps intermédiaire $M$ tel que
$$F_0 \lhd M < F_2 \;.$$
En translatant alors l'extension galoisienne $M \gal F_0$ par l'extension $F_1 / F_0$, on
obtient l'extension galoisienne $M F_1 / F_1$, et donc
$$F_1 \unlhd M F_1 \leq F_2 \:.$$
Comme l'extension $(F_2 \galsimple \nongal F_1)$ est galsimple non galoisienne, 
son extension galoisienne quotient $(M F_1 / F_1)$ est nécessairement triviale :
$$M F_1 = F_1 \quad \text{i.e.} \qquad M \leq F_1 \;.$$
De sorte que l'on a la tour
$$F_0 \lhd M \leq F_1 \;.$$
Comme l'extension $F_1 \galsimple F_0$ est galsimple, on a donc $M=F_1$.
Et finalement l'extension $(M \gal F_0)=(F_1 \gal F_0)$ est galoisienne :
contradiction.
\end{proof}

\gespace
La proposition \ref{prop:5fin} suivante fournit une classe infinie d'extensions galsimples non
galoisiennes immédiatement utilisables car sur le corps $\mathbb Q$ des nombres
rationnels. Nous nous en servirons dans les exemples du chapitre 6.

\gespace
Rappelons d'abord le critère d'irréductibilité des polynômes $X^n-a$ que nous
utiliserons également au chapitre 6

\mespace
{\it {\bf Theorem.} \cite[p.297, Th. 9.1]{La} \\
Let $K$ be a field and $n \geq2$ an integer. Let $a \in K$, $a \neq 0$. Assume
that for all prime numbers $p$ such that $p | n$ we have $a \notin K^p$, and if
$4 | n$ then $a \notin -4K^4$. Then $X^n -a$ is irreducible in $K[X]$.}

\mespace
\begin{prop} \label{prop:5fin}
Soit $L={\mathbb Q}(\alpha) / {\mathbb Q}$ une extension radicale où 
$\alpha^n = a \in \mathbb Q$ avec\\
(1) $a$ positif : $a \in \mathbb Q_+$, et $n \geq 3$ impair,\\
(2) pour tout diviseur $d$ de $n$ : $d\mid n $, $a \notin \mathbb Q^d$.\\
Alors l'extension $L/\mathbb Q$ est galsimple non galoisienne.
\end{prop}

\mespace
\begin{proof}
D'après le théorème précédent, 
le polynôme $X^n - a$ est irréductible sur $\mathbb Q$, d'où $X^n -
a=Irr(\alpha,{\mathbb Q},L)$ et $[L:{\mathbb Q}]=n$. Soit alors $I$ un corps
intermédiaire entre $\mathbb Q$ et $L$. D'après la proposition de la
démonstration de l'exemple 1.10 du chapitre 2, 
il
existe nécessairement un diviseur $\delta$ de $n$ tel que $I={\mathbb Q}(\beta)$
avec $\beta^\delta=a$. Raisonnons par l'absurde en supposant que $I/{\mathbb Q}$
soit une extension galoisienne non triviale : $I \gal {\mathbb Q}$, $[I :
{\mathbb Q}] \geq 2$. Le polynôme $X^\delta - a$ est aussi irréductible sur
${\mathbb Q}$ car
$$d|\delta \quad \implique \quad d|n \quad \implique \quad a \notin {\mathbb
Q}^d$$
et le critère d'irréductibilité s'applique. Donc $X^\delta - a = Irr(\beta,{\mathbb Q},X)$
et 
$$2 \leq [I : {\mathbb Q}]=\delta \mid n \text{ impair.}$$
Or dire que $I={\mathbb Q}(a^{1/\delta})/{\mathbb Q}$ est galoisienne implique
que $\mu_\delta \subseteq I$.

\begin{figure}[!h]
\begin{center}
\vskip -6mm
\includegraphics[width=5.5cm]{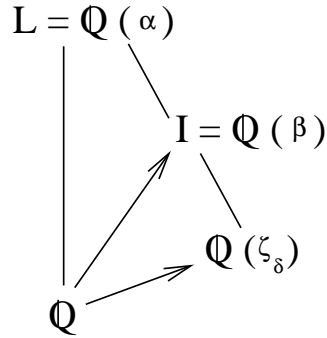}
\end{center}
\vskip -5mm
\rm{\caption{\label{fig:8}\leg Galsimplicité de ${\mathbb
Q}(a^{1/n})/{\mathbb Q}$ contredite}}
\vskip 1mm
\end{figure}

 Et on sait que $[{\mathbb Q}(\mu_\delta) : {\mathbb Q}]=
\varphi(\delta)$ où $2 \leq \delta$ impair est divisible par un nombre premier
impair $p$. Il en résulte que l'entier pair $p-1$ divise $[I:{\mathbb Q}]$
impair : contradiction. Il est donc prouvé que $L/{\mathbb Q}$ est galsimple.
\end{proof}

\addtocontents{toc}{\mespace\pespace}
\chapter{THÉORÈME $M$}
\addtocontents{lof}{\gespace}
\addtocontents{lof}{\noindent Chapitre \thechapter}
\addtocontents{lof}{\pespace}
\addtocontents{toc}{\pespace}
\gespace
Ce chapitre 6 détaille un résultat non publié de Richard Massy (confer
remerciements) que nous sommes convenus d'appeler le "théorème $M$"
\footnote{\, plutôt que l'ambigu jeu de mots "corps de Massy" !}(Sect. 1). Ce théorème $M$
est la clef de la généralisation à toutes les extensions finies des théorèmes de
dissociation précédents pour les extensions galtourables. Il découvre qu'à toute
extension finie $L/K$ est attachée un invariant unique, son "corps
d'intourabilité $M(L/K)$" qui est un corps intermédiaire au-delà duquel
l'extension n'est plus galtourable. Le rôle central du corps d'intourabilité
$M(L/K)$ est mis en lumière section 2 : il dissocie l'extension $L/K$ en une
extension quotient galtourable maximale $M(L/K) \galtou K$ et une sous-extension
d'intourabilité maximale $L/M(L/K)$, ou bien triviale, ou bien galsimple non
galoisienne. Dans la dernière section, nous réalisons tous les corps
cyclotomiques ${\mathbb Q} (\zeta_{n})$ comme corps d'intourabilité d'une classe
infinie d'extensions $L/{\mathbb Q}$ avec $L/{\mathbb Q} (\zeta_{n})$ galsimple,
donc non triviale, et non galoisienne : $(L \galsimple \nongal {\mathbb Q}
(\zeta_{n}))$ (cf. Chap.2, Not. 1.9). 

\gespace
\gespace
\section{Quatrième théorème de dissociation}
\mespace
Il s'agit  ici d'énoncer et de prouver le résultat central suivant pour
généraliser les théorèmes de dissociation du chapitre 4, qui se limitaient au
cas galtourable.\mespace

\begin{Th} ($4^{\text{ème}}$ théorème de dissociation, dit "Théorème
$M$")\label{th:6A}\index{Théorème!de dissociation}\index{Théorème!M}\index{Dissociation!$4^{\text{ème}}$ Théorème (de)}\\
Pour toute extension finie $L/K$ , il existe un corps intermédiaire $M$ et un
seul entre $K$ et $L$, vérifiant à la fois les deux propriétés suivantes :

\pespace
(1) L'extension $M/K$ est galtourable ;

\pespace
(2) La sous-extension $L/M$ est soit triviale, soit galsimple non galoisienne.\\
\end{Th}
\vskip -4mm
\noindent{\it Scholie. }Dans nos notations (Chap. 2, Not. 1.9), 
ce théorème $M$ se symbolise comme suit :
$$\forall L/K \quad [L:K] < \infty \;, \; \exists ! \,M \quad K \leq M \leq L
\quad (M \galtou K \;,\; L=M \text{ ou } (L \galsimple \nongal M)) \;.$$

\mespace
\begin{proof}
{\emphase Existence.} Le théorème est trivial pour $L=K$ car $K \gal K$ implique $K
\galtou K$ (Chap.2, Déf.1.4, Scholie (2)). 
Procédons par récurrence sur $n:= [L:K] \geq 2$ en supposant l'existence d'un
corps intermédiaire de l'énoncé pour toute extension de degré $\leq n-1$.
Deux cas se présentent :\\
\clearpage
\noindent - Ou bien $L/K$ est galsimple. Si elle est galoisienne (resp. non galoisienne)
$M=L$ (resp. $M=K$) convient.\\
- Ou bien $L/K$ n'est pas galsimple. Comme elle n'est pas triviale, cela
signifie par définition (Chap. 2, Déf. 1.6.(2)) 
qu'il existe un corps $K_1$ tel que
$$K \lhd K_1 < L \;.$$
En particulier $[L:K_1] \leq n-1$ et l'hypothèse de récurrence s'applique : il
existe un corps $M$ tel que
$$(M \galtou K_1 \;,\; L=M \text{ ou } (L \galsimple \nongal M)) \;.$$
Or $K_1/K$ est galtourable en tant qu'extension galoisienne, donc par le Fait
2.7 du chapitre 2,
$$(K_1 \galtou K \;,\; M \galtou K_1) \quad \implique \quad (M \galtou K) \;;$$
et le corps $M$ remplit les conditions de l'énoncé.\\

{\emphase Unicité. }Soit $M'$ un autre corps intermédiaire entre $K$ et $L$ satisfaisant
les deux conditions de l'énoncé. Comme la translatée d'une extension galtourable
est toujours une extension galtourable (cf. Chap. 2, Cor 2.4 ou Th. 2.3), 
avoir $M/K$ galtourable implique que l'extension $MM'/M'$ est galtourable.

\begin{figure}[!h]
\begin{center}
\vskip -6mm
\includegraphics[width=1.5cm]{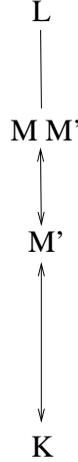}
\end{center}
\vskip -5mm
\rm{\caption{\label{fig:6-1} \leg Translatée de $M/K$ par $M'$}}
\vskip -1mm
\end{figure}

Considérons une tour galoisienne $(F)$ de $MM'/M'$
$$(F) \qquad M'=F_0 \unlhd \dots \unlhd F_i \unlhd \dots \unlhd F_m=MM' \;,$$
et sa tour stricte associée $(F_<)$ (qui existe même pour une extension triviale : cf.
Chap. 3, remarque. 2.2.(1)). D'après le corollaire 2.6 du chapitre 3, 
cette tour stricte associée $(F_<)$ est nécessairement galoisienne :
$$(F_<) \qquad M'=F_{<0} \lhd \dots \lhd F_{<j} \lhd \dots \lhd F_{<h}=MM' \;.$$
Supposons que l'on ait $h \geq 1$ de sorte que
$$M' \lhd F_{<1} \leq MM' \leq L \;.$$
En particulier $L \neq M'$, d'où $(L \galsimple \nongal M')$ par définition de
$M'$. Mais alors, on ne peut avoir $M' \lhd F_{<1} < L$ par la galsimplicité de
$L/M'$, ni $F_{<1}=L$ puisque $L/M'$ n'est pas galoisienne. La seule possibilité
est donc que $h=0$, autrement dit $MM'=M'$ i.e. $M \leq M'$. En échangeant les
rôles de $M$ et $M'$, on prouve de même que $M' \leq M$, et finalement $M=M'$.
\end{proof}

\gespace
\begin{rem} \label{rem:6B}
Nous venons de démontrer que toute extension finie se dissocie en une extension
galtourable et, éventuellement, une extension galsimple. Le théorème $M$
justifie donc à lui seul l'introduction de la galtourabilité et de la
galsimplicité.
\end{rem}

\gespace
\begin{defn} \label{def:6C}
Dans les notations du théorème $M$ (Th. \ref{th:6A}) précédent, nous appelons
$M=M(L/K)$ "Le corps d'intourabilité\index{Corps d'intourabilité} de $L/K$".
\end{defn}

\gespace
\begin{cor} \label{cor:6D}
Une extension finie $L/K$ est galtourable si et seulement si son corps
d'intourabilité coïncide avec son corps sommet. Autrement dit, on a
l'équivalence
$$(L \galtou K) \quad \ssi \quad M(L/K)=L \;.$$
En particulier, on a toujours l'égalité $M(M(L/K)/K)=M(L/K)$.
\end{cor}

\mespace
\begin{proof}
Si $L/K$ est galtourable, le corps $L$ lui-même vérifie les deux conditions du
théorème \ref{th:6A}. Par unicité du corps d'intourabilité, on en déduit
$M(L/K)=L$. Inversement dans ce cas, $L/K$ est galtourable puisqu'il en est
ainsi de $M(L/K)/K$ par définition.
\end{proof}

\gespace
\begin{ex} \label{ex:6E}
Dans les notations 3.1 du chapitre 5 ; 
$$E_0={\mathbb Q} \;, \quad E_\infty= {\displaystyle\bigcup_{n \in {\mathbb
N}}}\,{\mathbb Q}(\sqrt[2^n]{2}) \;,$$
il n'existe pas de corps intermédiaire entre $E_0$ et $E_\infty$ tel que les
conditions (1) et (2) du théorème \ref{th:6A} soient vérifiées.
\end{ex}

\mespace
\begin{proof}
C'est clair d'après le Fait 3.4 du chapitre 5. 
\end{proof}

\gespace
\gespace
\section{Le rôle central du corps d'intourabilité}
\mespace
Nous allons prouver que le corps d'intourabilité du théorème $M$ commande la
galtourabilité (resp. l'intourabilité) des extensions quotients (resp. des
sous-extensions) de toute extension finie en induisant deux maximalités
déterminantes.

\gespace
\begin{propdef} \label{propetdef:6F}
Soient $L/K$ une extension finie et $M(L/K)$ son corps d'intourabilité (Déf.
\ref{def:6C}).

\pespace
(1) Pour toute extension quotient galtourable $F \galtou K$ de $L/K$ : $K \leq F
\leq L$, on a nécessairement $F \leq M(L/K)$.

\pespace
(2) L'extension galtourable $M(L/K) \galtou K$ est maximale dans l'ensemble des
extensions quotients galtourables de $L/K$ muni de la relation d'ordre du (1) du
lemme 2.3 du chapitre 1.

\pespace
(3) L'extension $M(L/K) \galtou K$ est la seule extension quotient de $L/K$
vérifiant la maximalité du (2).\\

Nous appelons $M(L/K) \galtou K$ "l'extension quotient galtourable maximale de
$L/K$."\index{Extension!quotient galtourable maximale}
\end{propdef}

\mespace
\begin{proof}
Posons $M:=M(L/K)$.
\pespace
(1) Soit $F \galtou K$ une extension quotient galtourable de $L/K$. D'après le
corollaire 2.4 du chapitre 2, 
l'extension $FM \galtou M$ est également galtourable. Elle admet ainsi une tour
galoisienne que l'on peut supposer stricte (Chap. 3, Cor. 2.6) : 
$$M=F_0 \lhd \dots \lhd F_n=F \, M \;.$$
Raisonnons par l'absurde en supposant $M \neq FM$. Dès lors on a $n \geq 1$, et
en particulier $M \lhd F_1 \leq L$. L'extension $L/M$ est donc non triviale.
C'est qu'elle est galsimple non galoisienne en vertu du théorème $M$. Or avoir
$F_1=L$ contredit le fait que $L/M$ est non galoisienne, et avoir $F_1 \neq L$
contredit la galsimplicité de $L/M$. Ceci prouve que nécessairement $M=FM$, i.e.
$F \leq M$. 

\pespace
\noindent (2) Pour toute extension quotient galtourable $F \galtou K$, on a $F
\leq M$ d'après le (1) précédent, ce qui signifie, par définition de la relation
d'ordre, que
$$(F \galtou K) \; \leq \; (M \galtou K)\;.$$

\pespace
\noindent (3) Si $(M' \galtou K)$ est une extension quotient galtourable
maximale de $L/K$, on a $M' \leq M$ d'après le (1), i.e. $(M' \galtou K) \; \leq
\; (M \galtou K)$, d'où $M'=M$ par la maximalité de $M' \galtou K$.
\end{proof}

\gespace
La proposition suivante induira la notion de "tour d'élévation" du chapitre 7
final.

\mespace
\begin{prop} \label{prop:6G}
Soit $L/K$ une extension finie. Pour tout corps intermédiaire $F$ : $K \leq F
\leq L$, le corps d'intourabilité de l'extension quotient $F/K$ est inclus dans
celui de l'extension $L/K$. Précisément :\\
$$M(F/K) \; \leq \; F \, \cap \, M(L/K) \;.$$
\end{prop}

\mespace
\begin{proof}

L'extension $M(F/K) \galtou K$ est galtourable (Th. \ref{th:6A}).
D'après le (1) de la proposition \& définition \ref{propetdef:6F}, on a donc
nécessairement
$$M(F/K) \; \leq \;  M(L/K) \;.$$

\pespace $ $\\
\begin{figure}[!h]
\begin{center}
\vskip -10mm
\includegraphics[width=4.7cm]{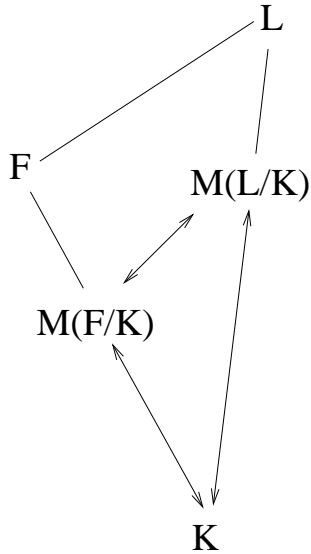}
\end{center}
\vskip -5mm
\rm{\caption{\label{fig:6-2} \leg Corps d'intourabilité d'une extension quotient}}
\vskip -3mm
\end{figure}

\end{proof}

\gespace
\begin{propdef} \label{propetdef:6H}
Soient $L/K$ une extension finie et $M(L/K)$ son corps d'intourabilité (Déf.
\ref{def:6C}).

\pespace
\noindent (0) Nous appelons "sous-extension d'intourabilité de $L/K$"
\index{Sous-extension!d'intourabilité}toute
sous-extension \\ $L/F$ , $K \leq F \leq L$, ou bien triviale : $L=F$, ou bien galsimple
non galoisienne : $(L \galsimple \nongal F)$.

\pespace
\noindent (1) Pour toute sous-extension d'intourabilité $L/F$ de $L/K$, on a
nécessairement $M(L/K) \leq F$.

\pespace
\noindent (2) L'extension $(L/M(L/K))$ est maximale dans l'ensemble des
sous-extensions d'intourabilité de $L/K$ muni de la relation d'ordre du (1) du
lemme 2.3 du chapitre 1. 

\pespace
\noindent (3) L'extension $(L/M(L/K))$ est la seule sous-extension d'intourabilité
de $L/K$ vérifiant la maximalité du (2).\\

Nous appelons $(L/M(L/K))$ "la sous-extension d'intourabilité maximale de
$L/K$"\index{Sous-extension!d'intourabilité maximale}.
\end{propdef}

\mespace
\begin{proof}
(1) Soit $L/F$ une sous-extension d'intourabilité de $L/K$. Quand $F=L$, le
résultat est trivial. Supposons $F \neq L$. D'après le (0), c'est donc que $L/F$
est une extension galsimple non galoisienne : $(L \galsimple \nongal F)$.
Distinguons alors deux cas.\\
- Si $F$ est égal au corps d'intourabilité de $F/K$, on déduit du théorème $M$
(Th. \ref{th:6A}) dans $L/K$ que
$$\text{\Large (} (F=M(F/K) \galtou K) \;,\quad (L \galsimple \nongal F)
\text{\Large )} \quad \implique \quad F=M(L/K) \;,$$
et l'on a le résultat voulu.\\
- Si $F \neq M(F/K)$, on a l'extension galsimple non galoisienne $(F \galsimple
\nongal M(F/K))$, et par la proposition 4.2 du chapitre 5 : 
$$\text{\Large (} (F \galsimple \nongal M(F/K)) \;,\quad (L \galsimple \nongal F)
\text{\Large )} \quad \implique \quad \text{\Large (}L \galsimple \nongal M(F/K)
\text{\Large )} \;.$$
On en déduit, une fois encore par le théorème $M$, que
$$\text{\Large (} (M(F/K) \galtou K) \;,\quad (L \galsimple \nongal M(F/K))
\text{\Large )} \quad \implique \quad M(L/K)=M(F/K) \leq F\;.$$

\pespace
\noindent (2) Soit $L/F$ une sous-extension d'intourabilité de $L/K$. D'après le
(1), on a $M(L/K) \leq F$, ce qui signifie $(L/F) \leq (L/M(L/K))$.

\pespace
\noindent (3) Si $L/M'$ est une sous-extension d'intourabilité maximale de
$L/K$, on déduit du (1) que $M(L/K) \leq M'$, i.e. $(L/M') \leq (L/M(L/K))$,
d'où $M'=M(L/K)$ par la maximalité de $L/M'$.
\end{proof}

\gespace
\begin{defn} \label{def:6I}
Soient $L/K$ une extension finie et $M(L/K)$ son corps d'intourabilité. 
\pespace
\noindent (1) Nous appelons "degré de galtourabilité de
$L/K$"\index{Degré!de galtourabilité}, et nous notons
$[L:K]_{gal}$, le degré de l'extension quotient galtourable maximale de $L/K$ (Prop. \&
Déf. \ref{propetdef:6F}) :
$$[L:K]_{gal} := [M(L/K) : K] \;.$$

\pespace
\noindent (2) Nous appelons "degré d'intourabilité de
$L/K$"\index{Degré!d'intourabilité}, et nous notons
$[L:K]_{int}$, le degré de la sous-extension d'intourabilité maximale de $L/K$
(Prop. \& Déf. \ref{propetdef:6H}) :
$$[L:K]_{int} := [L : M(L/K)] \;.$$

\pespace
\noindent (3) Nous appelons "degré de tourabilité de
$L/K$"\index{Degré!de tourabilité}, et nous notons
$[L:K]_{tour}$, le couple d'entiers formé par le degré de galtourabilité et le
degré d'intourabilité de $L/K$ :
$$[L:K]_{tour} :=([L:K]_{gal} ,[L:K]_{int}) \;.$$
\end{defn}

\gespace
\gespace
\section{Exemples de corps d'intourabilité}
Nous avons déjà exhibé une classe infinie d'extensions simples non galoisiennes
${\mathbb Q}(\theta) / {\mathbb Q}$ où $n \geq 3$ et $\theta^n - \theta -1 = 0$
(Chap. 2, Ex. 1.10.(iii)). 
Ceci prouve en particulier que :

\mespace
\begin{prop} \label{prop:6J}
Tout entier naturel distinct de $0$ et $2$ est un degré d'intourabilité
(Déf. 2.4.(2)).
\end{prop}

\pespace
\noindent{\it Scholie. }Une extension quadratique étant nécessairement galoisienne, $2$
est exclu.

\pespace
Cependant, le corps d'intourabilité des extensions induites par les polynômes 
$\theta^n - \theta -1 = 0$ est toujours égal à ${\mathbb Q}$. Dans cette
section, nous exhibons une classe infinie d'extensions finies où les extensions
quotients galtourables maximales, ainsi que les sous-extensions d'intourabilité
maximales, sont non triviales.

\pespace
Nous utiliserons le Fait élémentaire suivant, dans lequel une racine primitive
$\nu^{\text{ème}}$ de l'unité est notée génériquement $\zeta_\nu \quad ( \nu \in
{\mathbb N} \setminus \{0\})$.

\mespace
\begin{fait} \label{fait:6K}
Pour tout couple d'entiers $(m,n)$ premiers entre eux, on a le parallélogramme
galoisien cyclotomique
$$[{\mathbb Q}, {\mathbb Q}(\zeta_m), {\mathbb Q}(\zeta_{mn}), {\mathbb
Q}(\zeta_n)] \;.$$
\end{fait}

\begin{proof}
Écrivons $m=p_1^{e_1} \, \dots \, p_r^{e_r}$ et $n=q_1^{f_1} \, \dots \,
q_s^{f_s}$ où les $p_i$ et les $q_j$ sont deux à deux étrangers. La
décomposition en facteurs premier de $mn$ est donc $mn=p_1^{e_1} \, \dots \,
p_r^{e_r} q_1^{f_1} \, \dots \, q_s^{f_s}$, et d'après \cite[p.74, Th.2]{La2}
$$\begin{array}{rl}
{\mathbb Q}(\zeta_{mn})	&={\mathbb Q}(\zeta_{p_1^{e_1}}) \dots {\mathbb Q}(\zeta_{p_r^{e_r}})
\, {\mathbb Q}(\zeta_{q_1^{f_1}})  \dots {\mathbb Q}(\zeta_{q_s^{f_s}})\\
			&={\mathbb Q}(\zeta_{m}) \, {\mathbb Q}(\zeta_{n})\;.
\end{array}$$
De plus
$$ {\mathbb Q}(\zeta_{m}) \cap {\mathbb Q}(\zeta_{n})= {\mathbb Q} \qquad
\text{(\cite[p.11]{Wa} ou \cite[p.75]{La2})} \;.$$
D'où la conclusion.
\end{proof}

\mespace
\begin{lem} \label{lem:6L} 
Soient deux entiers $d$ et $n$ tels que 
$$d \geq 2 \; , \quad 4 \nmid d \;, \quad n \geq 1 \;, \quad pgcd(d,n)=1 \;.$$
Notons $E_n:={\mathbb Q} (\zeta_n)$ le $n^{\text{ème}}$ corps cyclotomique sur
$\mathbb Q$. Pour tout nombre premier $l \in \mathbb N$ ne divisant pas $n$ et
tout $\rho \in \mathbb C$ tel que $\rho^d=l$, le polynôme minimal de $\rho$ sur
$E_n$ est $X^d-l$ et $E_n(\rho) / E_n$ est de degré $d$ :
$$Irr(\rho, E_n,X) = X^d-l \; , \quad [E_n(\rho) : E_n]=d \;.$$
\end{lem}

\mespace
\begin{proof}
Soit $q$ l'un quelconque des nombres premiers divisant $d$ ($\geq 2$).
Raisonnons par l'absurde en supposant que $l \in E_n^q$, i.e. $l= e_n^q$ avec
$e_n \in E_n$. A fortiori pour les idéaux engendrés dans l'anneau ${\mathbb Z}
[\zeta_n]$ des entiers de $E_n$ 
$$l \, {\mathbb Z} [\zeta_n] = (e_n \, {\mathbb
Z}[\zeta_n])^q \;,$$
de sorte que, pour tout idéal premier $\mathcal P$ de ${\mathbb
Z}[\zeta_n]$ 
$$\text{ord}_{\mathcal P}(l \, {\mathbb Z} [\zeta_n]) = q \,\text{ord}_{\mathcal
P}(e_n \, {\mathbb Z} [\zeta_n]) \;,$$
d'où $\text{ord}_{\mathcal P}(e_n \, {\mathbb Z} [\zeta_n]) \geq 0$. L'idéal $e_n
\, {\mathbb Z} [\zeta_n]$ est donc entier, et se décompose en idéaux premiers
dans l'anneau de Dedekind ${\mathbb Z} [\zeta_n]$ 
$$e_n \, {\mathbb Z} [\zeta_n] = {\frak P}_1^{v_1} \dots \, {\frak P}_r^{v_r}
\quad, \qquad v_i > 0 \quad (i=1, \dots , r) \;.$$
On en déduit
$$l \, {\mathbb Z} [\zeta_n] = {\frak P}_1^{q\,v_1} \dots \,{\frak P}_r^{q\,v_r}
\quad, \qquad v_i > 0 \quad (i=1, \dots , r)$$
avec $q \geq 2$. 
Ceci exprime que $l \, {\mathbb Z}$ se ramifie dans $E_n$. D'après
\cite[p.74, Th. 2]{La2}, ceci implique $l \mid n$ : contradiction. C'est donc
que, pour tout $q$ divisant $d$, $l \notin E_n^q$. Comme $4 \nmid d$, on déduit
alors du critère d'irréductibilité rappelé avant la proposition 4.3 
du chapitre 5  que $X^d-l$ est irréductible dans $E_n[X]$. D'où la conclusion. 
\end{proof}

\pespace
\begin{rem}
Une autre démonstration du Lemme \ref{lem:6L} est d'utiliser la galsimplicité de
${\mathbb Q}(\rho) / {\mathbb Q}$ (cf. Chap. 5, Prop. 4.3). 
\end{rem}

\gespace
On en déduit le

\mespace
\begin{fait} \label{fait:6M} 
Dans les notations du lemme \ref{lem:6L}, supposons de plus que $d$ soit impair.
Pour tout entier $\delta \neq d$ divisant $d$ : $\delta \mid d$, on a
$$\mu_p \,\cap \, E_n(\rho^\delta)= \iit$$
quel que soit le nombre premier $p$ divisant $\frac{d}{\delta}$ : $p \mid
\frac{d}{\delta}$. En particulier, pour tout entier $\delta'$, multiple de
$\delta$, distinct de $\delta$ et divisant $d$ : $\delta' \mid d$, l'extension
$(E_n(\rho^\delta)/E_n(\rho^{\delta'}))$ n'est pas galoisienne.
\end{fait}

\mespace
\begin{proof}
Prouvons d'abord le résultat pour $\delta=1$. Raisonnons par l'absurde en
supposant que $\mu_p \,\cap \, E_n(\rho) \neq \iit$. Comme $pgcd(d,n)=1$
(cf. lemme \ref{lem:6L}), $p$ ne divise pas $n$, et par le Fait \ref{fait:6K}, on
a le parallélogramme galoisien 
$$[{\mathbb Q}, {\mathbb Q}(\zeta_n), {\mathbb Q}(\zeta_{pn}), {\mathbb
Q}(\zeta_p)] \quad\text{(cf. F{\tiny IG}. 28)}\;.$$

\begin{figure}[!h]
\begin{center}
\vskip -6mm
\includegraphics[width=7.5cm]{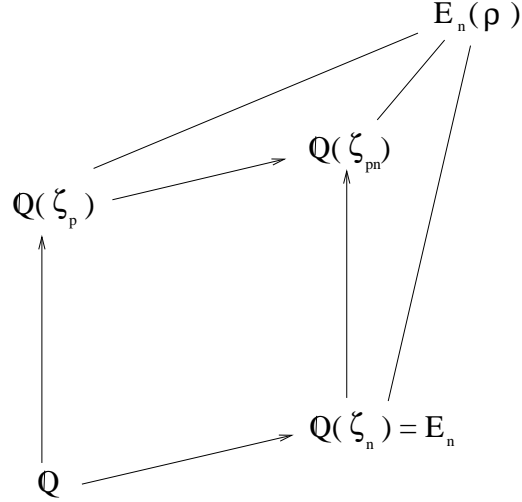}
\end{center}
\vskip -3mm
\rm{\caption{\label{fig:6-3} \leg Parallélogramme $[{\mathbb Q}, {\mathbb Q}(\zeta_n), {\mathbb Q}(\zeta_{pn}), {\mathbb
Q}(\zeta_p)]$}}
\vskip 1mm
\end{figure}

\noindent Or $[E_n(\rho) : E_n]=d$ par lemme \ref{lem:6L} précédent tandis que 
$$[{\mathbb Q}(\zeta_{pn}) : E_n] = [{\mathbb Q}(\zeta_{p}) : {\mathbb Q}] = p-1
\;.$$
Mais alors $p-1$ pair divise $d$ impair : absurde. Ceci prouve l'égalité $\mu_p
\,\cap \, E_n(\rho)= \iit$. Soit maintenant $\delta'$ un entier différent
de $1$ divisant $d$ : $1 \neq \delta' \mid d$. Appliquons le lemme \ref{lem:6L}
à $\frac{d}{\delta'}$ et $\rho^{\delta'}$ : comme
$(\rho^{\delta'})^{\frac{d}{\delta'}}=\rho^d=l$, on a
$$[E_n(\rho^{\delta'}) : E_n] = \frac{d}{\delta'} \;.$$
On déduit alors de $[E_n(\rho^{\delta}) : E_n]=d$, que
$$[E_n(\rho) : E_n(\rho^{\delta'})] = \delta' \;.$$
Il en résulte que le polynôme minimal de
$\rho$ sur $E_n(\rho^{\delta'})$ est :
$$Irr(\rho, E_n(\rho^{\delta'}),X) = X^{\delta'} -
\rho^{\delta'}=\prod_{j=0}^{\delta'-1}(X-\zeta_{\delta'}^j \, \rho) \;.$$
Avoir $E_n(\rho^{\delta}) / E_n(\rho^{\delta'})$ galoisienne impliquerait donc que
$\zeta_{\delta'} \in E_n(\rho^{\delta})$, et a fortiori que $\zeta_p \in
E_n(\rho^{\delta})$ pour tout nombre premier $p$ divisant $\delta' \neq 1$ :
contradiction d'après ce qui précède puisque par définition $\delta'$ divise
$d$. Il est donc prouvé que $E_n(\rho^{\delta}) \,\nongal E_n(\rho^{\delta'})$
pour le cas $\delta=1$.

Soit enfin $\delta$ un entier différent de $d$ divisant $d$ : $d\neq \delta \mid
d$. Posons
$$P:=\rho^{\delta} \;, \quad r:=\frac{d}{\delta} \;, \quad
\delta":=\frac{\delta'}{\delta}\;.$$
Clairement
$$P^r=\rho^d=l \;,\quad P^{\delta"}=\rho^{\delta'} \;.$$
De plus $\delta' \neq \delta$ i.e. $\delta" \neq 1$ et $\delta"$ divise $r$ car
$\delta'$ divise $d$. On peut donc appliquer le résultat précédent pour
$\delta=1$ avec la substitution
$$\left(\begin{array}{ccc}
d	&\rho	&\delta'\\
r	&P	&\delta"
\end{array}\right)$$
qui nous dit que l'extension 
$$(E_n(P) / E_n(P^{\delta"}))=(E_n(\rho^{\delta}) / E_n(\rho^{\delta'}))$$
n'est pas galoisienne.
\end{proof}

\gespace
Nous sommes maintenant en mesure de prouver le
\mespace
\begin{Th} \label{th:6N} 
Pour tous entiers : $d \geq 3$ impair et $n \geq 1$, tels que l'on ait $pgcd(d,n)=1$, le couple $(\varphi(n),d)$ (où $\varphi$ désigne l'indicateur d'Euler) est un degré de tourabilité (Déf. \ref{def:6I}).
\end{Th}

\noindent{\it Scholie .} Cet énoncé sera amélioré au théorème \ref{th:6P} (cf.
infra).

\gespace
\begin{proof}
Comme dans le Fait \ref{fait:6M}, on se place dans les notations du lemme
\ref{lem:6L}. Décomposons $d$ en produit de nombres premiers $p_i$ non
nécessairement distincts : $d={ \prod\limits_{i=0}^{k}}\,p_i$ , et posons
$$\forall m \in \{1, \dots, k+1\} \quad \delta_m:=d / (\prod_{j=0}^{m-1}p_j)\;,
\quad L_m:=E_n(\rho^{\delta_m})\;.$$
Prouvons, par récurrence finie sur $m$, que pour tout entier dans $\{1, \dots,
k+1 \}$ l'extension $L_m / E_n$ est galsimple non galoisienne :
$$\forall m \in \{1, \dots, k+1\} \quad (L_m \galsimple \nongal E_n) \;.$$
Pour $m=1$, on déduit directement du Fait \ref{fait:6M} appliqué avec
$\delta'=d$ que l'extension $L_1:=E_n(\rho^{\delta_1}) / E_n$ n'est pas
galoisienne. De plus, comme $(\rho^{\delta_1})^{p_0}=\rho^d=l$, on sait par le
lemme \ref{lem:6L} que
$$[L_1 : E_n]=p_0 \;,$$
de sorte que $L_1 / E_n$ est une extension simple (Chap. 2, Ex. 1.10.(iii)),
donc galsimple. Supposons maintenant le résultat vrai pour $m \in \{ 1, \dots ,
k \}$ et démontrons
le pour $m + 1$. Par définition
$$\delta_m = p_m \, \delta_{m+1} \;.$$
D'autre part, en vertu du lemme \ref{lem:6L} à nouveau 
$$[L_m = E_n(\rho^{\delta_m}) : E_n] = \frac{d}{\delta_m} \;.$$
On en déduit
$$\quad [L_{m+1} :E_n ] = p_m  \frac{d}{\delta_m} \;,$$
et
$$[L_{m+1} : L_m] = p_m \;.$$
Ceci prouve que l'extension $L_{m+1} / L_m$ est simple, donc galsimple.
Et d'après le Fait \ref{fait:6M}, elle est non galoisienne :
$$L_{m+1} = E_n(\rho^{\delta_{m+1}}) \nongal E_n(\rho^{\delta_m}) = L_m \;.$$
De plus, par l'hypothèse de récurrence, on a $(L_m \galsimple \nongal E_n)$.
D'après la proposition 4.2 du chapitre 5, on en déduit l'extension
galsimple non galoisienne $(L_{m+1} \galsimple \nongal E_n)$. Ceci achève le
raisonnement par récurrence.

En particulier pour $m=k+1$, on a prouvé que l'extension $(E_n(\rho) / E_n)$ 
est galsimple non galoisienne. Finalement, comme l'extension $E_n={\mathbb
Q}(\zeta_n) / {\mathbb Q}$ est galoisienne, donc galtourable, il résulte du
théorème $M$ (Th. \ref{th:6A}) que l'extension $(E_n(\rho) / E_n)$ est de degré
de tourabilité $(\varphi(n),d)$.
\end{proof}

\gespace
Pour généraliser le théorème \ref{th:6N} précédent, rappelons que

\mespace
\begin{prop} \cite[AVII.61, 5)]{Bo2} \label{prop:6O}\\
Soit $G$ un groupe commutatif fini. Pour tout entier $q$ diviseur de l'ordre de
$G$ : $q \, \big|  \mid \! G \!\mid$, il existe un sous-groupe de $G$
d'ordre $q$. 
\end{prop}

\gespace
\begin{Th} \label{th:6P}
Pour tous entiers : $d \geq 3$ impair et $n \geq 1$, tels que l'on ait $pgcd(d,n)=1$, le couple $(n,d)$ est un degré de tourabilité.
\end{Th}

\mespace
\begin{proof}
Comme $pgcd(d,n^2)=1$, on sait par le théorème \ref{th:6N} que le couple
$(\varphi(n^2),d)$ est un degré de tourabilité. Précisément, le corps cyclotomique
$E_{n^2}={\mathbb Q}(\zeta_{n^2})$ est le corps d'intourabilité de l'extension
$(E_{n^2}(\rho) / {\mathbb Q})$ où $\rho^d=l$ est un nombre premier ne divisant pas
$n$ (lemme \ref{lem:6L}). Retenons en particulier que
$$(E_{n^2}(\rho) \galsimple \nongal E_{n^2}) \;,\quad [E_{n^2}(\rho) : E_{n^2}]=d
\;.$$
Par ailleurs, il résulte directement de la formule explicitant l'indicateur
d'Euler que $n$ divise $\varphi(n^2)$. Appliquons la proposition \ref{prop:6O} au groupe abélien $Gal(E_{n^2} / {\mathbb Q})$ : comme $n$ divise son ordre
$\varphi(n^2)$, il existe un sous-groupe $H$ d'ordre $\varphi(n^2)/n$ :
$$\mid \! Gal(E_{n^2}/ {\mathbb Q}) \! \mid \,=\, \frac{\varphi(n^2)}{n} \, n \,=\, \mid \! H \! \mid \, n\;.$$
Considérons alors le corps des invariants dans $E_{n^2}$ de $H$ : $F_n :=
E_{n^2}^H$. Par le théorème d'Artin, $Gal(E_{n^2} / F_n)=H$, d'où
$$[F_n : {\mathbb Q}] = \frac{[E_{n^2} : {\mathbb Q}]}{[E_{n^2} : F_n]} =
\frac{n \, \mid H \mid}{\mid H \mid} = n \;;$$
et l'extension $F_n / {\mathbb Q}$ est galtourable (puisqu'elle est abélienne
!). En adjoignant $\rho$ à $F_n$, nous allons prouver que l'extension $F_n(\rho)
/ {\mathbb Q}$ est de degré de tourabilité égal à $(n,d)$. En translatant
l'extension galoisienne $E_{n^2} \gal F_n$ par $F_n(\rho) / F_n$ , nous obtenons
l'extension galoisienne $E_{n^2}(\rho) \gal F_n(\rho)$ avec
$$Gal(E_{n^2}(\rho) / F_n(\rho)) \isomto Gal(E_{n^2} / (F_n(\rho) \cap E_{n^2}))$$
(cf. Chap. 2, Th. 2.1). On en déduit en particulier l'inégalité
$$[E_{n^2}(\rho) : F_n(\rho)] \leq [E_{n^2} : F_n] \;.$$

\begin{figure}[!h]
\begin{center}
\vskip -6mm
\includegraphics[width=6.0cm]{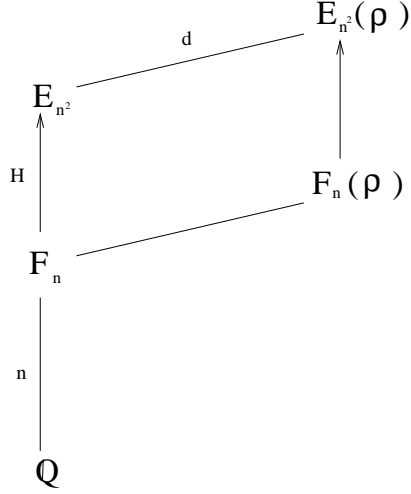}
\end{center}
\vskip -3mm
\rm{\caption{\label{fig:6-4} \leg $(n,d)$ degré de tourabilité}}
\vskip 1mm
\end{figure}

Or
$$[E_{n^2}(\rho) : E_{n^2}][E_{n^2} : F_n] = [E_{n^2}(\rho) : F_n] =
[E_{n^2}(\rho) : F_{n}(\rho)][F_{n}(\rho) :F_n] \;.$$
L'inégalité précédente oblige donc à avoir
$$[E_{n^2}(\rho) : E_{n^2}] \leq [F_{n}(\rho) :F_n] \;.$$
Mais $\rho$ annule le polynôme $X^d-l \in F_n[X]$ ; donc $[F_{n}(\rho) : F_n]
\leq d$. Comme on a rappelé que $[E_{n^2}(\rho) : E_{n^2}]=d$, ceci prouve que
$$[F_{n}(\rho) :F_n] = d\;.$$
Par ailleurs,
$$[E_{n^2}(\rho) : F_{n}(\rho)] = \frac{[E_{n^2}(\rho) : F_{n}]}{[F_{n}(\rho) :
F_n]} = \frac{[E_{n^2}(\rho) : F_{n}]}{d} = \frac{[E_{n^2}(\rho) :
F_{n}]}{[E_{n^2}(\rho) : E_{n^2}]} = [E_{n^2} : F_n] \;.$$
Mais aussi
$$[E_{n^2}(\rho) : F_{n}(\rho)] = [ E_{n^2} : F_n(\rho) \cap E_{n^2}] \;.$$
Il en résulte donc que
$$F_n(\rho) \cap E_{n^2} = F_n \;.$$
Montrons maintenant que l'extension $F_n(\rho) / F_n$ est galsimple. Supposons
qu'il existe un corps intermédiaire $N$ galoisien sur $F_n$ :
$$F_n \unlhd N \leq F_n(\rho) \;.$$
En translatant l'extension galoisienne $N \gal F_n$ par $E_{n^2} / F_n$, on
obtient l'extension galoisienne $N \, E_{n^2} \gal E_{n^2}$ qui est un quotient de
l'extension galsimple non galoisienne $(E_{n^2}(\rho) \galsimple \nongal
E_{n^2})$ (cf. supra). Par conséquent
$$E_{n^2} \unlhd N \, E_{n^2} \leq E_{n^2}(\rho) \quad \implique \quad 
\left\{\begin{array}{c}
 N \, E_{n^2} = E_{n^2}(\rho) \\
\text{ou}\\ 
E_{n^2}=  N \, E_{n^2} 
\end{array}\right.
\;.$$
Mais on ne peut avoir $N \, E_{n^2} = E_{n^2}(\rho)$ puisque $E_{n^2}(\rho)$
n'est pas galoisien sur $E_{n^2}$. C'est donc que $E_{n^2}=  N \, E_{n^2}$, i.e.
$N \leq E_{n^2}$. Ainsi
$$F_n \leq N = N \cap E_{n^2} \leq F_n(\rho) \cap E_{n^2} = F_n \;,$$
d'où $N=F_n$, ce qui établit la galsimplicité de $F_n(\rho) / F_n$. De plus, si
$F_n(\rho) / F_n$ était galoisienne, il en serait de même, par translation avec
$E_{n^2} / F_n$, de l'extension $E_{n^2}(\rho) \nongal E_{n^2}$ : contradiction.
Finalement, $F_n(\rho) / F_n$ est galsimple non galoisienne : $(F_n(\rho)
\galsimple \nongal F_n)$, ce qui prouve que l'extension $F_n(\rho) / {\mathbb Q}$
est degré de tourabilité $(n,d)$.
\end{proof}

\addtocontents{toc}{\mespace\pespace}
\chapter{TOURS D'ÉLÉVATION ET \hskip-17mm { } \newline DISSOCIATION DES EXTENSIONS FINIES}
\addtocontents{lof}{\gespace}
\addtocontents{lof}{\noindent Chapitre \thechapter}
\addtocontents{lof}{\pespace}
\addtocontents{toc}{\pespace}
\gespace
Grâce au théorème $M$, détaillé au chapitre précédent, nous sommes en mesure de
généraliser aux extensions finies quelconques les analogues aux théorèmes de
Schreier et de Jordan-Hölder du chapitre 4 pour les extensions galtourables.
Nous allons voir, grâce à la notion de tour d'élévation, que des définitions
tout à fait canoniques conduisent à des énoncés très similaires, bien que leurs
démonstrations soient différentes.

\gespace
\gespace
\section{Tours d'élévation}
\mespace

\begin{Thdef} \label{thdef:7A}\index{Théorème!de la tour d'élévation}(dit "de la tour d'élévation\footnote{\,Le
dictionnaire Le Robert donne la définition suivante : {\it Élévation.
$\diamond$ Géom. Projection sur un plan vertical parallélement à une des faces
de l'objet.} On peut considérer que c'est ce qu'évoque la figure 30.}") \\
Soit $L/K$ une extension finie quelconque. Toute tour
$$(F) \qquad K=F_0 \leq F_1 \leq \dots \leq F_i \leq \dots \leq F_m=L$$
de $L/K$ induit une tour galtourable (cf. Chap. 2, Déf. 1.5) 
constituée des corps d'intourabilité (Chap. 6, Déf. 1.3) 
sur $K$ de chacun des corps de $(F)$ :
$$\begin{array}{r}
(M)  \quad K=M_0 := M(F_0/K) \lessgtr M_1:= M(F_1/K) \lessgtr \dots \lessgtr
M_i:=M(F_i/K) \lessgtr \dots \\
\dots \lessgtr M_m:=M(F_m/K)=M(L/K) \;.
\end{array}$$

Nous appelons la tour $(M)$ "la tour d'élévation de $M(L/K)/K$ associée à la
tour $(F)$"\index{Tour!d'élévation}, et la notons
$$(M) = ({\mathcal El}[M(L/K) \galtou K, (F)]) \;.$$

Nous disons en abrégé que "$(E)$ est une tour d'élévation de $M(L/K)$" si et
seulement s'il existe une tour $(F)$ de $L/K$ telle que $(E) = ({\mathcal
El}[M(L/K) \galtou K, (F)])$.

\begin{figure}[!h]
\begin{center}
\vskip -1mm
\includegraphics[width=5.4cm]{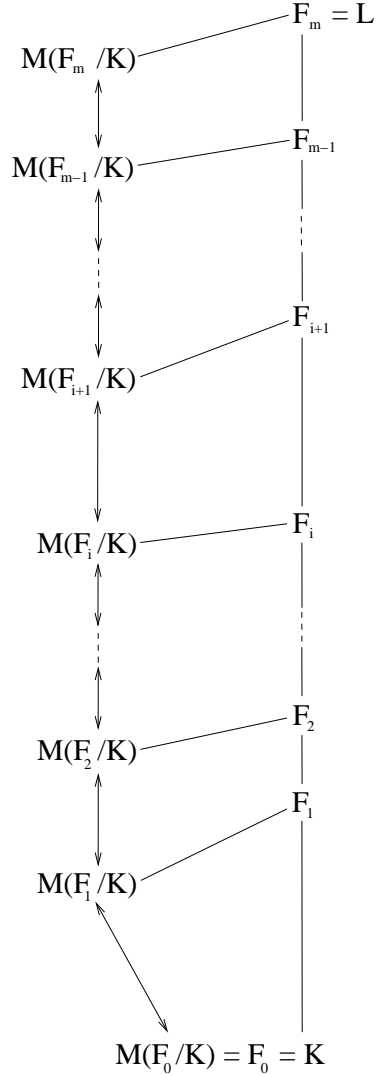}
\end{center}
\vskip -5mm
\rm{\caption{\label{fig:7-1} \leg Tour d'élévation de $M(L/K)$ associée à $(F)$}}
\vskip -3mm
\end{figure}
\end{Thdef}

\mespace
\begin{proof}
D'après la proposition 2.2 du chapitre 6, 
on a
$$M(F_{i-1}/K) \, \leq \, F_{i-1} \cap M(F_i/K) \qquad (i=1,\dots,m)$$
et donc $M_{i-1} \leq M_i \quad (i=1,\dots,m)$.
Or les extensions $M_i/K$ étant galtourables, il en est de même des
sous-extensions $M_i/M_{i-1} \quad (i=1,\dots,m)$ en vertu de la proposition 2.9
du chapitre 2. 
D'où la conclusion.
\end{proof}

Nous voulons ensuite introduire, à partir de la définition précédente, les tours
d'élévation de l'extension $L/K$ elle-même, et plus seulement de son extension
quotient galtourable maximale $M(L/K) \galtou K$. On sait, par le théorème $M$,
que la sous-extension $L/M(L/K)$ peut être ou bien triviale, ou bien galsimple
non galoisienne. Cette alternative pose une difficulté : rajouter
systématiquement le corps $L$ à une tour de $M(L/K) \galtou K$, c'est le répéter
lorsque l'extension $L/K$ est galtourable. Or une telle répétition rend
impossible l'obtention de tours strictes, et par suite de tours de composition.
Nous évitons cet écueil par la définition suivante

\mespace
\begin{defn} \label{def:7B}
Soient $L/K$ une extension et $M$ un corps intermédiaire entre $K$ et $L$ : $K
\leq M \leq L$. Soit de plus 
$$(E) \qquad K=E_0 \leq E_1 \leq \dots \leq E_m=M$$
une tour de $M/K$. Nous appelons "tour de $L/K$ induite par
$(E)$"\index{Tour!induite}, et nous
notons
$$((E) \dasharrow L)$$
la tour de $L/K$ définie de la façon suivante
$$((E) \dasharrow L) := \left\{\begin{array}{ll}
(E)						&\text{si } M=L\\
K=E_0 \leq E_1 \leq \dots \leq E_m=M < L	&\text{si } M\neq L
\end{array}\right. \;.$$
\end{defn}

\gespace
\begin{defn}  \label{def:7C}
Soient $L/K$ une extension finie et
$$(F) \qquad K=F_0 \leq F_1 \leq \dots \leq F_m=L$$
une tour quelconque de $L/K$.

Nous appelons "tour d'élévation de $L/K$ associée à $(F)$"\index{Tour!d'élévation}, et nous notons
$({\mathcal El}[L/K, (F)])$, la tour de $L/K$ induite par la tour d'élévation
associée à $(F)$ de l'extension quotient galtourable maximale de $L/K$ :
$$({\mathcal El}[L/K, (F)]):=(({\mathcal El}[M(L/K) \galtou K, (F)]) \dasharrow
L) \;.$$

Nous disons en abrégé que "$(E)$ est une tour d'élévation de $L/K$", si et
seulement s'il existe une tour $(F)$ de $L/K$ telle que $(E)=({\mathcal El}[L/K,
(F)])$.
\end{defn}

\gespace
\begin{rem} \label{rem:7D}
Les définitions \ref{thdef:7A} et \ref{def:7C} coïncident lorsque l'extension
$L/K$ est galtourable. En particulier, il résulte du Th. \& Déf. \ref{thdef:7A}
que toute tour d'élévation d'une extension $L/K$ galtourable est galtourable.
\end{rem}

\gespace
Tirons des définitions précédentes les résultats directs suivants

\mespace
\begin{fait} \label{fait:7E}
Dans les notations de la définition \ref{def:7B}, on a l'équivalence
$$(E) \text{ stricte} \qquad \SSI \qquad ((E) \dasharrow L) \text{ stricte.}$$
\end{fait}

\mespace
\begin{proof}
Le résultat est évident si $M=L$. Si $M \neq L$, il est immédiat puisque la
dernière marche de $((E) \dasharrow L)$ est toujours non triviale.
\end{proof}

\gespace
\begin{prop} \label{prop:7F}
Soit $L/K$ une extension (galtourable) finie. Pour toute tour galtourable $(T)$
de $L/K$, la tour d'élévation de $L/K$ associée à $(T)$ est égale à $(T)$ :
$$({\mathcal El}[L/K, (T)])=(T) \;.$$
\end{prop}

\mespace
\begin{proof}
Notons
$$(T) \qquad K=T_0 \lessgtr T_1 \lessgtr \dots \lessgtr T_i \lessgtr \dots
\lessgtr T_m=L$$
une tour galtourable de $L/K$. Considérons les tours ratio (Chap. 3, Déf. 3.1.(2))
$$(T_r) \qquad K=T_0 \lessgtr T_1 \lessgtr \dots \lessgtr T_r \qquad (r=0,
\dots ,m)$$
Ce sont des tours galtourables. D'après le corollaire 1.10 du chapitre 3,
les extensions $T_r/ K \; (r=0,\dots ,m)$ sont donc galtourables. Et l'on
déduit du corollaire 1.4 du chapitre 6 
que 
$$\forall r \in \{0,\dots ,m\} \qquad M(T_r/K)=T_r \;.$$
D'où la conclusion.
\end{proof}

\mespace
\begin{lem} \label{lem:7G}
Soient $L/K$ une extension quelconque et $M$ un corps intermédiaire entre $K$ et
$L$ : $K \leq M \leq L$. Pour tout entier $r$ et toute tour
$$(F) \qquad K=F_0 \leq F_1 \leq \dots \leq F_r=M$$
de $M/K$, on a l'égalité
$$(rat_r ((F) \dasharrow L))=(F) $$
où $(rat_r ((F) \dasharrow L))$ désigne la tour ratio à l'indice $r$ (cf. Chap.
3, Déf. 3.1.(2)) 
de la tour de $L/K$ induite par $(F)$ (Déf. \ref{def:7B}).
\end{lem}

\mespace
\begin{proof}
- Si $M=L$, $((F) \dasharrow L)=(F)$ et $(rat_r (F)=(F))$ (cf. Chap. 3, Fait
3.2).\\ 
- Si $M \neq L$, i.e. $M < L$, on a
$$((F) \dasharrow L) \qquad K=F_0 \leq F_1 \leq \dots \leq F_r=M < L\;.$$
Il en découle, par définition d'une tour ratio, que
$$(rat_r ((F) \dasharrow L)) \qquad K=F_0 \leq F_1 \leq \dots \leq F_r=M$$
i.e. $(rat_r ((F) \dasharrow L))=(F)$.
\end{proof}

\gespace
Le lemme précédent nous permet d'énoncer la proposition suivante, dont l'intérêt
est de faire passer d'une tour de $M(L/K) \galtou K$ à une tour de $L/K$.

\mespace
\begin{prop} \label{prop:7H}
Soient $L/K$ une extension finie et $M(L/K)$ son corps d'intourabilité (Chap. 6,
Déf. 1.3). Pour toute tour
$$(E) \qquad K=E_0 \leq E_1 \leq \dots \leq E_i \leq \dots \leq E_{m-1} \leq
E_m=M(L/K)$$
de l'extension quotient galtourable maximale de $L/K$, on a
$$(E)=(rat_m ((E) \dasharrow L))$$
et
$$({\mathcal El}[M(L/K) \galtou K, (inf_L(E))])= ({\mathcal El}[M(L/K) \galtou
K, (E)])$$
où $(inf_L(E))$ désigne la tour inflatée à $L$ de $(E)$ (cf. Chap. 3, Déf.
3.1.(3)).
\end{prop}

\mespace
\begin{proof}
La première égalité est immédiate en vertu du lemme \ref{lem:7G}. Et par
définition d'une tour inflatée
$$(inf_L(E)) \qquad K=E_0 \leq E_1 \leq \dots \leq E_i \leq \dots \leq E_{m-1} \leq
L \;.$$
Par application directe du Th. \& Déf. \ref{thdef:7A}, on en déduit la tour
d'élévation de $M(L/K) \galtou K$ :
$$\begin{array}{r}
({\mathcal El}[M(L/K) \galtou K, (inf_L(E))]) \qquad K=M(E_0/K) \lessgtr
M(E_1/K) \lessgtr \dots \qquad \qquad \qquad\\
\dots \lessgtr M(E_{m-1}/K) \lessgtr M(L/K) \;.
\end{array}$$
D'autre part, on a vu au corollaire 1.4 du chapitre 6 
que
$$M(M(L/K)/K) = M(L/K) \;.$$
Pour obtenir l'égalité des tours d'élévation de l'énoncé, il suffit alors
d'appliquer à nouveau le Th. \& Déf. \ref{thdef:7A} à la tour $(E)$.
\end{proof}

\gespace
\begin{prop} \label{prop:7I}
Soient $L/K$ une extension finie quelconque et
$$(E) \qquad K=E_0 \leq \dots \leq E_j \leq \dots \leq E_n=L$$
une tour de $L/K$. On a l'équivalence :\\
$(E)$ est une tour d'élévation de $L/K$ si et seulement si $(E)$ est induite par
une tour galtourable de $M(L/K) \galtou K$.
\end{prop}

\mespace
\begin{proof}
Supposons que $(E)$ soit une tour d'élévation de $L/K$. Par la définition
\ref{def:7C}, elle est induite par une tour $(M)$ d'élévation de $M(L/K) \galtou
K$ ; autrement dit il existe une tour $(F)$ de $L/K$ telle que $(M)$ soit la
tour d'élévation de $M(L/K) \galtou K$ associée à $(F)$. Et l'on a vu au Th.
\& Déf. \ref{thdef:7A} que $(M)$ est une tour galtourable.

\indent Inversement, supposons que $(E)$ soit induite par une tour galtourable
$(T)$ de $M(L/K) \galtou K$ : $(E)=((T) \dasharrow L)$. D'après la proposition
\ref{prop:7F}
$$(T) =  ({\mathcal El}[M(L/K)\galtou K, (T)]) \;,$$
et par la proposition \ref{prop:7H}
$$({\mathcal El}[M(L/K) \galtou K, (T)]) = ({\mathcal El}[M(L/K) \galtou K,
(inf_L(T))])\;.$$
Finalement
$$\begin{array}{rl}
(E)	&=({\mathcal El}[M(L/K) \galtou K, (inf_L(T))]) \dasharrow L)\\
\\
	&=({\mathcal El}[L/K, (inf_L(T))])
\end{array}$$
(cf. Déf. \ref{def:7C}).
\end{proof}

\gespace
Prouvons enfin le résultat suivant.
\begin{prop} \label{prop:7J}
Soient $L/K$ une extension finie, et $(F)$ une tour stricte de $L/K$ telle que
la tour d'élévation $({\mathcal El}[L/K, (F)])$ de $L/K$ associée à $(F)$
soit une tour stricte. Alors, tout raffinement galoisien strict de $({\mathcal
El}[L/K, (F)])$ est encore une tour d'élévation de $L/K$.
\end{prop}

\mespace
\begin{proof}
Distinguons deux cas.\\
\noindent (1)  $M(L/K)=L$. Il résulte directement des définitions \ref{def:7C}
\& \ref{def:7B} que la tour d'élévation de $L/K$ associée à $(F)$ est une tour
galtourable de $L/K$. Par la proposition 1.8 du chapitre 3, 
tout raffinement galoisien de celle-ci est encore une tour galtourable de $L/K$
; et donc une tour d'élévation de $L/K$ en vertu de la proposition
\ref{prop:7F}.

\noindent (2) $M(L/K) < L$. Pour
$$(F) \qquad K=F_0 < \dots < F_i < \dots < F_m=L \;,$$
on a dans ce cas
$$\begin{array}{r}({\mathcal El}[L/K,(F)])  \qquad K=M_0 := M(F_0/K) \lessgtr \dots \lessgtr M_i:=M(F_i/K) \lessgtr \dots \qquad \qquad\qquad \\
\dots \lessgtr M_m:=M(L/K) < M_{m+1}:=L\;.
\end{array}$$
Soit $(R)$ un raffinement galoisien strict de $({\mathcal El}[L/K,(F)])$.
D'après la définition 1.3.(4) et la remarque 1.2.(2) du chapitre 3, il s'écrit
$$\begin{array}{r}
(R) \qquad K=R_0 < \dots < R_{j_1}=M_1 < \dots <R_j < \dots < R_{j_i}=M_i <
\dots \qquad \qquad\qquad\\
\dots < R_{j_m}=M_m < \dots < R_{j_{m+1}}=L
\end{array}$$
où 
$$0=j_0 < j_1 < \dots < j_m < j_{m+1}\;.$$
D'autre part, il résulte de la proposition 3.6.(1) du chapitre 3 que la tour
ratio de $(R)$ à l'indice $j_m$ (Chap. 3, Déf. 3.1.(2))
$$\begin{array}{r}
(rat_{j_m}(R)) \qquad K=R_0 < \dots < R_{j_1}=M_1 < \dots <R_j < \dots \qquad \qquad\qquad \\
\dots < R_{j_i}=M_i < \dots < R_{j_m}=M_m=M(L/K) \end{array}$$
est un raffinement galoisien de la tour
$$(rat_m({\mathcal El}[L/K, (F)])) \qquad K=M_0 \lessgtr \dots \lessgtr
M_i \lessgtr \dots \lessgtr M_m=M(L/K) \;.$$
Cette dernière étant galtourable, on déduit de la proposition 1.8 du chapitre 3
que $(rat_{j_m}(R))$ est aussi une tour galtourable de $M(L/K) \galtou K$. En
vertu de la proposition \ref{prop:7I} précédente, il suffit maintenant pour
conclure de prouver que la tour $(R)$ est induite par la tour galtourable
$(rat_{j_m}(R))$. Raisonnons par l'absurde en supposant que cela ne soit pas le
cas. Comme par définition on a l'implication
$$j_m+1=j_{m+1} \quad \implique \quad (R)=((rat_{j_m}(R)) \dasharrow L) \;,$$
cela signifierait que
$$j_m+1 \neq j_{m+1} \quad \implique \quad j_m+1 < j_{m+1}  \;;$$
et le raffinement $(R)$ étant strict,
$$M(L/K)=R_{j_m} < R_{j_m+1} < R_{j_{m+1}} =L \;.$$
Il est de plus galoisien ; donc par la condition (RAFG) de la définition 1.3 du
chapitre 3, 
on en déduirait que $R_{j_m+1}$ est galoisien sur $R_{j_m}$. Mais avoir
$$M(L/K) \lhd R_{j_m+1} < L$$
contredit la galsimplicité de $L/M(L/K)$ (cf. Théorème $M$ du chapitre 6).
\end{proof}

\gespace
\gespace
\section{Tour de composition et Théorèmes de dissociation}
\gespace
Nous avons introduit au chapitre 4 la notion de "tour de composition
galoisienne" d'une extension galtourable. Nous allons maintenant définir sa
généralisation à n'importe quelle extension finie.

\gespace
\begin{defn} \label{def:7K}
Soit $L/K$ une extension finie quelconque. Nous appelons "tour de composition de
$L/K$"\index{Tour!de composition} toute tour d'élévation de $L/K$ stricte qui n'admet aucun raffinement
galoisien propre.
\end{defn}

\gespace
Cette notion de tour de composition pour n'importe quelle extension finie
généralise celle de tour de composition galoisienne pour les extensions
galtourables (Chap. 4, Sect. 1). 
En effet :

\mespace
\begin{lem} \label{lem:7L}
Soit $L \galtou K$ une extension galtourable finie. Toute tour de composition
galoisienne de $L/K$ est une tour de composition de $L/K$ au sens de la
définition \ref{def:7K} précédente.
\end{lem}

\mespace
\begin{proof}
Soit $(T)$ une tour de composition galoisienne de $L/K$ (Chap. 4, Déf. 1.1.(1)).
Il suffit de montrer que $(T)$ est une tour d'élévation de $L/K$. Or $(T)$ est a
fortiori une tour galtourable ; donc par la proposition \ref{prop:7F},
$(T)=({\mathcal El}[L/K, (T)])$.
\end{proof}

\gespace
\begin{prop} \label{prop:7M}
Soit $L \galtou K$ une extension galtourable finie. L'ensemble non-vide des
tours de composition galoisiennes de $L/K$ est égal à l'ensemble des tours de
composition de $L/K$ au sens de la définition \ref{def:7K}.
\end{prop}

\mespace
\begin{proof}
L'ensemble des tours de composition galoisiennes de $L/K$ est non-vide en vertu du second théorème de dissociation (Chap. 4, Th. 4.2). 
Et par le lemme \ref{lem:7L}, il est inclus dans celui des tours de composition
de $L/K$ au sens de la définition ci-dessus.

Prouvons l'autre inclusion en considérant une tour de composition $(C)$ de $L/K$
comme définie au \ref{def:7K}. C'est en particulier une tour d'élévation de $L/K$,
et elle est galtourable parce que $L/K$ l'est (remarque \ref{rem:7D}). Selon la
proposition 1.9 du chapitre 3, 
elle admet donc un raffinement galoisien $(R)$ qui est une tour galoi\-sienne. Or,
en tant que tour de composition, $(C)$ n'admet pas de raffinement galoisien
propre. Ceci établit que le raffinement $(R)$ de $(C)$ est trivial (Chap. 3,
Déf. 1.3.(2)). 
De plus, $(C)$ est stricte en tant que tour de composition. Par définition d'une
tour stricte associée (Chap. 3, Prop. \& Déf. 2.1), 
on en déduit que $(C)=(R_<)$. Comme $(R)$ est une tour galoisienne, le
corollaire 2.6 du chapitre 3 
assure finalement que $(C)$ est une tour de composition galoisienne.
\end{proof}

\gespace
Nous avons donné une caractérisation des tours d'élévation (Prop. \ref{prop:7I}) ; nous sommes maintenant en mesure de faire de même pour les tours de composition.

\mespace
\begin{prop} \label{prop:7N}
Soient $L/K$ une extension finie et
$$(C) \qquad K=C_0 \leq \dots \leq C_i \leq \dots \leq C_m=L$$
une tour de $L/K$. On a l'équivalence :\\
$(C)$ est une tour de composition si et seulement si elle est induite par une
tour de composition galoisienne de l'extension quotient galtourable maximale de
$L/K$.
\end{prop}

\mespace
\begin{proof}
Supposons que $(C)$ soit une tour de composition de $L/K$. C'est une tour
d'élévation de $L/K$, donc par la proposition \ref{prop:7I}, elle est induite
par une tour galtourable $(T)$ de $M(L/K) \galtou K$. Montrons que $(T)$ est en
fait une tour de composition galoisienne de $M(L/K) \galtou K$. D'après le Fait
\ref{fait:7E}, elle est stricte. Notons $r$ la hauteur de $(T)$. D'après le
lemme \ref{lem:7G}, 
$$(T)=(rat_r((T) \dasharrow L))=(rat_r(C)) \;.$$
Raisonnons par l'absurde en supposant qu'il existe un raffinement galoisien
propre $(R)$ de $(T)$. D'après la proposition 3.7.(2) du chapitre 3, 
$(res_r(C))$ et $(R)$ induisent un raffinement galoisien propre de $(C)$ :
contradiction, puisque $(C)$ n'en admet pas.
Finalement, $(T)$ est une tour stricte sans aucun raffinement galoisien propre.
C'est de plus une tour d'élévation d'après la proposition \ref{prop:7F},
puisqu'elle est galtourable. C'est donc bien une tour de composition de $L/K$.

Réciproquement, supposons que $(C)$ soit induite par une tour de composition
galoisienne $(T)$ de $M(L/K) \galtou K$.
D'après le Fait \ref{fait:7E}, $(C)$ est stricte, tandis que par la proposition
\ref{prop:7I}, c'est une tour d'élévation de $L/K$. Soit $r$ la hauteur de
$(T)$. Raisonnons par l'absurde en supposant l'existence d'un raffinement
galoisien propre de $(C)$. Comme $(T)$ n'admet pas de raffinement galoisien
propre, on déduit alors du (1) de la proposition 3.7 du chapitre 3 
qu'il existe un raffinement galoisien propre de la tour resteinte $(res_r(C))$.
Or celle-ci est :\\
- Ou bien la tour triviale (Chap. 2, Déf. \& Conv. 1.1.(1))
$$(res_r(C)) \qquad M(L/K)=F_0=L$$
lorsque $L/K$ est galtourable. Mais cette tour triviale est une tour de
composition galoisienne (Chap. 4, Fait 1.2) 
et n'admet donc aucun raffinement galoisien propre.\\
- Ou bien la tour stricte
$$(res_r(C)) \qquad M(L/K) < L \;.$$
Notons
$$(R) \qquad M(L/K)=R_0 \leq \dots \leq R_j \leq \dots \leq R_n=L$$
son raffinement galoisien susmentionné. Par définition d'un raffinement propre
(Chap. 3, Déf. \& Conv. 1.1.(2)), 
l'ensemble $\{j \in \{1, \dots , n-1\} \mid R_0 < R_j < R_n \}$ est non
vide, et l'on peut considérer son plus petit élément $j_0$. Alors, d'une part
la minimalité de $j_0$ implique que $R_{j_0-1}=M(L/K)$, d'autre part la condition
(RAFG) de définition d'un raffinement galoisien (Chap. 3, Déf. 1.3.(4)) conduit
à 
$$M(L/K) = R_{j_0-1} \lhd R_{j_0} < R_n =L \;.$$
Mais ceci contredit la galsimplicité de $L/M(L/K)$ (Chap. 6, Th.
1.1).\\
\indent L'existence d'un raffinement galoisien propre de $(C)$ est donc
impossible, ce qui finit de prouver que $(C)$ est une tour de composition de
$L/K$ au sens de la définition \ref{def:7K}.
\end{proof}

\gespace
Nous allons enfin pouvoir prouver les derniers théorèmes de dissociation qui
sont les analogues pour les extensions finies quelconques des $1^{\text{er}}$ et
$3^{\text{ème}}$ théorèmes de dissociation du chapitre 4 pour les extensions
galtourables. Mais nous n'avons jusqu'ici défini l'équivalence de deux tours
d'une même extension que lorsque ces tours sont galoisiennes. La définition
suivante de l'équivalence de deux tours induites implique en particulier celle
de l'équivalence de deux tours de composition non galoisiennes en vertu de la
proposition
\ref{prop:7N} précédente.

\mespace
\begin{defn} \label{def:7O}
Soient $L/K$ une extension finie quelconque, $(T)$ et $(T')$ deux tours
galoisiennes de l'extension quotient galtourable maximale $M(L/K) \galtou K$ de
$L/K$. Nous disons que les tours induites de $L/K$ par $(T)$ et $(T')$ sont
équivalentes\index{Tour!equivalente@équivalente} si et seulement si les tours galoisiennes $(T)$ et $(T')$ le sont
au sens de la définition 1.1.(2) du chapitre 4 :
$$((T) \dasharrow L) \, \sim \, ((T') \dasharrow L) \quad
\stackrel{\text{Déf.}}{\SSI} \quad (T) \sim (T') \;.$$
\end{defn}

\gespace
\begin{Th} ($5^{\text{ème}}$ théorème de
dissociation)\label{th:7P}\index{Dissociation!$5^{\text{ème}}$ Théorème (de)}\index{Théorème!de dissociation}\\
Deux tours d'élévation d'une même extension finie quelconque admettent des
raffinements galoisiens qui sont des tours d'élévation équivalentes de cette
extension.
\end{Th}

\mespace
\begin{proof}
Soient $L/K$ une extension finie et $(E^1)$, $(E^2)$ deux tours d'élévation de
$L/K$. Lorsque $L/K$ est galtourable, il en est de même de $(E^1)$ et $(E^2)$
par définition ; donc il suffit d'utiliser le $1^{\text{er}}$ théorème de
dissociation bis (Chap. 4, Th. 3.4) 
avec la proposition \ref{prop:7F}.

\indent Supposons désormais $L/K$ non galtourable, i.e. telle que $M(L/K) < L$
(Chap. 6, Cor. 1.4). 
Les tours d'élévation $(E^1)$ et $(E^2)$ sont respectivement induites par des
tours galtourables $(T^1)$ et $(T^2)$ de l'extension quotient galtourable
maximale $M(L/K) \galtou K$ (cf. Déf. \ref{def:7C} et Th. \& Déf.
\ref{thdef:7A}). 
En notant $r_i \quad (i=1,2)$ la hauteur de la tour
$$(T^i) \qquad K=T_0^i \lessgtr \dots \lessgtr T_{r_i}^i=M(L/K) \;,$$
on a par le lemmme \ref{lem:7G}
$$(T^i) = (rat_{r_i}(E^i)) \;,$$
et donc
$$(E^i) \qquad K=E^i_0=T^i_0 \lessgtr \dots \lessgtr E^i_{r_i}=T^i_{r_i}=M(L/K)
< E^i_{r_i+1}=L \;.$$

Ceci permet d'appliquer le (2) de la proposition 3.6 du chapitre 3.\\ 
- D'une part, la tour restreinte $(res_{r_i}(E^i))$ est un raffinement galoisien
d'elle-même en vertu du (2) de la remarque 1.4 du chapitre 3.\\ 
- D'autre part, on sait par le théorème 3.4 du chapitre 4 que $(T^1)$ et
$(T^2)$ admettent des raffinements galoisiens $(T'^1)$ et $(T'^2)$ qui sont des
tours galoisiennes équivalentes.

On déduit alors de la proposition 3.6 précitée 
qu'il existe un raffinement galoisien $(E'^i)$ de $(E^i)$ tel que
$$(res_{j_{r_i}}(E'^i))=(res_{r_i}(E^i)) \quad \text{et} \quad
(rat_{j_{r_i}}(E'^i))=(T'^i) \;.$$
Par conséquent :
$$(E'^i) \quad K=E'^i_0 \unlhd \dots \unlhd E'^i_{j_0}=E^i_0 \unlhd \dots
\unlhd E'^i_{j_{r_i}}=E^i_{r_i}=M(L/K) < E'^i_{j_{r_i}+1}=E^i_{r_i+1}=L \;.$$
En particulier $(E'^i)$ est la tour de $L/K$ induite par $(T'^i)$
$$(E'^i)=((T'^i) \dasharrow L) \qquad (i=1,2) \;.$$
Comme $(T'^1)$ et $(T'^2)$ sont des tours galoisiennes, a fortiori galtourables, de
$M(L/K) \galtou K$, on déduit de la proposition \ref{prop:7I} que les
raffinements galoisiens $(E'^1)$ et $(E'^2)$ sont des tours d'élévation de
$L/K$. Enfin, elles sont équivalences au sens de la définition \ref{def:7O}
puisqu'il en est ainsi de $(T'^1)$ et $(T'^2)$.
\end{proof}

\gespace
\begin{Th} ($6^{\text{ème}}$ théorème de
dissociation)\label{th:7Q}\index{Dissociation!$6^{\text{ème}}$ Théorème (de)}\index{Théorème!de dissociation}\\
Soit $L/K$ une extension finie quelconque.
\pespace
\noindent (1) Toute tour d'élévation stricte de $L/K$ admet un raffinement galoisien qui
est une tour de composition de $L/K$.
\pespace
\noindent (2) Deux tours de composition de $L/K$ sont équivalentes.
\end{Th}

\mespace
\begin{proof}
(1) Si $L/K$ est galtourable, i.e. si $L=M(L/K)$ (Chap. 6, Cor. 1.4), 
toute tour d'élévation de $L/K$ est galtourable puisque
$$({\mathcal El}[L/K, (F)]) = ({\mathcal El}[M(L/K) \galtou K, (F)]) \quad
\text{(Th. \& Déf. \ref{thdef:7A}).}$$
D'après le $3^{\text{ème}}$ théorème de dissociation bis (Chap. 4, Th. 4.4.(1)),
une tour d'élévation stricte de $L/K$ admet donc un raffinement qui est une tour
de composition galoisienne de $L/K$. Par le Fait 1.5.(2) du chapitre 3,
ce raffinement est galoisien, et c'est une tour de composition de $L/K$ en vertu
de la proposition \ref{prop:7M}.

Plaçons-nous maintenant dans le cas nouveau où $M(L/K) < L$. Soit $(E)$ une tour
d'élévation stricte de $L/K$. Elle est induite par une tour d'élévation $(T)$ de
$M(L/K) \galtou K$ (Déf. \ref{def:7C}) qui est stricte d'après le Fait
\ref{fait:7E} du présent chapitre, et galtourable (Th. \& Déf. \ref{thdef:7A}) :
$$(T) \qquad K=T_0 \lessgtr \dots \lessgtr T_i \lessgtr \dots \lessgtr
T_r=M(L/K) \;;$$
d'où 
$$\begin{array}{r}
(E)=((T) \dasharrow L) \qquad K=E_0:=T_0 \lessgtr \dots \lessgtr E_i:=T_i \lessgtr \dots \qquad\qquad\qquad\\
\dots \lessgtr E_r:=T_r=M(L/K) < E_{r+1}:=L \;.
\end{array}$$
En particulier, par définition (Chap. 3, Déf. 3.1), 
$$(T)=(rat_r(E))$$
et
$$(res_r(E)) \qquad E_r=M(L/K) < E_{r+1}=L \;.$$
Nous allons raffiner les tours $(rat_r(E))$ et $(res_r(E))$. D'une part, il est
clair que $(res_r(E))$ est un raffinement galoisien de lui-même (Chap. 3,
remarques 1.4.(2)). 
D'autre part $(T)$ étant une tour galtourable stricte, on déduit du (1) du
théorème 4.4 du chapitre 4 
que $(rat_r(E))$ admet un raffinement galoisien $(C)$ qui est une tour de
composition galoisienne de $M(L/K) \galtou K$. Il résulte alors du (2) de la
proposition 3.6 du chapitre 3 
qu'il existe un raffinement galoisien $(E')$ de $(E)$ tel que l'on ait à la
fois
$$(res_{j_r}(E'))=(res_{r}(E)) \quad , \quad
(rat_{j_r}(E'))=(C) \;.$$
Par conséquent, la tour $(E')$ s'écrit
$$\begin{array}{r}
(E') \qquad K=E'_0=C_0 \lhd \dots \lhd E'_{j_i}=C_{j_i}=E_i \lhd \dots \lhd
E'_j=C_j \lhd \dots \qquad\qquad\qquad\\
\dots \lhd E'_{j_r}=C_{j_r}=E_r=M(L/K) < E'_{j_r+1}=E_{r+1}=L \;.
\end{array}$$
On constate en particulier que la tour $(E')$ est induite par la tour de
composition galoisienne $(C)$ de $M(L/K) \galtou K$ :
$$(E') = ((C) \dasharrow L) \;.$$
En vertu de la proposition \ref{prop:7N}, $(E')$ est donc une tour de
composition de $L/K$.
\end{proof}

\addtocontents{toc}{\gespace\gespace\gespace}
\bibliographystyle{plain}
\bibliography{plouplou}
\listoffigures
\printindex
\tableofcontents
\end{document}